\documentclass[a4paper,11pt,oneside]{scrbook}

\usepackage[left=2.5cm,right=2.5cm,top=2.7cm,bottom=2.7cm]{geometry} 
\usepackage[utf8]{inputenc}
\usepackage[english]{babel}
\usepackage{amssymb}
\usepackage{amsmath}
\usepackage{latexsym}
\usepackage[bookmarks]{hyperref}
\usepackage{amsthm}
\usepackage{graphicx}
\usepackage{bbm}
\usepackage{verbatim}
\usepackage{epigraph}
\usepackage{mathtools}
\usepackage{authblk}
\usepackage[pdflatex]{crop}
\usepackage{tikz}
\usepackage{tikz-cd}
\usepackage{wasysym}
\usepackage{adjustbox}
\usepackage{microtype}
\usepackage{chngcntr}
\usepackage{tocbibind}
\usepackage{csquotes}
\usepackage[capitalize]{cleveref}
\usepackage{stmaryrd}	
\usetikzlibrary{shapes.geometric,fit,arrows,babel}
\usepackage{pgfplots}
\usepgfplotslibrary{polar}
\pgfplotsset{compat=newest}

\tikzset{%
    bullet/.style={
       fill=black,
       circle,
       minimum width=1pt,
       inner sep=1pt
     },
     relation/.style={
       -,
       thick,
       shorten <=2pt,
       shorten >=2pt
     },
     function/.style={
       ->,
       thick,
       shorten <=2pt,
       shorten >=2pt
     },
     every fit/.style={
       ellipse,
       draw,
       inner sep=0pt
     }
}

\makeatletter
\newcommand*{\currentname}{\@currentlabelname}
\makeatother

\usepackage[T1]{fontenc}
\usepackage{newpxtext,newpxmath}

\setkomafont{disposition}{\normalfont\bfseries}
\newcommand{\auth}[0]{Paolo Perrone}
\newcommand{\tit}[0]{Notes on Category Theory}
\newcommand{\subtit}[0]{with examples from basic mathematics}
\newcommand{\kw}[0]{Categories, Category Theory}

\definecolor{darkblue}{rgb}{0,0,0.5}
\definecolor{darkred}{rgb}{0.5,0,0}

\hypersetup{
pdfauthor={\auth},%
pdftitle={\tit},%
colorlinks, linktocpage=true, pdfstartpage=1, pdfstartview=FitV,%
breaklinks=true, pdfpagemode=UseNone, pageanchor=true, pdfpagemode=UseOutlines,%
plainpages=false, bookmarksnumbered, bookmarksopen=true, bookmarksopenlevel=1,%
hypertexnames=true, pdfhighlight=/O,%
urlcolor=darkblue, linkcolor=black, citecolor=black, 
}
\numberwithin{equation}{section}

\theoremstyle{plain}
\newtheorem{thm}{Theorem}[section]
\newtheorem{lemma}[thm]{Lemma}
\newtheorem{prop}[thm]{Proposition}
\newtheorem{cor}[thm]{Corollary}

\newtheorem{deph}[thm]{Definition}

\newtheorem*{idea}{Idea}
\theoremstyle{definition}
\newtheorem{remark}[thm]{Remark}
\newtheorem{eg}[thm]{Example}

\newtheorem*{caveat}{Caveat}
\newtheorem{ex}[thm]{Exercise}

\DeclareMathOperator{\Hom}{Hom}
\DeclareMathOperator{\End}{End}

\DeclareMathOperator{\Aut}{Aut}

\newcommand{\Z}{\mathbb{Z}}
\newcommand{\N}{\mathbb{N}}

\newcommand{\C}{\mathbb{C}}
\newcommand{\R}{\mathbb{R}}

\newcommand{\cat}[1]{{\mathbf{#1}}} 
\newcommand{\ar}[2][]{\arrow{#2}{#1}}
\newcommand{\mono}[2][]{\arrow[hookrightarrow]{#2}{#1}} 
\newcommand{\uni}[2][]{\arrow[dashrightarrow]{#2}{#1}} 
\newcommand{\nat}[2][]{\arrow[Rightarrow]{#2}{#1}} 
\newcommand{\pullback}{\arrow[dr, phantom, "\ulcorner", very near start]}
\newcommand{\pushout}{\arrow[ul, phantom, "\lrcorner", very near start]}
\newcommand{\B}{\cat{B}} 
\newcommand{\op}{\mathrm{op}} 
\DeclareMathOperator{\Cone}{Cone}
\newcommand{\Set}{\cat{Set}}
\newcommand{\Grp}{\cat{Grp}}
\newcommand{\Vect}{\cat{Vect}}

\newcommand{\e}{\varepsilon}

\DeclareMathOperator*{\colim}{colim}
\newcommand{\id}{\mathrm{id}} 
\newcommand{\ladj}{\dashv} 

\newcommand{\meet}{\wedge}

\newcommand{\join}{\vee}

\DeclareMathOperator*{\bigintersection}{\bigcap}
\newcommand{\union}{\cup}

\let\originalleft\left
\let\originalright\right
\renewcommand{\left}{\mathopen{}\mathclose\bgroup\originalleft}
\renewcommand{\right}{\aftergroup\egroup\originalright}

\title{\tit}
\subtitle{\subtit}
\author{\auth}
\affil{\small \href{http://www.paoloperrone.org}{www.paoloperrone.org}}
\date{\small Unrevised version, last updated February 2021}

\begin{document}

\begin{titlepage}
	
	\phantom{.}  
	\vspace{3cm}

	\begin{center}
	
	{\Huge \textbf{ Notes on Category Theory}\par}
	\vspace{0.3cm}
	{\Large \textbf{with examples from basic mathematics}}
	\vspace{1cm}
	
	{\LARGE Paolo Perrone\par}
	{\small \href{http://www.paoloperrone.org}{www.paoloperrone.org}}\\[1.5cm]
	
	{\small Unrevised version, last updated February 2021}
	
	\vspace{3cm}
	
		\setlength{\fboxsep}{15pt}
		\fbox{
			\begin{minipage}{0.33\textwidth}
				\centering
				\includegraphics[width=0.8\linewidth]{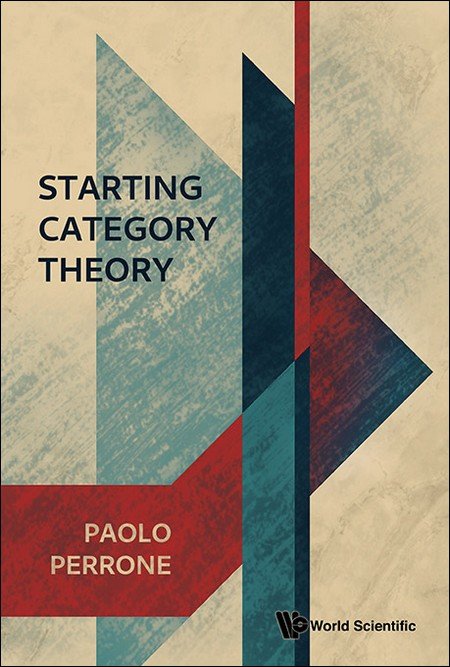}
			\end{minipage}
			\begin{minipage}{0.6\textwidth}
				\begin{center}
					\textbf{These notes have been revised and extended to a book!}
					
					Paolo Perrone, 
					\emph{Starting Category Theory},
					
					World Scientific, 2024.
					
					\href{https://doi.org/10.1142/13670}{doi.org/10.1142/13670}
					
					There is even an extra chapter on monoidal categories. Check it out!
				\end{center}
			\end{minipage}
		}
	\end{center}

\end{titlepage}

\tableofcontents


\newpage
\chapter*{About these notes}
\addcontentsline{toc}{chapter}{\currentname}

These notes were originally developed as lecture notes for a course that I taught at the Max Planck Institute of Leipzig in the Summer semester of 2019.
They are rather different from other material on category theory, partly in content, and largely in form:
\begin{itemize}
 \item The audience of the course was very varied, there were algebraic geometers and topologists, as well as computer scientists, physicists and chemists;
 \item The lecture notes were always written \emph{after} the lecture had taken place, to reflect what was discussed in class, and include all the questions, remarks, and views of the participants. 
\end{itemize}

These notes should be well-suited to anyone that wants to learn category theory from scratch and has a scientific mind.
There is no need to know advanced mathematics, nor any of the disciplines where category theory is traditionally applied, such as algebraic geometry or theoretical computer science. The only knowledge that is assumed from the reader is linear algebra.
Other assets of these notes are:
\begin{itemize}
 \item Thorough explanation of the Yoneda lemma and its significance, with both intuitive explanations, detailed proofs, and many specific examples;
 \item A treatment of monads and comonads on an equal footing, since comonads in the literature are often unjustly overlooked as ``just the dual to monads'';
 \item A detailed treatment of the relationship between categories and directed multigraphs. From the applied point of view, this shows why categorical thinking can help whenever some process is taking place on a graph. From the pure math point of view, this can be seen as the 1-dimensional first step into the theory of simplicial sets;
 \item Many examples from different areas of mathematics (group theory, graph theory, probability,\dots). Not every example is helpful for every reader, but hopefully every reader can find at least one helpful example per concept.  The reader is encouraged to read all the examples, this way they may even learn something new about a different field.
\end{itemize}

One thing that is unfortunately not contained in these notes, since they are already quite long, is string diagrams. For those, we refer the interested reader for example to \cite{marsden-string} and \cite{sevensketches}.

While the tone of these notes is informal, the mathematics is rigorous, and every theorem is proven. 
I make no claim of scientific originality. Some of the technical content, and some of the examples, are taken from Emily Riehl's book \cite{ctcontext}, which was the recommended textbook for the ``pure math'' part of the course.
Another considerable part, such as the adjoint functor theorem for preorders, is taken from Brendan Fong and David Spivak's book \cite{sevensketches}, which was the recommended textbook for the ``applied'' part of the course.
Finally, many ideas about how to present certain mathematical concepts are novel, and are to be credited not only to me, but also to the participants of the course.
(Part of the treatment of monads already appeared in the introduction of my PhD thesis \cite{thesis}.)

Since these are informal lecture notes, most results are presented without citing the original work in which they were discovered. If however you find that you should be cited for a result used here, or that some citation here is incorrect or incomplete, please contact me.

Just as well, feel free to contact me if you find any mistakes, or if you find that something is not clearly written. (Thanks to all the readers who reported such mistakes already.)

\begin{flushright}
 Paolo Perrone.
\end{flushright}

\section*{Acknowledgements}
\addcontentsline{toc}{section}{\currentname}

There are many people who helped me developing these notes, helped me understanding the ideas behind these notes, and helped me finding ways to explain these ideas to new learners. 
\begin{itemize}
 \item First of all, I want to thank all the participants of the course, without whom these notes wouldn't exist. I want to thank in particular Emma Chollet, Wilmer Leal, Guillermo Restrepo and Sharwin Rezagholi, who came up with many original ideas during the lectures, now recorded in these notes.
 \item I also want to thank Carmen Constantin, Jules Hedges, Jürgen Jost, Slava Matveev, David Spivak, and more recently Walter Tholen, for the interesting discussions, some of which were reflected in the way I taught this course and wrote these notes.
 \item Finally I want to thank Tobias Fritz, who had the patience to teach category theory to me. 
\end{itemize}

\section*{Notation and conventions}
\addcontentsline{toc}{section}{\currentname}

We will follow for the most part the notation and conventions of \cite{ctcontext}. 
The typesetting is slightly different. 
\begin{itemize}
 \item We denote categories by boldface capitalized words, such as $\cat{C}$ and $\cat{Set}$.
 \item We denote functors by uppercase letters, such as $F:\cat{C}\to\cat{Set}$. (We compose functors on the left.)
 \item We denote objects of a category by uppercase letters such as $X,Y,A,B$.
 \item We denote morphisms of a category by lowercase letters, such as $f:A\to B$. (We compose morphisms on the left.)
 \item We denote natural transformations by lowercase Greek letters, such as $\alpha:F\Rightarrow G$. 
\end{itemize}
If you don't know what any of the above are, don't worry, everything will be introduced shortly.

\section*{About citing this work}
\addcontentsline{toc}{section}{\currentname}
If you want to cite these notes, you can also cite the book \cite{startingcats}, which is a published reference. 
Most of the numbers for definitions, theorems, etc.\ in the book are the same as here, except for the following:
\begin{itemize}
	\item \Cref{1.2.10};
	\item \Cref{1.3.30};
	\item Exercises \ref{1.3.40} and \ref{1.3.41} (in the book there are definitions instead);
	\item In the book there are extra exercises at the end of \Cref{sec_nat};
	\item \Cref{1.5.20};
	\item Examples \ref{grptoset} and \ref{2.1.9}
	\item \Cref{2.3.9};
	\item Most of \Cref{sec_tens};
	\item In the book there is a new section, 3.2.7, on sequential and filtered limits;
	\item Most of \Cref{sec_fun_lim,proofrepfulcomplete,sec_adj_lim};
	\item In \Cref{adjunctions,secmonads}, \emph{sharp} and \emph{flat} are switched;
	\item In the book there is a new exercise, 4.2.31;
	\item \Cref{5.4.19};
	\item \Cref{5.5.4} (in the book there's a lemma instead);
	\item In the book there is a new chapter, 6, on monoidal categories.
\end{itemize}
Of course, the page numbers are different as well.

In general, I recommend to use the \emph{book} as a reference, since there have been quite some revisions and corrections compared to the material here.

\vspace{1cm}

We are now ready to get started.

\newpage
\chapter{Basic concepts}

Whenever we introduce a new concept we will first write the rigorous definition, and right after give an intuitive interpretation.

\section{Categories}

\begin{deph}\label{defcat}
 A \emph{category} $\cat{C}$ consists of
 \begin{itemize}
  \item A collection $\cat{C}_0$, whose elements are called the \emph{objects} of $\cat{C}$ and are usually denoted by uppercase letters $X$, $Y$, $Z$,\dots
  \item A collection $\cat{C}_1$, whose elements are called the \emph{morphisms} or \emph{arrows} of $\cat{C}$ and are usually denoted by lowercase letters $f$, $g$, $h$,\dots
 \end{itemize}
 such that
 \begin{itemize}
  \item Each morphism has assigned two objects, called \emph{source} and \emph{target}, or \emph{domain} and \emph{codomain}. We denote the source and target of the morphism $f$ by $s(f)$ and $t(f)$, respectively. If the morphism $f$ has source $X$ and target $Y$, we also write $f:X\to Y$, or more graphically, $X\xrightarrow{f} Y$. 
  \item Each object $X$ has a distinguished morphism $\id_X:X\to X$, called \emph{identity morphism}.
  \item For each pair of morphisms $f,g$ such that $t(f)=s(g)$ there exists a specified morphism $g\circ f$, called the \emph{composite morphism}, such that $s(g\circ f) = s(f)$ and $t(g\circ f) = t(g)$. More graphically: 
  $$
  f:X\to Y \,,\; g:Y\to Z \qquad \rightsquigarrow \qquad g\circ f: X\to Z.
  $$
 \end{itemize}
 These structures need to satisfy the following axioms.
 \begin{itemize}
  \item Unitality: for every morphism $f:X\to Y$, the compositions $f\circ \id_X$ and $\id_Y\circ f$ are both equal to $f$.
  \item Associativity: for $f:X\to Y$, $g:Y\to Z$, and $h:Z\to W$, the compositions $h\circ(g\circ f)$ and $(h\circ g)\circ f$ are equal. 
 \end{itemize}
\end{deph}

A category is a very general structure, the objects and morphisms can be anything, provided that they satisfy the properties given above. Below we will give some examples of categories arising from standard mathematical practice. 

\begin{caveat}
 The composite $g\circ f$ means \emph{first} applying $f$, then $g$:
 $$
\begin{tikzcd}
 X \ar{r}{f} \ar[bend right,swap]{rr}{g\circ f} & Y \ar{r}{g} & Z .
\end{tikzcd}
$$
This choice of order, which is unfortunately incompatible with diagrams, comes from traditional mathematics: in expressions such as 
$$
\sin(\cos(x))
$$
one must first take the cosine, and then the sine. A way to avoid confusion is to read $g\circ f$ as ``$g$ \emph{after} $f$''. 
\end{caveat}

Here is a depiction of associativity: given three composable morphisms
$$
\begin{tikzcd}
 X \ar{r}{f} & Y \ar{r}{g} & Z \ar{r}{h} & W
\end{tikzcd}
$$
we have a priori two different ways to compose them. We could first compose $f$ and $g$,
$$
\begin{tikzcd}
 X \ar{r}{f} \ar[bend right,swap]{rr}{g\circ f} & Y \ar{r}{g} & Z \ar{r}{h} & W
\end{tikzcd}
$$
and then compose the resulting $g\circ f$ with $h$, to obtain
$$
\begin{tikzcd}
 X \ar{r}{f} \ar[bend right]{rr}{g\circ f} \ar[bend right,swap]{rrr}{h\circ(g\circ f)} & Y \ar{r}{g} & Z . \ar{r}{h} & W
\end{tikzcd}
$$
Alternatively, we could first compose $g$ with $h$, 
$$
\begin{tikzcd}
 X \ar{r}{f} & Y \ar{r}{g} \ar[bend right,swap]{rr}{h\circ g} & Z \ar{r}{h} & W
\end{tikzcd}
$$
and then compose $f$ with the resulting $h\circ g$, to obtain 
$$
\begin{tikzcd}
 X \ar{r}{f} \ar[bend right,swap]{rrr}{(h\circ g)\circ f} & Y \ar{r}{g} \ar[bend right]{rr}{h\circ g} & Z \ar{r}{h} & W
\end{tikzcd}
$$
The associativity condition is then saying that the two methods give the same result, that is, $h\circ(g\circ f)=(h\circ g)\circ f$. 
We can therefore write the triple composition without brackets, $h\circ g\circ f$, without any ambiguity. 

A first intuitive way to think about a category is the following.

\subsection{Categories as relations}

Here is the first intuitive idea of what a category is.

\begin{idea} 
 A category is a collection of objects which are related to each other in a consistent way. 
\end{idea}

Let's see what this means by looking at examples.

\begin{eg}[sets and relations]\label{egeq}
 An \emph{equivalence relation} on a set $X$ is a relation $\sim$ satisfying the following properties.
 \begin{enumerate}
  \item Reflexivity: for each $x\in X$ we have $x\sim x$;
  \item Transitivity: for each $x,y,z\in X$, if $x\sim y$ and $y\sim z$, then $x\sim z$;
  \item Symmetry: for each $x,y\in X$, if $x\sim y$ then $y\sim x$ as well.
 \end{enumerate}
 An equivalence relation is a mathematical way of formalizing the idea of ``having something in common'', such as ``having the same shape''. For example, here is what we obtain if, in a set of shapes, we draw an arrow whenever we want to say ``has the same shape as'' (regardless of color and size):
 $$
 \begin{tikzcd}[column sep=small]
  & \circ \ar[out=60,in=120,loop,distance=1cm] \ar[shift left]{ddr} \ar[shift left]{ddl} \\
  \\
  \bigcirc \ar[out=180,in=240,loop,distance=1cm] \ar[shift left]{uur} \ar[shift left]{rr} && \bullet \ar[out=300,in=0,loop,distance=1cm] \ar[shift left]{ll} \ar[shift left]{uul}
 \end{tikzcd}
 \qquad
 \begin{tikzcd}
  \blacksquare \ar[out=150,in=210,loop,distance=1cm] \ar[shift left]{r} & \Box \ar[out=330,in=30,loop,distance=1cm] \ar[shift left]{l} 
  \end{tikzcd}
  \qquad
  \begin{tikzcd}
  \triangle \ar[out=330,in=30,loop,distance=1cm]
  \end{tikzcd}
 $$
 We note that:
 \begin{enumerate}
  \item $x$ has the same shape as itself;
  \item If $x$ has the same shape as $y$ and $y$ has the same shape as $z$, then $x$ has the same shape as $z$;
  \item If $x$ has the same shape as $y$, then $y$ has the same shape as $x$. 
 \end{enumerate}
 An equivalence relation defines a category in the following way. It intuitively looks like the picture above.
 \begin{itemize}
  \item The objects are the elements of $X$;
  \item Given $x,y\in X$, there exists a unique morphism $x\to y$ if and only if $x\sim y$. 
 \end{itemize}
 The identity at an object $x$ is the unique arrow $x\to x$ given by $x\sim x$. We know there is one, by reflexivity, and by definition there is only one, so we don't have to worry about \emph{which} arrow $x\to x$ is the distinguished one. Just as well, the composition is given by transitivity: if we have arrows $x\to y$ and $y\to z$ it means that $x\sim y$ and $y\sim z$, and so by transitivity $x\sim z$ as well, which means that there is an arrow $x\to z$. Again, we don't have to worry about which arrow $x\to z$ we want, since there is only one.
 The symmetry property says that if there is an arrow $x\to y$, then there is an arrow $y\to x$ as well. In other words, ``we can always go back''. This last property is not important to have a category, it is merely an extra property that this category will have. 
\end{eg}

\begin{eg}[sets and relations]\label{egorder}
 An \emph{order relation} on a set $X$ is a relation $\le$ satisfying the following axioms.
 \begin{enumerate}
  \item Reflexivity: for each $x\in X$ we have $x\le x$;
  \item Transitivity: for each $x,y,z\in X$, if $x\le y$ and $y\le z$, then $x\le z$;
  \item Antisymmetry: for each $x,y\in X$, if $x\le y$ and $y\le x$, then necessarily $x=y$.
 \end{enumerate}
 An order relation is a mathematical way of formalizing the idea of ``comparing sizes'', such as ``being smaller or equal''. For example:
 \begin{enumerate}
  \item $x$ is smaller or equal to itself;
  \item If $x$ is smaller or equal to $y$ and $y$ is smaller or equal to $z$, then $x$ is smaller or equal to $z$;
  \item If $x$ is smaller or equal to $y$ and $y$ is smaller or equal to $x$, then $x=y$. 
 \end{enumerate}
 An order relation defines a category in an analogous way to what happens with equivalence relations: objects are elements of $X$, and there is a unique arrow $x\to y$ if and only if $x\le y$. Again, reflexivity and transitivity are enough to have a category structure, antisymmetry is an additional property. 
\end{eg}

The most general relation which gives rise to a category is called a preorder, and it has no conditions on symmetry.

\begin{eg}[sets and relations, economics]\label{egpreorder}
 A \emph{preorder relation} on a set $X$ is a relation $\lesssim$ satisfying only the axioms:
 \begin{enumerate}
  \item Reflexivity: for each $x\in X$ we have $x\lesssim x$;
  \item Transitivity: for each $x,y,z\in X$, if $x\lesssim y$ and $y\lesssim z$, then $x\lesssim z$.
 \end{enumerate}
 Again, a preorder defines a category where there is a unique morphism between $x$ and $y$ if and only if $x\lesssim y$, with identities and composition given by reflexivity and transitivity.
 Equivalence relations and order relations are both special cases of preorder relations. 
 A non-trivial example of a preorder is the \emph{price}, in economics, i.e.~the relation of being ``cheaper or equally priced''. For example:
 \begin{enumerate}
  \item $x$ has lower or equal price to itself;
  \item If $x$ has lower or equal price to $y$ and $y$ has lower or equal price to $z$, then $x$ has lower or equal price to $z$.
 \end{enumerate}
 However, differently from a partial order, if $x$ has lower or equal price to $y$ and $y$ has lower or equal price to $x$ we don't have that necessarily $x=y$: they will simply have the same price. They are not the same object, but they are equivalent: you can trade one for the other one and vice versa. 
\end{eg}

\begin{ex}[sets and relations]\label{corepreorder}
 Let $(X,\lesssim)$ be a preorder. Prove that the relation $\sim$ given by 
 $$
 x\sim y \; \mbox{ if and only if } \; x\lesssim y \mbox{ and } y \lesssim x
 $$
 is an equivalence relation. 
\end{ex}

For all these examples, given two objects, there is at most one morphism between them. This is not the case in general, and sometimes the \emph{choice} of the morphism matters. 

\subsection{Categories as operations}\label{ssscatop}

Here is another way to think about categories. 

\begin{idea}
 A category is a collection of operations which can be composed in a consistent way.
\end{idea}

Here are some examples of categories where this interpretation is most helpful. We encourage the readers unfamiliar with group theory \emph{not} to skip the examples below: once one gets the basic idea of a group, many concepts in category theory are much easier to understand.

\begin{eg}[group theory]\label{eggroup}
 A \emph{group} is a nonempty set $G$ together with
 \begin{enumerate}
  \item A distinguished element $1\in G$ called the \emph{neutral element} or \emph{unit};
  \item A binary operation $G\times G\to G$ called \emph{multiplication}, which we denote by $(g,h)\mapsto g\cdot h$;
  \item For each element $g\in G$ an element $g^{-1}$ called the \emph{inverse};
 \end{enumerate}
 such that the following properties hold.
 \begin{enumerate}
  \item Unitality: for each $g\in G$, the multiplications $g\cdot 1$ and $1\cdot g$ are both equal to $g$;
  \item Associativity: for each $g,h,i\in G$, the multiplications $(g\cdot h)\cdot i$ and $g\cdot(h\cdot i)$ are equal;
  \item Inverse law: for each $g\in G$, the multiplications $g\cdot g^{-1}$ and $g^{-1}\cdot g$ are both equal to $1$.
 \end{enumerate}
\end{eg}

A group in mathematics is used to model the \emph{symmetries} of some structures, namely the ways in which we can act on some object while keeping it ``the same''. For example, rotations of the plane form a group, and also permutations of a set. 
A group is similar to a category: there is the neutral element $1$, which behaves like an identity, and there is the multiplication, which behaves like composition (it is even associative). How can we turn then a group into a category?
An important feature of a group $G$ is that \emph{we can compose any two elements of $G$}. There is no requirement of ``matching source and target'', as there is for a category. Therefore, in order to view a group as a particular category, we should make sure that the source and target of all morphisms always match. The only way to do this in general is to have \emph{a unique object}. Here is the construction in detail.

\begin{deph}
 Let $G$ be a group. The \emph{delooping} of $G$, denoted by $\B G$, is the following category.
 \begin{itemize}
  \item There is a single object $\bullet$;
  \item There is a morphism $\bullet \to \bullet$ for each element $g\in G$. We denote the morphism with the same letter, such as $g:\bullet \to \bullet$.
  \item The identity of the object $\bullet$ is the morphism given by $1\in G$;
  \item The composition is given by the multiplication of $G$. That is, the composition $g\circ h$ of the morphisms given by $g,h\in G$ is the morphism given by the element $g\cdot h$ of $G$. 
 \end{itemize} 
 Since associativity and unitality hold for elements of $G$, they also hold for the morphisms of $\B G$. Therefore, $\B G$ is a category.
\end{deph}

Graphically, $\B G$ looks like a point with loops, with one special loop (the identity):
$$
\begin{tikzcd}[row sep=huge]
  \bullet \ar[out=60,in=120,loop,"1",swap,distance=1.5cm] \ar[out=132,in=192,loop,"",swap,distance=1.5cm] \ar[out=204,in=264,loop,"",swap,distance=1.5cm] \ar[out=276,in=336,loop,"",swap,distance=1.5cm] \ar[out=348,in=408,loop,"",swap,distance=1.5cm]  
 \end{tikzcd}
$$

The term ``delooping'' comes from algebraic topology. There is a deep connection between algebraic topology and category theory, which is widely explored on the \href{http://ncatlab.org}{nLab}, and the delooping is part of this connection. For us, however, it will be merely a way of obtaining a category from a group (or monoid, see below). 

\begin{remark}
 The category $\B G$ contains the same information as the group $G$, it is just a different way to express it. This is why some authors prefer omitting the symbol $\B$ and simply say that ``a group is a certain category with a single object''. In this course we will keep the two notions distinct. Remember, however, that both encode the same data. 
\end{remark}

\begin{remark}
Since a group is usually modeling the ``symmetries of something'', sometimes we can make this delooping construction more concrete. For example, for the group of rotations of $\R^2$, we can take the object $\bullet$ to be $\R^2$ and the arrows to be rotations $\R^2\to\R^2$ (as linear maps). This will be made precise in the future. However, for now, the construction is purely formal.
\end{remark}

In order to form the category $\B G$ from the group $G$ we never used the inverse elements and the inverse law. These, just as the symmetry property of equivalence relations, are not needed to have a category, and merely give to $\B G$ an additional structure, namely an ``inverse'' morphism $g^{-1}$ to each morphism $g$ (more on this in \ref{iso}). Therefore, if we drop the inverse requirement, we can still get a category.

\begin{eg}
 A \emph{monoid} is a nonempty set $M$ together with
 \begin{enumerate}
  \item A distinguished element $1\in M$ called the \emph{neutral element} or \emph{unit};
  \item A binary operation $M\times M\to M$ called \emph{multiplication}, which we denote by $(g,h)\mapsto g\cdot h$;
 \end{enumerate}
 such that the following properties hold.
 \begin{enumerate}
  \item Unitality: for each $g\in M$, the multiplications $g\cdot 1$ and $1\cdot g$ are both equal to $g$;
  \item Associativity: for each $g,h,i\in M$, the multiplications $(g\cdot h)\cdot i$ and $g\cdot(h\cdot i)$ are equal.
 \end{enumerate}
 
 The \emph{delooping} of a monoid $M$ is the category $\B M$ with, just as for groups, a single object $\bullet$, and as morphisms the elements of $M$ with their identity and composition. 
\end{eg}

Just as a group is used to model the symmetries of some object, a monoid is used to model the transformations of an object which are not necessarily invertible. For example consider, instead of only the rotations of the plane $\R^2$, the set of \emph{all linear maps} $\R^2\to\R^2$, including the one that maps everything to $(0,0)$. These maps include the identity and can be composed, so they form a monoid. However, not all of them are invertible, so they don't form a group. There are many examples similar to this:
\begin{itemize}
 \item Given a vector space $V$, all linear maps $V\to V$ form a monoid;
 \item Given a topological space $X$, all continuous maps $X\to X$ form a monoid;
 \item Given a set $X$, all functions $X\to X$ form a monoid;
 \item ...and so on.
\end{itemize}

Here is another couple of examples from probability theory:
\begin{itemize}
 \item Given a measurable space $X$ and a Markov kernel $k:X\to X$, the set given by the identity $\id_X$ and the repeated applications of $k$, that is, the set 
 $$
 \{\id_X,\,k,\,k^2,\,k^3,\,\dots\}
 $$
 is a monoid. Sometimes this is called the \emph{Markov semigroup} generated by $k$.
 \item More generally, given a measurable space $X$, all Markov kernels $X\to X$ form a monoid. Sometimes this is called the \emph{full Markov semigroup on $X$}.
\end{itemize}

\begin{remark}\label{semigroup}
 In fields such as probability theory, people sometimes prefer the notion of ``semigroup'' to the one of monoid. A semigroup is like a monoid, but without the unit (and the unitality property). In category theory, by convention, we always use monoids. The difference is minimal, one can almost always add the unit, since it corresponds to ``doing nothing'', and it is largely only a matter of convention. (Requiring the identity is analogous to admitting the zero in the set of natural numbers.) 
\end{remark}

We conclude this part with an exercise.

\begin{ex}[important!]
 Prove that, in a group or in a monoid, the neutral element is unique. In other words, if both $1$ and $1'$ satisfy the properties of the neutral element, then $1=1'$. 
\end{ex}

\subsection{Categories as spaces and maps with extra structure}

In the case of the categories defined by equivalence relations, orders, and preorders, we didn't need to worry about the equality of composite arrows (for example to check associativity), since between any two objects, there is at most one arrow. In categories arising as delooping of groups or monoids, we didn't need to worry about matching source and target of arrows, since there is only a single object. In the most general case, one needs always to check both. Here is another idea, which is helpful to study the categories of this more general kind.

\begin{idea}
 A category is a collection of sets or spaces equipped with extra structure, and maps between them which are compatible with that structure.
\end{idea}

Categories of this form are usually named after their objects. 

\begin{eg}[sets and relations]
 The category $\Set$ is the category whose objects are sets, and whose morphisms are maps (functions) between them. 
\end{eg}
For each set $X$, the identity function is a function $X\to X$. Functions can be composed, and the composition is associative. Therefore $\Set$ is a category. 

\begin{eg}[several fields]
 Analogously to the category $\Set$, the following are categories.
 \begin{itemize}
  \item $\cat{Top}$ has topological spaces as objects and continuous maps as morphisms.
  \item $\cat{Mfd}$ has smooth manifolds as objects and smooth maps as morphisms.
  \item $\Vect$ has vector spaces as objects and linear maps as morphisms.
  \item $\cat{Meas}$ has as objects measurable spaces, and as morphisms measurable maps.
  \item $\Grp$ has groups as objects and group homomorphisms as morphisms, i.e.~maps $f:G\to H$ such that for each $g,g'\in G$, we have that $f(g\cdot g')=f(g)\cdot f(g')$.
  \item $\cat{Poset}$ has partially ordered sets as objects, and monotone maps as morphisms, i.e.~maps $f:X\to Y$ such that for each $x,x'\in X$ with $x\le x'$ in $X$, we have that $f(x)\le f(x')$ in $Y$.
 \end{itemize}
\end{eg}

A category with sets as objects may have many choices of morphisms, for example, instead of functions, we may choose injective functions, or bijections, or even relations. The choice of morphisms reflects the choice of \emph{context} which we want to consider, or the choice of \emph{structure} that we want to preserve (and study). 

\begin{eg}[graph theory]
For graphs there are many choices of morphisms, depending on what one wants to do with graphs (graph theory is so general and versatile that many choices are meaningful). For example, for undirected, unweighted graphs, a possible choice of morphisms $f:G\to H$ is functions between the sets of vertices which preserve the adjacency relation: if the vertices $x$ and $y$ are connected by an edge in $G$, then $f(x)$ and $f(y)$ must be connected by an edge in $H$. With this choice of morphisms, graphs and these morphisms form a category.
This is however only one of many possible choices.
\end{eg}

\begin{ex}
 In your field of mathematics (or physics, chemistry, computer science, economics,\dots) which structures do you work with the most? Can you construct a category with such structures? Are those structures better represented as objects or as morphisms of a category?
\end{ex}

Here is also an example of what is \emph{not} a category, because composition fails.

\begin{eg}[calculus]
 Consider the following convex functions on $\R$:
 \begin{itemize}
  \item $x\mapsto f(x) \coloneqq  x^2$;
  \item $x\mapsto g(x) \coloneqq  x^2 - 1$.
 \end{itemize}
 Both are even strictly convex (their graphs are parabolas, see the picture below). However, the composite $f\circ g:\R\to\R$ is not convex: we have
 $$
 f(g(x)) = (x^2 - 1)^2,
 $$
 which is equal to $0$ for $x=\pm 1$, and equal to $1$ for $x=0$. 
\end{eg}

\begin{center}
\begin{tikzpicture}
      \draw[->] (-2.5,0) -- (2.5,0) node[above left] {$x$};
      \draw[->] (0,-2.5) -- (0,2.5) node[below right] {$y$};
      \draw[scale=1,domain=-1.5:1.5,smooth,variable=\x,darkblue,thick] plot ({\x},{\x*\x)}) node[below right] {$f$};
\end{tikzpicture}
$\;$
\begin{tikzpicture}
      \draw[->] (-2.5,0) -- (2.5,0) node[above left] {$x$};
      \draw[->] (0,-2.5) -- (0,2.5) node[below right] {$y$};
      \draw[scale=1,domain=-1.8:1.8,smooth,variable=\x,darkblue,thick] plot ({\x},{\x*\x-1}) node[below right] {$g$};
\end{tikzpicture}
$\;$
\begin{tikzpicture}
      \draw[->] (-2.5,0) -- (2.5,0) node[above left] {$x$};
      \draw[->] (0,-2.5) -- (0,2.5) node[below right] {$y$};
      \draw[scale=1,domain=-1.58:1.58,smooth,variable=\x,darkblue,thick] plot ({\x},{(\x*\x-1)*(\x*\x-1)}) node[below right] {$f\circ g$};
\end{tikzpicture}
\end{center}

Therefore, convex functions on $\R$ do not form a category. The same is true for lower semicontinuous maps.

\begin{ex}[analysis]
 Show that there exists two lower semicontinuous maps $f,g:\R\to\R$ such that $f\circ g$ is not lower semicontinuous.
\end{ex}

Therefore, whenever you define a category, check that all the axioms hold. Not everything is automatically a category.

\begin{ex}[analysis]
 Can we give an additional requirement to convex maps, in such a way to make them closed under composition? Will they form a category that way? What about lower semicontinuous maps?
\end{ex}

\subsection{Set-theoretical considerations}

The content of this part is a bit technical, it will not be really needed to understand the rest of the course. It should nevertheless be mentioned, in order to have a rigorous treatment and to avoid possible confusion. Moreover, the notation given in \Cref{defhomset} will be used in the future.

You may know from your set theory knowledge that there is no such thing as the ``set of all sets'' (the reason being: does that set contain itself?). Therefore, if we want a ``category of all sets'', the objects of this category cannot form a set. That's why, in the definition of a category (\Cref{defcat}) we said \emph{collections} rather than sets: the objects and morphisms may in general be proper classes. 

\begin{deph}
 A category $\cat{C}$ is called \emph{small} if $\cat{C}_0$ and $\cat{C}_1$ are sets. 
 
 A category $\cat{C}$ is called \emph{locally small} if for every two objects $X$ and $Y$ of $\cat{C}$, the morphisms $X\to Y$ form a set. 
\end{deph}

Alternatively, one may speak of universes and \emph{small and large sets} instead of sets and proper classes. In this formalism, a category is small if $\cat{C}_0$ and $\cat{C}_1$ are small sets, and locally small categories are defined analogously.

Most categories of interest are locally small. In particular, all the examples of categories analyzed so far are locally small. 
By construction, moreover, the categories obtained from relations on a set or as delooping of a monoid are small. You can prove converses of this last statement by exercise:

\begin{ex}[sets and relations]
 Prove that a small category with at most a single arrow between any two objects is a preorder. (Hint: there is not much to prove.)
\end{ex}
Just as well:
\begin{ex}
 Prove that a locally small category with a single object is a monoid. 
\end{ex}

We conclude this section with a useful piece of notation, which we will use in the rest of this course.

\begin{deph}\label{defhomset}
 Let $X$ and $Y$ be objects of a locally small category $\cat{C}$. The \emph{hom-set} or \emph{hom-space} of $X$ and $Y$ is the set of morphisms of $\cat{C}$ from $X$ to $Y$. We will denote it by $\Hom_\cat{C}(X,Y)$. 
\end{deph}

From now on we will mostly work with locally small categories. 

\begin{caveat}
 Sometimes a mathematical structure can be an object of different categories. For example $\R$ and $\R^2$ are sets, topological spaces, vector spaces, and so on. This is why it is important to keep track of the category when we write the hom-sets. For example,
 \begin{itemize}
  \item $\Hom_\Set(\R,\R^2)$ is the set of \emph{all} functions $\R\to\R^2$, i.e.~the morphisms of $\Set$;
  \item $\Hom_\cat{Top}(\R,\R^2)$ is the set of all \emph{continuous} functions $\R\to\R^2$, i.e.~the morphisms of $\cat{Top}$;
  \item $\Hom_\Vect(\R,\R^2)$ is the set of all \emph{linear} functions $\R\to\R^2$;
  \item \dots and so on.
 \end{itemize}
\end{caveat}

If you want to know more about the set-theoretical issues in category theory, you can read these notes by Mike Shulman, \cite{shulman-sets}. However I recommend you read them after learning a bit more about category theory.

\subsection{Isomorphisms and groupoids}\label{iso}

For the case of equivalence relations, and for delooping of groups, we had categories with a special property, which intuitively said that ``we can always go back''. Let's now try to make this intuition precise.

\begin{deph}\label{defiso}
 Let $X$ and $Y$ be objects in a category $\cat{C}$. An \emph{isomorphism} is a pair of morphisms 
 $$
 \begin{tikzcd}
  X \ar[shift left]{r}{f} & Y \ar[shift left]{l}{g}
 \end{tikzcd}
 $$
 such that 
 \begin{enumerate}
  \item $g\circ f = \id_X$, and
  \item $f\circ g = \id_Y$. 
 \end{enumerate}
 If there exists an isomorphism between $X$ and $Y$, we say that $X$ and $Y$ are \emph{isomorphic}.
\end{deph}

We see in this definition that $g$ \emph{not only} is in the opposite direction as $f$, but also ``undoes $f$'', in the sense that applying $g$ after $f$ is like doing nothing at all. And just as well, $f$ undoes $g$.

\begin{caveat}
 The two conditions do not imply one another, as the following example shows. Therefore, in order to have an isomorphism, both conditions need to be satisfied separately.
\end{caveat}

\begin{eg}[calculus, linear algebra]
 Let $X=\R^2$ and $Y=\R$. Let $f:\R^2\to \R$ be the projection onto the $x$-axis $(x,y)\mapsto x$, and let $g:\R\to \R^2$ be the inclusion of the $x$-axis $x\mapsto (x,0)$. Then $f\circ g = \id_{\R}$, but $g\circ f \ne \id_{\R^2}$: in particular, $g\circ f(x,y) = g(x) = (x,0)$, which is in general different from $(x,y)$. 
\end{eg}

\begin{ex}[important!]
 Let $f:X\to Y$ be a morphism in a category. Suppose that both $g$ and $g':Y\to X$ are inverse to $f$. Prove that $g=g'$. Conclude that, in particular, inverses in a group are unique. 
\end{ex}

Here are some examples of isomorphisms.

\begin{eg}[several fields]
 \begin{itemize}
  \item In the category $\Set$, the isomorphisms are the bijective functions.
  \item In $\cat{Top}$, the isomorphisms are the homeomorphisms: invertible continuous maps whose inverse is continuous. Mind: this is strictly stronger than saying continuous and bijective (can you give a counterexample?).
  \item In $\cat{Mfd}$, the isomorphisms are the diffeomorphisms.
  \item In $\Vect$, the morphisms are bijective linear maps.
 \end{itemize}
\end{eg}

\begin{ex}[important!]
 Prove that the composition of invertible morphisms is invertible.
\end{ex}

Let's now generalize equivalence relations and groups, i.e.~categories where ``we can always go back''.

\begin{deph}
 A \emph{groupoid} is a category where all the morphisms are invertible.
\end{deph}

Here are some examples of groupoids. 

\begin{eg}[sets and relations]
 We have already seen equivalence relations. For an equivalence relation $\sim$ on a set $X$, we have that if $x\sim y$, i.e.~we have an arrow $x\to y$, then $y\sim x$, i.e.~we have an arrow $y\to x$. Checking that the two arrows invert each other is trivial, guaranteed by uniqueness. In detail, the composition $x\to y \to x$ has to be equal to the identity $x\to x$, because there is at most one arrow $x\to x$.
\end{eg}

\begin{eg}[group theory]
 The delooping $\B G$ of a group $G$ is a groupoid (this is why it's called ``groupoid'': it generalizes a group). The inverses are given by the inverses in $G$.
\end{eg}

Another example of a groupoid is the groupoid of sets and \emph{bijections}, since those are exactly the invertible morphisms. We can do this more generally. 

\begin{deph}\label{core}
 Let $\cat{C}$ be a category. The \emph{core} of $\cat{C}$ is the groupoid whose objects are the objects of $C$ and whose morphisms are the isomorphisms of $C$. 
\end{deph}

We have seen another example of a core already: in \Cref{corepreorder}, we were canonically obtaining an equivalence relation from a preorder. Therefore, the core of a preorder is an equivalence relation.

We conclude this section with a ``philosophical'' remark. In category theory, one usually never talks about two objects being \emph{equal}, but only \emph{isomorphic}. In some sense, this reflects standard mathematical practice. For example, as vector spaces, every three-dimensional space is isomorphic to $\R^3$. But are they really \emph{equal}? What would equality even mean in that context? What can be done in $\R^3$ can be done in any three-dimensional vector space, and that's what matters in the end. Therefore, isomorphism is the important notion between objects. 
When it comes to \emph{morphisms}, on the other hand, two morphisms are often required to be \emph{equal}. For example, in the definition of an isomorphism, $f^{-1}\circ f$ is \emph{equal} to the identity $\id_X$, not ``isomorphic''.
Again, this reflects standard mathematical practice: functions between spaces can be equal. For example, two functions $f,g:X\to Y$ are equal if and only if for every $x\in X$ we have $f(x)=g(x)$. 

One may be tempted to ask whether it is possible to also talk about \emph{isomorphisms of maps} instead of equalities, so that objects and morphisms are treated more on an equal footing. This is exactly the subject of \emph{higher category theory}, in which one can talk about morphisms between morphisms, and so on. Higher category theory is very helpful, for example, in algebraic geometry and algebraic topology. The \href{http://ncatlab.org}{nLab} is the standard online wiki on this subject. 
In this course we will for the most part do \emph{ordinary} category theory, where morphisms are allowed to be equal.

\subsection{Diagrams, informally}\label{infdiagrams}

One of the most powerful methods of category theory is \emph{reasoning in terms of diagrams}. 
We start with an informal (but consistent) definition of a diagram. A more formal definition will be given later on in terms of functors.\footnote{See \Cref{functordiagrams}.}

\begin{deph}[informal]
 A \emph{diagram} in a category $\cat{C}$ is a directed (multi-)graph formed out of objects and arrows of $\cat{C}$ such that:
 \begin{itemize}
  \item Each object and morphism may appear more than once in the diagram;
  \item Between any two objects there may be also more than one morphism;
  \item For each object $X$ in the diagram, the identity is implicitly present in the diagram (but generally not drawn);
  \item For each composable edges (arrows) $f:X\to Y$ and $g:Y\to Z$ which are appearing head-to-tail in the diagram, the composite $g\circ f:X\to Z$ is implicitly present in the diagram (but generally not drawn).
 \end{itemize}
\end{deph}

Since objects and morphisms may appear more than once in the diagram, as different vertices and edges, we will refer to vertices and edges to avoid ambiguity. 

\begin{ex}\label{exdiags}
 Which morphisms are contained in the following diagrams, implicitly or explicitly? (Without assuming any additional equalities other than those required by our axioms.)
 \begin{enumerate}
  \item $$
   \begin{tikzcd}
    X \ar{r}{f} & Y
   \end{tikzcd}
  $$
  \item 
  $$
   \begin{tikzcd}
    X \ar{r}{f} \ar{dr}[swap]{g\circ f} & Y \ar{d}{g} \\
    & Z
   \end{tikzcd}
  $$
  \item 
  $$
   \begin{tikzcd}
    X \ar{r}{f} \ar{dr}[swap]{h} & Y \ar{d}{g} \\
    & Z
   \end{tikzcd}
  $$
  \item 
  $$
   \begin{tikzcd}
    X \ar{r}{f} \ar{dr}[swap]{\id_X} & Y \ar{d}{g} \\
    & X
   \end{tikzcd}
  $$
  \item 
  $$
   \begin{tikzcd}
    X \ar[shift left]{r}{f} & Y \ar[shift left]{l}{g} 
   \end{tikzcd}
  $$
  \item 
  $$
   \begin{tikzcd}
    X \ar[shift left]{r}{f} & Y \ar[shift left]{l}{f^{-1}} 
   \end{tikzcd}
  $$
  \item 
  $$
   \begin{tikzcd}
    X \ar{r}{f} & X
   \end{tikzcd}
  $$
  \item $$
   \begin{tikzcd}
    X \ar{r}{f} \ar{d}{k} & Y \ar{d}{h} \\
    A \ar{r}{g} & B
   \end{tikzcd}
  $$
 \end{enumerate}
\end{ex}

\begin{deph}
 A diagram is \emph{commutative} (or it \emph{commutes}) if for each two vertices $X$ and $Y$ in the diagram, all compositions along paths of composable arrows connecting $X$ to $Y$ are equal.
\end{deph}

\begin{caveat}
 Not every diagram commutes!
\end{caveat}

\begin{eg}
\begin{itemize}
 \item The diagram 
 $$
   \begin{tikzcd}
    X \ar{r}{f} \ar{dr}[swap]{h} & Y \ar{d}{g} \\
    & Z
   \end{tikzcd}
  $$
  is commutative if and only if $g\circ f = h$. 
  \item The diagram
  $$
   \begin{tikzcd}
    X \ar{r}{f} \ar{d}{k} & Y \ar{d}{h} \\
    A \ar{r}{g} & B
   \end{tikzcd}
  $$
  is commutative if and only if $h\circ f = g\circ k$.
\end{itemize}
\end{eg}

\begin{ex}
 Which diagrams in \Cref{exdiags} commute? For those which do not commute, which morphisms should we require to be equal in order for the diagrams to be commutative?
\end{ex}

\begin{ex}[sets and relations]
 Prove that in a preorder (or in a poset, or in an equivalence relation) every diagram commutes. 
\end{ex}

Commutative diagrams can be pasted, as the following exercise shows.

\begin{ex}
 Consider the following diagram, made out of two adjacent squares:
 $$
   \begin{tikzcd}
    X \ar{r}{f} \ar{d}{k} & Y \ar{d}{h} \ar{r}{g} & Z \ar{d}{\ell}\\
    A \ar{r}{p} & B \ar{r}{q} & C
   \end{tikzcd}
  $$
  Suppose that the left square and the right square, separately, commute. Prove that the whole diagram then commutes. 
\end{ex}

More generally, commutative diagrams of any shape can be pasted to give again commutative diagrams. 

Recall the definition of an isomorphism (\Cref{defiso}). We can rewrite it purely in terms of commutative diagrams.

\begin{deph}[alternative]\label{altdefiso}
 Two morphisms $f:X\to Y$ and $g:Y\to X$ of a category $\cat{C}$ are inverse to each other, forming an isomorphism, if and only if the following two diagrams are commutative.
$$
\begin{tikzcd}
 X \ar{dr}[swap]{\id_X}\ar{r}{f} & Y \ar{d}{g} \\
 & X
\end{tikzcd}
\qquad
\begin{tikzcd}
 Y \ar{dr}[swap]{\id_Y}\ar{r}{g} & X \ar{d}{f} \\
 & Y
\end{tikzcd}
$$
\end{deph}
Again, keep in mind that both diagrams need to be checked separately, the two conditions are independent.

\subsection{The opposite category}

Another useful technique in category theory is ``reversing all the arrows''. 

\begin{deph}
 Let $C$ be a category. The \emph{opposite category} of $\cat{C}$, denoted as $\cat{C}^\op$, is the category obtained as follows.
 \begin{itemize}
  \item Objects are just the objects of $\cat{C}$;
  \item A morphism, denoted as $f^\op$, between two objects $X$ and $Y$ is a morphism $f$ in $\cat{C}$ from $Y$ to $X$ (mind the direction). Graphically:
  $$
  \begin{tikzcd}
   X \ar{d}{f^\op} \\ Y
  \end{tikzcd}
  \quad\leftrightsquigarrow\quad
  \begin{tikzcd}
   X  \\ Y \ar{u}[swap]{f}
  \end{tikzcd}
  $$
  \item The identity at each object $X$ is given by ${\id_X}^\op$;
  \item Composition is the same as in $\cat{C}$, but the order is reversed: $g^\op\circ f^\op$ is defined to be $(f\circ g)^\op$. In diagrams:
  $$
  \begin{tikzcd}
   X \ar{d}{f^\op} \\ Y \ar{d}{g^\op} \\ Z
  \end{tikzcd}
  \quad\leftrightsquigarrow\quad
  \begin{tikzcd}
   X  \\ Y \ar{u}[swap]{f} \\ Z \ar{u}[swap]{g}
  \end{tikzcd}
  $$
 \end{itemize}
\end{deph}

Here are examples on how the opposite category looks in practice.

\begin{eg}[sets and relations]
 Let $(X,\le)$ be a poset, i.e.~a set equipped with a partial order. As we have seen in \Cref{egorder}, this is in particular a category, with an arrow $x\to y$ if and only if $x\le y$. The opposite category has then an arrow $x\to y$ if and only if $y\le x$. We can consider it as the category obtained by the opposite relation, $\ge$ instead of $\le$. For example, for $X=\R$ with its usual order, $(\R,\le)^\op\cong (\R,\ge)$, ordered downward.
\end{eg}

\begin{eg}[group theory]
 Let $G$ be a group. We have seen in \Cref{ssscatop} that the delooping $\B G$ is a category. The opposite category $(\B G)^\op$ will again have a single object, and morphisms will be again elements of $G$, however the composition is reversed: $(g\circ h)^\op=h^\op \circ g^\op$.
\end{eg}

Since the axioms of categories are symmetric with respect to reversing all arrows, the following is true: \emph{a statement in a category $\cat{C}$ is true if and only if the dual statement is true in $\cat{C}^\op$}. 

\begin{caveat}
 A statement is true in $\cat{C}$ if and only if the dual statement is true in $\cat{C}^\op$, but in general, the two statements do not look the same.
\end{caveat}
This is why this principle is so powerful: for each result we prove in category theory there is always a dual result that we get for free. Here is an example.

\begin{ex}
 Let $f,g:X\to Y$ be morphisms in $\cat{C}$. Prove that $f=g$ if and only if $f^\op=g^\op$ in $\cat{C}^\op$.  (Hint: there is not much to prove.)
\end{ex}

Given the result of the exercise, the following statements then just follow as corollaries.

\begin{cor}
 A diagram in $\cat{C}$ commutes if and only if the dual diagram in $\cat{C}^\op$ commutes. 
\end{cor}

As an example, the following diagram
$$
   \begin{tikzcd}
    X \ar{r}{f} \ar{d}{k} & Y \ar{d}{h} \\
    A \ar{r}{g} & B
   \end{tikzcd}
  $$
  commutes if and only if $h\circ f = g\circ k$. This happens if and only if $f^\op\circ h^\op = k^\op\circ g^\op$, i.e.~if the dual diagram commutes:
$$
   \begin{tikzcd}
    X \ar[leftarrow]{r}{f^\op} \ar[leftarrow]{d}{k^\op} & Y \ar[leftarrow]{d}{h^\op} \\
    A \ar[leftarrow]{r}{g^\op} & B
   \end{tikzcd}
  $$

Moreover, since isomorphisms can be defined purely in terms of commutative diagrams (\Cref{altdefiso}), we get:

\begin{cor}
 A morphisms $f:X\to Y$ is invertible if and only if $f^\op:Y\to X$ is invertible. (What will the inverse be?)
 
 In particular, if $X$ and $Y$ are isomorphic in $\cat{C}$, then they are also isomorphic in $\cat{C}^\op$.
\end{cor}

In the next section we will see more examples of how reversing the arrows can give meaningful statements.

\section{Mono and epi}

So far, we have expressed in terms of diagrams the notion of isomorphism, which gives the idea of when two objects are, in some sense, ``the same''. This generalizes the notion of ``bijective map'' between sets, and gives the right notion of ``sameness'' in other categories (such as homeomorphisms between topological spaces). We now would like to present a diagrammatic approach to talk about subspaces and quotients, generalizing injective and surjective maps. In category theory these maps will be called \emph{monomorphisms} (from the Greek work meaning ``one'', as in ``one-to-one''), and \emph{epimorphisms} (from the Greek word meaning ``onto''). 

\subsection{Monomorphisms}

\begin{deph}\label{defmono}
 A morphism $m:X\to Y$ in a category $\cat{C}$ is called a \emph{monomorphism}, or \emph{mono}, if the following holds. For every object $A$ of $\cat{C}$ and for every pair of maps $f,g:A\to X$, i.e.~fitting into the diagram
 $$
 \begin{tikzcd}
 A \ar[shift left]{r}{f} \ar[shift right]{r}[swap]{g} & X \ar{r}{m} & Y
 \end{tikzcd}
 $$
 and such that $m\circ f = m\circ g$, we have that $f=g$ (i.e.~the whole diagram commutes). 
\end{deph}

Intuitively, a monomorphism is a map which ``does not map different things of $X$ to the same thing in $Y$''.  Let's try to see what this means, using the interpretation of categories in terms of spaces and maps. If $A$ is a space, then the image $f(A)$ of $A$ through $f$ will be a subspace of $X$ which, in general, may be different from the image $g(A)$ along $g$. We now want to say that if the images $f(A)$ and $g(A)$ are different in $X$, then applying $m$ to both will \emph{keep them different} in $Y$. In other words, if $f(A)\ne g(A)$, then $m\circ f(A)\ne m\circ g(A)$. Equivalently, if the images are equal in $Y$, that is, if $m\circ f(A) = m\circ g(A)$, then the images must have been equal already in $X$, that is, $f(A) = g(A)$. This is what the definition is saying. If this property holds for every $A$, then the map $m$ ``behaves as if it did not identify different things of $X$'', and we call it a monomorphism. 

Here are some examples. 

\begin{eg}[several fields]
 \begin{itemize}
  \item In $\Set$, the monomorphisms are the injective maps.
  \item In $\cat{Top}$, the monomorphisms are the injective continuous maps.
  \item In the category $\cat{FVect}$ of finite-dimensional vector spaces and linear maps, the monomorphisms are the injective linear maps.
  \item In the category $\Grp$, the monomorphisms are the injective group homomorphisms.
 \end{itemize}
\end{eg}

\begin{ex}[several fields]
 Prove the above claims (or at least, the ones related to a field of math which is familiar to you). Hint: a convenient choice of the object $A$ may help you. 
\end{ex}

In general, not every monomorphism looks like an injective map. For example:

\begin{ex}[sets and relations]
 Prove that in a preorder every morphism is a monomorphism.
\end{ex}

\begin{ex}[sets and relations; difficult!]
 What are the monomorphisms of $\Set^\op$? 
 
 (The answer to this question will hopefully be clear at the end of this section.)
\end{ex}

From the examples given above you may have noticed something: isomorphisms seem to be monomorphisms in all the categories seen so far. For example, in $\Set$, every invertible function is in particular injective. This is always the case.

\begin{prop}\label{isoismono}
 Every isomorphism of $\cat{C}$ is in particular a monomorphism of $\cat{C}$. 
\end{prop}

Before looking at the proof, try to do it yourself.

\begin{proof}
 Let $m:X\to Y$ be an isomorphism with inverse $m^{-1}$. Let now $A$ be any object of $\cat{C}$, and let $f,g:A\to X$ be such that $m\circ f = m\circ g$. We can apply $m^{-1}$ to both terms, and since $m^{-1}\circ m = \id_{X}$, we get
 \begin{align*}
  m^{-1}\circ m\circ f \;&=\; m^{-1}\circ m\circ g \\
  f \;&=\; g . \qedhere
 \end{align*} 
\end{proof}

\subsection{Split monomorphisms}

You may have noticed that, in the proof of \Cref{isoismono}, we only applied $m^{-1}$ on the left of $m$, never on the right. In other words, we only used the \emph{left} diagram of \Cref{altdefiso}, not the right one. Therefore, the proof above, works also for maps which admit only a ``left inverse''. Here is what this means in rigor.

\begin{deph}\label{defsplitmono}
 Let $m:X\to Y$ be a morphism of $\cat{C}$. A \emph{left inverse}, or \emph{retraction} of $m$ is a map $r:Y\to X$ such that $r\circ m = \id_X$. Equivalently, $r$ is such that the following diagram commutes:
 $$
 \begin{tikzcd}
  X \ar{dr}[swap]{\id_X} \ar{r}{m} & Y \ar{d}{r} \\
  & X
 \end{tikzcd}
 $$
 
 If $m$ admits a left inverse, we call $m$ a \emph{split monomorphism}, and call $r$ its \emph{splitting}.
\end{deph}

A split monomorphism is less than an isomorphism, since an isomorphisms is required to have a \emph{two-sided} inverse (there is an additional diagram that has to commute). In particular: 

\begin{prop}
 Every isomorphism is a split monomorphism. 
\end{prop}

Moreover, as we have said, the proof of \Cref{isoismono} only requires left inverses, so we have:

\begin{prop}\label{splitmonoismono}
 Every split monomorphism is a monomorphism. 
\end{prop}

In general, the converse is not true: you may know from topology, or from graph theory, that there are are injective maps which do not have a retraction. Let's see some examples more in detail.

\begin{eg}[sets and relations, linear algebra]\label{1.2.10}
 \begin{itemize}
  \item In $\Set$, every injective map has a left inverse. Therefore, every monomorphism is split. 
  \item In $\cat{FVect}$, every injective linear map has a left inverse. Therefore, again, every monomorphism is split.
 \end{itemize}
\end{eg}

Let's look now at the case of topological spaces.
\begin{eg}[geometry, topology]\label{egcircledisc}
 Let $S^1$ be the circle, and $D^2$ be the two-dimensional disc in $\R^2$. Let $m:S^1\to D^2$ be the embedding of the circle as the boundary of $D^2$. This map is injective and continuous, therefore it is a monomorphism of $\cat{Top}$. However, there is no retraction of $m$: any such retraction $r:D^2\to S^1$ would have to be such that $r\circ m=\id_{S_1}$, that is, it has to map the boundary of $D^2$ to the corresponding point of the circle $S^1$. This cannot be done in a continuous way: where would $r$ map the center of the disc? Any such assignment would have to ``pierce a hole'' somewhere in the disc. 
\end{eg}

Whenever a continuous map $m:X\to Y$ admits a continuous retraction $Y\to X$, we say that $X$ is a \emph{retract} of $Y$. As the previous example shows, this is a stronger condition than simply requiring $m$ to be injective (but of course, it is \emph{necessary} that $m$ is injective, by \Cref{{splitmonoismono}}).
Therefore, almost by definition, the split monomorphisms of $\cat{Top}$ are precisely the embeddings of retracts. 

A similar phenomenon happens in the category of groups.
\begin{ex}[group theory]
 Prove that the inclusion $\Z \hookrightarrow \R$ is a monomorphism in $\Grp$, but it is not split mono. 
\end{ex}

In a generic category, monomorphisms and split monomorphisms are not the same. This is something that may seem counterintuitive, because they are the same in $\Set$. But keep in mind that the generic case is what happens in $\cat{Top}$.

\begin{ex}[graph theory]
 In the category of graphs, are all monomorphisms split? What about the category $\cat{Poset}$ of partial orders and monotone maps?
\end{ex}

\subsection{Epimorphisms}

Let's now present the dual notion.
\begin{deph}\label{defepi}
 A morphism $e:X\to Y$ is called an \emph{epimorphism}, or \emph{epi}, if the following holds. For every object $A$ of $\cat{C}$ and every $f,g:Y\to A$, sitting in the diagram
 $$
 \begin{tikzcd}
  X \ar{r}{e} & Y \ar[shift left]{r}{f} \ar[shift right]{r}[swap]{g} & A
 \end{tikzcd}
 $$
 and such that $f\circ e = g\circ e$, we have that $f=g$ (i.e.~the full diagram commutes).
\end{deph}

Let's try to interpret this definition. An epimorphism, intuitively, is ``something having full image''.
Again, if we consider objects and morphisms as spaces and maps, the definition means the following. Suppose that the image of $e$ in $Y$ is the whole of $Y$. Then if any two maps $f$ and $g$ on $Y$ agree on the image of $e$, they \emph{must agree on the whole of $Y$}. That is, if $f\circ e = g\circ e$, then $f=g$. This is what the definition is saying. If this holds for each $f,g$ into $A$ and for each $A$, then $e$ ``behaves as if it had full image'', and we call it an epimorphism.

This notion of epimorphism is exactly the dual notion to the notion of monomorphism:
\begin{ex}
 Prove that $f:X\to Y$ is epi in $\cat{C}$ if and only if $f^\op:Y\to X$ is mono in $\cat{C}^\op$.
\end{ex}

Therefore, we get for free the dual statement to \Cref{isoismono} :

\begin{prop}\label{isoisepi}
 Every isomorphism is an epimorphism.
\end{prop}

This makes intuitive sense: for example, every invertible map is in particular surjective. 

\begin{ex}
You may not be comfortable yet with believing that \Cref{isoisepi} is true just because the dual statement is true. If that is the case, try to prove the proposition yourself from scratch, and notice that the diagrams appearing in the proof are dual to those appearing in the proof of \Cref{isoismono}. 
\end{ex}

Here are some examples of epimorphisms.
\begin{eg}[sets and relations, linear algebra, topology]
 \begin{itemize}
  \item In $\Set$, the epimorphisms are the surjective maps.
  \item In $\cat{FVect}$, the epimorphisms are the surjective linear maps.
  \item In $\cat{Top}$, the epimorphisms are the surjective continuous maps. 
 \end{itemize}
\end{eg}
Again, not in every category epimorphisms look like surjective maps. For example, in $\Set^\op$, an epimorphism is the opposite of an injective map. Again, in a preorder every arrow is an epimorphism. 
More interestingly, try to prove the following.
\begin{ex}[algebra]
 In the category $\cat{Mon}$ of monoids and monoid morphisms, the inclusion $\N\hookrightarrow\Z$ is an \emph{epimorphism}. 
\end{ex}

\subsection{Split epimorphisms}

\begin{deph}\label{defsplitepi}
 Let $e:X\to Y$ be a morphism of $\cat{C}$. A \emph{right inverse}, or \emph{section} of $e$ is a map $s:Y\to X$ such that $e\circ s = \id_Y$. Equivalently, $s$ is such that the following diagrams commutes:
 $$
 \begin{tikzcd}
  Y \ar{dr}[swap]{\id_Y} \ar{r}{s} & X \ar{d}{e} \\
  & Y
 \end{tikzcd}
 $$
 
 If $e$ admits a right inverse, we call $e$ a \emph{split epimorphism}, and call $s$ its \emph{splitting}.
\end{deph}

Again, this is the dual notion to the one of split monomorphism. Therefore we get the dual results:

\begin{prop}
 Every isomorphism is a split epimorphism.
\end{prop}

\begin{prop}
 Every split epimorphism is an epimorphism.
\end{prop}

Moreover, notice the following: if $s$ is a right inverse of $e$, then $e$ is a left-inverse of $s$. So in particular, if $e$ is split epi, then its section $s$ is split mono, and vice versa. 
The pair of maps 
$$
\begin{tikzcd}
 X \ar[shift left]{r}{e} & Y \ar[shift left]{l}{s}
\end{tikzcd}
$$
is sometimes called a \emph{retract}, or it is said (especially in the context of topological spaces) that $Y$ is a \emph{retract} of $X$.

This gives then an intuitive way to think about split monos and epis: \emph{a retract is at the same time a subspace and a quotient}. 
For example, $\R$ is at the same time a subspace and a quotient of $\R^2$, both in $\Vect$ and in $\cat{Top}$. The notion is not symmetric: $\R^2$ is not a retract of $\R$ (intuitively, the retract is often ``smaller'' in some sense). In \Cref{egcircledisc} we saw that the circle is a subspace of the disc, but not a quotient, therefore it is not a retract. 
This should give us an idea of what to expect as examples of split epimorphisms.

\begin{eg}[sets and relations, linear algebra]
 \begin{itemize}
  \item In $\Set$, every surjective map has a right inverse. Therefore, every epimorphism is split.\footnote{If you are familiar with set theory, try to show that the axiom of choice is equivalent to the statement that in $\Set$ every epimorphism is split. If you are not familiar with set theory, you can safely ignore this comment.}
  \item In $\cat{FVect}$, every surjective linear map has a right inverse. Therefore, again, every epimorphism is split.
 \end{itemize}
\end{eg}

\begin{eg}[topology]\label{closecircle}
 Consider the half-open interval $[0,1)$ and the map $e:[0,1)\to S^1$ which closes the interval to a circle (mapping $t$ to the pair $(\cos(2\pi\,t), \sin(2\pi\,t))$ if we consider the circle as a subset of $\R^2$). 
 This map is surjective, since it reaches the whole circle, so it is an epimorphism of $\cat{Top}$. However, it is not split: any splitting would have to map the circle back to the interval, and to do so, it would have to ``break the circle open''. This cannot be done in a continuous way.
\end{eg}
In $\cat{Top}$ the split epimorphisms are the retractions (or the quotients to a retract), and being a retraction is a strictly stronger requirement than just being a surjective map. 

\begin{ex}[algebra]\label{circleretract}
Consider the map $\R\to S^1$ of $\Grp$ mapping all integers to the neutral element of $S^1$ (as a time is mapped to a clock, or as $t\mapsto e^{2\pi i t}$ if we consider the circle as a subset of $\C$). This is surjective, so it is an epimorphism of $\Grp$. Show that it does not admit a section. 
\end{ex}

Let's observe again the map $e$ from \Cref{closecircle}. We have that $e:[0,1)\to S^1$ is even \emph{bijective}. That is, it is mono and epi. But it is not an isomorphism, in particular it is not split epi (or more intuitively, the circle and the interval are different spaces). Therefore, \emph{a map which is both epi and mono is not necessarily an isomorphism}. 

However, the following is true:
\begin{prop}\label{splitmonoepi}
 If $f:X\to Y$ is epi and split mono, it is an isomorphism. Dually, if it is mono and split epi, it is an isomorphism.
\end{prop}

This is why, in $\Set$, a map is invertible when it is injective and surjective: those are not just mono and epi, they are also \emph{split}. But keep in mind that this is not the general case.

\begin{ex}[important!]
 Prove \Cref{splitmonoepi}. (Hint: you only need to prove one of the two statements, the other one is dual.)
\end{ex}

\section{Functors and functoriality}

\begin{deph}
 Let $\cat{C}$ and $\cat{D}$ be categories. A \emph{functor} $F:\cat{C}\to\cat{D}$ consists of the following data. 
 \begin{itemize}
  \item For each object $X$ of $\cat{C}$, an object $FX$ of $\cat{D}$;
  \item For each morphism $f:X\to Y$ of $\cat{C}$, a morphism $Ff:FX\to FY$ of $\cat{D}$ (note that the domain and codomain of $Ff$ must be exactly $FX$ and $FY$);
 \end{itemize}
 such that the following \emph{functoriality axioms} hold.
 \begin{itemize}
  \item \emph{Unitality} (or sometimes \emph{normalization}): for every object $X$ of $\cat{C}$, $F(\id_X)=\id_{FX}$. That is, $F$ maps identities into identities.
  \item \emph{Compositionality} (or sometimes \emph{cocycle condition}): for every pair of composable morphisms
  $$
  \begin{tikzcd}
  X \ar{r}{f} & Y \ar{r}{g} & Z
  \end{tikzcd}
  $$
  in $\cat{C}$, we have that $F(g\circ f)=Fg\circ Ff$. That is, the following diagram must commute. 
  $$
  \begin{tikzcd}
  & FY \ar{dr}{Fg} \\
  FX \ar{ur}{Ff} \ar{rr}[swap]{F(g\circ f)} & & FZ
  \end{tikzcd}
  $$
  In other words, $F$ respects the composition of arrows.
 \end{itemize}
\end{deph}

Note that in the expression $F(\id_X)=\id_{FX}$ we have on the left the identity of $X$ in $\cat{C}$, and on the right the identity of $FX$ in $\cat{D}$. Just as well, in the expression $F(g\circ f)=Fg\circ Ff$ we have on the left the composition in $\cat{C}$, and on the right the composition in $\cat{D}$. 

In the following we look at some ways to interpret functors between different types of categories.

\subsection{Functors as mappings preserving relations}

\begin{idea}
 A functor is a mapping preserving or respecting the relations between the objects.
\end{idea}

\begin{eg}[sets and relations]
 Let $(X,\le)$ and $(Y,\le)$ be partial orders. As we have seen in \Cref{egorder}, these are in particular categories. A functor $F:(X,\le)\to (Y,\le)$ consists first of all of a mapping between the objects, that is, a function $F:X\to Y$. Moreover, to each arrow in $X$ there has to be a corresponding arrow of $Y$. That is, if $x\le x'$ in $X$, then $Fx \le Fx'$ in $Y$. The functoriality axioms are immediately satisfied, since between any two objects there is at most one morphism: the identity at $x\in X$ has to be mapped to the identity at $Fx\in Y$, since there is no other arrow $Fx\to Fx$, and the same is true for composition. 
 Therefore \emph{a functor between partial orders is the same thing as a monotone map}. If $x$ and $x'$ are related, ($x$ is less or equal than $x'$), then their images under $F$, the elements $Fx$ and $Fx'$, must be related too ($Fx$ is less or equal than $Fx'$). 
\end{eg}

\begin{eg}[sets and relations]\label{equequ}
 Let $(X,\sim)$ and $(Y,\sim)$ be equivalence relations. As we have seen in \Cref{egeq}, these are in particular categories. A functor $F:(X,\sim)\to(Y,\sim)$ is first of all a function $F:X\to Y$. Again, arrows of $X$ have to be mapped into arrows of $Y$, that is, if $x\sim x'$ in $X$, then $Fx \sim Fx'$ in $Y$. Again, the functoriality axioms are guaranteed by uniqueness.
 Therefore \emph{a functor between equivalence relations is the same thing as a map respecting the equivalence}.
 If $x$ and $x'$ are equivalent in $X$, then they have to be mapped into equivalent elements of $Y$. This is sometimes called an \emph{equivariant map for the equivalence relations}. 
\end{eg}

\begin{caveat}
 In the target category there may be \emph{more relations} than in the source category. In the partial order case, $F:(X,\le)\to (Y,\le)$, it could be that $x\nleq x'$, but still $Fx \leq Fx'$. Relations can be ``created''. The important thing is that they are not ``destructed''.
 Just as well, for $F:(X,\sim)\to(Y,\sim)$ it could be that $x\nsim x'$, but still $Fx \sim Fx'$. The important thing is that if $x\sim x'$, then $Fx \sim Fx'$.
 In the terminology of category theory, relations have to be \emph{preserved}, but not necessarily \emph{reflected}. 
\end{caveat}

\begin{ex}[sets and relations]\label{quotient}
 Let $(X,\sim)$ be an equivalence relation. Let $X/\sim$ be the \emph{quotient space} of the relation, that is, the set of equivalence classes. We can consider $X/\sim$ a category, where the only arrows are identities. Prove that the quotient map $(X,\sim)\to X/\sim$ assigning to each $x\in X$ its equivalence class $[x]\in X/\sim$ is a functor.
\end{ex}

\begin{ex}[sets and relations]
 More generally, let $(X,\lesssim)$ be a preorder. Consider the equivalence relation $\sim$ given by \Cref{corepreorder}. Define a partial order $\le$ on the quotient space $X/\sim$ in such a way that the quotient map $(X,\lesssim)\to (X/\sim, \le)$ is a functor. 
\end{ex}

For all these examples, by uniqueness we didn't have to worry about checking the functoriality axioms. In general, this is something that needs to be checked, as the next examples will show.

\subsection{Functors as mappings preserving operations}

\begin{idea}
 A functor is a mapping preserving or respecting the operations of our structures, and their composition.
\end{idea}

\begin{eg}[group theory]\label{eghomomorphism}
 Let $G$ and $H$ be groups. \emph{A function $f:G\to H$ is a group homomorphism if and only if the induced mapping $\B G \to \B H$ is a functor.} Let's see why. 
 First of all, $f$ is a group morphism if and only if the following properties are satisfied.
 \begin{itemize}
  \item Unit condition: $f(1)=1$, where on the left we have the unit of $G$, and on the right we have the unit of $H$.
  \item Composition condition: for each $g,g'\in G$, $f(g\cdot g') = f(g)\cdot f(g')$, where on the left we have composition in $G$, and on the right we have composition in $H$. 
 \end{itemize}
 Instead, a functor $\B G \to \B H$ consists of the following data.
 \begin{itemize}
  \item First of all, a mapping between the objects of $\B G$ to the objects of $\B H$. But since both categories have only a single object, this is trivial: we just map the single object of $\B G$ to the single object of $\B H$.
  \item Now, we need a map from the morphisms of $\B G$ to the morphisms of $\B H$. Their source and targets will always match, since there is just one object, therefore this amounts simply to a function $f:G\to H$ (remember that the morphisms of $\B G$ are just the elements of the group $G$, and the same for $H$).
  \item The function $f:G\to H$ has to map the identity to the identity and the composition to the composition, by functoriality. This says precisely that $f$ has to be a group homomorphism.
 \end{itemize}
\end{eg}

\begin{eg}
 Just as well, let $M$ and $N$ be monoids. \emph{A function $f:M\to N$ is a monoid homomorphism if and only if it induces a functor $\B M \to \B N$.} The proof is the same, since in \Cref{eghomomorphism} we never used inverses.
\end{eg}

Here are two very important examples.

\begin{eg}[group theory]\label{egrepresentation}
 Let $G$ be a group. A \emph{linear representation} of $G$ is a functor $R:\B G\to \Vect$. (If you don't know what a linear representation is, you can take this as a definition.) Let's see what this means in practice. 
 \begin{itemize}
  \item First of all, we need to map the single object $\bullet$ of $\B G$ to an object of $\Vect$, that is, a vector space. In other words, we need to select a particular vector space, which will be the space where our representation acts. Let's call this space $V$, so $R(\bullet)=V$. In representation theory, one says that \emph{$G$ is acting on the space $V$}.
  \item For each morphism of $\B G$, that is, for each element $g\in G$, we need a linear map $Rg:V\to V$. This map is an isomorphism (see the next exercise).
  \item The assignment $g\mapsto Rg$ has to be such that the neutral element $1$ is mapped to the identity $\id_V$, and that composition is preserved. 
 \end{itemize}
\end{eg}

\begin{ex}[group theory, important!]
 Show that for each $g\in G$, the morphism $Rg$ is an isomorphism.
\end{ex}

Therefore, $R:\B G\to \Vect$ induces a group homomorphism $G\to\Aut(V)$. ($\Aut(V)$, called the \emph{automorphism group} of $V$, is the set of isomorphisms from $V$ to itself. If you are not familiar with it, prove that it has a natural group structure.)

A possible interpretation of a group representation is the following. A group is a rather abstract object, but it is usually an abstraction of the idea of \emph{symmetry} of some more concrete object. A representation is then a way to assign to the group $G$ a \emph{concrete} space where $G$ acts. For example, you may view the group $S^1$ as ``rotations'' in a plane. A way to make this intuition rigorous is then to take a vector space modeling our plane, for example $\R^2$, and viewing $S^1$ as acting on $\R^2$ via linear maps $\R^2\to\R^2$, the rotations, one for each element of $S^1$. 
This allows us to have a much more concrete idea of what $S^1$ looks like ``in practice''. 

Of course, just like not every function is a bijection, just as well in general the representation of a group may lose information. For example, one could represent $S^1\times S^1$ again on the plain $\R^2$ by letting the first factor $S^1$ act on the plane, and forgetting the second factor. This is a well-defined representation, but it is not a \emph{faithful} representation of the group $S^1\times S^1$. 

An analogous way to represent a group more concretely is, instead of letting it act on a vector space, letting it act on a set, by just permuting its elements.

\begin{eg}[group theory]\label{perm}
 Let $G$ be a group. A \emph{permutation representation} of $G$ is a functor $R:\B G\to \Set$.
\end{eg}

As in the vector space case, this picks out a particular set $X$, on which the group $G$ acts. A set equipped with such a group action is sometimes called a \emph{$G$-set}.

\begin{ex}[group theory]
 Show that a permutation representation of $G$ is the same as a set $X$ together with a group homomorphism $G\to\Aut(X)$. 
\end{ex}

Again, in general a permutation representation may lose information. 
However, for each group there is always a \emph{faithful} permutation representation, that is, a representation which does not lose any information:

\begin{thm}[Cayley]\label{thmcayley}
 Let $G$ be any group. Then there exists a set $X$ such that $G$ is isomorphic to a subgroup of $\Aut(X)$. 
\end{thm}

\begin{ex}[group theory, linear algebra]
 Using Cayley's theorem, derive an analogous statement for linear representations.
\end{ex}

Cayley's theorem can be considered a special case of a more general result in category theory, known as the \emph{Yoneda lemma}. We will look at this in more detail (and prove it) in the future.\footnote{\Cref{secyoneda}, see in particular \Cref{egcayley}.}

Let us now look at the most general case of functors, where neither objects nor morphisms between two given objects are unique. 

\subsection{Functors defining induced maps}

\begin{idea}
 A functor is a consistent way of defining new \emph{induced maps} from existing maps.
\end{idea}

\begin{eg}[sets and relations]\label{egpowerset}
 The \emph{power set functor} is a functor $P:\Set\to\Set$ defined in the following way. 
 \begin{itemize}
  \item On objects, it maps a set $X$ to its \emph{power set} $PX$, that is, the set of subsets of $X$. An element $S\in PX$ is a subset $S\subseteq X$. 
  \item On morphisms, it maps a function $f:X\to Y$ to the induced function $Pf:PX\to PY$ given by the \emph{image of subsets}: given a subset $S\subseteq X$, we obtain a subset of $Y$ as follows. We apply $f$ to all the elements in $S$, and we will get a bunch of elements of $Y$. These form a subset of $Y$, that is, an element of $PY$. Equivalently,
  $$
  (Pf)(S) \; \coloneqq  \; \{y\in Y \,|\, y = f(x) \mbox{ for some } x\in S\}.
  $$
 \end{itemize}
 In set theory, usually one denotes the image of subsets with the same letter as the original function, namely as $f(S)$. This makes sense, since the map $PX\to PY$ is \emph{induced} by $f:X\to Y$. However, technically, it is not the same map as $f$ (it has different domain and codomain). This intuitive idea of ``having a map which is still defined by $f$, but with different domain and codomain'' is made precise exactly by functoriality. So, in category-theoretical terminology, the map $f:X\to Y$ induces, via the \emph{functor} $P$, a map $Pf:PX\to PY$.
\end{eg}

\begin{ex}[sets and relation]
 Check that $P$ satisfies the functoriality axioms.
\end{ex}

Just as we can induce maps between subsets, we can also induce maps between probability measures. We first look at finitely supported probability measures, for readers which are not familiar with measure theory. 
More sophisticated constructions can be done as well, using measure theory, and will be given in the subsequent exercises.

\begin{eg}[basic probability]\label{egprob}
 The \emph{probability functor}, called also \emph{distribution functor}, is a functor $\mathcal{P}:\Set\to\Set$ defined as follows.
 \begin{itemize}
  \item On objects, it maps a set $X$ to the set of \emph{finitely supported probability measures} on $X$. Those are functions $p:X\to [0,1]$ with only finitely many nonzero entries, such that their sum is $1$, that is:
  $$
  \sum_{x\in X} p(x) \; = \; 1.
  $$
  These can be thought of as ``finite normalized histograms over the elements of $X$''. 
  \item On morphisms, to a function $f:X\to Y$ we assign the function $\mathcal{P}f:\mathcal{P}X\to\mathcal{P}(Y)$ given by the \emph{pushforward of measures} along $f$. In particular, given $p\in\mathcal\mathcal{P}X$, the probability measure $(\mathcal{P}f)(p)\in \mathcal{P}Y$ is given by
  $$
  (\mathcal{P}f)(p) \, (y) \; \coloneqq  \; \sum_{x\in f^{-1}(y)} p(x) .
  $$
  Intuitively, this ``moves the columns of the histogram on $X$ along $f$ to give a histogram on $Y$, stacking columns on top of each other whenever they end up over the same element''.
 \end{itemize}
 In the probability and measure theory literature, usually the map $\mathcal{P}f$ is written as $f_*$. In terms of random variables, the map giving \emph{image random variables} (which have the pushforward of the measure as law) is instead denoted again by $f$. Just like for the power set case, these notations reflect the fact that the map $\mathcal{P}f$ is induced by $f$, functorially. 
\end{eg}

Here are more general construction, for people with a background in probability (or who want to learn more about probability).

\begin{ex}[measure theory, probability]\label{giryfunctor}
 Consider the category $\cat{Meas}$ of measurable spaces and measurable maps. 
 Recall that the pushforward of measures in general is defined as follows. Let $f:X\to Y$ be measurable. Let $p$ be a measure on $X$ and let $A$ be a measurable subset of $Y$. Then, since $f$ is measurable, by definition $f^{-1}(A)$ is a measurable subset of $X$, and the pushforward of $p$ along $f$ is defined to be the assignment
 $$
 A \;\longmapsto\; (f_*p)(A) \coloneqq   p(f^{-1}(A)) .
 $$
 Now define a probability functor $\mathcal{P}:\cat{Meas}\to\cat{Meas}$, in such a way that
 \begin{itemize}
  \item To each measurable space $X$ we assign a space $\mathcal{P}X$ of probability measures on $X$;
  \item To each measurable map $f:X\to Y$ we get a map $\mathcal{P}f:\mathcal{P}X\to\mathcal{P}Y$ via the pushforward of measures. 
 \end{itemize}
 Keep in mind that, in order for $\mathcal{P}$ to be a functor $\cat{Meas}\to\cat{Meas}$, we need $\mathcal{P}X$ to be an object of $\cat{Meas}$, i.e.~a measurable set, and $\mathcal{P}f$ to be a morphism of $\cat{Meas}$, i.e.~a measurable function. 
 Therefore, when you define the spaces $\mathcal{P}X$, you have to equip them with a $\sigma$-algebra, and you have to do it in such a way that the map $\mathcal{P}f$ will be measurable. 
 
 One convenient way to do it is the following.
 Given $A\subseteq X$ measurable, consider the mapping $i_A:\mathcal{P}X\to[0,1]$ given by evaluation:
 $$
 i_A(p) \;\coloneqq\; p(A) .
 $$
 Equip now $\mathcal{P}X$ with the initial $\sigma$-algebra of all the maps $i_A$ in the form above, i.e.~the coarsest $\sigma$-algebra which makes all the maps $i_A$ measurable. Prove that this way, for every measurable map $f:X\to Y$, the map $\mathcal{P}f:\mathcal{P}X\to\mathcal{P}Y$ is measurable, and check the functoriality axioms, so that $\mathcal{P}:\cat{Meas}\to\cat{Meas}$ is a functor.
 
 This functor is called the \emph{Giry functor}.
\end{ex}

\begin{ex}[measure theory, probability]\label{radonfunctor}
 Consider the category $\cat{CHaus}$ of compact Hausdorff spaces and continuous maps. Assign to each object $X$ (which is a compact Hausdorff space) the space of Radon probability measures on $X$, equipped with the weak topology (which is again a compact Hausdorff space). Show that this assignment is part of a functor, with the assignment on morphisms again given by the pushforward of measures. 
 
 This functor is called the \emph{Radon functor}. 
\end{ex}

Here is again a functor giving an ``induced map'', this time coming from computer science.

\begin{eg}[basic computer science, combinatorics]\label{listfunctor}
Consider again the category $\cat{Set}$ of sets and functions. (Readers familiar with computer science may want to use types instead of sets.) 
Given a set $X$ we can form a new set $LX$ whose elements are \emph{lists of elements of $X$}. A list is an expression in the form $[x_1,\dots,x_n]$, where the $x_i$ are elements of $X$, and the length $n$ is finite but can be arbitrarily large (it can also be zero, giving the empty list $[\;]$). 

The list construction is functorial: given $f:X\to Y$ we can induce a function $Lf:LX\to LY$ by applying $f$ elementwise in the list. The empty list of $LX$ is mapped to the empty list of $LY$, and $[x_1,\dots,x_n]$ is mapped to $[f(x_1),\dots,f(x_n)]$.
Extending a function $f$ to \emph{lists of possible inputs}, giving as output the list of the results, is then the action of the functor on the morphisms.
This functionality is for example given (roughly) by \emph{map} in Python, \emph{fmap} in Haskell, and corresponds to the attribute \emph{Listable} in Mathematica.
\end{eg}

Here is another class of examples which are widely used in the applications of category theory: forgetful functors. These are functors which intuitively ``forget the extra structures''. Let's see what this means by examples.

\begin{eg}[topology]\label{forgetfulfunctor}
 Consider the following functor $U:\cat{Top}\to\Set$. 
 \begin{itemize}
  \item On objects, it maps a topological space $X$ to the \emph{underlying set} $X$. That is, it considers $X$ only as a set now, ``forgetting'' the topology. (The letter $U$ comes from ``underlying''.)
  \item On morphisms, it maps a continuous function $f:X\to Y$ to the \emph{underlying function} $f:X\to Y$, considered now only a functions between sets, ``forgetting'' the continuity. 
 \end{itemize}
 Here are some things to keep in mind:
 \begin{itemize}
  \item A continuous function is in particular a function, so we can always do this. However, not every function is continuous: therefore, the functions between the \emph{sets} $X$ and $Y$ may in general be more than those obtained by forgetting the continuity. 
  \item One can define many possible topologies on the same set. So, in general, the ``underlying set'' assignment may map different spaces to the same set, it is many-to-one (up to isomorphism). 
 \end{itemize}
\end{eg}

This idea of ``forgetting the structure'' can be done for many structures, not just topologies. 

\begin{ex}[several fields]
 Define the following forgetful functors:
 \begin{itemize}
  \item $\Grp\to\Set$, forgetting the group structure;
  \item $\cat{Ring}\to\Grp$, forgetting the multiplication (but keeping the addition);
  \item Let $\cat{TopGrp}$ be the category of topological groups and continuous group homomorphisms. Define forgetful functors $\cat{TopGrp}\to\cat{Top}$ and $\cat{TopGrp}\to\Grp$ forgetting the group structures and the topology, respectively.
 \end{itemize}
\end{ex}

In category theory, many ideas have intuitively to do with \emph{structures}, \emph{preserving structures}, and \emph{forgetting structures}. The idea of a forgetful functor is a way of making this intuition rigorous. 

For the people with a more pure math background, here is a classic example of a functor: the fundamental group. The construction is quite long, and the details will be given as exercise for who is interested. (This is difficult to people without a background in topology. If you cannot solve this exercise, but you still would like to know about the fundamental group, an excellent reference is \cite{hatcher}.)

\begin{eg}[algebraic topology]\label{fundamentalgroup}
 Let $\cat{Top}_*$ be the category of \emph{pointed topological spaces}, defined as follows.
 \begin{itemize}
  \item Objects are topological spaces with a distinguished point, that is, pairs $(X,x)$ where $X$ is a topological space and $x$ is a point of $X$, usually called the \emph{base point}.
  \item A morphism $f:(X,x)\to(Y,y)$ is a continuous map $f:X\to Y$ which preserves the base point, that is, such that $f(x)=y$. 
 \end{itemize}
 
 Let now $(X,x)$ be a pointed topological space. A \emph{loop in $X$ based at $x$} is a continuous function $l:[0,1]\to X$ such that $f(0)=f(1)=x$. Intuitively, this looks indeed like a loop inside $X$ from the point $x$ to itself. 
 
 Consider now loops $l,m:[0,1]\to X$ at $x$. The \emph{concatenation of loops} $lm$ is the loop at $x$ given as follows. First we walk along the loop $l$, and then we walk along the loop $m$. This takes twice the time, that is, this gives a function $[0,2]\to X$ rather than $[0,1]\to X$. To have a function $[0,1]\to X$, then, we have to walk twice as fast. Formally, $lm:[0,1]\to X$ is the loop at $x$ given by:
 $$
 lm(t) \; \coloneqq  \; \begin{cases}
                l(2t) & 0\le t \le 1/2 \\
                m(2t - 1) & 1/2 < t \le 1 .
                \end{cases}
 $$
 
 A \emph{homotopy} between the loops $l,m:[0,1]\to X$ at $x$ is a continuous map $h:[0,1]\times[0,1]\to X$ such that:
 \begin{itemize}
  \item For each $s\in [0,1]$, the map $t\mapsto h(s,t)$ is a loop in $X$ based at $x$;
  \item The map $h$ is equal to $l$ at $s=0$ and equal to $m$ at $s=1$. That is, for each $t\in [0,1]$, $h(0,t)=l(t)$ and $h(1,t)=m(t)$.
 \end{itemize}
 Intuitively, $h$ is a way to deform the loop $l$ continuously into the loop $m$, while keeping the base points fixed. 
 If there exists a homotopy $h$ between $l$ and $m$ we say that $l$ and $m$ are \emph{homotopic}. This is an equivalence relation (see the next exercise), and the equivalence classes are called \emph{homotopy classes}. We denote the space of homotopy classes by $\pi_1(X,x)$
 
 The concatenation of loops induces a concatenation between the homotopy classes which equips $\pi_1(X,x)$ with a group structure. The unit is given by the constant loop at $x$, and inverses are given by ``walking the loop backwards'' (again, see the next exercise). We call therefore $\pi_1(X,x)$ the \emph{fundamental group of $X$ at $x$}. 
 
 Let now $f:(X,x)\to(Y,y)$ be a base point-preserving continuous function. We can map a loop at $x$ to a loop of $y$ by just applying $f$ to it, that is, given a loop $l$ at $x\in X$, we form the loop
 $$
 \begin{tikzcd}
  \left[0,1\right] \ar{r}{l} & X \ar{r}{f} & Y . 
 \end{tikzcd}
 $$
 Intuitively, this applies $f$ to each point of the loop $l$, and since $f$ is continuous, the resulting points in $Y$ will form again a loop, based at $y$. We then get a function from loops at $x$ in $X$ to loops at $y$ in $Y$. This function respects homotopy equivalence (see the next exercise), and therefore it induces a map between the equivalence classes, $\pi_1(X,x) \to \pi_1(Y,y)$. We denote this resulting map $\pi_1(f)$.
 
 The assignment given by $(X,x)\mapsto \pi_1(X,x)$ and $f\mapsto \pi_1(f)$ is a functor $\cat{Top_*}\to\Grp$.
\end{eg}

\begin{ex}[algebraic topology]
In the notation of the previous example,
\begin{itemize}
 \item Show that homotopy of loops is an equivalence relation.
 \item Show that the concatenation of loops up to homotopy makes $\pi_1(X,x)$ a group. (Hint: to form inverses, walk the loop backward. Why is it an inverse, up to homotopy?)
 \item Show that if $l$ and $m$ are homotopic, then applying $f$ to both gives homotopic loops. Conclude that $f$ will induce a well-defined mapping $\pi_1(f): \pi_1(X,x) \to \pi_1(Y,y)$.
 \item Show, moreover, that the induced map $\pi_1(f): \pi_1(X,x) \to \pi_1(Y,y)$ is a group homomorphism. 
 \item Verify the functoriality axioms and conclude that $\pi_1:\cat{Top_*}\to\Grp$ is a functor.
\end{itemize}
\end{ex}

\begin{ex}[graph theory, algebraic topology]
 Can one construct a similar functor on a category of graphs?
\end{ex}

In different fields of mathematics, the functoriality condition $F(g\circ f) = Ff\circ Fg$ appears under different names. One name under which this condition is known is \emph{chain rule}. This is the case, for example, of the chain rule of derivatives.

The usual \emph{derivative} of calculus (or \emph{differential}, or \emph{gradient}, or \emph{Jacobian matrix}) can be thought of as a \emph{linear approximation} of a differentiable function $f:X\to Y$ in the neighborhood of a point $x$. Linear maps live between vector spaces, therefore, in order to have a linear map, we need to first replace our original spaces $X$ and $Y$ with vector spaces, and then replace $f$ with a linear map. This is precisely what we can do with a functor: we assign to each space $X$ a vector space, and to each map $f$ a linear map, the derivative. 
There are many ways of constructing this in practice. Here is one.

\begin{eg}[calculus]\label{derivative}
 Define the category of \emph{pointed Euclidean spaces} $\cat{Euc_*}$ as follows.
 \begin{itemize}
  \item As objects, we take Euclidean spaces ($\R^n$ for different $n\in \N$) with a distinguished point $x\in \R^n$. Let's denote these objects by $(\R^n,x)$.
  \item As morphisms $f:(\R^n,x)\to(\R^m,y)$ we take smooth (i.e.~differentiable infinitely many times) functions $\R^n\to \R^m$ such that $f(x)=y$. 
 \end{itemize}
 These correspond to the functions that we want to differentiate, at the point at which we want to take the differential. 
 
 The derivative is now a functor $D:\cat{Euc_*}\to\Vect$ defined in the following way.
 \begin{itemize}
  \item On objects, it maps $(\R^n,x)$ to $\R^n$ (now seen as a vector space).
  \item On morphisms, it maps $f:(\R^n,x)\to(\R^m,y)$ to the derivative $Df|_x$ of the function $f$ at the point $x$, which is a linear function $\R^n\to\R^m$ (usually represented by the Jacobian matrix).
 \end{itemize}
 The intuition is that on objects it maps $(\R^n,x)$ to the space of ``vectors starting at $x$'', which is canonically isomorphic to $\R^n$. On morphisms, the differential gives the map between vectors at $x$ and vectors at $f(x)\in \R^m$ which is given by the differential.
 This is functorial because:
 \begin{itemize}
  \item The derivative of the identity map $(\R^n,x)\to(\R^n,x)$ is just the identity of $\R^n$ (the identity matrix).
  \item Consider now composable maps
  $$
  \begin{tikzcd}
   (\R^n,x) \ar{r}{f} & (\R^m,y) \ar{r}{g} & (R^p,z) .
  \end{tikzcd}
  $$
  We have that, \emph{by the chain rule of derivatives},
  $$
  D(g\circ f)|_x \;=\; Dg|_{y} \circ Df|_x ,
  $$
  where the composition on the right is the composition of linear maps. Equivalently, in components,
  $$
  \dfrac{\partial(g\circ f)^k}{\partial x^i} \;=\; \sum_{j=1}^m \dfrac{\partial g^k}{\partial y^j}\,\dfrac{\partial f^j}{\partial x^i} 
  $$
  for each $i=1,\dots,n$ and $k=1,\dots,p$.
 \end{itemize}
 Therefore, the derivative is a functor, with functoriality guaranteed by the chain rule.
\end{eg}

The readers with a background in differential geometry can also try to do the same with manifolds. The intuition is the very same, replacing the spaces and the maps with \emph{linear approximations}.

\begin{ex}[differential geometry]
 Define the category of \emph{pointed manifolds} $\cat{Mfd_*}$ as follows.
 \begin{itemize}
  \item As objects, we take smooth manifolds $M$ with a distinguished point $x\in M$. Let's denote these objects by $(M,x)$.
  \item As morphisms $f:(M,x)\to(N,y)$ we take smooth maps $M\to N$ such that $f(x)=y$. 
 \end{itemize}
 Consider the following assignment.
 \begin{itemize}
  \item On objects, it maps $(M,x)$ to the tangent space $T_xM$, which is a vector space.
  \item On morphisms, it maps $f:(M,x)\to(N,y)$ to the derivative $Df|_x$ of the function $f$ at the point $x$, which is a linear function $T_xM\to T_yN$.
 \end{itemize}
 Show that this is a functor $\cat{Mfd_*}\to\Vect$.
\end{ex}

\begin{ex}[differential geometry; difficult!]
 Can you construct a functor which takes the tangent \emph{bundle} of a space, rather than the tangent space at a single point?
\end{ex}

For the readers with a background on information theory or dynamical systems, here is another exercise. 
\begin{ex}[information theory; difficult!]
It is well-known that Shannon entropy also satisfies a sort of chain rule: for random variables $X,Y$ on finite state spaces,
$$
H(X,Y) \;=\; H(X) + H(Y|X) .
$$
Let now $\cat{FinProb}$ be the category of \emph{finite probability spaces and measure-preserving maps}, that is,
\begin{itemize}
 \item An object $(X,p)$ consists of a finite set $X$ equipped with a probability measure $p$;
 \item A morphism $f:(X,p)\to(Y,q)$ is a function $f:X\to Y$ such that $f_*p=q$ (the latter denotes the pushforward of measures, see \Cref{egprob}). 
\end{itemize}
Show now that, thanks to the chain rule, entropy is functorial on $\cat{FinProb}$. The main question is: \emph{to which category}?
\end{ex}

\subsection{Functors and cocycles}

In some fields of mathematics functoriality appears under the name of \emph{cocycle condition}, or more specifically, \emph{1-cocycle condition}. The reason for this terminology comes from algebraic topology. 
The intuitive idea is that a 1-cocycle is a map which only depends on the endpoints. For example, given a 2-simplex (triangle) $$
\begin{tikzcd}
A \ar[-]{dr}[swap]{h} \ar[-]{r}{f} & B \ar[-]{d}{g} \\
& C
\end{tikzcd}
$$
in a simplicial complex, a 1-cocycle $F$ with coefficients in an abelian group (say, $\Z$) satisfies the identity  
$$
Ff + Fg \; = \;Fh.
$$
This is analogous to a functor in the following sense: suppose that we have a commutative triangle in a category $\cat{C}$
$$
\begin{tikzcd}
A \ar{dr}[swap]{h} \ar{r}{f} & B \ar{d}{g} \\
& C
\end{tikzcd}
$$
and a functor $F:\cat{C}\to \B \Z$. Then we have
$$
Ff \circ Fg \; = \; F(g\circ f) \;= \;Fh.
$$
Both condition can be intuitively interpreted as ``$F$ only depends on the endpoints''.
This intuition can be made precise and concrete, and the resulting correspondence between functors and cocycles is actually very deep. This is however quite advanced in its full generality, well outside the scope of the present lecture. If you are interested, you can look at the \href{http://ncatlab.org/nlab/show/cocycle}{nLab page about cocycles (link)}, maybe first starting at \href{https://ncatlab.org/nlab/show/motivation+for+sheaves\%2C+cohomology+and+higher+stacks}{this introductory page (link)}.
What is important is to keep in mind that \emph{a functor is similar to a 1-cocycle}. That's why in different fields of math people call ``cocycle conditions'' some conditions that make some mappings analogous to cocycles. More than often, those mappings are actually \emph{functors}.

The term ``cocycle condition'' is also used outside geometry and topology: for example, in stochastic processes. In particular, we have that \emph{the cocycle condition for a Markov process} can be seen as an instance of functoriality. 
\begin{eg}[probability]
 Let $X$ be a measurable space. A \emph{continuous-time Markov process} on $X$ is a collection of Markov kernels $K_{s,t}:X\to X$ for each $s\le t\in \R$ satisfying the following conditions.
 \begin{itemize}
  \item Unitarity: For each $t\in \R$, $K_{t,t}=\id_X$.
  \item Cocycle condition: for each $r\le s \le t\in \R$, 
  $$
  K_{r,s} \circ K_{s,t}\;=\; K_{r,t} .
  $$
 \end{itemize}
 These conditions are precisely saying that the assignment $(r,s)\mapsto K_{r,s}$ is a functor from $(\R,\le)$ (which is a partial order, hence a category) into the category of measurable spaces and Markov kernels. The same can be said for discrete time processes, for processes defined only for positive time, and so on.
\end{eg}
Strictly speaking, we don't really have a cohomology theory for Markov processes (at least not in an obvious way). The term ``cocycle'' here is rather used because it's \emph{analogous} to a cocycle. But actually it is a \emph{functor}, and functors are analogous to 1-cocycles. 

People with a background in geometry and topology may find the following example helpful: the (Čech) cocycle condition for vector and principal bundles is also a form of functoriality.
\begin{eg}[algebraic topology, differential geometry]\label{1.3.30}
 Let $X$ be a topological space, and let $\mathcal{U}=(U_i)_{i\in I}$ be an open cover of $X$. The \emph{Čech groupoid} $\cat{C}(\mathcal{U})$ of the cover $\mathcal{U}$ is the following groupoid.
 \begin{itemize}
  \item The objects of $\cat{C}(\mathcal{U})$ are the open sets $U_i$ of the cover $\mathcal{U}$.
  \item There is an isomorphism $U_i\to U_j$ if and only if $U_i$ and $U_j$ have nonempty intersection.
  \item Given three sets $U_i$, $U_j$ and $U_k$ with nonempty pairwise intersection, the diagram of isomorphisms
  \begin{equation}\label{cechcocycle}
  \begin{tikzcd}
   U_i \ar{rr}{\cong} && U_j \ar{dl}{\cong} \\
   & U_k \ar{ul}{\cong}
  \end{tikzcd}
  \end{equation}
  commutes if and only if the triple intersection $U_i\cap U_j \cap U_k$ is nonempty.
 \end{itemize}
 The definition of vector bundle in terms of local trivializations and transition functions can be seen as a functor $F:\cat{C}(\mathcal{U})\to\Vect$. All the objects $FU_i$ in the image of $F$ are isomorphic whenever $X$ is connected, and correspond to the \emph{fiber} of the bundle. 
 To each morphism of $\cat{C}(\mathcal{U})$, i.e.~to each intersection $U_i\cap U_j$, there corresponds an invertible linear map $F_{ij}:FU_i\to FU_j$, or  \emph{transition function}, with inverse $F_{ji}$. By functoriality, the transition functions have to satisfy:
 \begin{itemize}
  \item Identity: for $i=j$, the transition function $F_{ii}$ has to be the identity;
  \item Composition: for every commutative diagram in the form \eqref{cechcocycle}, the $F$-image of that diagram has to commute as well. In other words, whenever a triple intersection $U_i\cap U_j\cap U_k$ is nonempty, then $F_{jk}\circ F_{ij}=F_{ik}$. This is usually called the \emph{cocycle condition for transition functions}.  
 \end{itemize}
Note that a vector bundle is not simply defined as an open cover $\mathcal{U}$ together with a functor $\cat{C}(\mathcal{U})\to\cat{Vect}$, since different choices of covers may still give the same bundle. To obtain the definition of vector bundle, you have to take equivalence classes of such structures (via the usual refinements of covers).
\end{eg}
In this case there actually \emph{is} a cohomology theory where this construction is a cocycle, namely \emph{Čech cohomology}. However, to make the analogy precise, one has to take a non-Abelian version of this theory, and so, on the surface, it does not look like the usual cohomology (if you are interested, see \href{http://ncatlab.org/nlab/show/nonabelian+cohomology}{the nLab entry on nonabelian cohomology}).

\subsection{Functors, mono and epi}\label{funepimono}

We have said that functors preserve relations between objects. Let's now see what they preserve more concretely, and also what they do \emph{not} preserve. 

Let's start with a simple example. Let $F:\cat{C}\to\cat{D}$ be a functor, and consider the following diagram in $\cat{C}$:
$$
\begin{tikzcd}
 X \ar[shift left]{r}{f} \ar[shift right,swap]{r}{g} & Y
\end{tikzcd}
$$
If we apply $F$, we get a diagram in $\cat{D}$:
$$
\begin{tikzcd}
 FX \ar[shift left]{r}{Ff} \ar[shift right,swap]{r}{Fg} & FY
\end{tikzcd}
$$
Now, 
\begin{remark}
 If the first diagram commutes (in $\cat{C}$), then the second diagram commutes (in $\cat{D}$).
\end{remark}
The reason is actually trivial: the first diagram commutes if and only if $f=g$. But that implies that $Ff=Fg$, which means exactly that the second diagram commutes.
This can be done for each pair of objects in a diagram, and for each pair of arrows between them. Therefore,
\begin{cor}
 Functors preserve commutative diagrams. If a diagram commutes in $\cat{C}$, its image under a functor $F:\cat{C}\to\cat{D}$ commutes in $\cat{D}$.
\end{cor}

This is trivial, but very important. Here is another very important fact.

\begin{prop}
 Let $F:\cat{C}\to\cat{D}$ be a functor, and suppose that we have an isomorphism 
 $$
 \begin{tikzcd}
  X \ar[shift left]{r}{f} & Y \ar[shift left]{l}{g}
 \end{tikzcd}
 $$
 in $\cat{C}$. Then the diagram 
 $$
 \begin{tikzcd}
  FX \ar[shift left]{r}{Ff} & FY \ar[shift left]{l}{Fg}
 \end{tikzcd}
 $$
 is an isomorphism in $\cat{D}$. In particular, $F$ maps isomorphic objects of $\cat{C}$ to isomorphic objects of $\cat{D}$.
\end{prop}

\begin{proof}
 By definition of isomorphism, we have that 
 $$
 g\circ f = \id_X \quad \mbox{and} \quad f\circ g = \id_Y.
 $$ 
 Applying $F$ we get  
 $$
 F(g\circ f) = F(\id_X) \quad \mbox{and} \quad F(f\circ g) = F(\id_Y) ,
 $$ 
 which by functoriality of $F$ gives
 $$
 Fg\circ Ff = \id_{FX} \quad \mbox{and} \quad Ff\circ Fg = \id_{FY} .
 $$ 
 This means precisely that the pair $(Ff,Fg)$ is an isomorphism between $FX$ and $FY$.
\end{proof}

Recall that split monomorphisms and split epimorphisms satisfy by definition only one of the two conditions of isomorphism (Definitions \ref{defsplitmono} and \ref{defsplitepi}), namely, $g\circ f = \id_X$ means that $f$ is split mono, and $g$ is split epi. Therefore, the same proof gives us also the following result.

\begin{prop}
 Let $F:\cat{C}\to\cat{D}$ be a functor, and let $m:X\to Y$ be a split monomorphism of $\cat{C}$, with retraction $r:Y\to X$. Then $Fm$ is a split monomorphism of $\cat{D}$, with retraction $Fr$.
 
 Dually, let $e:X\to Y$ be a split epimorphism of $\cat{C}$, with section $s:Y\to X$. Then $Fe$ is a split epimorphism of $\cat{D}$ with section $Fs$.
\end{prop}

In particular, functors map split monomorphisms to split monomorphisms and split epimorphisms to split epimorphisms.
However,

\begin{caveat}
 If a monomorphism $f$ is not split, its image $Ff$ under a functor $f$ may not be a monomorphism. The same is true for epimorphisms.
\end{caveat}

Here is a simple counterexample.
\begin{eg}[algebra]
 The inclusion map $i:(\N,+)\hookrightarrow(\Z,+)$ is an epimorphism of the category $\cat{Mon}$ of monoids and monoid homomorphisms. Take now the forgetful functor $U:\cat{Mon}\to\Set$ mapping a monoid to the underlying set and a monoid homomorphism to the underlying function. The function $Ui:\N\to\Z$ is not surjective, so it is not an epimorphism of $\Set$.
\end{eg}

Another example, which is quite helpful in practice, is the following. We encourage readers unfamiliar with graph theory to try to understand this example and the ones that will follow on the same topic: graphs and categories are very intimately related.
\begin{eg}[graph theory]\label{graphcomponents}
 Let $\cat{Graph}$ be the category of undirected graphs and functions between the vertices preserving adjacency. There is a functor $\pi_0:\cat{Graph}\to\Set$ which:
 \begin{itemize}
  \item To each graph $G$ it assigns the set of the connected components of $G$, denoted by $\pi_0(G)$;
  \item To each morphism $f:G\to H$ it assigns the induced maps between the connected components, in the following way. For each connected component $[x]\in\pi_0(G)$ take a representative vertex $x\in[x]$. Map it with $f$ to $f(x)\in H$, and take its connected component $[f(x)]\in\pi_0(H)$. This is well-defined, since if we take another $x'\in[x]$, then there is a path of edges $x\to x'$, and since $f$ preserves adjacency, there is a path of edges $f(x)\to f(x')$. Therefore $[f(x)]=[f(x')]$. 
 \end{itemize}
 
 Consider now the morphism $f$ between the following two graphs:
 \begin{center}
 \begin{tikzpicture}[baseline=(current  bounding  box.center),>=stealth]
  \node[bullet] (tl) at (0,1) {};
  \node[bullet] (bl) at (0,0) {};
  \node[bullet] (tr) at (1,1) {};
  \node[bullet] (br) at (1,0) {};
  
  \draw[relation] (tl) -- (bl);
  \draw[relation] (tr) -- (br);
  
  \node[label=below:$G$] () at (0.5,-0.5) {};

  \node[bullet] (tl2) at (4,1) {};
  \node[bullet] (bl2) at (4,0) {};
  \node[bullet] (tr2) at (5,1) {};
  \node[bullet] (br2) at (5,0) {};
  
  \draw[relation] (tl2) -- (bl2);
  \draw[relation] (tr2) -- (br2);
  \draw[relation] (bl2) -- (br2);
  
  \node[label=below:$H$] () at (4.5,-0.5) {};

  \draw[function,dotted] (tl) to [bend left] (tl2);
  \draw[function,dotted] (tr) to [bend left] (tr2);
  \draw[function,dotted] (bl) to [bend right] (bl2);
  \draw[function,dotted] (br) to [bend right] (br2);
  
  \node[label=below:$f$] () at (2.5,-0.5) {};
 \end{tikzpicture}
 \end{center}
 The map $f:G\to H$ is injective, and injective maps preserving adjacency are monomorphisms of $\cat{Graph}$. However, the induced map $\pi_0(f):\pi_0(G)\to\pi_0(H)$ is the following function:
 \begin{center}
 \begin{tikzpicture}[baseline=(current  bounding  box.center),>=stealth]
  \node[bullet] (tl) at (0,0) {};
  \node[bullet] (bl) at (1,0) {};
  
  \node[label=below:$\pi_0(G)$] () at (0.5,-0.5) {};

  \node[bullet] (tl2) at (4.5,0) {};
  
  \node[label=below:$\pi_0(H)$] () at (4.5,-0.5) {};

  \draw[function,dotted] (tl) to [bend left] (tl2);
  \draw[function,dotted] (bl) to [bend right] (tl2);

  \node[label=below:$\pi_0(f)$] () at (2.5,-0.5) {};
 \end{tikzpicture}
 \end{center}
 which is clearly not injective. Therefore, even if $f$ is mono, $\pi_0(f)$ is not.
\end{eg}

We can use this fact to our own advantage. Suppose that $f$ were \emph{split} mono. Then $\pi_0(f)$ would have been split mono too. The fact that $\pi_0(f)$ is not mono (and so in particular not split mono) necessarily means that $f$ cannot be split. In other words,

\begin{cor}
 Let $F:\cat{C}\to\cat{D}$ be a functor. Let $f:X\to Y$ be mono (resp.~epi). Suppose that $Ff$ is \emph{not} mono (resp.~epi). Then $f$ is not split mono (resp.~split epi). In other words, $f$ does not admit a retraction (resp.~section).
\end{cor}

This is quite useful in practice: \emph{functors can be used to prove nonexistence of retractions and sections}. Proving that something does not exist is often quite hard, and functors give us a tool to do exactly that. Sometimes one says that functors \emph{detect obstructions to the existence of retractions and sections}. This is the case for example for many functors in homological algebra. 
Let's show how this is used for example in algebraic topology. The following example is a category-theoretical solution to \Cref{circleretract}. 

\begin{eg}[algebraic topology]
 Consider the inclusion map $i:S^1\to D^2$. This is injective and continuous, so it is a monomorphism of $\cat{Top}$. Let's now show that it has no retract. Fix for convenience a point $x\in S^1$ and call $y\coloneqq i(x)\in D^2$ so that $i:(S^1,x)\to (D^2,y)$ is a morphism of the category $\cat{Top}_*$ of pointed topological spaces. Consider now the fundamental group functor $\pi_1:\cat{Top}_*\to\Grp$ (\Cref{fundamentalgroup}). The fundamental group of $S^1$ is (isomorphic to) $(\Z,+)$, as loops are generated by going once around the perimeter, while the fundamental group of $D^2$ is the trivial group $0$, since the disk is contractible. Therefore the map $\pi_1(i):\pi_1(S^1)\to\pi_1(D^2)$ is the zero map $\Z\to 0$, sending everything to zero. This map is not injective, so it is not a monomorphism of $\Grp$. Therefore $i$ is not split mono, i.e.~it has no retract. 
\end{eg}

The same can be done for spheres and discs of higher dimension, using higher homotopy groups. The nonexistence of this retraction can be used, for example, to give a very short proof of Brouwer's fixed point theorem.

\subsection{What is not a functor?}

Of course, \emph{many} things are not functors. 
Here are however two examples of a construction which \emph{may seem functorial at a first look, but is not}. Given a set $X$, we can assign to it its automorphism group $\Aut(X)$, the group of all invertible functions from $X$ to itself. (If $X$ is finite this is usually called the \emph{permutation group} or \emph{symmetric group} of $X$.)

One may be tempted to think of $\Aut$ as a functor $\Set\to\Grp$, but there is no nontrivial way to make this construction functorial. The reason is what it would do on \emph{morphisms}: a function $f:X\to Y$ does not induce a well-defined mapping $\Aut(X)\to\Aut(Y)$.\footnote{Of course, one could always pick the map sending everything to the identity of $\Aut(Y)$, but there is no \emph{nontrivial} way.} 

Just as well, the set of endomorphisms $\End(X)$, i.e.~the functions from $X$ to itself, is a monoid. Again, this does not give a functor $\Set\to\cat{Mon}$.

This is true in most categories (where $\Aut$ and $\End$ are the isomorphisms and endomorphisms in the category). There are categories where this construction can be made functorial in a convenient way, but in the generic case this does not work, so $\Aut$ and $\End$ are not functors.

\begin{ex}[group theory]
 Take the category of finite sets and \emph{injective} maps. Can you make $\Aut$ functorial in a nontrivial way on this category?
\end{ex}

\subsection{Contravariant functors and presheaves}\label{presheaf}

Some constructions in mathematics seem to be functorial, except that \emph{they reverse the direction of the arrows and the order of composition}. Here is an example. Take the category of sets (or vector spaces, or topological spaces, et cetera), and fix an object, for example $\R$. Given a set (or object) $X$, we can form the \emph{space of functions} $X\to\R$. This is also a set, namely $\Hom_\Set(X,\R)$. To make this functorial, we should say what this does on morphisms. So, given a function $f:X\to Y$, we want to construct a function $\Hom_\Set(X,\R)\to \Hom_\Set(Y,\R)$ induced by $f$. However, there is no natural choice of such a function. What we \emph{can} do, instead, is induce a function the other way, namely $\Hom_\Set(Y,\R)\to \Hom_\Set(X,\R)$, in the following way. Given $g\in \Hom_\Set(Y,\R)$, that is, $g:Y\to\R$, we can \emph{precompose it} with $f$ to obtain the function $g\circ f:X\to \R$,
$$
\begin{tikzcd}
 X \ar{r}{f} & Y \ar{r}{g} & \R .
\end{tikzcd}
$$
This is assignment is like a functor $\Set\to\Set$, except that \emph{it reverses all the arrows}. Therefore, it is not quite a functor $\Set\to\Set$, but rather, a functor $\Set^\op\to\Set$. 

\begin{ex}[important!]\label{1.3.40}
 Prove that this indeed is a functor $\Set^\op\to\Set$. How is composition of arrows preserved?
\end{ex}

More in general, given categories $\cat{C}$ and $\cat{D}$ a functor $\cat{C}^\op\to\cat{D}$ behaves like a usual functor, except that the direction of the arrows is reversed. 

\begin{ex}\label{1.3.41}
 You may be wondering, why $\cat{C}^\op\to\cat{D}$ and not $\cat{C}\to\cat{D}^\op$? Prove that the two concepts are actually the same (modulo notation). 
\end{ex}

\begin{remark}
 In some areas of mathematics, instead of a functor $\cat{C}^\op\to\cat{D}$, one speaks of a \emph{contravariant functor} $\cat{C}\to\cat{D}$ (and the ordinary functors are called \emph{covariant functors}). This is analogous to what happens in linear algebra, where one could alternatively speak of covariant and contravariant vectors, or of vectors belonging to the dual space. 
 In category theory it is customary not to use the word ``contravariant'' except informally, and stick with functors $\cat{C}^\op\to\cat{D}$.
 
 Whichever convention you choose, be consistent: avoid saying things like ``a contravariant functor $\cat{C}^\op\to\cat{D}$'' -- this could lead to confusion.
\end{remark}

\begin{eg}[linear algebra]\label{egdoubledual}
 Consider the category $\Vect$ of vector spaces and linear maps. The \emph{dual space} is a functor $\Vect^\op\to\Vect$ (or, a contravariant functor) constructed in the following way, similar to the function spaces above.
 \begin{itemize}
  \item It assigns to each vector space $V$ its dual space $V^*$, that is, the vector space of linear functionals (i.e.~linear maps) $V\to\R$;
  \item It assigns to each linear map $f:V\to W$ the map $f^*:W^*\to V^*$. The map $f^*$ takes an element $\omega$ of $W^*$, which is a linear functional $\omega:W\to\R$, and gives the linear functional $V\to \R$ (i.e.~element of $V^*$) given by the composition
  $$
  \begin{tikzcd}
  V \ar{r}{f} & W \ar{r}{\omega} & \R .
  \end{tikzcd}
  $$
 \end{itemize}
 This functor reverses all the arrows. We now want to talk about the \emph{double dual}, so we apply the functor again, and this times the arrows are back to the normal direction. That is, the double dual is a functor $\Vect\to\Vect$ which explicitly works as follows.
 \begin{itemize}
  \item It assigns to each vector space $V$ its \emph{double dual space} $V^{**}$, that is, the vector space of linear functionals $V^*\to\R$;
  \item It assigns to each linear map $f:V\to W$ the map $f^{**}:V^{**}\to W^{**}$. 
  This map takes an element $k$ of $V^{**}$, which is a linear functional $k:V^*\to\R$, and gives the linear functional $W^*\to \R$ (i.e.~element of $W^{**}$) given by the composition
  $$
  \begin{tikzcd}
  W^* \ar{r}{f^*} & V^* \ar{r}{k} & \R ,
  \end{tikzcd}
  $$
  where $f^*$ is the dual map to $f$ defined above. 
 \end{itemize}
\end{eg}

Now a very important definition.

\begin{deph}
 Let $\cat{C}$ be a category. A \emph{presheaf} on $\cat{C}$ is a functor $\cat{C}^\op\to\Set$.
\end{deph}

The reason why we pick $\Set$ as the target category is that, in our setting, \emph{the hom-spaces are sets}. That is, given any two objects $X$ and $Y$ of $\cat{C}$, then $\Hom_\cat{C}(X,Y)$ is a \emph{set}.
(Of course, for this to be technically correct, we need to be in a locally small category. Let's assume we are, from now on.)

\begin{remark}
 There are settings, such as algebraic geometry, where one works in categories whose hom-sets admit additional structure. For example in the category of abelian groups the hom-sets can be themselves considered abelian groups. In that case, one may be interested in functors to $\cat{AbGrp}$ instead of $\Set$. 
 In other applications, one may be interested in different structures that one can put on the hom-sets. These additional structures are called \emph{enrichments} and are the topic of \emph{enriched category theory}. In this course we will only study ordinary category theory, where hom-sets are just sets. (If you are interested in enriched category theory, this course can still help you: most concepts in ordinary category theory have analogues in the enriched case.)
\end{remark}

The ``function space functor'' $\Hom_{\Set}(-,\R)$ is for example a presheaf on $\Set$. More generally, given a category $\cat{C}$ and an object $X$ of $\cat{C}$, one can analogously form a presheaf $\Hom_\cat{C}(-,X):\cat{C}^\op\to\Set$. 

Not all presheaves are in this form. Here is a more general example, from basic graph theory. 

\begin{eg}[graph theory]\label{egmultigraph}
 Let $\cat{Par}$ be the category consisting only of the objects and morphisms specified by this diagram of ``parallel arrows''
 $$
 \begin{tikzcd}
  V \ar[shift left]{r}{s} \ar[shift right, swap]{r}{t} & E
 \end{tikzcd}
 $$
 plus the implicit identities at the two objects. A presheaf on $\cat{Par}$, that is, a functor $F:\cat{Par}^\op\to\Set$, consists explicitly of the following.
 \begin{itemize}
  \item Two distinguished sets $FV$ and $FE$;
  \item Two functions $Fs,Ft:FE\to FV$ (mind the direction).
 \end{itemize}
 These can be interpreted as \emph{directed multigraphs}. The set $FV$ is the set of vertices, the set $FE$ is the set of edges, and the maps $Fs$ and $Ft$ associate to each edge its source and target, respectively.
 In this case, between any two vertices there can be multiple edges, and there can be loops. (It looks a bit like a category, but mind that identities and compositions here are missing.)
\end{eg}

Presheaves come originally from algebraic geometry, but are not only useful in algebraic geometry. As the example above shows, they can also be used to understand graphs. This will be of use later on.

Here is another example of presheaf, which historically has been the first, and which still is very important. 

\begin{eg}[topology]\label{presheavesonopen}
 Let $X$ be a topological space. Let $O(X)$ be the topology of $X$, i.e.~the poset of open subsets of $X$, ordered by inclusion. As a poset, $O(X)$ can be thought of as a category. We can form a presheaf $O(X)^\op\to\Set$ in the following way.
 \begin{itemize}
  \item To each open subset $U\in O(X)$ we associate the set of continuous functions $U\to\R$.
  \item To each inclusion $U\subseteq V$, which is a morphism of $O(X)$, we need a mapping from functions $V\to\R$ to functions $U\to\R$. As such mapping we take the restriction: given $g:V\to\R$, we can restrict $g$ to the subset $U\subseteq V$, to get the function $g|_U:U\to \R$. This is again continuous.
 \end{itemize}
This preserves the identities and composition, therefore it is a functor $O(X)^\op\to\Set$, i.e.~a presheaf on $O(X)$. Note that, by construction, this reverses the arrows.
\end{eg}

\begin{remark}
 The term \emph{presheaf} is one of the few words in category theory whose translation to other languages is not obvious. In German the word is \emph{Prägarbe}, in French \emph{prefaisceau}, in Spanish \emph{prehaz}, in Italian \emph{prefascio}. It derives from the word \emph{sheaf}, which is itself inspired by agriculture (such as in \emph{a sheaf of grain}). We will not cover sheaves in this course. 
\end{remark}

\section{Natural transformations}\label{sec_nat}

\begin{deph}
 Let $\cat{C}$ and $\cat{D}$ be categories, and let $F$ and $G$ be functors $\cat{C}\to\cat{D}$. A \emph{natural transformation $\alpha$ from $F$ to $G$}, consists of the following data.
 \begin{itemize}
  \item For each object $C$ of $\cat{C}$, a morphism $\alpha_C:FC\to GC$ in $\cat{D}$, called the \emph{component of $\alpha$ at $C$};
  \item For each morphism $f:C\to C'$ of $\cat{C}$, the following diagram has to commute (\emph{naturality condition}):
  $$
  \begin{tikzcd}
   FC \ar{r}{Ff} \ar{d}{\alpha_C} & FC' \ar{d}{\alpha_{C'}} \\
   GC \ar{r}{Gf} & GC' 
  \end{tikzcd}
  $$
 \end{itemize}
\end{deph}

We denote natural transformations as double arrows, $\alpha: F \Rightarrow G$, to distinguish them in diagrams from functors (which are denoted by single arrows):
 $$
 \begin{tikzcd}[column sep=large]
  \cat{C} \ar[bend left,""{below,name=F}]{r}{F} \ar[bend right,""{above,name=G}]{r}[swap]{G} & \cat{D}
  \ar[Rightarrow,from=F,to=G,"\alpha"]
 \end{tikzcd}
 $$
The notation suggests that these are ``arrows between functors''. We will make this precise shortly. Let's first give some intuition.

\subsection{Natural transformations as systems of arrows}

\begin{idea} 
 A natural transformation is a consistent system of arrows between the images of two functors.
\end{idea}

\begin{eg}[sets and relations]\label{egpointwiseorder}
 Let $(X,\le)$ and $(Y,\le)$ be partial orders, and consider monotone maps $f,g:(X,\le)\to(Y,\le)$. We can view $f$ and $g$ as functors between two categories, where the arrows of $Y$ (and $X$) are just inequalities $y\le y'$. A natural transformation between $f$ and $g$ is, for each $x\in X$, an arrow $f(x)\to g(x)$, that is, for each $x\in X$, we must have $f(x)\le g(x)$. Traditionally, one says that $f\le g$ pointwise. The naturality condition is trivial here, since in a partial order all diagrams commute.
 So there is a unique natural transformation $f\Rightarrow g$ if and only if $f\le g$ pointwise.
\end{eg}

\begin{ex}[sets and relations]
 Suppose that $Y$ is an equivalence relation instead of a partial order. What do natural transformations between functors into $Y$ look like?
\end{ex}

\subsection{Natural transformations as structure-preserving mappings}\label{equivariant}

\begin{idea} 
 A natural transformation is a mapping between functors preserving specified actions, symmetries, or other structures. 
\end{idea}

\begin{eg}[group theory]\label{egmorphrepr}
 Let $G$ be a group. We know (\Cref{egrepresentation}) that a linear representation of $G$ is a functor $\cat{B}G\to \Vect$, mapping the single object of $\cat{B}G$ to a specified vector space $V$, on which the group $G$ ``acts''. 
 Let now $R,S:\cat{B}G\to \Vect$ be linear representations acting on $V$ and $W$, respectively. A natural transformation $\alpha:R\to S$ consists of the following.
 \begin{itemize}
  \item For each object of $\cat{B}G$ we need a morphism of $\Vect$ between the images of this object under $R$ and $S$. As $\cat{B}G$ has a single object and its images under $R$ and $S$ are $V$ and $W$, this amounts to a linear map $\alpha:V\to W$.
  \item For each morphism of $\cat{B}G$, that is, for each $g\in G$, the following diagram has to commute:
  $$
  \begin{tikzcd}
   V \ar{d}{\alpha} \ar{r}{Rg} & V \ar{d}{\alpha} \\
   W \ar{r}{Sg} & W
  \end{tikzcd}
  $$
  That is, for each vector $v\in V$, acting with $g$ commutes with the map $\alpha$. Denoting the two actions of $G$ by just $g\cdot$, this reads
  $$
  \alpha( g\cdot v) \; = \;  g\cdot \alpha(v).
  $$
 \end{itemize}
 A map commuting with a group action is called an \emph{equivariant map}. The natural transformations between two representation of $G$ are then the linear $G$-equivariant maps. In representation theory, these are also known as \emph{morphisms of representations}. 
\end{eg}

\begin{eg}[group theory]\label{egmorphsym}
 A group $G$ acting on a space encodes an idea of \emph{symmetry} of the given space. Analogously with the linear case, a \emph{$G$-space} is a functor $\cat{B}G\to\cat{Top}$. 
 For example, the group $D_3$, the group of symmetries of an equilateral triangle, can also act on the following spaces via rotations and reflections:
 \begin{center}
 \def\ra{3.1}
  \begin{tikzpicture}[scale=0.5]
   \begin{polaraxis}[grid=none, axis lines=none]
     \addplot[thick,mark=none,domain=0:360,samples=300,white] {1};
     \addplot[thick,mark=none,domain=0:360,samples=300] {-sin(x*3)};
   \end{polaraxis}
   
   \node[fill,circle,inner sep=0pt,minimum size=2pt] (c) at (16,3.45) {};
   \draw[thick] (c) -- ++(90:\ra);
   \draw[thick] (c) -- ++(210:\ra);
   \draw[thick] (c) -- ++(330:\ra);
   \draw[thick] (c) circle (\ra);
 \end{tikzpicture}
\end{center}
 The natural transformations between $G$-spaces are exactly the $G$-equivariant continuous maps. These can be interpreted as the maps which respect the symmetry: take for example denote by $\phi$ the map from the figure on the left to the figure on the right which collapses the petals into segments, and then includes them into figure on the right.
 \begin{center}
 \def\ra{3.1}
  \begin{tikzpicture}[scale=0.5]
   \begin{polaraxis}[grid=none, axis lines=none]
     \addplot[thick,mark=none,domain=0:360,samples=300,white] {1};
     \addplot[thick,mark=none,domain=0:360,samples=300] {-sin(x*3)};
   \end{polaraxis}
   
   \draw[->] (6.8,3.45) -- (8.8,3.45);
   
   \node[fill,circle,inner sep=0pt,minimum size=2pt] (c) at (12,3.45) {};
   \draw[thick] (c) -- ++(90:\ra);
   \draw[thick] (c) -- ++(210:\ra);
   \draw[thick] (c) -- ++(330:\ra);
   
   \draw[->] (15.7,3.45) -- (17.7,3.45);
   
   \node[fill,circle,inner sep=0pt,minimum size=2pt] (d) at (22,3.45) {};
   \draw[thick] (d) -- ++(90:\ra);
   \draw[thick] (d) -- ++(210:\ra);
   \draw[thick] (d) -- ++(330:\ra);
   \draw[thick] (d) circle (\ra);
 \end{tikzpicture}
\end{center}
 Rotating $120$ degrees counterclockwise before and after applying the map gives the same result, as the following illustration shows,
 \begin{center}
 \def\ra{3.1}
 \begin{tikzcd}[column sep=huge, row sep=large]
  \begin{tikzpicture}[scale=0.5, baseline=(current  bounding box.center)]
   \begin{polaraxis}[grid=none, axis lines=none]
     \addplot[thick,mark=none,domain=0:360,samples=300,white] {1};
     \addplot[thick,mark=none,domain=0:360,samples=300] {-sin(x*3)};
     \node[fill,circle] (p) at (90,1) {};
   \end{polaraxis}
  \end{tikzpicture}
  \ar{r}{\phi} \ar{d}{\circlearrowleft 120^{\circ}} & 
  \begin{tikzpicture}[scale=0.5, baseline=(current  bounding box.center)]
   \node[fill,circle,inner sep=0pt,minimum size=2pt] (d) at (0,0) {};
   \draw[thick] (d) -- ++(90:\ra);
   \draw[thick] (d) -- ++(210:\ra);
   \draw[thick] (d) -- ++(330:\ra);
   \draw[thick] (d) circle (\ra);
   \node[fill,circle,inner sep=0pt,minimum size=6pt] (p) at (90:\ra) {};
 \end{tikzpicture}
  \ar{d}{\circlearrowleft 120^{\circ}} \\
  \begin{tikzpicture}[scale=0.5, baseline=(current  bounding box.center)]
   \begin{polaraxis}[grid=none, axis lines=none]
     \addplot[thick,mark=none,domain=0:360,samples=300,white] {1};
     \addplot[thick,mark=none,domain=0:360,samples=300] {-sin(x*3)};
    \node[fill,circle] (p) at (210,1) {};
   \end{polaraxis}
  \end{tikzpicture}
  \ar{r}{\phi} & 
  \begin{tikzpicture}[scale=0.5, baseline=(current  bounding box.center)]
   \node[fill,circle,inner sep=0pt,minimum size=2pt] (d) at (0,0) {};
   \draw[thick] (d) -- ++(90:\ra);
   \draw[thick] (d) -- ++(210:\ra);
   \draw[thick] (d) -- ++(330:\ra);
   \draw[thick] (d) circle (\ra);
   \node[fill,circle,inner sep=0pt,minimum size=6pt] (p) at (210:\ra) {};
 \end{tikzpicture}
 \end{tikzcd}
\end{center}
 where we marked with a dot one of the petals to make the rotation visible. 
\end{eg}

\begin{eg}[dynamical systems]\label{egmorphdynsys}
 Let $M$ be a monoid, for example $(\N,+)$. A \emph{topological dynamical system with monoid $M$} is a functor $\cat{B}M\to\cat{Top}$. This picks a space $X$ and lets the dynamics indexed by the monoid $M$ act on the given space $X$. Given dynamical systems $F,G:\cat{B}M\to\cat{Top}$ acting on the spaces $X$ and $Y$ respectively, a \emph{morphism of dynamical systems} is a natural transformation $\alpha:F\to G$. This amounts to the following: first of all, we need a map $\alpha:X\to Y$. Moreover, for each $m\in M$, which we can interpret as a time interval, the following diagram has to commute.
 $$
  \begin{tikzcd}
   X \ar{d}{\alpha} \ar{r}{Fm} & X \ar{d}{\alpha} \\
   Y \ar{r}{Gm} & Y
  \end{tikzcd}
  $$
 In a way, this means that $\alpha$ is compatible with the dynamics.
 
 Keep in mind that, in the field of dynamical systems,
 \begin{itemize}
  \item People prefer to talk about semigroups rather than monoids (see \Cref{semigroup});
  \item People may only want to restrict to \emph{compact} spaces instead of all topological spaces;
  \item People are mostly interested in \emph{surjective} maps $\alpha:X\to Y$.
 \end{itemize}
\end{eg}

\begin{eg}[graph theory]\label{egmorphgraph}
 We have seen in \Cref{egmultigraph} that presheaves on the category $\cat{Par}$ (defined there)
 can be seen as directed multigraphs. Now let $F,G:\cat{Par}^\op\to\Set$ be presheaves (i.e.~multigraphs). A natural transformation $\alpha:F\to G$ consists of two functions $\alpha_E:FE\to GE$ and $\alpha_V:FV\to GV$ such that the following diagram commutes.
 $$
 \begin{tikzcd}
  & FE \ar{rrr}{\alpha_E} \ar{ddl}[swap]{Fs} \ar{dr}{Ft} &&& GE  \ar{dr}{Gt} \\
  && FV \ar{rrr}[near start]{\alpha_V} &&& GV\\
  FV \ar{rrr}{\alpha_V} &&& GV \ar[leftarrow, crossing over]{uur}[near end]{Gs}
 \end{tikzcd}
 $$
 This means that, for an edge $e\in FE$, the source of $\alpha_E(e)$ is exactly the result of applying $\alpha_V$ to the source of $e$. The same is true for targets. In other words, $\alpha$ is a mapping between edges and vertices \emph{preserving the incidence relation}.
\end{eg}

\subsection{Natural transformations as canonical maps}

\begin{idea}
 A natural transformation is a mapping or assignment which is canonical, systematic, or ``natural'' (hence the name).
\end{idea}

\begin{eg}[sets and relations]\label{egsingletonmap}
 Consider the power set functor $P:\Set\to\Set$ of \Cref{egpowerset}. Given a set $X$, there is a canonical embedding of $X$ into its power set $PX$, namely the map $\sigma:X\to PX$ which assign to each $x\in X$ the singleton $\{x\}\in PX$. This map is natural in the following sense. Let $f:X\to Y$ be any function. Then the following diagram commutes:
 $$
 \begin{tikzcd}
  X \ar{r}{f} \ar{d}{\sigma} & Y \ar{d}{\sigma} \\
  PX \ar{r}{Pf} & PY
 \end{tikzcd}
 $$
 Indeed, let $x\in X$. Then on one side of the diagram,
 $$
 \sigma(f(x)) \;=\; \{f(x)\} ,
 $$
 and on the other side,
 $$
 Pf(\sigma(x)) \;=\; \{f(x)\} ,
 $$
 recalling that the action of $Pf$ on a set is just applying $f$ to all the elements of the set (in this case the set $\{x\}$ has only one element).
 
 Note that the top of the diagram has $f:X\to Y$, without any functor applied to it, or, with \emph{the identity functor} applied to it. Therefore $\sigma$ is a natural transformation from the identity functor on $\Set$ to $P$, that is, $\sigma:\id_\Set\Rightarrow P$. 
\end{eg}

What does it mean that the map above is \emph{canonical}? This means the following. Let $p:X\to X$ be a permutation, i.e.~a bijective map, which we can view as a ``relabeling'' of the elements of $X$. By naturality of $\sigma$, the following diagram commutes:
$$
 \begin{tikzcd}
  X \ar{r}{p} \ar{d}{\sigma} & X \ar{d}{\sigma} \\
  PX \ar{r}{Pp} & PX
 \end{tikzcd}
 $$
This can be interpreted as follows: if we relabel the elements of $X$ and we relabel the elements of $PX$ (i.e.~the subsets of $X$) accordingly, then the map $\sigma$ will look the same after the relabeling. In other words, the assignment $\sigma:X\to PX$ \emph{does not depend on the way we label the elements}. In this sense the map $\sigma$ is canonical. 

Here is an analogous map for the probability functor of~\Cref{egprob}. Again, we first give it for the finitely-supported case, without using measure theory -- the equivalents for more general probability measures are given in the next exercises. 

\begin{ex}[basic probability]\label{setdelta}
 Consider the probability functor $\mathcal{P}$ on $\Set$, given in \Cref{egprob}. 
 Given $x\in X$ define $\delta_x\in PX$ to be the function 
 $$
 \delta_x(y) \;\coloneqq\;
 \begin{cases}
  1 & x=y \\
  0 & x\ne y .
 \end{cases}
 $$
 Show that the map $\delta:X\to \mathcal{P}X$ assigning to each $x\in X$ the measure $\delta_x$ is a natural transformation from the identity functor to $\mathcal{P}$, analogous to the singleton map for the power set.
\end{ex}

Let's now see the analogous for general probability measures, not necessarily finitely supported.

\begin{ex}[measure theory, probability]\label{measdelta}
 Let $X$ be a measurable space, and let $x\in X$. Recall that the \emph{Dirac measure at $x$} is defined to be the one mapping each measurable set $A\subseteq X$ to 
 $$
 \delta_x(A) \;\coloneqq\;
 \begin{cases}
  1 & x\in A \\
  0 & x\notin A .
 \end{cases}
 $$
 Consider now either the Giry functor $\mathcal{P}$ on $\cat{Meas}$, given in \Cref{giryfunctor}, or the Radon functor $\mathcal{P}$ on $\cat{CHaus}$, given in \Cref{radonfunctor}.
 Show that in both cases, the map $\delta:X\to \mathcal{P}X$ assigning to each $x\in X$ the measure $\delta_x$ is again a natural transformation from the identity functor to $\mathcal{P}$.
\end{ex}

\begin{ex}[basic computer science, combinatorics]\label{singlelist}
 Construct a natural transformation from the identity functor to the list functor of \Cref{listfunctor}, analogous to the ones given in the exercises above.
\end{ex}

A similar phenomenon to the ``natural maps'' above can be seen also in vector spaces. In particular, you probably know from linear algebra that a finite-dimensional vector space $V$ is always isomorphic to its dual $V^*$, and so also to the dual of the dual, i.e.~the double dual $V^{**}$. However, in some way, $V^{**}$ is ``a lot more similar to $V$ than $V^*$ is''. That's why sometimes (such as in differential geometry) $V^{**}$ is directly identified with $V$. (Hence the term ``dual'': it's as if there were only two of them, $V$ and $V^*$.)
In terms of natural transformation we can make this intuition precise: there is a \emph{natural} isomorphism $V\to V^{**}$ (but not $V\to V^*$). 
 
\begin{eg}[linear algebra]
 Consider the double dual functor $(-)^{**}:\Vect\to\Vect$ of \Cref{egdoubledual}, mapping a vector space $V$ to its double dual $V^{**}$, i.e.~the space of linear functionals $V^*\to\R$. Every element $v$ of $V$ defines canonically an element of the double dual $V^{**}$ by \emph{evaluation}. Let's see what this means. An element of $V^{**}$ is a linear functional $V^*\to \R$. We can map a linear functional $\omega\in V^*$ to $\R$ by just feeding it $v$, that is: $\omega\mapsto \omega(v)\in\R$. This way from $v\in V$ we get a linear map $V^*\to\R$, i.e.~an element of $V^{**}$. Let's call this map $\eta:V\to V^{**}$.
 Let's now prove that this map is natural. This means that for every linear map $f:V\to W$, the following diagram must commute:
 $$
 \begin{tikzcd}
  V \ar{r}{f} \ar{d}{\eta} & W \ar{d}{\eta} \\
  V^{**} \ar{r}{f^{**}} & W^{**}
 \end{tikzcd}
 $$
 That is, for each $v\in V$ we need to prove that $\eta(f(v))$ and $f^{**}(\eta(v))$ are equal (as elements of $W^{**}$). Now elements of $W^{**}$ are linear maps $W^*\to\R$, so let $\omega\in W^*$. We have to prove that 
 $$
 \eta(f(v))(\omega) \;=\; f^{**}(\eta(v))(\omega).
 $$
 We have, by the definition of $\eta$:
 $$
 \eta(f(v))(\omega) \;=\; \omega(f(v)) .
 $$
 On the other hand, by definition of $f^{**}$ (\Cref{egdoubledual}), again of $\eta$, and of $f^*$,
 $$
 f^{**}(\eta(v))(\omega) \;=\; \eta(v)(f^*\omega) \;=\; (f^*\omega)(v) \;=\; \omega(f(v)) .
 $$
 Therefore $\eta$ is natural, i.e.~it is a natural transformation from the identity of $\Vect$ to the double dual functor.
 
 You probably know from linear algebra that for finite-dimensional vector spaces, $\eta$ is even an isomorphism. Therefore a finite-dimensional vector space $V$ is \emph{naturally} isomorphic to its double dual $V$. If the space is not finite-dimensional, the map $\eta$ is not an isomorphism, but it is still natural. 
\end{eg}

Again, let's see why $\eta$ is ``canonical''. Let $c:V\to V$ be a linear isomorphism, which we can interpret as a ``change of coordinates''. Then the following diagram commutes.
$$
 \begin{tikzcd}
  V \ar{r}{c} \ar{d}{\eta} & V \ar{d}{\eta} \\
  V^{**} \ar{r}{c^{**}} & V^{**}
 \end{tikzcd}
 $$
We can interpret this as follows: if we change the coordinates of $V$, and also change those of $V^{**}$ accordingly, then the map $\eta$ will look the same after the change of coordinates. That is, $\eta$ \emph{does not depend on the choice of coordinates.}
In particular, if $V$ is finite-dimensional, then $\eta$ is an isomorphism identifying the elements of $V$ with those of $V^{**}$, and this identification does not depend on the choice of coordinates. Again, in this sense the map $\eta$ is canonical. This idea is made precise by the naturality condition -- actually, the naturality condition is precisely this idea, except that the map $c$ is allowed to be an arbitrary morphism, not necessarily an isomorphism.

For finite-dimensional vector spaces, we have seen that the map $\eta$ is an isomorphism. Such natural transformations are called natural isomorphisms:

\begin{deph}
 Let $\cat{C}$ and $\cat{D}$ be categories, let $F,G:\cat{C}\to\cat{D}$ be functors, and let $\alpha:F\Rightarrow G$ be a natural transformation. We call $\alpha$ a \emph{natural isomorphism} if for each object $C$ of $\cat{C}$, the component $\alpha_C:FC\to GC$ is an isomorphism.
 In that case, we also say that the functors $F$ and $G$ are \emph{(naturally) isomorphic}.
\end{deph}

Some authors use the term \emph{natural equivalence} instead of \emph{natural isomorphism}. Equivalence is however a different concept from isomorphism in category theory (we will study it in \Cref{equcat}), so we will not use the term \emph{equivalence} in this context.

\subsection{What is not natural?}

In order to understand better what naturality really means, here are examples of constructions which may seem natural, but are not. 

\begin{eg}[linear algebra]\label{egsingledual}
 The \emph{single} dual $V^*$ is not naturally isomorphic to $V$, not even in finite dimension. A reason is that, on the category of vector spaces, the single dual functor is contravariant, it reverses all the arrows (see \Cref{egdoubledual}). Therefore there can be no natural transformation from the identity to the single dual: one is a functor $\Vect\to\Vect$, the other one is a functor $\Vect^\op\to\Vect$. 
 But there is a deeper reason: even if we try to generalize the notion of naturality appropriately, all isomorphisms $V\to V^*$ that one can construct (in finite dimension) involve the choice of a basis (such as the dual basis), and this is not natural. 
 
 For example, let $V$ be an $n$-dimensional vector space, and let $e_1,\dots,e_n$ be a basis of $V$. We can construct the dual basis $e^1,\dots,e^n$ of $V^*$, as the unique one such that $e^i(e_j)=\delta^i_j$. 
 Given these bases, one can construct an isomorphism $d:V\to V^*$ by setting $d(e_i)\coloneqq e^i$ for all $i$. This corresponds to ``copying the coordinates'', in the sense that we map a vector $v$ of $V$ to the vector of $V^*$ that \emph{has the same coordinates as $V$} (where the coordinates in $V$ are with respect to the basis $e_1,\dots,e_n$, and the coordinates in $V^*$ are with respect to the dual basis). That is, almost by definition,
 $$
 d\left( \sum_{i=1}^n v^i e_i \right) \; = \; \sum_{i=1}^n v^i e^i .
 $$
 However, this assignment depends crucially on the choice of the first basis, a different choice gives us a different map (show this!). Therefore, the map $d$ is not natural.
\end{eg}

\begin{ex}[linear algebra]
 You may know from linear algebra that a finite-dimensional vector space equipped with a positive-definite scalar product is ``canonically'' isomorphic to its dual space. Can you make this statement precise in terms of natural transformations? (The question mostly asks: in which category?)
\end{ex}

Here is another well-known failure of a map to be natural in linear algebra (and quantum mechanics).

\begin{eg}[linear algebra, quantum physics]\label{nocloning}
 Consider the category $\Vect$. The tensor product of vector spaces gives us a functor $F:\Vect\to\Vect$ sending:
 \begin{itemize}
  \item A vector space $V$ to the vector space $FV\coloneqq V\otimes V$;
  \item A linear map $f:V\to W$ to the linear map $Ff:V\otimes V \to W\otimes W$, defined on a basis as 
  $$
  Ff(e_i\otimes e_j) \;\coloneqq\; f(e_i)\otimes f(e_j) .
  $$
 \end{itemize}
 The expression of the map $f\otimes f$ does not depend on the choice of the coordinates used (prove it!), so that we have a well-defined functor. 
 
 However, suppose we want a natural transformation $\alpha:\id\Rightarrow F$, of components $\alpha:V\to V\otimes V$. It turns out that \emph{there is no such natural map except the zero map}. This is known in the quantum information literature as the \emph{no-cloning theorem}, since it can be interpreted in quantum physics as the fact that we cannot duplicate a quantum state. 
 One may be tempted to define such a map, for example, in the following way using coordinates:
 $$
 \alpha\left( \sum_{i=1}^n v^i e_i \right) \; \coloneqq \; \sum_{i=1}^n v^i e_i\otimes e_i.
 $$
 However, just as in \Cref{egsingledual}, this map depends crucially on the choice of the basis, and so it is not natural. In physics, this means that this map cannot encode meaningful physical information.
\end{eg}

\subsection{Functor categories and diagrams}\label{functordiagrams}

\begin{deph}
 Let $\cat{C}$ and $\cat{D}$ be categories. The \emph{functor category} $[\cat{C},\cat{D}]$ is constructed as follows.
 \begin{itemize}
  \item Objects are functors $F:\cat{C}\to\cat{D}$;
  \item Morphisms are natural transformations $\alpha:F\Rightarrow G$.
 \end{itemize}
\end{deph}

Note that the functors are the \emph{objects} of the category, not the morphisms. 

\begin{ex}[important!]\label{vertcomp}
 Prove that $[\cat{C},\cat{D}]$ is indeed a category, where the identity of a functor $F$ is the natural transformation of components $\id_{FC}:FC\to FC$, and the composition of two natural transformations 
 $\alpha:F\Rightarrow G$ and $\beta:G \Rightarrow H$ is the natural transformation $\beta\circ\alpha :F\Rightarrow H$ of components 
 $$
 \begin{tikzcd}
  FC \ar{r}{\alpha_C} & GC \ar{r}{\beta_C} & HC .
 \end{tikzcd}
 $$
 Why are the identity and the composite natural transformations natural again?
\end{ex}

\begin{eg}[several fields]\label{fcats}
 The functor category generalizes the following known constructions.
 \begin{itemize}
  \item If $\cat{D}$ is a partial order, $[\cat{C},\cat{D}]$ is the partial order of functions into $\cat{D}$ given by the pointwise order (see \Cref{egpointwiseorder}).
  \item If $\cat{C}=\cat{B}G$ for a group $G$ and $\cat{D}=\Vect$, the category $[\cat{C},\cat{D}]$ is the category of linear representations of $G$ and $G$-equivariant (linear) maps (see \Cref{egmorphrepr}).
  \item Analogously, if instead $\cat{D}=\cat{Top}$, the category $[\cat{C},\cat{D}]$ is the category of $G$-spaces and $G$-equivariant (continuous) maps (see \Cref{egmorphsym}).
  \item Again in the same spirit, if instead $\cat{C}=\cat{B}M$ for $M$ a monoid, the category $[\cat{C},\cat{D}]$ is the category of dynamical systems with dynamics indexed by $M$, and their morphisms (see \Cref{egmorphdynsys}).
  \item If $\cat{C}$ is the category $\cat{Par}$ with only two parallel arrows of \Cref{egmultigraph}, then $[\cat{C},\cat{D}]$ is the category of directed multigraphs and incidence-preserving mappings (see \Cref{egmorphgraph}). We denote such category by $\cat{MGraph}$.
 \end{itemize}
\end{eg}

\begin{caveat}
 The functor category $[\cat{C},\cat{D}]$ may be not locally small, even if $\cat{C}$ and $\cat{D}$ are (why?). 
\end{caveat}

Now that we know about functors, we are also ready for the rigorous definition of a diagram. We had defined diagrams explicitly, but informally, in \Cref{infdiagrams}. Here is the real, but more abstract, definition.

\begin{deph}[this time for real]
 Let $\cat{C}$ be a category, and $\cat{I}$ be a small category.\footnote{Some authors allow diagrams to be indexed by any category, not necessarily small.} A \emph{diagram in $\cat{C}$ of shape $\cat{I}$} is a functor $\cat{I}\to\cat{C}$. 
 
 The \emph{category of $\cat{I}$-shaped diagrams in $\cat{C}$} is the functor category $[\cat{I},\cat{C}]$.
\end{deph}

This is analogous to how one defines a loop in a topological space $X$ as a continuous map $S^1\to X$, or a curve as a continuous map $\R\to X$. We have a category $\cat{I}$ of a certain shape, for example
$$
\begin{tikzcd}
 & \bullet \ar{dr} \ar{dl} \\
 \bullet \ar{rr} && \bullet 
\end{tikzcd}
$$
(identity and compositions are not drawn, but are there!), and therefore a functor $D:\cat{I}\to\cat{C}$ will pick a diagram in $\cat{C}$ as its ``image'', such as
$$
\begin{tikzcd}
 & X \ar{dr}{g} \ar{dl}[swap]{f} \\
 Y \ar{rr}{h} && Y 
\end{tikzcd}
$$
Note that $D$ doesn't need to be injective in any sense, therefore objects and morphisms can appear more than once in the diagram. 
Also note that this triangle doesn't need to be commutative. However, if the original triangle in $\cat{I}$ commutes, then this has to commute in $\cat{C}$ too (and we have a commutative diagram).

A morphism of diagrams $\alpha:D\Rightarrow D'$ is a natural transformation. This means that we have a diagram:
$$
\begin{tikzcd}[row sep=small]
 & X \ar{dd}[near end]{\alpha} \ar{dr}{g} \ar{dl}[swap]{f} \\
 Y\ar{dd}{\alpha} \ar{rr}[near end]{h} && Y \ar{dd}{\alpha} \\
 & A \ar{dr}{m} \ar{dl}[swap]{n} \\
 B \ar{rr}{p} && C 
\end{tikzcd}
$$
such that all the vertical squares (the ones involving $\alpha$) are commutative.
Readers familiar with algebraic topology may find it helpful to think of this diagram as something analogous to a \emph{homotopy} between two embeddings $S^1\to X$.

\subsection{Whiskering and horizontal composition}

Consider three categories $\cat{C},\cat{D},\cat{E}$, a functor $F:\cat{C}\to\cat{D}$, functors $H,I:\cat{D}\to\cat{E}$, and a natural transformation $\beta:H\Rightarrow I$. The situation is better represented as the following diagram.
$$
 \begin{tikzcd}[column sep=large]
  \cat{C} \ar{r}{F} & \cat{D} \ar[bend left,""{below,name=H}]{r}{H} \ar[bend right,""{above,name=I}]{r}[swap]{I} & \cat{E}
  \ar[Rightarrow,from=H,to=I,"\beta"]
 \end{tikzcd}
$$
Then we also have a natural transformation $H\circ F \Rightarrow I\circ F$, that is, between the top composition  and the bottom composition in the diagram. This is obtained as follows. The natural transformation $\beta$ has as components arrows $\beta_D:HD\to ID$ for each object $D$ of $\cat{D}$. In particular, if we take a $D$ in the image of $F$, that is $D=FC$ for some object $C$ of $\cat{C}$, there is an arrow $\beta_{FC}:HFC\to IFC$ of $\cat{E}$. We can define therefore a natural transformation, which we denote by $\beta F: H\circ F \Rightarrow I\circ F$, where
\begin{itemize}
 \item Its component at each object $C$ of $\cat{C}$, is the arrow given by the component $\beta_{FC}:HFC\to IFC$ of $\beta$, applied to $FC$;
 \item For each arrow $f:C\to C'$ of $\cat{C}$, the naturality diagram 
 $$
 \begin{tikzcd}
  HFC \ar{r}{HFf} \ar{d}{\beta_{FC}} & HFC' \ar{d}{\beta_{FC'}} \\
  IFC \ar{r}{HFf} & IFC' 
 \end{tikzcd}
 $$
 commutes, because of the naturality of $\beta$.
\end{itemize}
We call the natural transformation $\beta F$ the \emph{whiskering} of the natural transformation $\beta$ on the left by the functor $F$, since the functor $F$ is added as a sort of ``whisker'' on the left of $\beta$. 

Analogously, now suppose that we have the situation depicted in the following diagram.
$$
 \begin{tikzcd}[column sep=large]
  \cat{C} \ar[bend left,""{below,name=F}]{r}{F} \ar[bend right,""{above,name=G}]{r}[swap]{G} & \cat{D}
  \ar{r}{H} & \cat{E}
    \ar[Rightarrow,from=F,to=G,"\alpha"] 
 \end{tikzcd}
$$
Again, we get a natural transformation $HF\Rightarrow HG$, which we denote as $H\alpha$, as follows.
For each object $C$ of $\cat{C}$, the natural transformation $\alpha$ gives an arrow $\alpha_C:FC\to GC$ of $\cat{D}$. We can now apply $H$ to this arrow to get the arrow $H(\alpha_C):HFC\to HGC$ of $\cat{E}$. This is how the natural transformation $H\alpha:HF\Rightarrow HG$ is constructed.
\begin{itemize}
 \item Its component at each object $C$ of $\cat{C}$ is given by the arrow $H\alpha_C:HFC\to HGC$ of $\cat{E}$;
 \item For each arrow $f:C\to C'$ of $\cat{C}$, consider the following diagrams.
 $$
  \begin{tikzcd}
   FC \ar{r}{Ff} \ar{d}{\alpha_C} & FC' \ar{d}{\alpha_{C'}} \\
   GC \ar{r}{Gf} & GC' 
  \end{tikzcd} 
  \qquad
  \begin{tikzcd}
   HFC \ar{r}{HFf} \ar{d}{H\alpha_C} & HFC' \ar{d}{H\alpha_{C'}} \\
   HGC \ar{r}{HGf} & HGC' 
  \end{tikzcd} 
  $$
 The diagram on the left commutes, by naturality of $\alpha$. The diagram on the right can be obtained from the diagram on the left by applying the functor $H$, and functors preserve commutative diagrams, therefore the diagram on the right commutes too. The right diagram is precisely the naturality condition for $H\alpha_C$.
\end{itemize}
We call the natural transformation $H\alpha_C$ the \emph{whiskering} of $\alpha$ on the right by $H$.
Note that for each object $C$ of $\cat{C}$, the map $H\alpha_C$ may be interpreted both as $H(\alpha_C)$, that is, the map obtained by applying $H$ to the component $\alpha_C$ of $\alpha$, or as $(H\alpha)_C$, that is, the component at $C$ of $H\alpha$. These two coincide by definition, so writing $H\alpha_C$ does not lead to ambiguity.

\begin{ex}[important!]\label{exfunwhisk}
 Consider the situation depicted in the following diagram
 $$
 \begin{tikzcd}[column sep=large]
  \cat{C} \ar[bend left=60,""{below,name=F}]{r}{F} \ar[""{above,name=G},""{below,name=GG}]{r}[near start]{G} \ar[bend right=60,""{above,name=P}]{r}[swap]{P} & \cat{D}
  \ar[""{above,name=I},""{below,name=II}]{r}{H} & \cat{E}
    \ar[Rightarrow,from=F,to=G,"\alpha"] 
  \ar[Rightarrow,from=GG,to=P,"\gamma"]
 \end{tikzcd}
$$
Prove, using functoriality of $H$, that $H(\gamma\circ\alpha) = (H\gamma)\circ (H\alpha)$.
(How about the dual situation?)
\end{ex}

Suppose now that we have the following situation. 
$$
 \begin{tikzcd}[column sep=large]
  \cat{C} \ar[bend left,""{below,name=F}]{r}{F} \ar[bend right,""{above,name=G}]{r}[swap]{G} & \cat{D}
  \ar[bend left,""{below,name=H}]{r}{H} \ar[bend right,""{above,name=I}]{r}[swap]{I} & \cat{E}
    \ar[Rightarrow,from=F,to=G,"\alpha"] 
  \ar[Rightarrow,from=H,to=I,"\beta"]
 \end{tikzcd}
$$
One could view this situation as follow. First we whisker $\alpha$ on the right with $H$,
$$
 \begin{tikzcd}[column sep=large]
  \cat{C} \ar[bend left,""{below,name=F}]{r}{F} \ar[bend right,""{above,name=G}]{r}[swap]{G} & \cat{D}
  \ar[bend left,""{below,name=H}]{r}{H}  & \cat{E}
    \ar[Rightarrow,from=F,to=G,"\alpha"] 
 \end{tikzcd}
$$
then we whisker $\beta$ on the left with $G$,
$$
 \begin{tikzcd}[column sep=large]
  \cat{C} \ar[bend right,""{above,name=G}]{r}[swap]{G} & \cat{D}
  \ar[bend left,""{below,name=H}]{r}{H} \ar[bend right,""{above,name=I}]{r}[swap]{I} & \cat{E}
  \ar[Rightarrow,from=H,to=I,"\beta"]
 \end{tikzcd}
$$
and then we compose the two, ``gluing the diagrams'', 
$$
 \begin{tikzcd}[column sep=large]
  \cat{C} \ar[bend left,""{below,name=F}]{r}{F} \ar[bend right,""{above,name=G}]{r}[swap]{G} & \cat{D}
  \ar[bend left,""{below,name=H}]{r}{H} \ar[bend right,""{above,name=I}]{r}[swap]{I} & \cat{E}
    \ar[Rightarrow,from=F,to=G,"\alpha"] 
  \ar[Rightarrow,from=H,to=I,"\beta"]
 \end{tikzcd}
$$
Alternatively, we could first whisker $\beta$ with $F$ on the left,
$$
 \begin{tikzcd}[column sep=large]
  \cat{C} \ar[bend left,""{below,name=F}]{r}{F} & \cat{D}
  \ar[bend left,""{below,name=H}]{r}{H} \ar[bend right,""{above,name=I}]{r}[swap]{I} & \cat{E}
  \ar[Rightarrow,from=H,to=I,"\beta"]
 \end{tikzcd}
$$
whisker $\alpha$ with $I$ on the right,
$$
 \begin{tikzcd}[column sep=large]
  \cat{C} \ar[bend left,""{below,name=F}]{r}{F} \ar[bend right,""{above,name=G}]{r}[swap]{G} & \cat{D}
  \ar[bend right,""{above,name=I}]{r}[swap]{I} & \cat{E}
    \ar[Rightarrow,from=F,to=G,"\alpha"] 
 \end{tikzcd}
$$
and again compose the two, 
$$
 \begin{tikzcd}[column sep=large]
  \cat{C} \ar[bend left,""{below,name=F}]{r}{F} \ar[bend right,""{above,name=G}]{r}[swap]{G} & \cat{D}
  \ar[bend left,""{below,name=H}]{r}{H} \ar[bend right,""{above,name=I}]{r}[swap]{I} & \cat{E}
    \ar[Rightarrow,from=F,to=G,"\alpha"] 
  \ar[Rightarrow,from=H,to=I,"\beta"]
 \end{tikzcd}
$$
the diagram would look the same. This turns out to yield exactly the same result, so that the diagram is not ambiguous. In other words:

\begin{prop}\label{horwhisker}
 Let $\cat{C}$, $\cat{D}$ and $\cat{E}$ be categories. Let $F,G:\cat{C}\to\cat{D}$ and $H,I:\cat{D}\to\cat{E}$ be functors, and let $\alpha:F\Rightarrow G$ and $\beta:H\Rightarrow I$ be natural transformations. Then 
 $$
 (\beta G)\circ(H\alpha) \; = \; (I\alpha)\circ(\beta F) .
 $$
\end{prop}
\begin{proof}
 It suffices to check that these two composite natural transformations have the same components. So let $C$ be an object of $\cat{C}$. The component of $(\beta G)\circ(H\alpha)$ at $C$ is, plugging in the definition of whiskering, the arrow of $\cat{E}$ given by the composition
 $$
 \begin{tikzcd}
  HFC \ar{r}{H\alpha_C} & HGC \ar{r}{\beta_{GC}} & IGC .
 \end{tikzcd}
 $$
 Analogously, the component of $(I\alpha)\circ(\beta F)$ at $C$ is the arrow of $\cat{E}$ given by the composition
 $$
 \begin{tikzcd}
  HFC \ar{r}{\beta_{FC}} & IFC \ar{r}{I\alpha_C} & IGC .
 \end{tikzcd}
 $$
 These two compositions are equal if and only if the following diagram commutes,
 $$
 \begin{tikzcd}
  HFC \ar{r}{H\alpha_C} \ar{d}{\beta_{FC}} & HGC \ar{d}{\beta_{GC}} \\
  IFC \ar{r}{I\alpha_C} & IGC
 \end{tikzcd}
 $$
 but this diagram does commute: this is granted by naturality of $\beta$.
\end{proof}

Naturality, in other words, means that writing these diagrams does not lead to ambiguity. 

\begin{deph}
 Consider the hypothesis of \Cref{horwhisker}, fitting in the diagram 
 $$
 \begin{tikzcd}[column sep=large]
  \cat{C} \ar[bend left,""{below,name=F}]{r}{F} \ar[bend right,""{above,name=G}]{r}[swap]{G} & \cat{D}
  \ar[bend left,""{below,name=H}]{r}{H} \ar[bend right,""{above,name=I}]{r}[swap]{I} & \cat{E}
    \ar[Rightarrow,from=F,to=G,"\alpha"] 
  \ar[Rightarrow,from=H,to=I,"\beta"]
 \end{tikzcd}
$$
We call the resulting natural transformation $HF\Rightarrow IG$, which can be written equivalently (by \Cref{horwhisker}) as  $(\beta G)\circ(H\alpha)$ or as $(I\alpha)\circ(\beta F)$, the \emph{horizontal composition of $\alpha$ and $\beta$}. We denote it simply by juxtaposition, $\beta\alpha$. 
\end{deph}

\begin{caveat}
 The horizontal composition $\beta\alpha$ is \emph{different} from the ordinary composition $\beta\circ\alpha$, which we can call \emph{vertical composition}. In this case, the vertical composition $\beta\circ\alpha$ is not even defined, since the target of $\alpha$ (the functor $G$) is not equal to the source of $\beta$ (the functor $H$). 
\end{caveat}

However, vertical and horizontal composition are related by the so-called \emph{interchange law}:

\begin{prop}[interchange law] 
Consider the situation depicted in the following diagram. 
 $$
 \begin{tikzcd}[column sep=large]
  \cat{C} \ar[bend left=60,""{below,name=F}]{r}{F} \ar[""{above,name=G},""{below,name=GG}]{r}[near start]{G} \ar[bend right=60,""{above,name=P}]{r}[swap]{P} & \cat{D}
  \ar[bend left=60,""{below,name=H}]{r}{H} \ar[""{above,name=I},""{below,name=II}]{r}[near start]{I} \ar[bend right=60,""{above,name=Q}]{r}[swap]{Q} & \cat{E}
    \ar[Rightarrow,from=F,to=G,"\alpha"] 
  \ar[Rightarrow,from=H,to=I,"\beta"]
  \ar[Rightarrow,from=GG,to=P,"\gamma"]
  \ar[Rightarrow,from=II,to=Q,"\delta"]
 \end{tikzcd}
$$
Then the diagram is not ambiguous, in the sense that the vertical composition of the horizontal compositions
$$
(\delta\gamma)\circ(\beta\alpha)
$$
is equal to the horizontal composition of the vertical compositions
$$
(\delta\circ\beta)(\gamma\circ\alpha) .
$$
\end{prop}
\begin{proof}
 By definition, $\beta\alpha$ can be written as $(I\alpha)\circ(\beta F)$. Just as well, $\delta\gamma$ can be written as $(\delta P)\circ(I\gamma)$. Now, using \Cref{exfunwhisk} and \Cref{horwhisker},
 \begin{align*}
 (\beta\alpha)\circ(\delta\gamma) \;&=\; (\delta P)\circ(I\gamma)\circ (I\alpha)\circ(\beta F) \\
 \;&=\; (\delta P)\circ I(\gamma\circ\alpha)\circ(\beta F) \\
 \;&=\; (\delta P)\circ (\beta P) \circ H(\gamma\circ\alpha) \\
 \;&=\; (\delta \circ\beta) P \circ H(\gamma\circ\alpha) ,
 \end{align*}
 which is a way to write $(\delta\circ\beta)(\gamma\circ\alpha)$. 
\end{proof}

\begin{ex}
 Try to understand this proof graphically, by drawing the whiskerings explicitly.
\end{ex}

\section{Studying categories by means of functors}

The main philosophy of category theory is studying objects not by looking ``inside'' of them, but rather, in terms of the morphisms between them. This can be done for categories too, they can be studied in terms of functors.

\subsection{The category of categories}

Categories form themselves a category, with functors as the morphisms. In order to avoid size issues (such as talking about the category of \emph{all} categories), we have to restrict to small categories.

\begin{deph}
 The category $\cat{Cat}$ has
 \begin{itemize}
  \item As objects, small categories;
  \item As morphisms, the functors between them.
 \end{itemize}
 The identities are given by the identity functors, and the composition by the composition of functors. 
\end{deph}

Given small categories $\cat{C}$ and $\cat{D}$, the functors between them are the elements of the set $\Hom_\cat{Cat}(\cat{C},\cat{D})$. However, as we saw in \Cref{functordiagrams}, they also are \emph{objects of a category}, namely the functor category $[\cat{C},\cat{D}]$. 
As we said at the end of \Cref{iso}, in a set elements can be equal or not, while in a category objects can be rather \emph{isomorphic} or not -- equality is not really a helpful concept inside a category. In the same line, it is much more helpful to talk about \emph{isomorphisms} between functors than equality.
It is much more convenient, in mathematical practice, to think of functors $\cat{C}\to\cat{D}$ as the objects of a category than as the elements of a set. In other words, the category $[\cat{C},\cat{D}]$ is mathematically a more useful construction than the \emph{set} $\Hom_\cat{Cat}(\cat{C},\cat{D})$. 

Therefore the category $\cat{Cat}$, as defined above, is not very much used in category theory.
There is a more refined notion of $\cat{Cat}$ in which the hom-spaces $\Hom_\cat{Cat}(\cat{C},\cat{D})$ are themselves \emph{categories} instead of just sets. This is called a \emph{2-category} and it is one of the core concepts of higher category theory, which is beyond the scope of this basic course (for the interested reader, the \href{http://ncatlab.org}{nLab} has a lot of material on this). 

In particular, the notions of mono and epi of $\cat{Cat}$ are not a convenient way to study categories, since these concepts are based on the concept of equality of arrows and not isomorphism of functors. Much more helpful notions are given in \Cref{ffeso}.

\subsection{Subcategories}

\begin{deph}
 Let $\cat{C}$ be a category. A \emph{subcategory} $\cat{S}$ of $\cat{C}$ consists of a subcollection of the objects of $\cat{C}$ and a subcollection of the morphisms of $\cat{C}$ such that 
 \begin{itemize}
  \item For each object $S$ of $\cat{S}$, the identity of $S$ is also in $\cat{S}$;
  \item For each pair of composable arrows $f:X\to Y$ and $g:Y\to Z$ of $\cat{S}$, their composite $g\circ f$ is also in $\cat{S}$.
 \end{itemize}
\end{deph}

\begin{deph}
 A subcategory $\cat{S}$ of $\cat{C}$ is called \emph{wide} if all objects of $\cat{C}$ are also objects of $\cat{S}$ (but $\cat{S}$ could have less morphisms).
  
 A subcategory $\cat{S}$ of $\cat{C}$ is called \emph{full} if given any two objects $X$ and $Y$ of $\cat{S}$ (and $\cat{C}$), all their morphisms in $\cat{C}$ are also morphisms in $\cat{S}$. In other words, 
 $$
 \Hom_\cat{S}(X,Y) \; \cong \;\Hom_\cat{C}(X,Y) .
 $$
 (But not all objects of $\cat{C}$ are necessarily objects of $\cat{S}$ too.)
\end{deph}

\begin{eg}[several fields]
 The following are wide subcategories.
 \begin{itemize}
  \item The category $\cat{Lip}$ of metric spaces and \emph{Lipschitz} maps is a wide subcategory of the category $\cat{Met}$ of metric spaces and continuous maps.
  \item The category of sets and \emph{injective} maps is a wide subcategory of $\Set$.
  \item For every category, its core (\Cref{core}) is a wide subcategory.
  \item For $S$ a subgroup of $G$, the category $\cat{B}S$ is a wide subcategory of $\cat{B}G$.
 \end{itemize}
 The following are full subcategories.
 \begin{itemize}
  \item The category $\cat{AbGrp}$ of abelian groups and group homomorphisms is a full subcategory of $\Grp$.
  \item The category $\cat{FVect}$ of finite-dimensional vector spaces and linear maps is a full subcategory of $\Vect$.
  \item The category $\cat{CHaus}$ of compact Hausdorff spaces and continuous maps is a full subcategory of $\cat{Top}$.
  \item A subposet $(S,\le)$ of $(X,\le)$ with the order induced from $(X,\le)$ is a full subcategory of $(X,\le)$ seen as a category.
 \end{itemize}
\end{eg}

\subsection{Full, faithful, essentially surjective}\label{ffeso}

A function $f:X\to Y$ between sets can be
\begin{itemize}
 \item Injective: for $x,x'\in X$ we have $f(x)=f(x')$ in $Y$, then $x=x'$ in $X$;
 \item Surjective: for each $y\in Y$ there exists $x\in X$ such that $f(x)=y$.
\end{itemize}

For functors we have instead three possible properties. These roughly translate to ``injective on arrows'', ``surjective on arrows'', and ``surjective on objects up to isomorphism''.

\begin{deph}
 A functor $F:\cat{C}\to\cat{D}$ can be
 \begin{itemize}
  \item \emph{faithful}: for every objects $C,C'$ of $\cat{C}$ and every pair of arrows $f,f':C\to C'$ of $\cat{C}$, we have that if $Ff=Ff'$ in $\cat{D}$, then $f=f'$ in $\cat{C}$.
  \item \emph{full}: for every objects $C,C'$ of $\cat{C}$ and every arrow $g:FC\to FC'$ of $\cat{D}$, there exists an arrow $f:C\to C'$ of $\cat{C}$ such that $Ff=g$;
  \item \emph{fully faithful} if it is full and faithful;
  \item \emph{essentially surjective}: for every object $D$ of $\cat{D}$ there exist an object $C$ of $\cat{C}$ and an isomorphism $FC\to D$.
 \end{itemize}
\end{deph}

\begin{remark}\label{onarrows}
 Given objects $C,C'$ of $\cat{C}$, a functor $F:\cat{C}\to\cat{D}$ induces a function 
$$
\Hom_\cat{C}(C,C') \;\xrightarrow\;\Hom_\cat{D}(FC,FC') .
$$
This function is injective for all $C,C'$ precisely if $F$ is faithful, surjective for all $C,C'$ precisely if $F$ is full, and bijective for all $C,C'$ precisely if $F$ is fully faithful.
\end{remark}

Compare the definition of essential surjectivity with the definition of surjectivity of functions. An essentially surjective functor is ``surjective up to isomorphism''. This is the only version of surjectivity that we need between categories, since we look at \emph{isomorphic objects} rather than equal.
Fully faithful functors are ``injective up to isomorphism'', as the following proposition shows. 

\begin{prop}\label{ffinj}
 Let $F:\cat{C}\to\cat{D}$ be a fully faithful functor. Let $C,C'$ be objects in $\cat{C}$, and let $\phi:FC\to FC'$ be an isomorphism of $\cat{D}$. Then there is a unique isomorphism $\tilde{\phi}:C\to C'$ such that $F\tilde{\phi}=\phi$. 
 
 In particular, if $F(C)$ and $F(C')$ are isomorphic (in $\cat{D}$), then $C$ and $C'$ are isomorphic (in $\cat{C}$).
\end{prop}

\begin{proof}
 Let $\phi:F(C)\to F(C')$ be an isomorphism, with inverse $\psi$. Since $F$ is fully faithful, necessarily $\phi=F\tilde{\phi}$ for a unique $\tilde{\phi}:C\to C'$, and $\psi=F\tilde{\psi}$ for a unique $\tilde{\psi}:C'\to C$. It now suffices to prove that $\tilde{\phi}$ and $\tilde{\psi}$ are inverse to each other. Now as $\phi$ and $\psi$ are inverse to each other, $\psi\circ\phi=\id_{FC}$, or in other words, $F\tilde{\psi}\circ F\tilde{\phi}=\id_{FC}$. By functoriality, the last equation is equivalent to $F(\tilde{\psi}\circ \tilde{\phi})=F(\id_C)$. Since $F$ is faithful, this implies $\tilde{\psi}\circ \tilde{\phi}=\id_C$. In the same way, from $\phi\circ\psi=\id_{FC'}$ one can conclude $\tilde{\phi}\circ \tilde{\psi}=\id_{C'}$. Therefore $\tilde{\phi}$ and $\tilde{\psi}$ are inverse to each other.
\end{proof}

\begin{eg}[algebra, analysis]
 \begin{itemize}
  \item The forgetful functor $\cat{AbGrp}\to\Grp$ is fully faithful. It is not essentially surjective: not every group is isomorphic to an abelian group.
  More generally, a full subcategory always defines a fully faithful ``inclusion'' functor, and vice versa. 
  \item The forgetful functor $\cat{Lip}\to\cat{Met}$ is faithful and essentially surjective, but not full: not every continuous map is Lipschitz. 
 \end{itemize}
\end{eg}

Let's see what happens for posets, equivalence relations and preorders.

\begin{eg}[sets and relations]
 Let $(X,\le)$ and $(Y,\le)$ be posets, and let $f:(X,\le)\to (Y,\le)$ be monotone. 
 \begin{itemize}
  \item $f$ is always trivially faithful, since for any two objects $x,x'\in X$ there can be up to one unique arrow between them.
  \item $f$ is full precisely if it is \emph{order-reflecting}: that is, $f(x)\le f(x')$ implies $x\le x'$. Note the direction of the implication: this is not just monotonicity. 
  \item Since two elements in a poset are isomorphic if and only if they are equal, essential surjectivity is just surjectivity.
 \end{itemize}

 For equivalence relations and preorders, again every functor is automatically faithful, and full precisely if it \emph{reflects} (and not just preserves) the relation. Essential surjectivity is more subtle: given $f:(X,\sim)\to (Y,\sim)$, the function $f$ is essentially surjective if for every $y\in Y$ there exists $x\in X$ such that $f(x)\sim y$. In other words, the image of $f$ is not necessarily the whole of $Y$, but \emph{it hits every equivalence class of $Y$}. 
\end{eg}

Let's now see what happens for groups (and analogously, for monoids).

\begin{eg}[group theory]
 Let $G$ and $H$ be groups. A group homomorphism $f:G\to H$ corresponds to a functor $F:\cat{B}G\to\cat{B}H$, and vice versa. We have that
 \begin{itemize}
  \item $F$ is always trivially essentially surjective, since $\cat{B}H$ has a single object, which is the image of the single object of $\cat{B}G$.
  \item $F$ is faithful exactly if $f$ is injective by \Cref{onarrows}.
  \item $F$ is full exactly if $f$ is surjective by \Cref{onarrows}.
  \item Therefore, $F$ is fully faithful if and only if $f$ is a group isomorphism.
 \end{itemize}
\end{eg}

The term \emph{faithful} originally comes from representation theory, and here is why.

\begin{eg}[group theory]
 Let $G$ be a group. A linear representation of $G$ is a functor $R:\cat{B}G\to\Vect$. Let $V$ be the unique object in the image of $R$, that is, the space on which $G$ acts. The functor $R$ is faithful if and only if different elements $g\neq h\in G$ give different linear maps $V\to V$. In other words, for every $g\neq h\in G$ there exists $v\in V$ such that $g\cdot v \neq h\cdot v$.
 In representation theory, a representation with this property is called a \emph{faithful representation of $G$}. The intuitive idea is that this representation respects all the structure of $G$ without forgetting anything. As the terminology suggests, it ``represents $G$ faithfully''.
\end{eg}

\begin{remark}
One may ask, \emph{why do we have three properties for functors}? Why not two, as for functions, and why not four, two for objects and two for morphisms? One possible answer will be given by \Cref{thmeqcat}. Another one is the following. A set can be seen as a category in which the only morphisms are the identities. In particular, given $x,x'\in X$, there is an arrow $x\to x'$ if and only if $x=x'$. A function $f:X\to Y$ is therefore injective if and only if \emph{it is surjective on arrows}: if there is an arrow $f(x)\to f(x')$ (that is, $f(x)=f(x')$) then there is an arrow $x\to x'$ (that is, $x=x'$). In other words, surjectivity and injectivity can be seen as surjectivity on objects and on arrows, respectively. Injectivity on arrows is guaranteed by the fact that between $x$ and $x'$, in a set, there cannot be more than one arrow. In a category, however, there can be many arrows, and so injectivity on arrows becomes an issue too. Again, injectivity on arrows can be seen as surjectivity between ``arrows between arrows'', and so on.
This idea becomes precise, and more general, in the context of higher category theory.
\end{remark}

\subsection{Equivalences of categories}\label{equcat}

\begin{deph}
 Let $\cat{C}$ and $\cat{D}$ be categories. An \emph{equivalence of categories} between $\cat{C}$ and $\cat{D}$ consists of a pair of functors $F:\cat{C}\to\cat{D}$ and $G:\cat{D}\to\cat{C}$, and natural isomorphisms $\eta:G\circ F \Rightarrow \id_\cat{C}$ and $\epsilon: F\circ G \Rightarrow \id_\cat{D}$. 
 
 In that case we call $G$ the \emph{pseudoinverse} of $F$ (and vice versa). 
\end{deph}

Compare this with the definition of isomorphism in a category (\Cref{defiso}). In some sense, here we require $G$ to be the the inverse of $F$ ``only up to isomorphism''. Readers familiar with topology may find this analogous to a \emph{homotopy equivalence} between two spaces (rather than a homeomorphism).

\begin{eg}[sets and relations, group theory]
 The following are examples of equivalences of categories.
 \begin{itemize}
  \item Two posets are equivalent as categories if and only if they are isomorphic as posets.
  \item Two sets equipped with equivalence relations are equivalent as categories if and only if they have isomorphic quotients.
  \item Two groups $G$ and $H$ are isomorphic if and only if $\cat{B}G$ and $\cat{B}H$ are equivalent categories. The same is true for monoids.
 \end{itemize}
\end{eg}

\begin{ex}
 If any of the statements above does not convince you, try to prove it explicitly by plugging in the definition. 
\end{ex}

In the category of sets, a function between sets is a bijection (i.e.~it is invertible) if and only if it is both injective and surjective. We have a similar statement for functors, where instead of injectivity and surjectivity we have the properties of \Cref{ffeso}.

\begin{thm}\label{thmeqcat}
 A functor $F:\cat{C}\to\cat{D}$ defines an equivalence of categories if and only if it is fully faithful and essentially surjective. 
\end{thm}

This theorem is quite convenient, since it is often much easier to check the specified properties for $F$ than to find a specific pseudoinverse. 
In order to prove the theorem, we make use of the following lemma, which is itself interesting.

\begin{lemma}[{\cite[Lemma~1.5.10]{ctcontext}}]\label{uniquef}
 Let $f:X\to Y$ be a morphism in a category $\cat{C}$. Consider isomorphisms $\phi:X\to X'$ and $\psi:Y\to Y'$. Then there is a unique morphism $X'\to Y'$ such that the following diagram commutes
 $$
 \begin{tikzcd}
 X \ar{r}{f} \ar{d}{\cong}[swap]{\phi} & Y \ar{d}{\cong}[swap]{\psi} \\
 X' \ar{r} & Y'
 \end{tikzcd}
 $$
 and this morphism is given by $\psi\circ f\circ \phi^{-1}$.
\end{lemma}

\begin{proof}[Proof of \Cref{uniquef}]
 First of all, the morphism $\psi\circ f\circ \phi^{-1}$ makes the diagram commute, since $\psi\circ f\circ \phi^{-1}\circ \phi$ = $\psi\circ f$. Now suppose that there is any other morphism $g:X'\to Y'$ making the diagram commute. That is, $g\circ\phi = \psi\circ f$. By applying $\phi^{-1}$ on the right on both sides we get $g = \psi\circ f \circ \phi^{-1}$.
\end{proof}

We are now ready to prove the theorem. 

\begin{proof}[Proof of \Cref{thmeqcat}]
 Let $F:\cat{C}\to\cat{D}$ be a functor with pseudoinverse $G:\cat{D}\to\cat{C}$, with natural isomorphisms $\eta:G\circ F \Rightarrow \id_\cat{C}$ and $\epsilon: F\circ G \Rightarrow \id_\cat{D}$.
 Let's prove that $F$ is faithful.
 Let $C$ and $C'$ be objects of $\cat{C}$.
 Suppose that $f,f':C\to C'$ are morphisms of $\cat{C}$ such that $Ff=Ff'$ in $\cat{D}$. By naturality of $\eta$, these two diagrams commute. 
 $$
 \begin{tikzcd}
  GFC \ar{d}{\cong}[swap]{\eta_C} \ar{r}{GFf} & GFC' \ar{d}{\cong}[swap]{\eta_{C'}} \\
  C \ar{r}{f} & C'
 \end{tikzcd}
 \qquad\begin{tikzcd}
  GFC \ar{d}{\cong}[swap]{\eta_C} \ar{r}{GFf'} & GFC' \ar{d}{\cong}[swap]{\eta_{C'}} \\
  C \ar{r}{f'} & C'
 \end{tikzcd}
 $$
 Now if $Ff=Ff'$ then $GFf=GFf'$, which are the arrows at the top of the two diagrams, and by uniqueness (\Cref{uniquef}) it must be that $f=f'$. So $F$ is faithful (and by the same argument, $G$ is faithful too).
 Let's now prove that $F$ is full.  
 Let $g:FC\to FC'$ be a morphism of $\cat{D}$. We have to prove that there is a morphism $f:C\to C'$ of $\cat{C}$ such that $Ff=g$. So consider the morphism $Gg$ of $\cat{C}$, fitting into the following diagram. 
 $$
 \begin{tikzcd}
  GFC \ar{d}{\cong}[swap]{\eta_C} \ar{r}{Gg} & GFC' \ar{d}{\cong}[swap]{\eta_{C'}} \\
  C & C'
 \end{tikzcd}
 $$
 We can construct the morphism $f\coloneqq \eta_{C'}\circ Gg\circ \eta_C^{-1}:C\to C'$ as the unique arrow that fills the diagram above to a commutative diagram. By naturality of $\eta$, the following diagram commutes,
 $$
 \begin{tikzcd}
  GFC \ar{d}{\cong}[swap]{\eta_C} \ar{r}{GFf} & GFC' \ar{d}{\cong}[swap]{\eta_{C'}} \\
  C \ar{r}{f} & C'
 \end{tikzcd}
 $$
 Again by uniqueness (\Cref{uniquef}), we then must conclude that $GFf=Gg$. But since $G$ is faithful, this implies $Ff=g$. Therefore $F$ is full.
 Let's now turn to essential surjectivity. Let $D$ be an object of $\cat{D}$. We have to find an object $C$ of $\cat{C}$ such that $FC\cong D$. So define $C\coloneqq GD$. We have that $FC=FGD\cong D$ via the map $\epsilon_D:FGD\to D$. Therefore $F$ is essentially surjective. We have proven one direction of the theorem.
 
 Now the converse. Let $F$ be fully faithful and essentially surjective. We need to construct $G:\cat{D}\to\cat{C}$ with $G\circ F\cong\id_\cat{C}$ and $F \circ G\cong \id_\cat{D}$.
 Since $F$ is essentially surjective, for every object $D$ of $\cat{D}$ we can choose an object $C$ of $\cat{C}$ and an isomorphism $\phi_D:FC\to D$. We define then, on objects, $GD\coloneqq C$. For the morphisms we proceed as follows. Given $g:D\to D'$ in $\cat{D}$, we can apply \Cref{uniquef} and obtain the $\phi_{D'}\circ g\circ \phi_{D}^{-1}:FGD\to FGD'$. Since $F$ is full, the latter map can be written as $Ff$ for some morphism $f:C_D\to C_{D'}$ of $\cat{C}$. We then define, for each $g:D\to D'$, $Gg\coloneqq f$. This assignment is functorial: let $g:D\to D'$ and $g':D'\to D''$ in $\cat{D}$. Then 
 \begin{align*} 
 FG(g'\circ g) \;=\; \phi_{C''}\circ g'\circ g\circ \phi_{C}^{-1} \;&=\; (\phi_{C''}\circ g'\circ \phi_{C'})\circ (\phi_{C'}^{-1}\circ g\circ \phi_{C}^{-1}) \\
 \;&=\; FGg'\circ FGg \;=\; F(Gg'\circ Gg),
 \end{align*}
 and since $F$ is faithful, this implies $G(g'\circ g)=Gg'\circ Gg$. 
 Now let $C$ be an object of $\cat{C}$. By construction of $G$, we have that $G\circ FC$ is an object of $\cat{C}$ admitting an isomorphism $\phi_{FC}:FGFC\to FC$. 
 By \Cref{ffinj}, since $F$ is fully faithful, we have an isomorphism $\tilde{\phi}_{FC}:GFC\to C$. This isomorphism is natural: let $f:C\to C'$ be a morphism of $\cat{C}$ and consider the following diagrams.
 $$
 \begin{tikzcd}
 GFC \ar{r}{GFf} \ar{d}{\tilde{\phi}_{FC}} & GFC' \ar{d}{\tilde{\phi}_{FC'}} \\
 C \ar{r}{f} & C'
 \end{tikzcd}
 \qquad
 \begin{tikzcd}
 FGFC \ar{r}{FGFf} \ar{d}{\phi_{FC}} & FGFC' \ar{d}{\phi_{FC'}} \\
 FC \ar{r}{Ff} & FC'
 \end{tikzcd}
 $$
 The diagram on the left commutes exactly if the isomorphism $\tilde{\phi}_{FC}:GFC\to C$ is natural. Since $F$ is fully faithful, the diagram commutes if and only if it does after applying $F$, which gives the diagram on the right. But the diagram on the right commutes, since by construction of $G$, $FGFf = \phi_{FC'}\circ Ff\circ \phi_{FC}^{-1}$. Therefore we have a natural isomorphism $G\circ F \cong \id_\cat{C}$. 
 Just as well, let $D$ be an object of $\cat{D}$. We have y construction of $G$ that $GD$ is an object of $\cat{C}$ such that $FGD$ is isomorphic to $D$ via the map $\phi_D$. Again, this isomorphism is natural: let $g:D\to D'$ be a morphism. By construction of $G$, the map $FGg$ is equal to $\phi_{D'}\circ g\circ \phi_{D}^{-1}$, so the following diagram commutes.
 $$
 \begin{tikzcd}
  FGD \ar{r}{FGg} \ar{d}{\phi_D} & FGD'\ar{d}{\phi_{D'}} \\
  D \ar{r}{g} & D'
 \end{tikzcd}
 $$
 Therefore, we have a natural isomorphism $F\circ G \cong \id_\cat{D}$, and hence, an equivalence of categories between $\cat{C}$ and $\cat{D}$. 
\end{proof}

Here is a first example, written as a corollary.
\begin{cor}
 A monotone function $f:(X,\le)\to(Y,\le)$ is an isomorphism of partial orders if and only if it is order-reflecting and surjective.
\end{cor}

A more suggestive example comes from linear algebra. Some people consider a vector in finite dimensions ``just an array of numbers'', and a linear map ``just a matrix''. Of course they are not \emph{quite} the same, but one thing is true: everything that can be done with abstract finite-dimensional vector spaces can also be done with their representation as arrays of numbers, and vice versa, if one is careful.
This can be made precise in the following way: \emph{there is an equivalence of categories between $\cat{FVect}$ and the category of matrices}. In the example below we show it for vector spaces over the reals, but the same can be done for complex numbers, or any other field.

\begin{eg}[linear algebra]\label{matcat}
 First of all, we want to define a category $\cat{Mat}$ whose morphisms are matrices with their usual matrix composition. As you know from linear algebra, a $m\times n$-matrix $A$ (that is, a matrix with $n$ rows and $m$ columns) can only be composed on the left with a matrix $B$ that has $n$ columns (the same $n$ as the rows of $A$), and an arbitrary number of rows. Therefore, as objects we take natural numbers, which keep track of the numbers of rows and columns in such a way that composable pairs of arrows correspond to composable pairs of matrices. That is:
 \begin{itemize}
  \item The objects of $\cat{Mat}$ are natural numbers $n\in \N$, including zero.
  \item For objects (natural numbers) $m,n\ne 0$, the morphisms $M:m\to n$ are the $m\times n$ matrices $M$ with real entries.
  \item If $m$ or $n$ is zero, we just assign a unique morphism $m\to n$, which we can see as a ``zero-dimensional matrix''.
  \item The identity of $n$ is just the $n\times n$ identity matrix, and the composition is given by the usual composition of matrices.
 \end{itemize}
 
 Let's now construct a functor $F:\cat{Mat}\to\cat{FVect}$ as follows. 
 \begin{itemize}
  \item $F$ maps each object (natural number) $n$ of $\cat{Mat}$ to the vector space $\R^n$, and $0$ to the zero vector space;
  \item $F$ maps each $m\times n$ matrix $M:m\to n$ to the linear map $\R^m\to\R^n$ represented by the matrix $M$.
 \end{itemize}
 This preserves identities and composition, and so it is a functor. Now,
 \begin{itemize}
  \item $F$ is faithful: different $m\times n$ matrices give different linear maps $\R^m\to\R^n$;
  \item $F$ is full: all linear maps $\R^m\to\R^n$ arise in this way, i.e.~can be represented by a matrix;
  \item $F$ is essentially surjective: every finite-dimensional vector space is isomorphic to $\R^n$ for some $n$. 
 \end{itemize}
 Therefore, by \Cref{thmeqcat}, $F$ defines an equivalence of categories.

 In particular, as in the proof of \Cref{thmeqcat}, \Cref{uniquef} implies the following: let $V$ and $W$ be vector spaces of dimension $m$ and $n$, respectively. Then there are isomorphisms $\phi_V:\R^m\to V$ and $\phi_W:\R^n\to V$. These can be seen as \emph{choices} of bases for $V$ and $W$.
 There are actually many more possible isomorphisms than these two, which correspond to different choices of bases, but let's keep these fixed for the moment. 
 Now given a linear map $f:V\to W$, by \Cref{uniquef} there exists a \emph{unique} linear map $\R^m\to \R^n$ making this diagram commute,
 $$
 \begin{tikzcd}
  \R^m \ar{r} \ar{d}{\cong}[swap]{\phi_V} & \R^n \ar{d}{\cong}[swap]{\phi_W} \\
  V \ar{r}{f} & W
 \end{tikzcd}
 $$
 namely, the map $\phi_W^{-1}\circ f\circ \phi_V$. This map can be seen as the \emph{representation of $f$ as a matrix given the chosen bases of $V$ and $W$}. If we keep $\phi_V$ and $\phi_W$ fixed, this representation is unique, and it uniquely specifies $f$. 
 But of course, if we let $\phi_V$ and $\phi_W$ vary, then the representation of $f$ will be different.
\end{eg}

As the example above shows, an equivalence of categories between $\cat{C}$ and $\cat{D}$ can be interpreted, intuitively, as the fact that \emph{whatever can be done in $\cat{C}$ can be equivalently done in $\cat{D}$ too, and vice versa}. At a first look, the two categories may look differently, but they really encode the same mathematical methods, they have the same expressive power.

\begin{ex}\label{1.5.20}
 A category is called \emph{connected} if between any two objects $X$ and $Y$ there is at least an arrow (in at least one direction). Prove that a connected groupoid is equivalent to a category in the form $\cat{B}G$ for some group $G$. 
 
 Is an analogous statements true for connected categories and monoids? (Hint: no. Can you give a counterexample?)
\end{ex}

\newpage
\chapter{Universal properties and the Yoneda lemma}

\section{Representable functors and the Yoneda embedding theorem}

\subsection{Extracting sets from objects}

We have said in \Cref{presheaf} that the category $\Set$ plays a special role in (ordinary) category theory, because the arrows between two given objects (in a locally small category) form a set. 
Given an object $X$ in a category $\cat{C}$, we can usually extract many sets out of $X$, which can be thought of as ``capturing part of the structure'' whenever they are functorial. Let's see some examples.

\begin{eg}[topology]\label{toptoset}
 Let $\cat{C}=\cat{Top}$. Given a topological space $X$, we can consider for example the following sets.
 \begin{enumerate}
  \item The underlying set $U(X)$, which is the set of \emph{points} of $X$;
  \item The set $\mathit{Curve}(X)$ of continuous curves in $X$, with or without endpoints;
  \item The set $\mathit{Loop}(X)$ of continuous loops, or closed curves in $X$;
  \item The set $\pi_0(X)$ of path-connected components of $X$;
  \item The set $O(X)$ of open sets of $X$, i.e.~the topology of $X$.
 \end{enumerate}
 All these constructions are functorial in $X$ in the following ways. Let $f:X\to Y$ be a continuous function.  
 \begin{enumerate}
  \item There is an ``underlying function'' $Uf$ between the underlying sets of points $U(X)$ and $U(Y)$. So we have a functor $U:\cat{Top}\to\Set$ (see also \Cref{forgetfulfunctor}).
  \item If $c$ is a continuous curve in $X$, then the image $f_*c$ of $c$ under $f$ is a continuous curve in $Y$. Therefore we get an induced map $f_*:\mathit{Curve}(X)\to \mathit{Curve}(Y)$ (we can set $\mathit{Curve}(f)\coloneqq f_*$). So we have a functor $\mathit{Curve}:\cat{Top}\to\Set$.
  \item Just as well, we get an induced map $f_*:\mathit{Loop}(X)\to \mathit{Loop}(Y)$. Again, we have a functor $\mathit{Loop}:\cat{Top}\to\Set$.
  \item If $x,y\in X$ are in the same path component of $X$, then $f$, by continuity, has to map them to the same path component of $Y$. This induces a well-defined map $\pi_0(X)\to\pi_0(Y)$ (compare with the graph case \Cref{graphcomponents}.) Therefore we have a functor $\pi_0:\cat{Top}\to\Set$. 
  \item Let $U$ be an open set of $Y$. Then its preimage $f^{-1}(U)$ is an open set of $X$. Therefore there is a well-defined function $f^{-1}:O(Y)\to O(X)$ (note the direction of the arrow!). So we have a \emph{presheaf} $O:\cat{Top}^\op\to\Set$. 
  (In general, the \emph{image} of an open set under a continuous map is not open -- we really need to take preimages here.)
 \end{enumerate}
\end{eg}

The case of graphs and multigraphs is analogous.
\begin{eg}[graph theory]\label{graphtoset}
 Let $\cat{C}$ be the category $\cat{MGraph}$ of directed multigraphs and their morphisms (\Cref{egmorphgraph}). Given a multigraph $G$, the following are functorial assignments into $\Set$. (How are they functorial?)
 \begin{enumerate}
  \item The set $\mathit{Vert}(G)$ of its vertices;
  \item The set $\mathit{Edge}(G)$ of its edges;
  \item The set $\mathit{Chain}_2(G)$ of 2-chains in $G$, i.e.~pairs of edges which are head-to-tail (a bit like composable morphisms in a category). More generally, the set $\mathit{Chain}_n(G)$ of $n$-chains, or $n$-walks in $G$, i.e.~series of $n$ consecutive edges.
  \item The set $\pi_0(G)$ of connected components (see also \Cref{graphcomponents}).
 \end{enumerate}
 (Note that the word ``path'', in graph theory, is reserved for those chains which do not cross the same vertices more than once.)
\end{eg}

Here are some examples for the category of groups.
\begin{eg}[group theory]\label{grptoset}
 The following are functorial assignments $\Grp\to\Set$. Given a group $G$,
 \begin{enumerate}
  \item The underlying set $U(G)$;
  \item The set of elements of $G$ of order $2$, or more generally, of order $n$ for fixed $n$. This means, the elements $g\in G$ such that $g^n=e$. For example, $-1$ has order $2$ in $(\R_{\ne 0},\cdot)$.
 \end{enumerate}
\end{eg}

\begin{ex}[group theory]
 Can you give more examples of functors $\Grp\to\Set$? What about presheaves $\Grp^\op\to\Set$?
\end{ex}

\begin{ex}\label{yourcatset}
 Take a category $\cat{C}$ of your choice, useful to your field. Can you give examples of functors $\cat{C}\to\Set$ or presheaves $\cat{C}^\op\to\Set$ which extract useful information from the objects?
\end{ex}

\subsection{Representable functors}

In many of the examples in the section above, the given functors are actually of a very simple form: they are in the form $\Hom_{C}(S,-):\cat{C}\to\Set$ for some object $S$. 

\begin{deph}
 Let $\cat{C}$ be a category. A functor $F:\cat{C}\to\Set$ is called \emph{representable} if it is naturally isomorphic to the functor $\Hom_{C}(S,-):\cat{C}\to\Set$ for some object $S$ of $\cat{C}$. In that case we call $S$ the \emph{representing object}. 
 
 A presheaf $F:\cat{C}^\op\to\Set$ is called \emph{representable} if it is naturally isomorphic to the functor $\Hom_{C}(-,S):\cat{C}^\op\to\Set$ for some object $S$ of $\cat{C}$. Again, in that case we call $S$ the \emph{representing object}. 
\end{deph}

Here is a way to interpret representable functors and representable presheaves. 
A representable functor $\Hom_{C}(S,-):\cat{C}\to\Set$ takes an objects $X$ of $\cat{C}$ and gives the set $\Hom_{C}(S,X)$ of arrows from $S$ to $X$. Intuitively, the features that it extracts from $X$ are exactly all the possible ways of mapping $S$ into $X$, as if $S$ were a sort of \emph{probe} that we use to explore the structure of $X$.

Dually, a representable presheaf $\Hom_{C}(-,S):\cat{C}^\op\to\Set$ takes an objects $X$ of $\cat{C}$ and gives the set $\Hom_{C}(X,S)$ of arrows from $X$ to $S$. Intuitively, the features that it extracts from $X$ are all the possible ways of mapping $X$ to $S$, or all the possible \emph{observations} of $X$ with values in $S$ (think of $S$ as a ``screen'' onto which $X$ can be projected in many possible ways). 

\begin{eg}[topology]
 Let's see which functors from \Cref{toptoset} are representable. Let $X$ be a topological space.
 \begin{enumerate}
  \item Let $x$ be a point of $X$. If we denote by $1$ the one-point space, then there exists a (continuous) map $1\to X$ which picks out exactly $x$, that is, which maps the unique point of $1$ to $x$. This can be done for every point $x$ of $X$, and different maps $1\to X$ pick out necessarily different points. Therefore there is a bijection between points of $X$ and maps $1\to X$, that is, $U(X)\cong \Hom_\cat{Top}(1,X)$. This bijection is natural in $X$ (why?), and so the functor $U$ is representable, and the representing object is the one-point space $1$. 
  \item Analogously, curves in $X$ with endpoints correspond to continuous maps from $[0,1]$ to $X$ and curves without endpoints correspond to continuous maps from $(0,1)$ (or equivalently, $\R$) to $X$. Therefore the functor $\mathit{Curve}$ is representable.
  \item In the same way, the functor $\mathit{Loop}$ is representable, and the representing object is the circle $S^1$. That is, loops in $X$ correspond to continuous maps $S^1\to X$. 
  \item The functor $\pi_0$ is not representable. (It is however a sort of homotopy-version of the forgetful functor $U$.) 
  \item Consider now the topology presheaf $O:\cat{Top}^\op\to\Set$, which assigns to each set $X$ its topology. Denote by $S$ the Sierpinski space, which is the space with two elements $0$ and $1$, with the topology given by $\{\varnothing,\{1\},\{0,1\}\}$, but not $\{0\}$. Given a topological space $X$, consider a function $f:X\to S$. Necessarily, $f^{-1}(\varnothing)=\varnothing$ and $f^{-1}(S)=X$. Therefore, $f$ is continuous if and only if $f^{-1}(1)$ is open in $X$. Vice versa, given any open set $U\subseteq X$, the function $X\to S$ mapping $U$ to $1$ and everything else to $0$ is continuous. Therefore there is a bijection between open sets of $X$ and continuous maps $X\to S$, that is, $O(X) \cong \Hom_\cat{Top}(X,S)$. This is natural in $X$, and so the presheaf $O$ is representable by the space $S$.
 \end{enumerate}
\end{eg}

In the interpretation given above, for example, the functor $\mathit{Loop}$ extracts from a space $X$ the features that can be obtained by ``probing'' $X$ with the space $S^1$, i.e.~by looking at all the possible ways of mapping $S^1$ into $X$.
Just as well, the topology presheaf $O$ extracts from a space $X$ all the possible ways in which $X$ can be mapped continuously to the Sierpinski space $S$, i.e.~all possible ways to observe $X$ with ``resolution'' $S$. 

\begin{eg}[graph theory]\label{g0g1}
 Let's see which functors from \Cref{graphtoset} are representable. Let $G$ be a directed multigraph.
 \begin{enumerate}
  \item Denote by $G_0$ the graph with one single vertex and no edges. For each vertex $v$ of $G$ there is a unique morphism $G_0\to G$ picking out exactly $v$. Therefore $\mathit{Vert}(G)\cong \Hom_\cat{MGraph}(G_0,G)$. This is natural in $G$, so we have that $\mathit{Vert}$ is representable by $G_0$. 
  \item Denote by $G_1$ the graph with $2$ vertices and a unique edge between them. For each edge $e$ of $G$ there is a unique morphism $G_1\to G$ picking out exactly the edge $e$. Again, this assignment is natural. Therefore $\mathit{Edge}$ is representable by the graph $G_1$.
  \item More generally, for each $n\in \N$, denote by $G_n$ be the graph of $n+1$ vertices and $n$ edges forming a chain. The functor $\mathit{Chain}_n$ is representable by the graph $G_{n}$.
  \item The set $\pi_0$ of connected components is not representable, just as for topological spaces.
 \end{enumerate}
\end{eg}

\begin{eg}[group theory]\label{2.1.9}
 The functors given in \Cref{grptoset} are representable. Let $G$ be a group.
 \begin{enumerate}
  \item The elements of $G$ are in bijection with group homomorphisms $(\Z,+)\to G$. Every element $g\in G$ defines a group homomorphism $f:\Z\to G$ defined by $f(1)=g$ (necessarily $f(0)=e$, $f(-1)=g^{-1}$, $f(2)=g^2$, and so on). Conversely, every morphism $f:\Z\to G$ gives an element of $G$ by looking at where it maps $1\in \Z$. Therefore $U:\Grp\to\Set$ is representable by $\Z$.
  Note that we cannot, as we did for topological spaces, use the one-point group $\{e\}$ as representing object: any group homomorphism from $\{e\}$ to $G$ can only have $e\in G$ as image.
  \item The elements of $G$ of order $n$ are in bijection with group homomorphisms $\Z/n\to G$.
 \end{enumerate}
\end{eg}

\begin{eg}[group theory]\label{BGrepr}
 Let $G$ be a group. We have seen (\Cref{perm}) that a functor $F:\cat{B}G\to\Set$ is a permutation representation of $G$, that is, a set $X$ (which the image $F\bullet$ of the unique object $\bullet$ of $\cat{B}G$) equipped with an action of $G$ on $X$. When is this functor representable?
 
 Plugging in the definition: $F$ is representable if for some object of $\cat{B}G$ (there is only one possible object, namely $\bullet$), $F$ is naturally isomorphic to $\Hom_{\cat{B}G}(\bullet, -)$. This in turns means that for every object of $\cat{B}G$ (again, there is only one, $\bullet$), we have that $F\bullet \cong \Hom_{\cat{B}G}(\bullet, \bullet)$, and that this bijection is natural. Now, $F\bullet$ is exactly the set $X$ on which $G$ is acting. And $\Hom_{\cat{B}G}(\bullet, \bullet)$ is the set of arrows of $\cat{B}G$ from $\bullet$ to $\bullet$, which by definition is the set of elements of $G$. Therefore, in order for $F$ to be representable, we need $X$ to be isomorphic to the underlying set $U(G)$ of $G$. That is, the group has to act on itself ($g\in G$ acts on the element $h\in G$ by mapping it to $gh\in G$). Moreover, the naturality condition says that the isomorphism $X\to U(G)$ has to be $G$-equivariant, that is, $X$ and $U(G)$ must also be isomorphic as $G$-sets. 
\end{eg}

\begin{ex}[graph theory]
 We have seen in \Cref{egmultigraph} that multigraphs are exactly presheaves on the category $\cat{Par}$ (defined there). 
 Prove that the representable presheaves on $\cat{Par}$ are, up to natural isomorphism, exactly the multigraphs $G_0$ and $G_1$, with a single vertex and a single edge respectively, appearing in \Cref{g0g1}. Which object of $\cat{Par}$ represents which presheaf?
\end{ex}

\begin{ex}[sets and relations]
 Show that the identity $\Set\to\Set$ is representable. What is the representing object?
\end{ex}

\begin{ex}[measure theory]
 Let $\cat{Meas}$ be the category of measurable spaces and measurable maps. To each measurable space $X$ assign the set $\Sigma(X)$ of its measurable subsets. Show that this is part of a presheaf $\cat{Meas}^\op\to\Set$. Is this representable?
\end{ex}

\begin{ex}[topology]\label{repronopen}
 Let $X$ be a topological space. According to \Cref{presheavesonopen}, presheaves on $X$ are presheaves on the poset $O(X)$ seen as a category. What do representable presheaves look like?
\end{ex}

\begin{ex}
 If you have done \Cref{yourcatset}, are the functors or presheaves (that you defined) representable?
\end{ex}

\subsection{The Yoneda embedding theorem}

Two natural questions may arise from the examples of the last section. 
\begin{itemize}
 \item Suppose that a functor (or presheaf) is representable. Then is the representing object necessarily unique (up to isomorphism)?
 \item Suppose that given two objects $X$ and $Y$ of $\cat{C}$, for each object $S$ we have that $\Hom_\cat{C}(S,X)$ is naturally isomorphic to $\Hom_\cat{C}(S,Y)$. That is, suppose that for each $S$, ``what $S$ sees in $X$ and $Y$ is the same''. Can we conclude that $X$ and $Y$ are isomorphic?
 (The same question can be asked with morphisms into $S$ instead of out of $S$.)
\end{itemize}

The answer to both question is \emph{yes}, and follows from one of the most important results in category theory.

\begin{thm}[Yoneda embedding]\label{yonedathm}
 Let $\cat{C}$ be a category, and let $X$ and $Y$ be objects of $\cat{C}$. There is a natural bijection of sets 
 $$
 \Hom_\cat{C} (X,Y) \; \cong \; \Hom_{[\cat{C}^\op,\Set]} \big( \Hom_\cat{C} (-,X) , \Hom_\cat{C} (-,Y) \big)
 $$
 between the morphisms of $\cat{C}$ from $X$ to $Y$ and the natural transformations from the presheaf $\Hom_\cat{C} (-,X):\cat{C}^\op\to\Set$ represented by $X$ to the presheaf $\Hom_\cat{C} (-,Y):\cat{C}^\op\to\Set$ represented by $Y$.
\end{thm}

The dual statement, replacing $\cat{C}$ by $\cat{C}^\op$ reads as follows. 
\begin{cor}
 Let $\cat{C}$ be a category, and let $X$ and $Y$ be objects of $\cat{C}$. There is a natural bijection of sets 
 $$
 \Hom_\cat{C} (X,Y) \; \cong \; \Hom_{[\cat{C},\Set]} \big( \Hom_\cat{C} (Y,-) , \Hom_\cat{C} (X,-) \big)
 $$
 between the morphisms of $\cat{C}$ from $X$ to $Y$ and the natural transformations from the functor $\Hom_\cat{C} (Y,-):\cat{C}\to\Set$ represented by $Y$ to the functor $\Hom_\cat{C} (X,-):\cat{C}\to\Set$ represented by $X$.
\end{cor}

Here are some consequences of this theorem, which give its intuitive interpretation.

\begin{cor}
 Let $\cat{C}$ be a category, and let $X$ and $Y$ be objects of $\cat{C}$.
 \begin{itemize}
 \item $X$ and $Y$ are isomorphic if and only if the functors (resp.~presheaves) that they represent are naturally isomorphic. In particular, if $X$ and $Y$ represent the same functor (resp.~presheaf), then they must be isomorphic. 
 \item $X$ and $Y$ are isomorphic if and only if for every object $S$ of $\cat{C}$, the sets $\Hom_\cat{C}(S,X)$ and $\Hom_\cat{C}(S,Y)$ (resp.~$\Hom_\cat{C}(X,S)$ and $\Hom_\cat{C}(Y,S)$) are naturally isomorphic. In particular, if $X$ and $Y$ are indistinguishable by $S$ for every $S$ in $\cat{C}$, then $X$ and $Y$ are necessarily isomorphic. 
 \end{itemize}
\end{cor}

In other words, \emph{each object of $\cat{C}$ is uniquely specified by the arrows into it (resp.~out if it)}, up to isomorphism. The objects of a category can be uniquely defined in terms of the \emph{role they play in the category}, in terms of their ``interaction with the whole''. 

\begin{caveat}
This statement can be thought of as rather ``philosophical'', and it is similar to axioms in philosophy that one can assume true or not (such as the \href{https://ncatlab.org/nlab/show/identity+of+indiscernibles}{\emph{identity of indiscernibles}}). In category theory, however, this is a \emph{theorem}, with a proof. It is true in every category.
\end{caveat}

\begin{ex}[relations, topology][difficult!]
 Prove \Cref{yonedathm} for the easier case of $\cat{C}$ being a partial order (for example, the topology of some space). What does the statement of the theorem look like? (Hint: \Cref{repronopen} can help you.)
\end{ex}

The proof of the theorem is given in the following section. 

\section{The Yoneda lemma}\label{secyoneda}

In order to prove \Cref{yonedathm} we will make use of the following statement, the \emph{Yoneda lemma}, which is at least as important as the theorem itself.

\begin{lemma}[Yoneda]\label{yonedalemma}
 Let $\cat{C}$ be a category, let $X$ be an object of $\cat{C}$, and let $F:\cat{C}^\op\to\Set$ be a presheaf on $\cat{C}$. 
 Consider the map 
 $$
 \Hom_{[\cat{C}^\op,\Set]} \big( \Hom_\cat{C} (-,X) , F \big) \longrightarrow FX
 $$
 assigning to a natural transformation $\alpha:\Hom_\cat{C} (-,X)\Rightarrow F$ the element $\alpha_X(\id_X)\in FX$, which is the value of the component $\alpha_X$ of $\alpha$ on the identity at $X$. 
 
 This assignment is a bijection, and it is natural both in $X$ and in $F$.
\end{lemma}

This is enough to prove the Yoneda embedding theorem already.

\begin{proof}[Proof of \Cref{yonedathm}]
 In the hypotheses of the Yoneda lemma, set $F$ to be the representable presheaf $\Hom_\cat{C} (-,Y):\cat{C}^\op\to\Set$.
\end{proof}

A possible interpretation of the Yoneda lemma is the following, which is rather ``trivial''. We have seen that a consequence of the Yoneda embedding theorem is that if all ways of observing two objects $X$ and $Y$ coincide, then $X$ and $Y$ must be isomorphic. The Yoneda lemma says \emph{why}: we can observe $X$ in such a way that does not lose any information, namely, mapping it to itself via the identity, $\id_X:X\to X$. This is an observation that trivially sees the whole of $X$ faithfully. Therefore, this one observation is sufficient to determine $X$ uniquely. And conversely, every other observation, which possibly loses information, is obtainable from this one, i.e.~it is a sort of ``corruption'' of this trivial observation. 
Of course, the actual statement is more complicated than this interpretation. The actual proof comes now.

\subsection{Proof of the Yoneda lemma}

First, let's see what the statement of \Cref{yonedalemma} really says. We are given an object $X$ of $\cat{C}$, and a presheaf $F:\cat{C}^\op\to\Set$. From the object $X$ we can form the representable presheaf $\Hom_\cat{C}(-,X):\cat{C}^\op\to\Set$. We can now look at the natural transformations $\Hom_\cat{C} (-,X)\Rightarrow F$. Let $\alpha$ be one such natural transformation. For each object $Y$ of $\cat{C}$, the component of $\alpha$ at $Y$ is a map between the sets $\alpha_Y:\Hom_\cat{C} (Y,X)\to FY$. Moreover, this is natural in $Y$, meaning that for every $f:Y\to Z$ in $\cat{C}$, the following diagram has to commute (note that both functors reverse the arrows).
$$
\begin{tikzcd}
 \Hom_\cat{C} (Z,X) \ar{r}{-\circ f} \ar{d}{\alpha_Z} & \Hom_\cat{C} (Y,X) \ar{d}{\alpha_Y} \\
 FZ \ar{r}{Ff} & FY
\end{tikzcd}
$$
Now, the component of $\alpha$ at $X$ is a map $\alpha_X:\Hom_\cat{C}(X,X)\to FX$. The identity at $X$ is an element of the set $\Hom_\cat{C}(X,X)$, and its image under $\alpha_X$ is an element $\alpha_X(\id_X)$ of $FX$. We can assign to $\alpha$ the element $\alpha_X(\id_X)$ of $FX$, and this gives a mapping from natural transformations $\alpha: \Hom_\cat{C} (-,X)\Rightarrow F$ to elements of $FX$. 

The lemma says that this mapping is a bijection, and it is natural. This is what we have to prove.

\begin{proof}[Proof of bijectivity]
 First of all, let $\alpha:\Hom_\cat{C} (-,X)\Rightarrow F$.
 Let $f:Y\to X$ in $\cat{C}$. By naturality of $\alpha$, the following diagram commutes.
 $$
 \begin{tikzcd}
  \Hom_\cat{C} (X,X) \ar{r}{-\circ f} \ar{d}{\alpha_X} & \Hom_\cat{C} (Y,X) \ar{d}{\alpha_Y} \\
 FX \ar{r}{Ff} & FY
 \end{tikzcd}
 $$
 If we start with the identity $\id_X\in \Hom_\cat{C} (X,X)$, the commmutativity of the diagram above says that
 \begin{equation}\label{uniqid}
  Ff(\alpha_X(\id_X)) \;=\; \alpha_Y(\id_X\circ f) \;=\; \alpha_Y(f) .
 \end{equation}
 
 Let's first show that the assignment $\alpha\mapsto \alpha_X(\id_X)$ is injective. Let $\alpha$ and $\beta$ be natural transformations $\Hom_\cat{C} (-,X)\Rightarrow F$, and suppose that $\alpha_X(\id_X)=\beta_X(\id_X)$. 
 We have to show that $\alpha=\beta$, that is, that at every object $Y$ of $\cat{C}$, the components $\alpha_Y$ and $\beta_Y:\Hom_\cat{C} (Y,X)\to FY$ are equal. So let $Y$ be such and object. For every $f\in\Hom_\cat{C} (Y,X)$, by \cref{uniqid}, we have
 $$
 \alpha_Y(f) \;=\; Ff(\alpha_X(\id_X)) \;=\; Ff(\beta_X(\id_X)) \;=\; \beta_Y(f).
 $$
 This happens for every $f\in\Hom_\cat{C}$ and for every object $Y$ of $\cat{C}$, so $\alpha=\beta$. 
 
 Let's now prove that the assignment is surjective. Let $p\in FX$. We have to show that there exists a natural transformation $\alpha:\Hom_\cat{C} (-,X)\Rightarrow F$ such that $\alpha_X(\id_X)=p$. Define now, for each object $Y$ of $\cat{C}$, the function $\alpha_Y:\Hom_\cat{C} (Y,X)\to FY$ mapping $f\in \Hom_\cat{C} (Y,X)$ to $Ff(p)\in FY$. In order to show that these functions are the components of a natural transformation $\alpha:\Hom_\cat{C} (-,X)\Rightarrow F$, we have to show that for each $g:Y\to Z$ of $\cat{C}$, the following diagram commutes,
 $$
 \begin{tikzcd}
 \Hom_\cat{C} (Z,X) \ar{r}{-\circ g} \ar{d}{\alpha_Z} & \Hom_\cat{C} (Y,X) \ar{d}{\alpha_Y} \\
 FZ \ar{r}{Fg} & FY
 \end{tikzcd}
 $$
 which means, for each $h\in\Hom_\cat{C} (Z,X)$, we should have $ Fg(\alpha_Z(h))=\alpha_Y(h\circ g)$. Now, since $F$ is functorial (reversing the direction of composition),
 $$
 Fg(\alpha_Z(h))\;=\; Fg(Fh(p)) \;=\; F(h\circ g)(p) \;=\; \alpha_Y(h\circ g).
 $$
 Therefore $\alpha$ is a natural transformation. Its component at $X$ maps the identity $\id_X$ to
 $$
 \alpha_X(\id_X) \;=\; F(\id_X) (p) \;= \; \id_{FX} (p) \;= \; p ,
 $$
 which is the desired element of $FX$.
\end{proof}

\begin{proof}[Proof of naturality]
 First we prove naturality in $X$. Let $h:X\to Y$ be a morphism of $\cat{C}$. 
 This induces a natural transformation $h_*:\Hom_\cat{C} (-,X) \Rightarrow \Hom_\cat{C} (-,Y)$ given by postcomposition with $h$. This means, for each object $A$, we get a component $(h_*)_A:\Hom_\cat{C} (A,X) \to \Hom_\cat{C} (A,Y)$ mapping $g\in \Hom_\cat{C} (A,X)$ to $h\circ g\in \Hom_\cat{C} (A,Y)$.
 Now the naturality in $X$ of the bijection of \Cref{yonedalemma} means that this diagram has to commute.
 $$
 \begin{tikzcd}
  \Hom_{[\cat{C}^\op,\Set]} \big( \Hom_\cat{C} (-,Y) , F \big) \ar{d}{\cong} \ar{r}{-\circ h_*} & \Hom_{[\cat{C}^\op,\Set]} \big( \Hom_\cat{C} (-,X) , F \big) \ar{d}{\cong} \\
  FY \ar{r}{Fh} & FX 
 \end{tikzcd}
 $$
 This means equivalently that for every natural transformation $\alpha: \Hom_\cat{C} (-,Y)\Rightarrow F$, we have that $Fh(\alpha_Y(\id_Y))$ is equal to $\alpha_X\circ h_*(\id_X)$. The latter is equal to 
 $$
 \alpha_X(h_*(\id_X)) \;=\; \alpha_X(\id_X\circ h) \;=\; \alpha_X(h),
 $$
 and this is equal to $Fh(\alpha_Y(\id_Y))$ by \Cref{uniqid}. (Note that in \Cref{uniqid}, $f$ is from $Y$ to $X$, while here, $h:X\to Y$.)
 
 Let's now turn to naturality in $F$. This means that for every presheaf $G:\cat{C}^\op\to\Set$ and every natural transformation $\beta:F\to G$, the following diagram has to commute:
 $$
 \begin{tikzcd}
 \Hom_{[\cat{C}^\op,\Set]} \big( \Hom_\cat{C} (-,X) , F \big) \ar{d}{\cong} \ar{r}{\beta\circ -} & \Hom_{[\cat{C}^\op,\Set]} \big( \Hom_\cat{C} (-,X) , G \big) \ar{d}{\cong} \\
  FX \ar{r}{\beta_X} & GX 
 \end{tikzcd}
 $$
 This means equivalently that that for every natural transformation $\alpha: \Hom_\cat{C} (-,X)\Rightarrow F$, we have that $\beta_X(\alpha_X(\id_X))$ has to be equal to $(\beta\circ\alpha)_X(\id_X)$. But this is guaranteed by definition of (vertical) composition of natural transformations, as given for example in \Cref{vertcomp}.
\end{proof}

\subsection{Particular cases}

\begin{eg}[graph theory]
 Consider the category $\cat{Par}$ from \Cref{egmultigraph}, given by the diagram
 $$
 \begin{tikzcd}
  V \ar[shift left]{r}{s} \ar[shift right]{r}[swap]{t} & E .
 \end{tikzcd}
 $$
 As we have seen, presheaves on $\cat{Par}$ encode directed multigraphs.
 The representable presheaves are, up to isomorphism, $\Hom_\cat{Par}(-,V)$ and $\Hom_\cat{Par}(-,E):\cat{Par}\to\Set$. Which multigraphs do these encode? 
 \begin{itemize}
  \item $\Hom_\cat{Par}(-,V)$ maps the object $V$ to the set arrows $V\to V$ of $\cat{Par}$, which contains just the identity of $V$, and maps $E$ to the set of arrows $E\to V$, which is empty (there are no such arrows in $\cat{Par}$). Therefore, $\Hom_\cat{Par}(V,V)$ is a singleton set, and $\Hom_\cat{Par}(E,V)$ is the empty set. The graph corresponding to $\Hom_\cat{Par}(-,V)$ is therefore a graph with a single node and no edges. In \Cref{g0g1} we called this graph $G_0$.
  \item $\Hom_\cat{Par}(-,E)$ maps $V$ to the set arrows $V\to E$ of $\cat{Par}$, which contains the two arrows $s$ and $t$, and maps $E$ to the set of arrows $E\to E$, which contains only the identity of $E$. Therefore, $\Hom_\cat{Par}(V,E)$ is a two-element set, $\Hom_\cat{Par}(E,E)$ is a singleton, and the morphisms $s$ and $t$ are mapped to functions $\Hom_\cat{Par}(E,E)\to \Hom_\cat{Par}(V,E)$ picking out the elements $s$ and $t$, respectively. The graph corresponding to $\Hom_\cat{Par}(-,E)$ is therefore a graph two distinct nodes and a single edge connecting them in one direction. In \Cref{g0g1} we called this graph $G_1$.
 \end{itemize}
 Therefore $G_0$ and $G_1$ are the graphs corresponding to representable presheaves on $\cat{Par}$. Mind that these are not the same as representable functors \emph{on} the category of multigraphs, as we had done in \Cref{g0g1}. But these constructions are related, as we show in a moment.
 
 Let's instance the Yoneda lemma, setting $\cat{C}$ equal to $\cat{Par}$, so that the category of presheaves $[\cat{C}^\op, \Set]$ becomes the category $\cat{MGraphs}$ of multigraphs and their morphisms (\Cref{egmorphgraph}). The resulting statement is that for each multigraph $G$ there are natural bijective correspondences
 $$
 \Hom_{\cat{MGraph}} \big( \Hom_\cat{Par} (-,V) , G \big) \longrightarrow GV 
 $$
 and
 $$
 \Hom_{\cat{MGraph}} \big( \Hom_\cat{Par} (-,E) , G \big) \longrightarrow GE .
 $$
 Since, as we have seen above, the representable presheaves $\Hom_\cat{Par} (-,V)$ and $\Hom_\cat{Par} (-,E)$ correspond to the graphs $G_0$ and $G_1$ respectively, and since $GV$ and $GE$ are the sets of vertices and edges of $G$, we can equivalently write the bijections as
 $$
 \Hom_{\cat{MGraph}} \big( G_0 , G \big) \;\cong\; \mathit{Vert}(G) 
 $$
 and
 $$
 \Hom_{\cat{MGraph}} \big( G_1 , G \big) \;\cong\; \mathit{Edge}(G) ,
 $$
 which is exactly what we had found in \Cref{g0g1}.
 
 Moreover, the Yoneda lemma says that this bijection is given by looking at where the identities of $V$ and $E$ are mapped. In particular,
 \begin{itemize}
  \item A multigraph morphism $f:G_0\to G$ corresponds to a vertex of $G$, and the correspondence is obtained by looking at where $f$ maps the unique vertex of $G_0$, since this vertex of $G_0$ is the unique element of $\Hom_\cat{Par}(V,V)$, i.e.~the identity of the object $V$ of $\cat{Par}$. The map $f$ can map the unique vertex of $G_0$ to any vertex of $G$, and different maps will pick out different vertices, so we have a bijection between vertices of $G$ and multigraph morphisms $G_0\to G$.
  \item A multigraph morphism $f:G_1\to G$ corresponds to an edge of $G$, and the correspondence is obtained by looking at where $f$ maps the unique edge of $G_1$, since this vertex of $G_1$ is the unique element of $\Hom_\cat{Par}(E,E)$, i.e.~the identity of the object $E$ of $\cat{Par}$. The map $f$ can map the unique edge of $G_1$ to any edge of $G$, and different maps will pick out different edges, so we have a bijection between edges of $G$ and multigraph morphisms $G_1\to G$.
 \end{itemize}
 This again follows the intuition of \Cref{g0g1}. 
\end{eg}

\begin{eg}[group theory]\label{egcayley}
 We know that for a group $G$, functors $\cat{B}G\to\Set$ are sets equipped with a $G$-action, or $G$-sets, and that the only representable functor $\cat{B}G\to\Set$ is given by $G$ acting on its underlying set $U(G)$. If we look at presheaves instead of functors, $\cat{B}G^\op\to\Set$, the situation is analogous, except that $G$ will act on the right instead of on the left. That is, presheaves $\cat{B}G^\op\to\Set$ correspond to \emph{right} $G$-sets, and the unique representable presheaf is given by the action on $G$ on its underlying set by \emph{right} multiplication, i.e.~$g$ acts on $G$ by mapping $h$ to $hg$ instead of $gh$. 
 
 The Yoneda embedding \Cref{yonedathm} says now the following, noting that $\cat{B}G$ has a single object $\bullet$:
 $$
 \Hom_{\cat{B}G} (\bullet,\bullet) \; \cong \; \Hom_{[{\cat{B}G}^\op,\Set]} \big( \Hom_{\cat{B}G} (-,\bullet) , \Hom_{\cat{B}G} (-,\bullet) \big). 
 $$
 Since the set $\Hom_{\cat{B}G} (\bullet,\bullet)$ consists by definition of the elements of $G$, and since the representable functors are given by the underlying set of $G$, the statement becomes
 $$
 G \; \cong \; \Hom_{\cat{GSet}} \big( U(G) , U(G) \big). 
 $$
 The group on the right is a subgroup of the full permutation group of the set $U(G)$. Therefore $G$ is isomorphic to a subgroup of a permutation group of some set. In other words, we have proven Cayley's \Cref{thmcayley}. 
 
 The Yoneda lemma says more: given any (right) $G$-set $X$, which is expressed by a presheaf $F:\cat{B}G\to\Set$, there is a natural bijection between elements of $X$ and $G$-equivariant maps $f:U(G)\to X$. This bijection is given by looking at where the identity element $e$ of $G$ will be mapped by $f$: every element of $X$ can be chosen to be the image of $e$, as long as the other elements of $G$ are ``rigidly'' mapped along the orbit. Moreover, different choices of $f$ will necessarily map the identity to different points. To see the latter statement, notice that, just as in the proof of \Cref{yonedalemma}, by equivariance (i.e.~naturality), for every $g\in G$,
 $$
 f(g) \;= \; f(e g) \;=\; f(e)\cdot g,
 $$
 so that if $f,f':U(G)\to X$ agree on $e$, they must agree on all of $U(G)$.
\end{eg}

\begin{eg}[linear algebra; {\cite[Corollary~2.2.9]{ctcontext}}] 
 Consider the matrix category $\cat{Mat}$ of \Cref{matcat}. Consider \emph{linear operation on rows} of matrices, such as ``multiply the second row by 2 and add it to the first one'', i.e.~mapping for example
 $$
 \begin{pmatrix}
a & b & c \\
d & e & f
\end{pmatrix} 
\;\longmapsto \;
 \begin{pmatrix}
a + 2d & b + 2e & c + 2f \\
d & e & f
\end{pmatrix} .
 $$
 Given any linear operation on rows, in order to apply it to a given matrix $M$ one can equivalently first apply the operation to the identity matrix, and then multiply $M$ by the resulting matrix. For example, we can first apply our operation above to the $2\times 2$ identity matrix,
$$
 \begin{pmatrix}
1 & 0 \\
0 & 1
\end{pmatrix} 
\;\longmapsto \;
 \begin{pmatrix}
1 + 2\cdot 0 & 0 + 2\cdot 1  \\
0 & 1
\end{pmatrix}  =
\begin{pmatrix}
1  & 2  \\
0 & 1
\end{pmatrix},
 $$
 and then multiply our given matrix by the image of the identity:
 $$
 \begin{pmatrix}
1  & 2  \\
0 & 1
\end{pmatrix}
\begin{pmatrix}
a & b & c \\
d & e & f
\end{pmatrix} 
\;= \;
 \begin{pmatrix}
a + 2d & b + 2e & c + 2f \\
d & e & f
\end{pmatrix} .
 $$
This is an instance of the Yoneda embedding theorem. Linear operations on rows (for example, of matrices of $n$ rows) correspond to natural transformations between presheaves $\phi:\Hom_{\cat{Mat}}(-,n)\Rightarrow\Hom_{\cat{Mat}}(-,n)$, since the naturality condition says precisely that the operation $\phi$ commutes with (other) matrices, i.e.~linear combinations. The Yoneda embedding theorem for $X=Y=n\in \N$ now says that there is a natural bijection 
$$
 \Hom_{\cat{Mat}} (n,n) \; \cong \; \Hom_{[{\cat{Mat}}^\op,\Set]} \big( \Hom_{\cat{Mat}} (-,n) , \Hom_{\cat{B}G} (-,n) \big). 
 $$
 that is, linear operations on $n$ rows are in bijection with elements of $\Hom_{\cat{Mat}} (n,n)$, which are by definition the $n\times n$ matrices. Moreover, this correspondence is given by looking at where the identity of $n$ is mapped, that is, by applying the operation to the $n\times n$ identity matrix, and then letting naturality (i.e.~linearity) do the rest.
\end{eg}

\section{Universal properties}\label{uniprop}

We have seen that thanks to the Yoneda lemma objects are specified, uniquely up to isomorphism,
by the functor or presheaf that they represent. We interpreted this as the fact that objects are uniquely specified by the way they interact with the rest of the category.  When an object is specified in this way, it is said to satisfy a universal property. 

\begin{deph}
 Let $\cat{C}$ be a category, and $X$ an object of $\cat{C}$. A \emph{universal property} of $X$ consists of either a functor $F:\cat{C}\to\Set$ or a presheaf $P:\cat{C}^\op\to\Set$ together with either a natural isomorphism $\Hom_\cat{C}(X,-)\Rightarrow F$ or $\Hom_\cat{C}(-,X)\Rightarrow P$.
\end{deph}

\begin{remark}
 By definition, the functor or the presheaf given as above is necessarily representable.
\end{remark}

\begin{remark}
 A natural isomorphism $\Hom_\cat{C}(-,X)\Rightarrow P$ is in particular a natural transformation. By the Yoneda lemma, such natural transformations are in bijection with elements of the set $PX$. Note that not all natural transformations are isomorphisms in general, so only \emph{some} of the elements of $PX$ give us a natural isomorphism. 
 
 Just as well, replacing $\cat{C}$ by $\cat{C}^\op$, we have that natural isomorphisms $\Hom_\cat{C}(X,-)\Rightarrow F$ correspond to some of the elements of the set $FX$. 
\end{remark}

A universal property for the object $X$ can be thought of as a condition of \emph{existence and uniqueness} of a specified map into $X$ (for presheaves) or out of $X$ (for functors), in the following way. Let $P:\cat{C}^\op\to\Set$ be a representable presheaf, and let $\alpha:\Hom_\cat{C}(-,X)\Rightarrow P$ be a chosen natural isomorphism, so that we have a universal property for $X$. Consider now a map $f:Y\to X$. We can form the usual naturality diagram
$$
\begin{tikzcd}
 \Hom_\cat{C} (X,X) \ar{r}{-\circ f} \ar{d}{\alpha_X}[swap]{\cong} & \Hom_\cat{C} (Y,X) \ar{d}{\alpha_Y}[swap]{\cong} \\
 PX \ar{r}{Pf} & PY
\end{tikzcd}
$$
where now the components of $\alpha$ are bijections.
Starting as usual with the identity of $X$ in the top left corner and denote $\alpha_X(\id_X)$ by $p$, we get that $Pf(p)=\alpha_Y(f)$. The fact that $\alpha_Y$ is a bijection means the following: for each element $x$ of $PY$ there exists a unique $f:Y\to X$ such that $\alpha_Y(f)=x$, i.e.~such that $\alpha_Y(\id\circ f)=x$. Since the diagram above commutes, equivalently, $\alpha_Y$ is a bijection if and only if \emph{for each element $x$ of $PY$ there exists a unique $f:Y\to X$ such that $Pf(p)=x$}. 
That is, there is \emph{exactly one} arrow $f$ into $X$ satisfying the condition that $Pf(p)=x$. 
Moreover, the element $p\in PX$ specifies a natural \emph{isomorphism} rather than just a natural transformation $\Hom_\cat{C}(-,X)\Rightarrow P$ if and only if for every object $Y$, $\alpha_Y$ is a bijection. That is, if for every object $Y$ and for every element $x$ of $PY$, there exists a unique $f:Y\to X$ such that $Pf(p)=x$.

Dually, a universal property given by a functor $\cat{C}\to\Set$ instead of a presheaf will give a condition for existence and uniqueness of arrows out of $X$. 

In order to have an interpretation for why we want $f$ to satisfy $Pf(p)=x$, let's see a couple of examples.
First, let's fix some notation. In diagrams, we will write a dashed arrow, such as
$$
\begin{tikzcd}
 X \ar[dashed]{r}{f} & Y
\end{tikzcd}
$$
whenever $f$ it satisfies a condition of existence and uniqueness coming from a universal property. 

Let's now see two examples of universal properties: the cartesian product of topological spaces, and the tensor product of vector spaces.

\subsection{The universal property of the cartesian product}\label{uniprod}

\begin{eg}[topology] 
 Let $\cat{C}$ be the category $\cat{Top}$. Let $X$ and $Y$ be objects of $\cat{Top}$, i.e.~topological spaces. Consider the presheaf $P$ given by 
 $$
 \Hom_\cat{Top}(-,X) \times \Hom_\cat{Top}(-,Y) \;:\; \cat{Top}^\op \to \Set.
 $$
 This presheaf maps a space $S$ to the set $\Hom_\cat{Top}(S,X) \times \Hom_\cat{Top}(S,Y)$, whose elements are pairs of arrows
 $$
 \begin{tikzcd}[sep=small]
  & S \ar{dr} \ar{dl} \\
  X && Y .
 \end{tikzcd}
 $$
 On morphisms, $P$ maps a continuous map $f:S\to T$ to the function
 $$
 \begin{tikzcd}[sep=huge]
 \Hom_\cat{Top}(T,X) \times \Hom_\cat{Top}(T,Y) \ar{r}{(-\circ f)\times (-\circ f)} & \Hom_\cat{Top}(S,X) \times \Hom_\cat{Top}(S,Y)
 \end{tikzcd}
 $$
 mapping the pair 
 $$
 \begin{tikzcd}[sep=small]
  & T \ar{dr}{b} \ar{dl}[swap]{a} \\
  X && Y
 \end{tikzcd}
 $$
 to the pair
 $$
 \begin{tikzcd}[sep=small]
  & S \ar{dr}{b\circ f} \ar{dl}[swap]{a\circ f} \\
  X && Y 
 \end{tikzcd}
 $$
 by precomposition with $f$ in both components. 
 Intuitively, one may think of this presheaf as of being a ``combined'' observation onto $X$ and $Y$, using two instruments, or two eyes.
 
 Now, is this presheaf representable? 
 The question can be unpacked as follows: is there an object $Z$ of $\cat{Top}$, i.e.~a topological space, such that maps into $Z$ have a natural bijection with pairs of maps into $X$ and $Y$? Or more intuitively, is there a space $Z$ such that observations with instrument $Z$ are the same as combined observations with the instruments $X$ and $Y$? 
 
 Following the guidelines above, a natural isomorphism 
 $$
 \Hom_\cat{Top}(-,Z) \;\Rightarrow\; \Hom_\cat{Top}(-,X) \times \Hom_\cat{Top}(-,Y) 
 $$
 is equivalently given by an element $p$ of $PZ$ such that for each object $S$ and for each element $x\in PS$ there exists a unique map 
 $$
\begin{tikzcd}
 S \ar[dashed]{r}{f} & Z
\end{tikzcd}
$$
such that $Pf(p)=x$. Let's now unpack this condition. 
 Since $PZ=\Hom_\cat{Top}(Z,X) \times \Hom_\cat{Top}(Z,Y)$, an element $p\in PZ$ is given by a pair of maps
 $$
 \begin{tikzcd}[sep=small]
  & Z \ar{dr}{p_2} \ar{dl}[swap]{p_1} \\
  X && Y .
 \end{tikzcd}
 $$
 Just as well, an element $x\in PS=\Hom_\cat{Top}(S,X) \times \Hom_\cat{Top}(S,Y)$ is given by a pair of maps
 $$
 \begin{tikzcd}[sep=small]
  & S \ar{dr}{f_2} \ar{dl}[swap]{f_1} \\
  X && Y .
 \end{tikzcd}
 $$
 Moreover, the condition $Pf(p)=x$ says that 
 $$
 \begin{tikzcd}[sep=small]
  & S \ar{dr}{p_2\circ f} \ar{dl}[swap]{p_1\circ f} \\
  X && Y .
 \end{tikzcd}
 \quad = \quad 
 \begin{tikzcd}[sep=small]
  & S \ar{dr}{f_2} \ar{dl}[swap]{f_1} \\
  X && Y 
 \end{tikzcd}
 $$
 which means equivalently that $f_1=p_1\circ f$ and $f_2=p_2\circ f$. 
 In other words, the condition, which is the universal property for $Z$, reads as follows: fix the two maps $p_1:Z\to X$ and $p_2:Z\to Y$. Then for every object $S$ and every pair of maps $f_1:S\to X$ and $f_2:S\to Y$, there exists a unique map $f:S\to Z$ such that this diagram commutes:
 $$
 \begin{tikzcd}
  & S \ar{dr}{f_2} \ar{dl}[swap]{f_1} \uni{d}{f} \\
  X & Z \ar{l}[swap]{p_1} \ar{r}{p_2} & Y 
 \end{tikzcd}
 $$
 In other words, maps $S\to Z$ are ``the same'' as pairs of maps $S\to X$ and $S\to Y$. 
 Now, can we find such object $Z$ and maps $p_1:Z\to X$ and $p_2:Z\to Y$? 
 
 As you probably know, we can. The object $Z$ is given by the cartesian product $X\times Y$, equipped with the product topology, and the maps $p_1:X\times Y\to X$ and $p_2:X\times Y\to Y$ are the two product projections. 
 The correspondence is given as follows. Given the maps $f_1$ and $f_2$, one can construct the map $f$ as the one mapping $s\in S$ to $(f_1(s),f_2(s))\in X\times Y$. Different choices of $f_1$ and $f_2$ give different maps $f:S\to X\times Y$, and all the maps $S\to X\times Y$ are in this form for some $f_1, f_2$, since the elements of $X\times Y$ are uniquely specified by their components. It remains to check that the resulting map $f$ is continuous, in order for it to be a morphism of $\cat{Top}$ -- for this see the next exercise. 
\end{eg}

\begin{ex}[topology]
 Show that for every topological space $S$, the map $f:S\to X\times Y$ is continuous (for the product topology on $X\times Y$) whenever $f_1:S\to X$ and $f_2:S\to Y$ are continuous. (Hint: pick convenient generating open sets of the product topology.)
\end{ex}

Therefore the product of two topological spaces is uniquely specified, up to homeomorphism, by its universal property. In some sense, it is the ``inevitable'' way of combining two spaces $X$ and $Y$ in terms of the maps into them.

\begin{remark}
 The maps $p_1:X\times Y\to X$ and $p_2:X\times Y\to Y$ are part of the universal property of the product. Indeed, in order to have a universal property, it is not enough to say that ``there is'' a natural isomorphism, one needs to specify which one. In this case, the choice of the natural isomorphism corresponds to the choice of the maps $p_1:X\times Y\to X$ and $p_2:X\times Y\to Y$. Intuitively, without a choice of these maps, it may be unclear how to relate $X\times Y$ with $X$ and $Y$, especially in presence of symmetries.
\end{remark}

We can construct a similar object in categories other than $\cat{Top}$.
\begin{ex}[linear algebra]
  Show that the cartesian product of vector spaces satisfies a similar universal property in $\Vect$.
\end{ex}
\begin{ex}[group theory]
  Show that the direct product of groups satisfies a similar universal property in $\Grp$.
\end{ex}
\begin{ex}\label{2.3.9}
 In a category pertaining to your field of application, can one give a similar construction?
\end{ex}

This construction, as we will see, is the special case of a very important universal construction, called a \emph{limit}. More on this will come in the next lectures.

\begin{ex}[topology]\label{prodfunctor}
 Given topological spaces $X$ and $Y$, the space $X\times Y$ is again a topological space. Now fix $X$. Show that the assignment $Y\mapsto X\times Y$ is part of a functor $X\times -:\cat{Top}\to\cat{Top}$. (Hint: what could it do on morphisms?)
\end{ex}

\subsection{The universal property of the tensor product}\label{sec_tens}

Let $V$, $W$, and $U$ be vector spaces. A map $f:V\times W\to U$ is called \emph{bilinear} if it is linear in $V$ and in $W$ separately. Note that this is not the same as being linear as a map between the vector spaces $V\times W$ and $U$. For example, the map $\R\times \R\to \R$ given by $(x,y)\mapsto xy$ is bilinear but not linear. 
A geometric example of a bilinear map is the (signed) area enclosed by the parallelogram of two vectors. A scalar product on a vector space is another such example. 

\begin{ex}[linear algebra]\label{bilinearlinear}
 Let $V$, $W$, $U$, and $U'$ be vector spaces. Let $f:V\times W\to U$ be a bilinear map and $g:U\to U'$ be a linear map. Show that $g\circ f: V\times W\to U'$ is bilinear.
\end{ex}

Let's now look at a new universal property, encoded by bilinearity.

\begin{eg}[linear algebra]
 Let $V$, $W$, and $U$ be vector spaces.
 Denote by $\mathrm{Bilin}(V,W; U)$ the set of bilinear maps $V\times W\to U$. If we keep $V$ and $W$ fixed, this gives a functor $\mathrm{Bilin}(V,W; -):\Vect\to\Set$ in the following way. Given a linear map $g:U\to U'$, we get a function $g\circ -:\mathrm{Bilin}(V,W; U)\to \mathrm{Bilin}(V,W; U')$ mapping a bilinear function $f: V\times W\to U$ to the bilinear function $g\circ f: V\times W\to U'$, which is again bilinear by \Cref{bilinearlinear}. 
 
 Now, is this functor representable? We look for a space $Z$ and a natural isomorphism
 $$
 \Hom_\Vect(Z,-) \; \Rightarrow \; \mathrm{Bilin}(V,W; -) .
 $$
 Following the guidelines above, this amounts to an element of the set $q\in \mathrm{Bilin}(V,W; Z)$ such that for every vector space $S$ and for each element $b\in \mathrm{Bilin}(V,W; S)$ there exists a unique linear map 
 $$
 \begin{tikzcd}
 Z \ar[dashed]{r}{f} & S
\end{tikzcd}
 $$
 such that $\mathrm{Bilin}(V,W; f)(q)=b$. Let's unpack this definition. The element $q\in \mathrm{Bilin}(V,W; Z)$ is a bilinear map $q:V\times W\to Z$, and this map has to satisfy the property that for each vector space $S$ and for each bilinear map $b:V\times W\to S$, there exists a unique linear map $f$ such that $f\circ q=b$, i.e.~such that the following diagrams commutes.
 $$
 \begin{tikzcd}
 V\times W \ar{d}[swap]{q} \ar{dr}{b} \\
 Z \ar[dashed,swap]{r}{f} & S
 \end{tikzcd}
 $$
 (Note that the solid maps are bilinear, not linear, so this is technically not a diagram of $\Vect$. It's just here for convenience.)
 
 Again, as you probably know, such an object $Z$ exists, and it is given by the tensor product of vector spaces $V\otimes W$. The map $q:V\times W\to V\otimes W$ is the tensor product of vectors $(v,w)\mapsto v\otimes w$, which is bilinear. (Note that we are using the same symbol $\otimes$ for two different purposes, as in $V\otimes W$ and as in $v\otimes w$. Can you tell the difference?)
 Here is how the bijection works. The elements of $V\otimes W$ are linear combinations of vectors in the form $v\otimes w$ for some $v\in V$ and $w\in W$. Therefore a linear map $f:V\otimes W\to S$ is uniquely specified by what it does on vectors in the form $v\otimes w$, i.e.~in the image of $q$. Given a multilinear map $b:V\times W\to S$, define $f:V\otimes W \to S$ to be the linear map specified by $f(v\otimes w)=b(v,w)$. This is linear since $b$ is bilinear. A different choice of $b$ gives necessarily a different $f$, since the $v\otimes w$ span the whole space $V\otimes W$, and every linear map $V\otimes W\to S$ arises this way.
\end{eg}

This gives the universal property of the tensor product. Note, again, that the map $q$ is part of the universal property, it is part of the tensor product structure. 
 Intuitively, the tensor product is the ``unique way to combine $V$ and $W$ in terms of bilinear maps on them''.

\begin{ex}[linear algebra; difficult!]
 Fix a vector space $V$. Show that the assignment $W\mapsto V\otimes W$ is part of a functor $V\otimes -:\Vect\to\Vect$. 
\end{ex}

\newpage
\chapter{Limits and colimits}

\section{General definitions}

\begin{deph}
 Let $\cat{C}$ be a category. Let $\cat{J}$ be a small category, and let $X$ be an object of $\cat{C}$. The \emph{constant diagram at $X$} (indexed by $\cat{J}$) is the functor $X:\cat{J}\to \cat{C}$ assigning:
 \begin{itemize}
  \item To each object of $\cat{J}$, always the same object $X$ of $\cat{C}$;
  \item To each morphism of $\cat{J}$, the identity morphism $\id_X$. 
 \end{itemize}
\end{deph}

This is analogous to a constant function.

\begin{deph}
 Let $\cat{J}$ be a small category, let $F:\cat{J}\to\cat{C}$ be a diagram (i.e.~a functor), and let $X$ be an object of $\cat{C}$. A \emph{cone over $F$} with tip $X$ is a natural transformation from the constant diagram at $X$ to the functor $F$. A \emph{cone under $F$}, or \emph{cocone}, with bottom $X$ is a natural transformation from the functor $F$ to the constant diagram at $X$.
\end{deph}

Explicitly, a cone over $F$ is the following assignment: for each object $J$ in $\cat{J}$ we have a morphism $\alpha_J:X\to FJ$ of $\cat{C}$, in such a way that for every morphism $m:J\to J'$ of $\cat{J}$, this triangle commutes
$$
\begin{tikzcd}
 & X \ar{dl}[swap]{\alpha_J} \ar{dr}{\alpha_{J'}} \\
 FJ \ar{rr}{Fm} && FJ' .
\end{tikzcd}
$$

\begin{eg}
 Consider the diagram
 $$
 \begin{tikzcd}[row sep=small]
  & B \ar{dr} \\
  A \ar{ur} \ar{dr} && D \\
  & C \ar{ur}
 \end{tikzcd}
 $$
 where the objects $A$, $B$, $C$ and $D$ are in the form $FJ$ for some objects $J$ of $\cat{J}$, and the arrows of the diagram are in the form $Fm:FJ\to FJ'$ for some arrows $m:J\to J'$ of $\cat{J}$.
 A cone and a cocone over this diagram look as follows,
 $$
 \begin{tikzcd}[row sep=small, column sep=large]
  & X \ar{dddl} \ar{dddr} \ar[bend right=15]{dd} \ar[bend left]{dddd} \\ \\
  & B \ar{dr} \\
  A \ar{ur} \ar{dr} && D \\
  & C \ar{ur}
 \end{tikzcd}
 \qquad
 \begin{tikzcd}[row sep=small, column sep=large]
  & B \ar{dr} \ar[bend left=23]{dddd} \\
  A \ar{dddr} \ar{ur} \ar{dr} && D \ar{dddl} \\
  & C \ar[crossing over]{ur} \ar[bend right=18]{dd}
  \\ \\   & Y 
 \end{tikzcd}
 $$
 hence the name.
 All the triangles involving $X$ and $Y$ are commutative (but the original diagram may not necessarily commute).
\end{eg}

The cone construction is functorial, as follows. Given a cone of tip $X$ and a morphism $f:X\to Y$, we get a cone of tip $Y$ via the following composition,
$$
 \begin{tikzcd}[row sep=small, column sep=large]
  & Y \ar{dddl} \ar{dddr} \ar[bend right=15]{dd} \ar[bend left]{dddd} \\ \\
  & B \ar{dr} \\
  A \ar{ur} \ar{dr} && D \\
  & C \ar{ur}
 \end{tikzcd}
 \quad\longmapsto \quad
 \begin{tikzcd}[row sep=tiny, column sep=large]
  & X \ar{dd}{f} \\ \\
  & Y \ar{dddl} \ar{dddr} \ar[bend right=15]{dd} \ar[bend left]{dddd} \\ \\
  & B \ar{dr} \\
  A \ar{ur} \ar{dr} && D \\
  & C \ar{ur}
 \end{tikzcd}
 $$
 (Note the direction of the arrows.)
 Here is the more rigorous definition. You can read it keeping the picture above in mind.
 
\begin{deph}
 Let $F:\cat{J}\to\cat{C}$ be a diagram. We construct the presheaf $\Cone(-,F):\cat{C}^\op\to\Set$ as follows.
 \begin{itemize}
  \item It maps an object $X$ of $\cat{C}$ to the set $\Cone(X,F)$ of cones over $F$ with tip $X$;
  \item It maps a morphism $f:X\to Y$ to the function $\Cone(Y,F) \to \Cone(X,F)$ given in the following way. It assigns to a cone $\alpha:Y \Rightarrow F$ of components $\alpha_J:Y\to FJ$ for each object $J$ of $\cat{J}$ the cone $\alpha\circ f:X\Rightarrow F$ of components $(\alpha\circ f)_J:=\alpha_j\circ f$.
 \end{itemize}
 
 Just as well, the functor $\Cone(F,-):\cat{C}\to\Set$ is defined as follows.
 \begin{itemize}
  \item It maps an object $X$ of $\cat{C}$ to the set $\Cone(F,X)$ of cocones under $F$ with bottom $X$;
  \item It maps a morphism $f:X\to Y$ to the function $\Cone(F,X) \to \Cone(F,Y)$ given in the following way. It assigns to a cocone $\alpha:F \Rightarrow X$ of components $\alpha_J:FJ\to X$ for each object $J$ of $\cat{J}$ the cocone $f\circ\alpha:F\Rightarrow Y$ of components $(f\circ\alpha)_J:=f\circ\alpha_j$. 
 \end{itemize}
\end{deph}

Limits and colimits are \emph{universal} cones and cocones. 

\begin{deph}
 Let $F:\cat{J}\to\cat{C}$ be a diagram. A \emph{limit} of $F$, if it exists, is an object $\lim F$ of $\cat{C}$ representing the presheaf $\Cone(-,F):\cat{C}^\op\to\Set$, together with its universal property.
 
 A \emph{colimit} of $F$, if it exists, is an object $\colim F$ of $\cat{C}$ representing the functor $\Cone(F,-):\cat{C}\to\Set$, together with its universal property.
\end{deph}

Let's see what this means concretely. The objects $\lim F$ and $\colim F$, whenever they exist, are equipped by definition with natural isomorphisms
$$
\Hom_\cat{C}(-, \lim F) \Rightarrow \Cone(-, F) \quad \mbox{and} \quad \Hom_\cat{C}(\colim F, -) \Rightarrow \Cone(F,-) .
$$
As we saw in \Cref{uniprop}, by the Yoneda lemma these natural transformations are uniquely specified by ``universal'' elements of the sets
$$
\Cone(\lim F, F) \quad \mbox{and} \quad  \Cone(F,\colim F) ,
$$
that is, cones in the following form.
$$
 \begin{tikzcd}[row sep=small, column sep=large]
  & \lim F \ar{dddl} \ar{dddr} \ar[bend right=15]{dd} \ar[bend left]{dddd} \\ \\
  & B \ar{dr} \\
  A \ar{ur} \ar{dr} && D \\
  & C \ar{ur}
 \end{tikzcd}
 \qquad
 \begin{tikzcd}[row sep=small, column sep=large]
  & B \ar{dr} \ar[bend left=23]{dddd} \\
  A \ar{dddr} \ar{ur} \ar{dr} && D \ar{dddl} \\
  & C \ar[crossing over]{ur} \ar[bend right=18]{dd}
  \\ \\   & \colim F
 \end{tikzcd}
 $$
Moreover, these cones are universal in the following sense. For the limit cone (the diagram on the left), given any (other) cone with any tip $X$, such as
$$
 \begin{tikzcd}[row sep=small, column sep=large]
  & X \ar{dddl} \ar{dddr} \ar[bend right=15]{dd} \ar[bend left]{dddd} \\ \\
  & B \ar{dr} \\
  A \ar{ur} \ar{dr} && D \\
  & C \ar{ur}
 \end{tikzcd}
 $$
 there is a unique map $u:X\to \lim F$ such that for each $J\in \cat{J}$, each component $\alpha_J:X\to FJ$ of the cone of tip $X$ factors uniquely through $u$ and through the component $\phi_J:\lim F\to FJ$ of the limit cone. That is, for each $J$ this triangle has to commute.
 $$
 \begin{tikzcd}
  X \uni{d}[swap]{u} \ar{dr}{\alpha_J} \\
  \lim F \ar[swap]{r}{\phi_J} & FJ .
 \end{tikzcd}
 $$
 In our example,
 $$
 \begin{tikzcd}[row sep=small, column sep=large]
  & X \ar{ddddl} \ar{ddddr} \ar[bend right=15]{ddd} \ar[bend left]{ddddd}  \\ \\ \\
  & B \ar{dr} \\
  A \ar{ur} \ar{dr} && D \\
  & C \ar{ur}
 \end{tikzcd}
 \quad= \quad
 \begin{tikzcd}[row sep=tiny]
  & X \uni{dd}{u} \\ \\
  & \lim F \ar{dddl} \ar{dddr} \ar[bend right=15]{dd} \ar[bend left]{dddd} \\ \\
  & B \ar{dr} \\
  A \ar{ur} \ar{dr} && D \\
  & C \ar{ur}
 \end{tikzcd}
 $$
 That is, 
 $$
 \begin{tikzcd}[row sep=small]
  &&& X \uni{dd}{u} \ar[bend right=15]{dddddll} \ar[bend left=15]{dddddrr}  \\ \\
  &&& \lim F \ar{dddll} \ar{dddrr} \ar[bend right=15]{dd} \ar[bend left]{dddd} \\ 
  \mbox{these triangles commute.}\ar[-]{rr} && \, \\
  &&& B \ar{drr} \\
  &A \ar{urr} \ar{drr} &&&& D \\
  &&& C \ar{urr}
 \end{tikzcd}
 $$
 
 For the colimit cone, given any cocone under $F$ with bottom $Y$, such as 
 $$
 \begin{tikzcd}[row sep=small, column sep=large]
  & B \ar{dr} \ar[bend left=23]{dddd} \\
  A \ar{dddr} \ar{ur} \ar{dr} && D \ar{dddl} \\
  & C \ar[crossing over]{ur} \ar[bend right=18]{dd}
  \\ \\   & Y 
 \end{tikzcd}
 $$
 there is a unique map $u:\colim F\to Y$ such that 
 for each $J\in \cat{J}$, each component $\beta_J:FJ\to Y$ of the cone of bottom $Y$ factors uniquely through $u$ and through the component $\psi_J:FJ\to\colim F$ of the colimit cone. That is, for each $J$ this triangle has to commute.
 $$
 \begin{tikzcd}
  FJ \ar{d}[swap]{\psi_J} \ar{dr}{\beta_J} \\
  \colim F \uni{r}[swap]{u} & Y .
 \end{tikzcd}
 $$
 In our example,
 $$
 \begin{tikzcd}[row sep=small, column sep=large]
  & B \ar{dr} \ar[bend left=23]{ddddd} \\
  A \ar{ddddr} \ar{ur} \ar{dr} && D \ar{ddddl} \\
  & C \ar[crossing over]{ur} \ar[bend right=18]{ddd}
  \\ \\ \\   & Y 
 \end{tikzcd}
 \quad= \quad
 \begin{tikzcd}[row sep=tiny]
  & B \ar{dr} \ar[bend left=28]{dddd} \\
  A \ar{dddr} \ar{ur} \ar{dr} && D \ar{dddl} \\
  & C \ar[crossing over]{ur} \ar[bend right=18]{dd}
  \\ \\   & \colim F \uni{dd}{u} \\ \\ 
  & Y
 \end{tikzcd}
 $$
 That is, 
 $$
 \begin{tikzcd}[row sep=small]
  &&& B \ar{drr} \ar[bend left=28]{dddd} \\
  &A \ar[bend right=15]{dddddrr} \ar{dddrr} \ar{urr} \ar{drr} &&&& D \ar{dddll} \ar[bend left=15]{dddddll} \\
  &&& C \ar[crossing over]{urr} \ar[bend right=18]{dd}  \\ 
  \mbox{these triangles commute.} \ar[-]{rr} && \,  \\
  &&& \colim F \uni{dd}{u} \\ \\ 
  &&& Y
 \end{tikzcd}
 $$
 
As we know by the Yoneda lemma, moreover, if $\lim F$ and $\colim F$ exist, they are unique up to isomorphism.

\begin{deph}
 A category $\cat{C}$ is called \emph{complete} if every diagram in $\cat{C}$ has a limit. $\cat{C}$ is called \emph{cocomplete} if every diagram in $\cat{C}$ has a colimit.
\end{deph}

Let's now see some practical examples of limits and colimits.

\section{Particular limits and colimits}

\subsection{The poset case: constrained optimization}

Let $(X,\le)$ be a poset. If we view it as a category, then a diagram in $X$ has the same information as a subset of $X$, because given any two elements, there is up to one arrow between them. 
So let $S\subseteq X$. Plugging in the definition of cone and cocone, we get the following.
A cone over $S$ with tip $x$ is an element $x$ together with an arrow $x\to s$ for each element $x\in S$. That is, it is an element $x\in X$ such that for all $s\in S$, $x\le s$. This is the same as a \emph{lower bound} for $S$. 
A cocone over $S$ with bottom $y$ is an element $y$ together with arrows $s\to y$, that is, $s\le y$, for each $s\in S$. Therefore it is the same as an \emph{upper bound} for $S$. Mind the possibly confusing direction of the arrow: if $x$ is the tip of a cone \emph{over} $S$ in the sense of category theory, it is actually \emph{below} $S$ in the partial order.

A cone of tip $x$ is a limit cone if the following universal property holds: given any (other) lower bound $z$ of $S$, we have a (necessarily unique) arrow $z\to x$, that is, $z\le x$. Equivalently, for every $z\in X$, we have that $z\le s$ for every $s\in S$ if and only if $z\le x$.

\begin{ex}[sets and relations]
 Prove that these two statements are indeed equivalent. 
\end{ex}

In other words, we are saying that $x$ is the \emph{greatest lower bound}, or \emph{infimum}, or \emph{meet} of $S$. This is usually denoted by $\inf S$ or $\meet S$.
Just as well, $y$ defines a colimit cone if and only if it is the \emph{least upper bound}, or \emph{supremum}, or \emph{join} of $S$. This is usually denoted by $\sup S$ or $\join S$. 

As you probably know from analysis, infima and suprema do not always exist, but when they do, they are unique.

In a way, the universal properties of limits and colimits in a poset correspond to solutions of a constrained optimization problem: the limit (meet) of $S$ is the largest element which is still below $S$, while the colimit (join) is the smallest element which is still above $S$. Because of this, one may view limits and colimits as a generalization of constrained optimization. This idea that limits and colimits are the largest or smallest objects satisfying a certain property can help the intuition also in most of the next examples.

\subsection{The group case: invariants and orbits}\label{invorbit}

Let $G$ be a group. We know that a functor (or a diagram) $F:\B G\to \Set$ is the same as a $G$-set, a set $X$ equipped with an action of $G$. Let's see what cones, cocones, limits, and colimits are. 

Plugging in the definition, a cone over $F$ is a set $S$ together with a map $f:S\to X$ such that for each $g\in G$, the following triangle commutes.
$$
\begin{tikzcd}[column sep=small]
 & S \ar{dl}[swap]{f} \ar{dr}{f} \\
 X \ar{rr}{g\cdot} && X
\end{tikzcd}
$$
In other words, for every $s\in S$ and $g\in G$, we must have $g\cdot f(s) = f(s)$. This means exactly that every element in the image of $f$ is an invariant element for the action. In particular, if $f$ is injective, $S$ is a subset of invariant elements. 

\begin{caveat}
 A \emph{set of invariant elements} is not the same thing as an \emph{invariant set}. If $S$ is a set of invariant elements, then applying $g$ to any $s\in S$ leaves $s$ where it is. Instead, if $S$ is an invariant set, applying $g$ to $s\in S$ may move it, but it will stay inside $S$. For example, for rotations in the plane, a circle centered at the origin is an invariant set (its points move). The origin is an invariant element.
\end{caveat}

Let us now wee what the limit is. Plugging in the definition, $\lim F$ is a cone such that for any (other) cone, say with tip $S$, there is a unique map $S\to \lim F$ making the following diagram commute for each $g\in G$.
$$
\begin{tikzcd}[column sep=small, row sep=small]
 & S \ar{ddl} \ar{ddr} \uni{d} \\
 & \lim F \ar{dl} \ar{dr} \\
 X \ar{rr}{g\cdot} && X
\end{tikzcd}
$$
So, first of all, the map $\lim F\to X$ must be injective. Here is why. Let $I\subseteq X$ be the image of this map. Then the inclusion $I\to X$ (which is injective) would also give a cone over $F$ (why?), and every map of a cone $S\to X$ would factor through $I$, that is, would be a composition $S\to I\to X$. But this is precisely the universal property of the limit, and so $I$ would be the limit. Since limits are unique up to isomorphism, we conclude that $\lim F \cong I$, and so the map $\lim F\to X$ must be injective too.
Therefore, $\lim F$ is a subset of invariant elements. Now which subset? Since every subset $S$ of invariant elements gives a cone over $F$, and by the universal property of the limit, the inclusion $S\to X$ has to factor as $S\to \lim F \to X$, it must necessarily be that $\lim F$ contains all the invariant elements.
Therefore, $\lim F\to X$ is exactly the \emph{set of all invariant elements for the action}. 
In terms of the natural isomorphism, the universal property reads: a function into $X$ picking out only $G$-invariant elements is the same as a function into $\lim F$.
In terms of constrained optimization, it is the largest subset of $X$ whose elements are all invariant for the $G$-action. Note that this set may be empty, but as a set, it always exists.

Let us now turn to cocones. A cocone under $F$, plugging in the definition, is a set $Y$ together with a map $p:X\to Y$ such that for each $g\in G$, the following diagram commutes.
$$
\begin{tikzcd}[column sep=small]
 X \ar{rr}{g\cdot} \ar{dr}[swap]{p} && X  \ar{dl}{p}  \\
 & Y
\end{tikzcd}
$$
In other words, for each $x\in X$ and $g\in G$, we must have $p(g\cdot x)=p(x)$. This is the same thing as an \emph{invariant function}: a function whose value does not change after applying $g$ to its argument.

We can view the function $p$ also in a different way, as follows. 
Any function $p:X\to Y$ can be seen as a \emph{partition} of $X$. That is, we can divide $X$ into different regions, according to the values that $p$ assumes. For the value $y\in Y$, we have the region $p^{-1}(y)\subseteq X$ of the elements $x\in X$ such that $p(x)=y$. For a different value $y'\in Y$, we have the region $p^{-1}(y')\subseteq X$ of the elements $x\in X$ such that $p(x)=y'$, and so on. We have that $X$ is the disjoint union
$$
X \;\cong\; \coprod_{y\in Y} p^{-1}(y) .
$$
Note that some of the regions may be empty (precisely when $p$ is not surjective).
Now, if and only if $p$ is an invariant function, we have that for all $y\in Y$, $p^{-1}(y)$ is an invariant set: if $x\in f^{-1}(y)$, then $f(x)=y$, but then also for each $g\in G$, $f(g\cdot x)=y$, that is, $g\cdot x\in f^{-1}(y)$.
(Mind, this is an invariant set, not necessarily a set of invariant elements, see the remark above.)
We are saying that $f$ is an invariant function if and only if, if we see it as a partition, all the regions of the partition are invariant sets, that is, $G$ acts on those regions separately, never moving elements from one region to another.
Therefore, a cocone under $F$ is the same thing as a \emph{partition of invariant sets}, also called 
\emph{invariant partition}.

What about the colimit of $F$? Plugging in the definition, $\colim F$ is a cocone such that for any (other) cone, say with bottom $Y$, there is a unique map $\colim F\to Y$ making the following diagram commute for each $g\in G$.
$$
\begin{tikzcd}[column sep=small, row sep=small]
 X \ar{ddr} \ar{dr} \ar{rr}{g\cdot} && X \ar{dl} \ar{ddl} \\
  & \colim F  \uni{d}  \\
   & Y   
\end{tikzcd}
$$
Dually to the case above, the map $X\to\colim F$ must be first of all surjective. Here is why. We can consider its image $I\subseteq \colim F$, and the map $X\to I$ is then surjective by construction. Again, any (other) map $X\to Y$ defining a cocone must factor through $\colim F$  by the universal property of the colimit, and therefore also through $I$ (why?). We can again conclude that $I\cong \colim F$, and so the map $X\to\colim F$ must be surjective as well. 
Therefore, in terms of partitions, the map $X\to\colim F$ corresponds to a partition where no region is empty. Now, let $O$ be an orbit of $G$. Consider the function $p$ from $X$ to the 2-element set $\{0,1\}$ mapping $O$ to $0$ and everything else to $1$. This function is invariant (why?). If we want this function to factor through $\colim F$, then the partition induced by $\colim F$ must be as fine as the one induced by $p$, that is, it needs to have $O$ as one of its regions. But this can be done for all orbits $O$ of $X$, and therefore the partition induced by $\colim F$ must precisely be the partition of $X$ into the orbits of the $G$-action.
In terms of the natural isomorphism, the universal property reads: an invariant function on $F$ is the same thing as a function on $\lim F$, that is, a function on the set of orbits.
In terms of constrained optimization, this is the finest $G$-invariant partition. Again, notice that such a partition always exists.

\begin{ex}[group theory, linear algebra]
 Let $F:\B G\to \Vect$ be a linear representation of $G$. What are its limit and colimit? 
\end{ex}

\begin{ex}[group theory, topology]
 Let $F:\B G\to \cat{Top}$ be an action of $G$ on a topological space. What are its limit and colimit? 
\end{ex}

\subsection{Products and coproducts}

Consider now the category $\Set$, and take two sets $A$ and $B$. They form the discrete diagram $F$ which simply looks as follows:
$$
\begin{tikzcd}
 A && B
\end{tikzcd}
$$
Discrete simply means that the only arrows present in the diagram are the identities (not drawn).
Now, a cone over $F$ consists of a set $X$ together with maps $f:X\to A$ and $g:X\to B$, i.e.
$$
\begin{tikzcd}
  & X \ar{dl}[swap]{f} \ar{dr}{g} \\
 A && B
\end{tikzcd}
$$
A limit cone over $F$ consists of a cone with the following universal property. Given any (other) cone over $F$, say with tip $X$, there is a unique map $X\to\lim F$ making the following diagram commute.
$$
\begin{tikzcd}
  & X \ar{ddl} \ar{ddr} \uni{d} \\
  & \lim F \ar{dl} \ar{dr} \\
 A && B
\end{tikzcd}
$$

You may recognize that this is exactly the universal property of the cartesian product $A\times B$, as we saw in \Cref{uniprod}. A way to see it is to look at the case of $X=1$, the one-element set. Maps $1\to A$ and $1\to B$ correspond to elements of $A$ and $B$, respectively. The universal property says in this case that an element of $A\times B$ is equivalently a pair of elements $(a,b)$ of $A$ and $B$ respectively. More generally, a function into $A\times B$ is the same as a pair of functions into $A$ and $B$.
Therefore, $A\times B$ is (up to isomorphism) the limit of the diagram $F$, and the maps forming the universal cone are the product projections.

\begin{deph}
 The limit of a discrete diagram is called the \emph{product}, or \emph{cartesian product}, or \emph{categorical product}. If the diagram contains two objects $A,B$ their product is usually denoted by $A\times B$, if the diagram contains many objects $\{A_i\}_{i\in I}$, their product is usually denoted by $\prod_i A_i$. 
\end{deph}

Intuitively, products are used to combine objects to form more complex ``composite'' objects, in such a way that the complex object can always be projected back to its simpler components via the projection maps $A\times B\to A$ and $A\times B\to B$ forming the universal cone. 
We have seen products of topological spaces already, in \Cref{uniprod}. The exercises in that section involve similar constructions for vector spaces and for groups.

\begin{ex}[graph theory]
 What is the product of two graphs in the category $\cat{MGraph}$ of multigraphs and their morphisms (\Cref{egmorphgraph})? 
\end{ex}

\begin{ex}\label{productfunctorial}
 Can we generalize \Cref{prodfunctor} to arbitrary categories (assuming all the necessary products exist)?
\end{ex}

What is the colimit of $F$? Reversing all the arrows, this is a cocone such that for each (other) cocone, say with bottom $Y$, there exists a unique map $\colim F\to Y$ such that the following diagram commutes:
$$
\begin{tikzcd}
 A \ar{dr} \ar{ddr} && B \ar{dl} \ar{ddl} \\
 & \colim F \uni{d} \\
 & Y
\end{tikzcd}
$$
This means that $\colim F$ is the set such that having a function on $\colim F$ is the same as having a function on $A$ and a function on $B$. In other words, $\colim F$ is the \emph{disjoint union} $A\sqcup B$: a function on the disjoint union of $A$ and $B$ is uniquely specified by its values on $A$ (which give a function on $A$) and its values on $B$ (which give a function on $B$), and vice versa. The universal maps forming the cocone are the canonical inclusions $A\to A \sqcup B$ and $B\to A\sqcup B$.

\begin{deph}
 The colimit of a discrete diagram is called the \emph{coproduct} or sometimes \emph{sum}. If the diagram contains two objects $A,B$, their coproduct is usually denoted by either $A+B$ or $A\sqcup B$, if the diagram contains many objects $\{A_i\}_{i\in I}$, their coproduct is usually denoted by $\coprod_i A_i$.
\end{deph}

Intuitively, coproducts are used to combine objects to form more complex objects by ``adding them together'', in such a way that the simple objects are contained in the complex object via the inclusion maps $A\to A \sqcup B$ and $B\to A\sqcup B$ forming the universal cocone. 

The symbol $\coprod$ is used because it is the upside-down version of the symbol $\prod$. 

\begin{ex}[linear algebra]\label{biproduct}
 Show that in $\cat{Vect}$, the coproduct of two vector spaces is their direct sum. Conclude that, for two (or finitely many) vector spaces, the product and the coproduct are isomorphic. 
\end{ex}

\begin{ex}[topology]
 Show that the coproduct in $\cat{Top}$ is given by the disjoint union, with the disjoint union topology.
\end{ex}

\begin{ex}[group theory]
 Show that in $\cat{Grp}$, the coproduct of two groups is given by their free product.
\end{ex}

\begin{ex}[graph theory]
 What is the coproduct of two graphs in $\cat{MGraph}$?
\end{ex}

\subsection{Equalizers and coequalizers}\label{eq}

Consider the following diagram in $\cat{Set}$, let's call it $F$ again, given by a parallel pair of arrows.
$$
\begin{tikzcd}
 A \ar[shift left]{r}{f} \ar[shift right,swap]{r}{g} & B
\end{tikzcd}
$$
Note that this in general is not a commutative diagram (unless $f=g$).

A cone over $F$ of tip $X$ consists of maps $p:X\to A$ and $q:X\to B$ such that the following triangles commute.
$$
\begin{tikzcd}
 & X \ar{dl}[swap]{p} \ar{dr}{q} \\
 A \ar{rr}{f} && B
\end{tikzcd}
\qquad
\begin{tikzcd}
 & X \ar{dl}[swap]{p} \ar{dr}{q} \\
 A \ar{rr}{g} && B
\end{tikzcd}
$$
This means, in particular, that $f\circ p = q = g\circ p$. So, since the map $q$ is just a composition, we can omit it from the diagram, and rewrite the condition saying that a cone of tip $X$ over $F$ is a map $p:X\to A$ as in the following diagram,
$$
\begin{tikzcd}
 X \ar{r}{p} & A \ar[shift left]{r}{f} \ar[shift right,swap]{r}{g} & B
\end{tikzcd}
$$
such that $f\circ p = g\circ p$. (Again, note that this does not imply $f=g$ unless $p$ is epi. Therefore the diagram above is not commutative in general.)
This means, in other words, that for every $x\in X$, $f(p(x))=g(p(x))$, or equivalently, that $f$ and $g$ agree on the image of $p$. In particular, if $p$ is injective, $X$ can be seen as a subset of $A$ on which $f$ and $g$ agree. Note that this does not imply that $f$ and $g$ necessarily disagree outside of $X$: $X$ is not necessarily the \emph{largest} subset on which $f$ and $g$ agree.  

What is now a limit of the diagram $F$?
First of all, just like it happened in \Cref{invorbit}, the universal map $\lim F \to A$ must be injective (see \Cref{eqmono}). Therefore, the limit of $F$ must be a subset of $A$ on which $f$ and $g$ agree.
Moreover, also just as it happened in \Cref{invorbit}, since any subset of $A$ on which $f$ and $g$ agree gives a cone over $F$, the limit of $F$ must contain all such subsets. In other words, $\lim F$ is precisely the largest subset of $A$ where $f$ and $g$ agree, and the universal map $\lim F\to A$ is the inclusion.

\begin{deph}
 The limit of a pair of parallel arrows $f,g$ is called the \emph{equalizer} of $f,g$. 
\end{deph}

Intuitively, equalizers are used to \emph{define subspaces by means of equations}. This is done systematically in algebraic geometry, but it appears also in many other fields of math.

\begin{eg}[basic geometry]
 In the category $\cat{Top}$, consider the diagram 
 $$
\begin{tikzcd}
 \R^2 \ar[shift left]{r}{f} \ar[shift right,swap]{r}{g} & \R
\end{tikzcd}
$$
where $f$ is the function given by $f(x,y)\coloneqq x^2 + y^2$, and $g$ is the constant function with value $1$. 
The equalizer of $f$ and $g$ is the largest subset $S$ of $\R^2$ such that for all of its points $(x,y)\in S$, $x^2+y^2=1$. This is precisely the unit circle $S^1$. 
\end{eg}

\begin{eg}[topology]
 Show that, generalizing the example above, the equalizer in $\cat{Top}$ of two maps $f,g:A\to B$ is given by the subset of $A$ on which $f$ and $g$ agree, equipped with the subspace topology inherited by $A$.
\end{eg}

\begin{eg}[linear algebra, group theory]
 Show that in the categories $\cat{Vect}$ and $\cat{Grp}$, given a map $f:A\to B$ and the constant zero map $0:A\to B$, the kernel of $f$ is the equalizer of $f$ and $0$. 
\end{eg}

\begin{ex}\label{eqmono}
 Let $m:E\to A$ be the equalizer of $f,g:A\to B$ in a category $\cat{C}$. Prove that $m$ is necessarily mono. (Hint: the idea is similar to what happens in \Cref{invorbit}.)
\end{ex}

\begin{ex}
 In the notation of the exercise above, prove that if and only if $m$ is epi, it is an isomorphism, and $f=g$. 
\end{ex}

What is a cocone under $F$? Similarly as above, a cocone under $F$ with bottom $Y$ is given by a map $q:B\to Y$ as in the following diagram,
$$
\begin{tikzcd}
A \ar[shift left]{r}{f} \ar[shift right,swap]{r}{g} & B \ar{r}{q} & Y
\end{tikzcd}
$$
such that $q\circ f=q\circ g$ (why exactly?). Note that again this does not imply that $f=g$ (unless $q$ is mono).
In other words, the function $q$ is such that for every $a\in A$, $q(f(a))=q(g(a))$. Intuitively, $q$ cannot tell whether the points come through $f$ or through $g$, it cannot tell the two maps apart. Let's make this precise. Consider the following relation on $B$: $b\sim b'$ if and only if there exists $a\in A$ such that $b=f(a)$ and $b'=g(a)$. Then our condition on $q$ says precisely that $q$ respects the relation, i.e.~$b\sim b'$ implies that $q(b)=q(b')$. If we view $q$ as a partition, we are saying that this partition cannot be finer than the relation $\sim$, that is, no two related elements $b\sim b'$ can come from different regions of the partition.

\begin{ex}[sets and relations]
 The relation $\sim$ defined above is in general not an equivalence relation. Show that, in any case, $q$ also respects the equivalence relation generated by $\sim$. Such an equivalence relation can be characterized as the smallest (finest) equivalence relation containing $\sim$, or as the one obtained from $\sim$ by enforcing reflexivity, symmetry, and transitivity.
\end{ex}

As you may probably expect, then, a colimit cone is given by the \emph{finest} function which respects this relation. In other words, analogously to what happens in \Cref{invorbit}, $\colim F$ is the quotient space of $B$ under the equivalence relation generated by $\sim$, and the universal map $Y\to\colim F$ is the quotient map.
In terms of partitions, $q$ corresponds precisely to the partition of $B$ into the equivalence classes generated by $\sim$.
This is because, first of all, the universal map is always surjective (\Cref{coeqepi}). Moreover, let $C\subseteq B$ be an equivalence class generated by $\sim$, and consider the map $q:B\to\{0,1\}$  mapping $C$ to $0$ and everything else to $1$. This map $q$ forms a cocone (why?). If we want $q$ to factor through the colimit, as the universal property prescribes, the only possibility is that the universal map $B\to \colim F$, seen as a partition, has $C$ as one of its regions. Again as in \Cref{invorbit}, this must be true for all such equivalence classes.

\begin{deph}
 The colimit of a pair of parallel arrows $f,g$ is called the \emph{coequalizer} of $f,g$. 
\end{deph}

Coequalizers are used to create quotient spaces canonically, by \emph{identification}. 

\begin{ex}[topology]
 Consider the following diagram in $\cat{Top}$,
 $$
 \begin{tikzcd}
 \bullet \ar[shift left]{r}{1} \ar[shift right,swap]{r}{0} & \left[0,1\right]
 \end{tikzcd}
 $$
 where $\bullet$ is the space of a single point, $[0,1]$ is the real unit interval, and the maps $0$ and $1$ map the unique point of $\bullet$ respectively to $0$ and $1\in[0,1]$. 
 Show that the coequalizer of these two maps is given by the circle $S^1$, as the space obtained by identifying $0$ and $1$ in the unit interval $[0,1]$.
\end{ex}

\begin{ex}\label{coeqepi}
 Prove the dual statement to \Cref{eqmono}, that is, that the universal map of the coequalizer of $f$ and $g$ is always epi.
 
 Conclude also that if and only if such map is also mono, then it is an isomorphism, and $f=g$. 
\end{ex}

\subsection{Pullbacks and pushouts}

Consider the following diagram, called a \emph{cospan}. Let's call the diagram $F$, as usual.
$$
\begin{tikzcd}
 & A \ar{d}{f} \\
 B \ar{r}{g} & C
\end{tikzcd}
$$
A cone over $F$ of tip $X$ corresponds to an object $X$ together with maps $p:X\to A$ and $q:X\to B$ such that this diagram commutes. 
$$
\begin{tikzcd}
 X\ar{d}{q} \ar{r}{p} & A \ar{d}{f} \\
 B \ar{r}{g} & C
\end{tikzcd}
$$
As before, the map $X\to C$ is implicit, since it is just a composition (of which maps?).

A limit cone is a cone such that for every (other) cone, say with tip $X$, there is a unique arrow $X\to \lim F$ such that the following diagram commutes.
$$
\begin{tikzcd}
 X \ar{dddr} \ar{drrr} \uni{dr} \\
 & \lim F \ar{rr} \ar{dd} && A \ar{dd}{f} \\ \\
 & B \ar{rr}{g} && C
\end{tikzcd}
$$

\begin{deph}
 The limit of a cospan $f:A\to C$, $g:B\to C$ is called \emph{pullback} or \emph{fibered product}. 
 As an object, it is usually denoted by $A\times_C B$. The maps giving the universal cone are denoted by $f^*g:A\times_C B\to A$ and $g^*f:A\times_C B \to B$.
 
 The universal cocone is sometimes called a \emph{pullback square} or \emph{cartesian} square, and it is denoted by putting either the symbol $\ulcorner$ or $\lrcorner$ in the corner of $A\times_C B$, as in one of these diagrams.
 $$
 \begin{tikzcd}
  A\times_C B \ar{d}[swap]{g^*f} \ar{r}{f^*g} \pullback & A \ar{d}{f} \\
  B \ar{r}{g} & C
 \end{tikzcd}
 \qquad
 \begin{tikzcd}
  A\times_C B \ar{d}[swap]{g^*f} \ar{r}{f^*g} \arrow[dr, phantom, "\lrcorner", very near start] & A \ar{d}{f} \\
  B \ar{r}{g} & C
 \end{tikzcd}
 $$
 (We will use the convention on the left.)
\end{deph}

Here are two exercises for people with a background in differential geometry. These also explain the names ``pullback'' and ``fibered product'', and the notation $f^*$.
\begin{ex}[differential geometry, topology]
 Let $p:E\to X$ be a fiber bundle. Let $f:Y\to X$ be a smooth map. Show that the map $f^*p:E\times_X Y\to Y$ forming the pullback
 $$
\begin{tikzcd}
 E\times_X Y\ar{d}[swap]{f^*p} \ar{r}{p^*f} \pullback & E \ar{d}{p} \\
 Y \ar{r}{f} & X
\end{tikzcd}
 $$
 is the pullback bundle of $p$ along $f$.
\end{ex}

\begin{ex}[differential geometry, topology]\label{productoffibers}
 Let $p:E\to X$ and $q:F\to X$ be fiber bundles on the same base space $X$. Show that the map $E\times_X F\to X$ obtained from the pullback
 $$
\begin{tikzcd}
 E\times_X F\ar{d}[swap]{q^*p} \ar{r}{p^*q} \pullback & E \ar{d}{p} \\
 F \ar{r}{q} & X
\end{tikzcd}
 $$
 is a fiber bundle whose fiber is the product of the fibers of $p$ and $q$.
\end{ex}

\begin{caveat}
 The notion of pullback in category theory does not always coincide with the notion of pullback in other fields of mathematics. In the case of bundles the two notion coincide, but keep in mind that this is not always the case. 
 This is even more evident in the dual case, where in category theory we have a construction called \emph{pushout} (see below), while in other fields of mathematics the dual to pullback is \emph{pushforward}. These are in general very different operations.  
\end{caveat}

To see how the pullback works, consider these two exercises in the category of sets.

\begin{ex}[sets and relations]
 In the category $\Set$, let $S$ and $T$ be subsets of $X$ together with the inclusion maps
 $$
 \begin{tikzcd}
  & S \mono{d} \\
  T \mono{r} & X
 \end{tikzcd}
 $$
 Show that the pullback of these diagrams is given by the intersection of $S$ and $T$. What are the maps of the universal cone?
\end{ex}

\begin{ex}[sets and relations]\label{pullbackone}
 In the category $\Set$, let $1$ be the one-element set. Show that $A\times_1 B \cong A \times B$. 
\end{ex}

Now, in more general categories.

\begin{ex}[topology, graph theory]
 In the categories $\cat{Top}$ and $\cat{Mgraph}$, do we have similar results as the previous two exercises?
\end{ex}

\begin{ex}\label{kernelpair}
 Consider a morphism $f:X\to Y$ in a category $\cat{C}$. Let's form the pullback of $f$ with itself (if it exists). 
 $$
 \begin{tikzcd}
  X \times_Y X \ar{r} \ar{d} \pullback & X \ar{d}{f} \\
  X \ar{r}{f} & Y
 \end{tikzcd}
 $$
 The two universal maps $X \times_Y X \to X$, which in this case are parallel maps, are called the \emph{kernel pair} of $f$. 
 Here is now the exercise: prove that $f$ is mono if and only if those maps coincide and are isomorphisms $X\times_Y X \cong X$. 
 
 This is why where the name ``kernel'' comes from: the kernel pair of $f$ is trivial if and only if $f$ is mono. 
\end{ex}

The next exercise can help with the intuition of ``what the kernel pair really does''.

\begin{ex}[sets and relations]
 In the category $\Set$, let $m:X\to Y$ be a function, and consider its kernel pair, denoting the resulting two maps $f,g:X \times_Y X \to X$. Consider the relation $\sim$ defined by $f$ and $g$ constructed as in \Cref{eq}, that is, for each $x,x'\in X$, we have $x\sim x'$ if and only if there exists $p\in X \times_Y X$ such that $f(p)=x$ and $g(p)=x'$. Show that: 
 \begin{itemize}
  \item In this case, $\sim$ is automatically an equivalence relation;
  \item $x\sim x'$ if and only if $f(x)=f(y)$, so that as a partition, $\sim$ is exactly the partition induced by $f$ whenever $f$ is surjective;
  \item if there exists $p\in X \times_Y X$ such that $f(p)=x$ and $g(p)=x'$, then this $p$ is unique. 
 \end{itemize}
\end{ex}

\begin{ex}[difficult!]\label{kercoeq}
 More generally, suppose that a morphism $f:X\to Y$ in a category $\cat{C}$ has a kernel pair (i.e.~it exists), and suppose moreover that $f$ is the coequalizer of \emph{something}. Show that $f$ is then (also) the coequalizer of its kernel pair.
 Why does this exercise generalize the previous exercise?
\end{ex}

Let us now consider the dual case. Consider this diagram, called a \emph{span}. Let's call this diagram $G$. 
$$
\begin{tikzcd}
 A \ar{r}{f} \ar{d}{g} & B \\
 C
\end{tikzcd}
$$
A cocone under $G$ of bottom $Y$ consists of an object $Y$ together with maps $B\to Y$ and $C\to Y$ such that the following diagram commutes. Again, there is no need to give the map $A\to Y$ (why?).
$$
\begin{tikzcd}
 A \ar{r}{f} \ar{d}{g} & B \ar{d} \\
 C \ar{r} & Y
\end{tikzcd}
$$

A colimit cocone is now a cocone such that for any (other) cocone, say with bottom $Y$, there is a unique map $\colim G\to Y$ such that the following diagram commutes:
$$
\begin{tikzcd}
 A \ar{rr}{f} \ar{dd}{g} && B \ar{dd} \ar{dddr} \\ \\
 C \ar{rr} \ar{drrr} && \colim G \uni{dr} \\
 && & Y
\end{tikzcd}
$$

\begin{deph}
 The colimit of a span is called \emph{pushout}. As an object, it is usually denoted by $B \sqcup_A C$. The maps giving the universal cone are denoted by $f_*g:A\to B \sqcup_A C$ and $g_*f:B\to B \sqcup_A C$.
 
 The universal cocone is sometimes called a \emph{pushout square} or \emph{cocartesian} square, and it is denoted by putting either the symbol $\ulcorner$ or $\lrcorner$ in the corner of $B \sqcup_A C$, as in one of these diagrams.
 $$
 \begin{tikzcd}
  A \ar{r}{f} \ar{d}{g} & B \ar{d}{f_*g}  \\
  C \ar{r}{g_*f} & B \sqcup_A C \pushout
 \end{tikzcd}
 \qquad
 \begin{tikzcd}
  A \ar{r}{f} \ar{d}{g} & B \ar{d}{f_*g}  \\
  C \ar{r}{g_*f} & B \sqcup_A C \arrow[ul, phantom, "\ulcorner", very near start]
 \end{tikzcd}
 $$
 (We will use the convention on the left.)
\end{deph}

Here are some exercises in the category of sets.

\begin{ex}[sets and relations]
 In the category $\Set$, let $\varnothing$ be the empty set. Show that $A\coprod_\varnothing B \cong A \coprod B$. (Hint: this is somewhat dual to \Cref{pullbackone}).
\end{ex}

\begin{ex}[sets and relations]
 More generally, let $S$ be a subset both of $A$ and of $B$, with the inclusion maps forming the following span.
 $$
  \begin{tikzcd}
   S \mono{d} \mono{r} & A \\
   B
  \end{tikzcd}
 $$
 Show that $A\coprod_S B$ is given by the \emph{non-disjoint} union of $A$ and $B$, i.e.~where the elements of $S$ are counted only once. What are the maps of the universal cocone? Also, why does this generalize the previous exercise?
 Can you also say why this is somewhat dual to \Cref{productoffibers}?
\end{ex}

In a generic category, we have the dual statement to \Cref{kernelpair}.

\begin{ex}
 Consider a morphism $f:X\to Y$ in a category $\cat{C}$. Let's form the pushout of $f$ with itself (if it exists). 
 The two universal maps $Y\to Y\coprod_X Y$ are called the \emph{cokernel pair} of $f$. 
 Prove that $f$ is epi if and only if those maps coincide and are isomorphisms. 
\end{ex}

\begin{ex}[difficult!]
 State and prove (independently) the dual statement to \Cref{kercoeq}.
\end{ex}

\subsection{Initial and terminal objects, trivial cases}

Let $\cat{O}$ be the empty category, i.e.~the one with no objects (and no morphisms). This category is small (why?), and there is a unique diagram $E:\cat{O}\to\cat{C}$ for every category $\cat{C}$, which we call the \emph{empty diagram} in $\cat{C}$.
The empty diagram looks like this:

\vspace{3cm}

A cone over the empty diagram $E$ of tip $X$ consists simply of an object $X$, with no other arrows except the identity of $X$ (why?). Just as well, a cocone of bottom $Y$ under $E$ is just $Y$ with its identity. So far everything is trivial, however, the \emph{limit} and \emph{colimit} of $E$ are more interesting. 

The \emph{limit} of the empty diagram, plugging in the universal property, is an object $\lim E$ such that for every (other) object $X$, there is a unique morphism
$$
\begin{tikzcd}
 X \uni{r} & \lim E .
\end{tikzcd}
$$

\begin{deph}
 The limit of the empty diagram in $\cat{C}$, if it exists, is called the \emph{terminal object} of $\cat{C}$. It is usually denoted by $1$. 
\end{deph}

The terminology comes from the fact that, since every object has a unique morphism to $1$, one may view $1$ as the ``endpoint'' of the category. This is particularly visible in the poset case:

\begin{eg}[sets and relations]
 Let $(X,\le)$ be a poset. A terminal object of $X$, if it exists, is an object $1$ such that for each $x\in X$, $x\le 1$. In other words, the terminal object is the top element of $X$, if it exists. The top element is also denoted by $\top$.
\end{eg}

The notation ``$1$'' comes instead from the following fact.

\begin{eg}[sets and relations]
 In $\Set$, the terminal object is the singleton set $1$, i.e.~a set with exactly one element. This is because for any other set $X$, there exists only one function $X\to 1$, namely the one that assigns to each $x\in X$ the unique element of $1$.
\end{eg}

\begin{ex}[topology]
 Show that the terminal object of $\cat{Top}$ is the one-point space.
\end{ex}

\begin{ex}[graph theory]
 Show that the terminal object of $\cat{MGraph}$ is a graph with a single vertex and\dots How many edges?
\end{ex}

\begin{caveat}
 The word ``terminal'' does \emph{not} mean that, in general, there are no arrows out of $1$ except the identity. In a poset this is true, but not in general, for example, not in $\cat{Set}$.
\end{caveat}

The colimit of the empty diagram $E$ is an object $\colim E$ such that for each (other) object $Y$, there is a unique arrow 
$$
\begin{tikzcd}
 \colim E \uni{r} & Y. 
\end{tikzcd}
$$

\begin{deph}
 The colimit of the empty diagram in $\cat{C}$, if it exists, is called the \emph{initial} or \emph{coterminal object} of $\cat{C}$. It is usually denoted by $0$. 
\end{deph}

\begin{eg}[sets and relations]
 Let $(X,\le)$ be a poset. An initial object of $X$, if it exists, is an object $0$ such that for each $x\in X$, $0 \le x$. In other words, the terminal object is the bottom element of $X$, if it exists. The bottom element is also denoted by $\bot$.
\end{eg}

The notation ``$0$'' comes instead from the following fact.

\begin{eg}[sets and relations]
 In $\Set$, the initial object is the empty set $\varnothing$. This is because for any other set $X$, there exists only one function $\varnothing \to X$, namely the one that doesn't assign anything. Note that this function still exists, even if it's a trivial case. (Just as zero is a trivial number, but still a number.)
 In particular there is a unique function $\varnothing\to\varnothing$, the identity.
\end{eg}

\begin{ex}[topology]
 Show that the initial object of $\cat{Top}$ is the empty space.
\end{ex}

\begin{ex}[graph theory]
 What is the initial object of $\cat{MGraph}$?
\end{ex}

\begin{ex}[linear algebra, group theory]\label{zeroobject}
 Show that the initial objects of $\Vect$ and $\Grp$ coincide with the terminal objects. An object which is both initial and terminal is sometimes called a \emph{zero object}.\footnote{This may be confusing, since both zero objects and initial objects (which are not necessarily terminal) are usually denoted as $0$.}
\end{ex}

\begin{caveat}
 Just as for terminal objects, the word ``initial'' does \emph{not} mean that, in general, there are no arrows into $0$ except the identity. In a poset this is true, and also in $\cat{Set}$ (why?), but not in general, for example, not in $\cat{Vect}$.
\end{caveat}

It turns out that \emph{every limit can be seen as a terminal object in some category}. Just as well, every colimit can be seen as an initial object in some category. 

\begin{deph}\label{slice}
 Let $D:\cat{J}\to\cat{C}$ be a diagram. The \emph{slice category over $D$}, denoted by $\cat{C}/D$, is the category whose:
 \begin{itemize}
  \item Objects are pairs $(X,\alpha)$ where $X$ is an object of $X$ and $\alpha$ is a cone over $D$ with tip $X$;
  \item A morphism $(X,\alpha)\to (X,\beta)$ is given by a morphism $f:X\to Y$ of $\cat{C}$ such that for every object $I\in\cat{J}$, the following triangle commutes.
  $$
  \begin{tikzcd}[column sep=small]
   X \ar{rr}{f} \ar{dr}[swap]{\alpha_X} && Y \ar{dl}{\beta_Y} \\
   & DI
  \end{tikzcd}
  $$
 \end{itemize}
\end{deph}

One can see the slice category as the category of cones over $D$ with all possible tips, and morphisms between them making the triangles of the cones commute.

\begin{ex}
 Show that the limit cone of $D$, if it exists, is the terminal object of $\cat{C}/D$.
\end{ex}

Therefore, complicated limits can be seen as simple limits in a complicated category.

\begin{ex}\label{coslice}
 Define the \emph{coslice category under $D$}, denoted by $D/\cat{C}$, in such a way that the colimit cone of $D$, if it exists, is the initial object of $D/\cat{C}$.
\end{ex}

Let's conclude this part by giving some trivial cases.

\begin{ex}
 Let $X$ be an object of $\cat{C}$, seen as a diagram (with only its identity). Show that its limit and colimit are both (isomorphic to) $X$.
\end{ex}

\begin{ex}
 Let now $f:X\to Y$ be a morphism of $\cat{C}$, seen as a diagram. Show that its limit is given by $X$ and its colimit is given by $Y$. Why does this generalize the previous exercise?
\end{ex}

\begin{ex}
 Consider the following diagrams.
 $$
 \begin{tikzcd}[row sep=tiny]
  & B \\
  A \ar{dr} \ar{ur}\\ 
  & C
 \end{tikzcd}
 \qquad \qquad
 \begin{tikzcd}[row sep=tiny]
  D \ar{dr} \\
  & F \\
  E \ar{ur}
 \end{tikzcd}
 $$
 Show that the limit of the diagram on the left is $A$, and the colimit of the diagram on the right is $F$. 
\end{ex}

\begin{ex}[important!]
 More generally, let $D:\cat{J}\to\cat{C}$ be a diagram, and suppose that $\cat{J}$ has an initial object $0$. Prove that $\lim D$ exists and is given by $D(0)$. What is the dual statement?
 Also, why does this generalize the previous two exercises?
\end{ex}

Therefore, limits and colimits can be seen as ways to ``complete the diagram'' by adding to it an initial or terminal object whenever such object does not exist already.

\section{Functors, limits and colimits}\label{sec_fun_lim}

We have seen that in general functors preserve commutative diagrams, but not the property of being mono or epi. 
In general, functors do not preserve limits and colimits either, as well as any other universal property. 

\begin{deph}
 A functor $F:\cat{C}\to\cat{C'}$ is called \emph{continuous} if it preserves all the limits which exist in $\cat{C}$. In other words, for every diagram $D$ in $\cat{C}$ such that $\lim D$ exists, then $\lim F(D)$ exists in $\cat{C'}$, and $\lim F(D) \cong F(\lim D)$ through an isomorphism compatible with the limit cones.
 
 A functor $F:\cat{C}\to\cat{C'}$ is called \emph{cocontinuous} if it preserves all the colimits which exist in $\cat{C}$. In other words, for every diagram $D$ in $\cat{C}$ such that $\colim D$ exists, then $\colim F(D)$ exists in $\cat{C'}$, and $\colim F(D) \cong F(\colim D)$ through an isomorphism compatible with the colimit cocones.
\end{deph}

Here are some examples and nonexamples of functors preserving particular limits.

\begin{eg}[topology]\label{topsetlim}
 Consider the forgetful functor $U:\cat{Top}\to\Set$. We have that (why?),
 \begin{itemize}
  \item The underlying set of the one-point space is the one-point set. Therefore $U$ preserves terminal objects (which are limits).
  \item The underlying set of the product of spaces is the cartesian product of sets. Therefore $U$ preserves products (which are limits).
  \item The underlying set of the empty space is the empty set. Therefore $U$ preserves initial objects (which are colimits).
  \item The underlying set of the disjoint union of spaces is the disjoint union of sets. Therefore $U$ preserves coproducts (which are colimits).
 \end{itemize}
\end{eg}

\begin{ex}[graph theory]
 Prove that the same is true for the forgetful functor $\cat{MGraph}\to\Set$. 
\end{ex}

One may be tempted to ask: are these functors continuous? What about cocontinuous? A partial answer will be given later on.
Not every forgetful functor behaves this nicely.

\begin{eg}[linear algebra]\label{vecsetlim}
 Consider the forgetful functor $U:\Vect\to\Set$. We know that binary products and coproducts in $\Vect$ coincide (\Cref{biproduct}), but not in $\Set$. Therefore, $U$ cannot preserve both. Just as well, initial and terminal objects in $\Vect$ coincide (\Cref{zeroobject}), but not in $\Set$. Therefore, again, $U$ cannot preserve both.
 In particular, this forgetful functors preserves products and terminal objects (which are limits), but not coproducts and initial objects (which are colimits). 
\end{eg}

\begin{ex}[group theory]
 Prove that the forgetful functor $U:\Grp\to\Set$ preserves products and terminal objects, but not coproducts and initial objects.
\end{ex}

\begin{ex}
 Recall that, by \Cref{kernelpair}, the property of being mono can be expressed in terms of a universal property, a limit. Give an example of a functor which does not preserve this limit. (Hint: in \Cref{funepimono} you can find functors which do not preserve monomorphisms.)
\end{ex}

\begin{ex}
 Can you give an example of the dual case to the previous exercise?
\end{ex}

\subsection{The power set and probability functors, and complexity}\label{complexity}

Here are two examples of functors which do not preserve products and coproducts for a deep, structural reason. 

\begin{eg}[sets and relations]\label{complexpowerset}
 Consider the power set functor $P:\Set\to\Set$ (\Cref{egpowerset}). Consider two sets $X$ and $Y$ with two or more elements each. 
 
 We can form the coproduct $X\sqcup Y$, i.e.~their disjoint union. The set $P(X\sqcup Y)$ has, as elements, subsets $S\subseteq X\sqcup Y$. Subsets of $X\sqcup Y$ may be subsets of $X$ only, of $Y$ only, or mixed, i.e.~may contain elements from both. On the other hand, the set $PX\sqcup PY$ is the disjoint union of the set of subsets of $X$ and the set of subsets of $Y$. The elements of $PX\sqcup PY$, in other words, are either subsets of $X$ or subsets of $Y$, but not mixed. For example, if $x\in X$ and $y\in Y$, the set $\{x,y\}$ is in $P(X\sqcup Y)$ but not in $PX\sqcup PY$. Therefore, $P(X\sqcup Y)\ncong PX\sqcup PY$. More specifically, we have $PX\sqcup PY \subsetneq P(X\sqcup Y)$.
 
 Consider now the product $X\times Y$, i.e.~their cartesian product. The set $P(X\times Y)$ has, as elements, subsets of the cartesian product $X\times Y$. On the other hand, the set $PX\times PY$ has, as elements, pairs consisting of a subset of $X$ and a subset of $Y$. Every subset of $X\times Y$ can be projected onto a subset of $X$ and onto a subset of $Y$, giving an element of $PX\times PY$. Therefore, we have a canonical map $P(X\times Y)\to PX\times PY$, which is surjective. However, this map is in general far from being injective: many subsets of $X\times Y$ have the same projections. For example, for $X=Y=\R$, the subsets of $\R^2$ in the following pictures all have the same projections onto the axes.
 
\begin{center}
\begin{tikzpicture}
      \draw[->] (-0.5,0) -- (4.5,0) ;
      \draw[->] (0,-0.5) -- (0,3.5) ;
      \draw[darkblue, very thick, fill=darkblue!10] (2.25,1.75) circle [x radius=1.25, y radius=0.75] ;
      \draw[dotted] (2.25,1) -- (0,1) ;
      \draw[dotted] (1,1.75) -- (1,0) ;
      \draw[dotted] (2.25,2.5) -- (0,2.5) ;
      \draw[dotted] (3.5,1.75) -- (3.5,0) ;
      \draw[darkblue, very thick] (0,1) -- (0,2.5) ;
      \draw[darkblue, very thick] (1,0) -- (3.5,0) ;
\end{tikzpicture}
$\;$
\begin{tikzpicture}
      \draw[->] (-0.5,0) -- (4.5,0) ;
      \draw[->] (0,-0.5) -- (0,3.5) ;
      \draw[darkblue, very thick, fill=darkblue!10] (1,1) rectangle (3.5,2.5) ;
      \draw[dotted] (1,1) -- (0,1) ;
      \draw[dotted] (1,1) -- (1,0) ;
      \draw[dotted] (3.5,2.5) -- (0,2.5) ;
      \draw[dotted] (3.5,2.5) -- (3.5,0) ;
      \draw[darkblue, very thick] (0,1) -- (0,2.5) ;
      \draw[darkblue, very thick] (1,0) -- (3.5,0) ;
\end{tikzpicture}
$\;$
\begin{tikzpicture}
      \draw[->] (-0.5,0) -- (4.5,0) ;
      \draw[->] (0,-0.5) -- (0,3.5) ;
      \draw[darkblue, very thick] (1,1) -- (3.5,2.5) ;
      \draw[dotted] (1,1) -- (0,1) ;
      \draw[dotted] (1,1) -- (1,0) ;
      \draw[dotted] (3.5,2.5) -- (0,2.5) ;
      \draw[dotted] (3.5,2.5) -- (3.5,0) ;
      \draw[darkblue, very thick] (0,1) -- (0,2.5) ;
      \draw[darkblue, very thick] (1,0) -- (3.5,0) ;
\end{tikzpicture}
\end{center}
\end{eg}

\begin{ex}[sets and relations]\label{forgetinteraction}
 Show that the canonical inclusion $PX\sqcup PY\to P(X\sqcup Y)$ can be given in terms of the universal property of the coproduct. Conclude that any category $\cat{C}$ with coproducts and any functor $\cat{C}\to\cat{C}$ admit canonically an analogous morphism. 
 
 Show that the canonical projection $P(X\times Y)\to PX\times PY$ satisfies a dual property. 
\end{ex}

\begin{eg}[probability]
 Consider the probability functor $\mathcal{P}:\Set\to\Set$ (\Cref{egprob}), or any of its more refined variants in the categories $\cat{Meas}, \cat{CHaus}$, et cetera. Consider again two sets $X$ and $Y$ (or spaces) with two or more elements each. 
 The set $\mathcal{P}(X\times Y)$ is the set of \emph{joint} probability distributions on $X$ and $Y$. On the other hand, the set $\mathcal{P}(X)\times \mathcal{P}(Y)$ is the set of \emph{pairs of marginals}, probability distributions separately on $X$ and on $Y$. Similarly to the case of the power set, many possible joint distributions correspond to the same marginals, and the different possibilities correspond to the different types of \emph{statistical interactions}, such as \emph{correlation}. 
 For example, for $X=Y$ and given a probability distribution $p$ on $X$, there are at least two probability distributions on $X\times X$ such that their marginals are both equal to $p$: the independent case $p\otimes p$, also called \emph{product probability}, and the perfectly correlated case, supported on the diagonal in $X\times X$. 
 
 By \Cref{forgetinteraction}, the map $\mathcal{P}(X\times Y)\to \mathcal{P}(X)\times \mathcal{P}(Y)$ is canonically given by the universal property of the product. One can interpret this map as the one ``forgetting the statistical interaction'': it maps a joint distribution simply to the pair of its marginals.
\end{eg}

A possible interpretation for this behavior, for both of these functors, is that \emph{the features extracted by these functors do not respect the composition of objects, either via products or coproducts}. In other words, we are free to compose our objects to obtain more complex ones, however, observing the more complex objects is \emph{different} from observing the parts separately. Observing the parts separately forgets ``something'', either elements (mixed subsets) or information (statistical interaction). We can think of functors such as $P$ and $\mathcal{P}$ as functors \emph{detecting} or \emph{exhibiting complexity}.\footnote{See \cite{sevensketches} for more on this. }

\begin{ex}[probability]
 Does $\mathcal{P}:\Set\to\Set$ preserve coproducts? (Hint: no, and this is again because some ``mixed'' elements are missing, similarly to the power set case.)
\end{ex}

\begin{ex}[sets and relations, probability]
 Does the power set functor preserve initial and terminal objects? What about the probability functor?
\end{ex}

\subsection{Continuous functors and equivalence}

The following exercises show how limits and colimits, and their preservation, behave under equivalences.

\begin{ex}[important!]\label{contiso}
 Let $F:\cat{C}\to\cat{D}$ be a functor preserving the limit of a diagram $D:\cat{J}\to\cat{C}$. Let $G$ be a functor naturally isomorphic to $F$. Show that $G$ also preserves the limit of $D$. 
 
 Conclude that if $F$ is continuous, then $G$ is continuous too. What is the dual statement?
\end{ex}

\begin{ex}
 Show that continuous and cocontinuous functors are stable under composition. 
\end{ex}

\begin{ex}[important!]
 Let $F:\cat{C}\to\cat{D}$ be a functor inducing an equivalence of categories. Show that $F$ is continuous (and hence cocontinuous\dots why?).
 
 Conclude that if a category $\cat{C}$ is complete (or cocomplete), then every category equivalent to $\cat{C}$ is complete (or cocomplete) too. 
\end{ex}

\subsection{The case of representable functors and presheaves}

Here is one of the most important properties of representable functors.

\begin{thm}\label{repfuncontinuous}
 Representable functors are continuous. 
\end{thm}

Let's see what this theorem means in detail. Let $\cat{C}$ be a category, let $J$ be a small category, let $D:\cat{J}\to\cat{C}$ be a diagram, and suppose that $\lim D$ exists in $\cat{C}$. Let now $R$ be an object of $\cat{C}$. Then we have that the limit in $\Set$ of the diagram
$$
\Hom_\cat{C}(R,D-):\cat{J}\to\Set
$$
exists, and is equal to the set $\Hom_\cat{C}(R,\lim D)$. 

The full proof of this theorem will be given in \Cref{proofrepfulcomplete}. 
For now, let's try to understand the result and its significance.

First of all, replacing $\cat{C}$ by $\cat{C}^\op$, we get the dual statement:
\begin{cor}\label{reprpresheaveslim}
 Let $P:\cat{C}^\op\to\Set$ be a representable presheaf. Then $P$ turns colimits into limits. That is, if $D:\cat{J}\to \cat{C}$ is a diagram admitting a colimit, then the limit of $P\circ D$ exists in $\Set$, and 
 $\lim P\circ D \cong P(\colim D)$.
\end{cor}

\begin{ex}[important!]
 Show that this corollary is indeed given by \Cref{repfuncontinuous} if one replaces $\cat{C}$ by $\cat{C}^\op$.
\end{ex}

For the case of products, the proof of the theorem is relatively straightforward. Let's consider binary products for simplicity. Take objects $X$ and $Y$ of $\cat{C}$ and suppose that their product $X\times Y$ exists. The universal property of the product says given any (other) object $R$ of $\cat{C}$ and maps $R\to X$ and $R\to Y$, there exists a unique map $R\to X\times Y$ making the following diagram commute.
$$
\begin{tikzcd}
 & R\uni{d} \ar{dl}\ar{dr} \\
 X & X\times Y \ar{l}{p_1} \ar{r}[swap]{p_2} & Y.
\end{tikzcd}
$$
In particular, there is a bijection between pairs of maps $R\to X$ and $R\to Y$, and maps $R\to X\times Y$. Now, the set of pairs of maps $R\to X$ and $R\to Y$ is exactly the cartesian product of the sets $\Hom_C(R,X)$ and $\Hom_C(R,Y)$, so there is a bijection
$$
\begin{tikzcd}
\Hom_C(R,X)\times \Hom_C(R,Y) \ar{r}{\cong} & \Hom_C(R,X\times Y) .
\end{tikzcd}
$$
This is an isomorphism in the category $\Set$, therefore, the representable functor $\Hom_C(R,-):\cat{C}\to\cat{Set}$ preserves binary products.

\begin{ex}
 Can you work out the case for presheaves and coproducts (independently)?
\end{ex}

Here is a possible intuitive interpretation of this statement. We have seen that representable functors can be thought of as experiments consisting of probing the objects of $\cat{C}$ with a given object $R$, by mapping $R$ into the objects of $\cat{C}$ in all possible ways. Now the statement of \Cref{repfuncontinuous}, for products, says that \emph{probing a composite system $X\times Y$ is the same as probing $X$ and $Y$ separately}. 

Dually, the statement with coproducts and presheaves can be interpreted as follows. We can interpret representable presheaves as experiments consisting of observing the objects of $\cat{C}$ by mapping them to a given object $R$. Now we have that \emph{observing the object $X\sqcup Y$ is the same as observing $X$ and $Y$ separately}.

On the other hand, we cannot in general conclude that representable presheaves behave well with products, or that representable functors behave well with coproducts. Probing a the coproduct of $X$ and $Y$ via objects $R$ is not the same as probing them separately - for example, a given set $R$ can be embedded into $X\sqcup Y$ in a way which hits partly both $X$ and $Y$, and this cannot be obtained by mapping $R$ separately into $X$ and $Y$. Just as well, observing a the product $X\times Y$ is not the same as observing $X$ and $Y$ separately: a function on $X\times Y$ in general does not depend only on the values of $X$ and $Y$ separately. For example, at the end of \Cref{complexpowerset}, observing only the projections we can't distinguish all the possible original subsets of $\R^2$.

\begin{ex}
 Can you give other concrete examples of the latter statements?
\end{ex}

By the way, by \Cref{repfuncontinuous}, the forgetful functors given in \Cref{topsetlim} and \Cref{vecsetlim} are indeed continuous, since we have seen that these functors are representable.

\section{Limits and colimits of sets}

\subsection{Completeness of the category of sets}

Consider a diagram of sets $D:\cat{J}\to\Set$. As $\cat{J}$ is assumed to be small, the objects of $\cat{J}$ form a set (or a small set). Therefore we can form the cartesian product of all the sets appearing in the diagram $D$, which is possibly infinite, but still well-defined:
$$
P \;\coloneqq\; \prod_{I\in\cat{J}_0} DI .
$$
For example, for the diagram
$$
 \begin{tikzcd}[row sep=small]
  & B \ar{dr} \\
  A \ar{ur} \ar{dr} && D \\
  & C \ar{ur}
 \end{tikzcd}
$$
the object $P$ is given by $A\times B\times C\times D$.
This cartesian product has projection maps $\pi_I: P \to DI$ for each object $I$ of $\cat{J}$. For the diagram above, this amounts to the projection maps 
$$
\begin{tikzcd}[row sep=small,column sep=small]
  & A\times B\times C\times D \ar{dddl}[swap]{p_1} \ar{dddr}{p_4} \ar[bend right=15,pos=0.65]{dd}[swap]{p_2} \ar[bend left]{dddd}[near end]{p_3} \\ \\
  & B \ar{dr} \\
  A \ar{ur} \ar{dr} && D \\
  & C \ar{ur}
 \end{tikzcd}
$$
In general, there is no reason to assume that the triangles of the diagram above commute. Just as well, in the general case, it is in general not true that given a morphism $m:I\to I'$ in $\cat{J}$, this diagram commutes,
$$
\begin{tikzcd}[column sep=small]
 &  P \ar{dl}[swap]{p_I} \ar{dr}{p_{I'}} \\
 DI \ar{rr}{Dm} && DI' 
\end{tikzcd}
$$
i.e.~it may not be true that for all $y\in P$, 
\begin{equation}\label{trianglecomm}
 Dm(p_I(y)) \;= \;p_{I'}(y) .
\end{equation}
However there is always a \emph{subset} $S_m$ of $P$ whose elements satisfy \Cref{trianglecomm} -- if no element of $P$ satisfies \Cref{trianglecomm}, then $S_m$ is the empty set, which is still a well-defined subset of $P$. 
Now, as the category $\cat{J}$ is small, also the morphisms of $\cat{J}$ form a set (or a small set). Therefore we can form the following intersection, indexed by all the morphisms appearing in the diagram $D$, which is possibly empty, but is still a well-defined subset of $P$:
$$
S \;\coloneqq\; \bigintersection_{m\in\cat{J}_1} S_m .
$$
More explicitly, $S$ is the subset of $P$ defined by the condition that for each element $y$ of $S$ and each morphism $m$ of $\cat{J}$, \Cref{trianglecomm} has to hold (and it is the largest such subset).

\begin{lemma}\label{explicitlimit}
 The object $S$ defined above, together with the maps $S\to DI$ given by
 $$
 \begin{tikzcd}
  S \mono{r}{i} & P \ar{r}{p_I} & DI 
 \end{tikzcd}
 $$
 where $i:S\to P$ is the inclusion, is a limit cone over the diagram $D$.
\end{lemma}

This gives an explicit way to write the limit of any diagram of sets. In particular, such limit always exists.
Therefore,
\begin{thm}
 The category $\Set$ is complete. 
\end{thm}
Keep in mind that, in the construction of $S$, we have crucially used the fact that $\cat{J}$ is small.

\begin{proof}[Proof of \Cref{explicitlimit}]
 First of all, by construction, the maps $p_I\circ i:S\to DI$ form a cone over $D$, since by construction, for each morphism $m:I\to I'$ of $\cat{J}$ the following diagram commutes,
 $$
 \begin{tikzcd}[column sep=small]
 &  S \ar{dl}[swap]{p_I\circ i} \ar{dr}{p_{I'}\circ i} \\
 DI \ar{rr}{Dm} && DI' .
 \end{tikzcd}
 $$ 
 Take now any (other) cone $\alpha$ over $D$, say of tip $X$. So for every object $I$ of $\cat{J}$ we have maps $\alpha_X:X\to DI$, and for each morphism $m:I\to I'$ of $\cat{J}$ this diagram commutes. 
 \begin{equation}\label{dmcomm}
 \begin{tikzcd}[column sep=small]
 &  X \ar{dl}[swap]{\alpha_I} \ar{dr}{\alpha_{I'}} \\
 DI \ar{rr}{Dm} && DI' .
 \end{tikzcd}
 \end{equation}
 Now by the universal property of the product, a tuple of maps $\alpha_X:X\to DI$ for each object $I$ of $\cat{J}$ gives a unique map $u:X\to P$, such that for each object $I$ of $\cat{J}$, this triangle commutes,
 \begin{equation}\label{picomm}
 \begin{tikzcd}
  X \uni{d}[swap]{u} \ar{dr}{\alpha_I} \\
  P \ar{r}[swap]{p_I} & DI
 \end{tikzcd}
 \end{equation}
 Moreover, the map $u$ actually factors through $S$, i.e.~for every $x\in X$, we have that $u(x)\in S$, since for all $m:I\to I'$ of $\cat{J}$, diagrams \eqref{dmcomm} and \eqref{picomm} say that 
 $$
 Dm(p_I(u(x))) \;=\; Dm(\alpha_I(x)) \;=\; \alpha_{I'}(x)\;=\; p_{I'}(u(x)) ,
 $$
 which means that $u(x)$ satisfies \Cref{trianglecomm}. 
 In conclusion, we have proven that for each limit cone of tip $X$ over $D$ there is a unique map $X\to S$ such that this diagram commutes,
 $$
 \begin{tikzcd}
  X \uni{d} \ar{dr}{\alpha_I} \\
  S \ar{r}[swap]{p_I\circ i} & DI
 \end{tikzcd}
 $$
 i.e.~that $S$ is the limit of $D$. 
\end{proof}

\begin{ex}[sets and relations]\label{SPeq}
 Write the subset $S\subseteq P$ as an equalizer of a pair of parallel maps $P\to P$. 
\end{ex}

\begin{ex}[sets and relations]
 If you have solved \Cref{SPeq}, rewrite the proof of \Cref{explicitlimit} without referring to the elements of the sets involved, use instead the universal properties of the products and equalizers. 
 In other words, you will show that every limit can be written as an equalizer of products. Conclude that every category where all products and equalizers exist is necessarily complete. 
\end{ex}

\begin{ex}[sets and relations; difficult!]\label{SPcoeq}
 In a somewhat dual way to \Cref{explicitlimit}, give an explicit formula of the colimit of a diagram of sets, as a quotient of the disjoint union $U$ of all the sets appearing in the diagram.
\end{ex}

\begin{ex}[sets and relations]
 If you have solved \Cref{SPcoeq}, express the quotient of the disjoint union $U$ as the coequalizer of a pair of parallel maps $U\to U$.
\end{ex}

\begin{ex}
 If you have solved the previous two exercises, repeat \Cref{SPcoeq} without referring to the elements of the sets involved, use instead the universal properties of the coproducts and coequalizers. 
 In other words, you will show that every colimit can be written as a coequalizer of coproducts. Conclude that every category where all coproducts and coequalizers exist is necessarily cocomplete. 
\end{ex}

\subsection{General proof of \Cref{repfuncontinuous}}\label{proofrepfulcomplete}

The statement of \Cref{repfuncontinuous} says the following. Suppose that the diagram $D:\cat{J}\to\cat{C}$  has a limit. Let $R$ be an object of $\cat{C}$. Then the diagram $\Hom_C(R,D-):\cat{J}\to\Set$ given by
$$
\begin{tikzcd}[column sep=huge]
 \cat{J} \ar{r}{D} & \cat{C} \ar{r}{\Hom_C(R,-)} & \Set
\end{tikzcd}
$$
has a limit, and 
$$
\lim_{I\in \cat{J}}\, \Hom_C(R,DI) \;\cong\; \Hom_C\big( R,\lim_{I\in \cat{J}} DI \big) .
$$

\begin{proof}[Proof of \Cref{repfuncontinuous}]
First of all, note that the action of the functor $\Hom_C(R,D-)$ on morphisms of $\cat{J}$ maps $m:I\to I'$ to the function 
$$
\begin{tikzcd}[row sep=0, column sep=large]
 \Hom_C(R,DI) \ar{r}{(Dm)\circ -} & \Hom_C(R,DI') \\
 f \ar[mapsto]{r} & (Dm)\circ f .
\end{tikzcd}
$$

Now, let's turn to prove the theorem. By \Cref{explicitlimit} we know that the limit of $\Hom_C(R,D-)$ exists, since every diagram of sets has a limit. Moreover, \Cref{explicitlimit} gives us a way to write down the limit explicitly. Namely, if we form the product
$$
P\;\coloneqq\; \prod_{I\in\cat{J}_0}  \Hom_C(R,DI) ,
$$
then the limit of $\Hom_C(R,D-):\cat{J}\to\Set$ is given by the subset $S\subseteq P$ of all $y\in P$ which satisfy $(Dm)\circ p_I(y)=p_{I'}(y)$ for all $m:I\to I'$ of $\cat{J}$, so that the following diagram commutes,
\begin{equation}\label{condformorph}
\begin{tikzcd}[column sep=small]
 & S \ar{dl}[swap]{p_I\circ i} \ar{dr}{p_{I'}\circ i} \\
 \Hom_C(R,DI) \ar{rr}{(Dm)\circ -} && \Hom_C(R,DI')
\end{tikzcd}
\end{equation}
where again $i:S\to P$ is the inclusion.
But now, $P$ is a product of sets in the form $\Hom_C(R,DI)$. The elements of $\Hom_C(R,DI)$ are exactly the functions $R\to DI$, so that the elements of $P$ are exactly tuples of arrows $R\to DI$ indexed by $I$ (varying over the objects of $\cat{J}$). In symbols, a generic element of $P$ is in the form 
$$
(f_I:R\to DI)_{I\in\cat{J}_0} ,
$$
which is a tuple of arrows in $\Set$ from $R$ to all the objects in the diagram $D$:
$$
\begin{tikzcd}
 & R \ar{dr} \ar{d}{f_{I'}} \ar{dl}[swap]{f_{I}} \\
 DI & DI' & ...
\end{tikzcd}
$$
Does this give a cone over $D$? Let's see. Condition~\eqref{condformorph} says that a tuple $(f_I:R\to DI)_{I\in\cat{J}_0}$ is in $S$ if and only if for every $m:I\to I'$ of $\cat{J}$,
$$
(Dm)\circ f_I \;=\; f_J ,
$$
i.e.~if for every $m:I\to I'$ of $\cat{J}$, this diagram commutes:
$$
\begin{tikzcd}[column sep=small]
 & R \ar{dl}[swap]{f_I} \ar{dr}{f_{I'}} \\
 DI \ar{rr}{Dm} && DI'
\end{tikzcd}
$$
which means indeed that $(f_I:R\to DI)_{I\in\cat{J}_0}$ is a cone over $D$ (and all cones over $D$ are in this form). So the elements of $S$ are precisely the cones over $D$ of tip $R$, that is,
$$
\lim_{I\in \cat{J}}\, \Hom_C(R,DI) \;\cong\; S \;\cong\; \Cone(R,D) .
$$
But now, by definition, the limit of $D$ is the object representing the functor $\Cone(-,D):\cat{C}\to\Set$, so that
$$
\Cone(R,D) \; \cong \; \Hom_C \big( R, \lim D \big) ,
$$
that is, 
$$
\lim_{I\in \cat{J}}\, \Hom_C(R,DI) \;\cong\; \Hom_C\big( R,\lim_{I\in \cat{J}} DI \big) . 
$$ 
\end{proof}

\begin{ex}\label{natisolimit}
 Is the isomorphism $\lim_{I\in \cat{J}} \Hom_C(R,DI) \to \Hom_C\big( R,\lim_{I\in \cat{J}} DI \big)$ natural? (By the way, in which argument?)
\end{ex}

\newpage
\chapter{Adjunctions}\label{adjunctions}

\section{General definitions}

\begin{deph}
 Let $\cat{C}$ and $\cat{D}$ be categories, and consider functors $F:\cat{C}\to\cat{D}$ and $G:\cat{D}\to\cat{C}$. 
 An \emph{adjunction} between $F$ and $G$ is a bijection
 $$
 \Hom_\cat{D}(FC, D) \; \xrightarrow{\cong} \;\Hom_\cat{C}(C, GD)
 $$
 for each object $C$ of $\cat{C}$ and $D$ of $\cat{D}$, natural both in $C$ and in $D$.
 
 Note that the notion is not symmetric in $F$ and $G$, or in $\cat{C}$ and $\cat{D}$. We call $F$ the \emph{left-adjoint} and $G$ the \emph{right-adjoint}, and denote the adjunction by $F \ladj G$. 
 
 Two maps $FC\to D$ and $C\to GD$ related by the bijection above are called \emph{transpose} or \emph{adjunct} to each other.
\end{deph}

It is useful to use the symbols $\sharp$ (sharp) and $\flat$ (flat) to denote the transposes. We denote by $f^\sharp$ a map $FC\to D$, and by $f^\flat$ its transpose $C\to GD$. Just as well, we can write the bijections as
$$
\begin{tikzcd}
 \Hom_\cat{D}(FC, D) \ar[shift left]{r}{\flat} & \Hom_\cat{C}(C, GD) . \ar[shift left]{l}{\sharp}
\end{tikzcd}
$$

Let's now see some more concrete examples of adjunctions.

\subsection{Free-forgetful adjunctions}

Many adjunctions appearing in mathematics are the so-called \emph{free-forgetful adjunctions}, where the right-adjoint is a forgetful functor, and the left-adjoint is a functor giving a ``free construction'', in a way that will be made precise shortly.

\begin{eg}[topology]\label{adjtopset}
 Consider the forgetful functor $U:\cat{Top}\to\Set$. This takes a topological space and returns the underlying set ``forgetting some structure''. In general, we cannot go back -- given a set, we cannot assign a topological space uniquely, many topological spaces have the same underlying set up to isomorphism (for example, $\R$ and $\R^2$). However, given a set $X$, there is always a ``trivial'' topology that we can put on $X$, the discrete topology, for which each subset of $X$ is declared open. Let's denote this space by $FX$. The discrete topology has the following property: given any topological space $S$, every function $f:FX\to S$ is continuous. Indeed, let $O\subseteq S$ be open. Then $f^{-1}(O)\subseteq FX$ is open, because every subset of $FX$ is. 
 Now given a function between sets $g:X\to Y$, the induced function $FX\to FY$ is continuous (because every function on $FX$ is). Let's denote this continuous map by $Ff:FX\to FY$. This way, $F$ is a functor $\Set\to\cat{Top}$. 
 
 Let us now prove that there is an adjunction $F \ladj U$. This amounts to natural bijections 
 $$
 \Hom_\cat{Top}(FX, S) \; \xrightarrow{\cong} \;\Hom_\cat{Set}(X, US) ,
 $$
 for each set $X$ and each topological space $S$.
 Here we construct the bijection, naturality of this construction will be the next exercise.
 Now consider a continuous function $FX\to S$. As simply a function, this is a map from the set $X$ to the underlying set of $S$, in other words, it is a function $X\to US$. Conversely, given a function $X\to US$, this gives a function $FX\to S$ which is necessarily continuous, since $FX$ is equipped with the discrete topology. Therefore \emph{continuous functions $FX\to S$ correspond bijectively to functions between the sets $X\to US$}. 
\end{eg}

\begin{ex}[topology]
 Show that the bijection above is natural in both arguments.
\end{ex}

A possible interpretation of this adjunction is the following: the functor $U$ ``forgets'' the structure of topological spaces, and the functor $F$ ``tries to recover it'' in some default way -- in this case, by equipping our sets with the discrete topology. The functors $F$ and $U$ are not inverse to each other, not even pseudoinverse -- an adjunction is a more nuanced relation between functors. 

\begin{ex}[topology]\label{radjtopset}
 Show that the forgetful functor $U:\cat{Top}\to\Set$ also has a \emph{right}-adjoint. (Hint: which topology is somewhat dual to the discrete topology?)
\end{ex}

\begin{eg}[linear algebra]\label{adjvectset}
 Consider now the forgetful functor $U:\cat{Vect}\to\Set$. This takes a vector space $V$ and forgets its vector space structure, returning its underlying set. In general, there is no way to go back, and not every set even admits a vector space structure. However, given a set $X$, we can always canonically form a vector space from $X$: the set of \emph{formal linear combinations} of the elements of $X$. Let's denote this latter set by $FX$, its elements are expressions in the form 
 $$
 a_1 x_1 + \dots + a_n x_n
 $$
 for all $a_i\in \R$ and $x_i\in X$, and $n$ finite (but arbitrarily large). Note that we are not \emph{actually computing} the expression above ($X$ is not a vector space), we are just \emph{writing it}, formally. The vector space structure on $FX$ is given by the usual rules of addition and scalar multiplication of expressions by \emph{expanding}, such as
 $$
 a_1 (b_1 x_1 + b_2 x_2) + a_2 (c_1 y_1 + c_2 y_2) \;=\; (a_1 b_1) x_1 + (a_1 b_2) x_2 + (a_2 c_1) y_1 + (a_2 c_2) y_2 , 
 $$
 and the set $X$ is a basis of $FX$ (why?). 
 Now let's make this construction functorial. Given a function between sets $f:X\to Y$, we can obtain the linear function $Ff:FX\to FY$ as
 $$
 Ff(a_1 x_1 + \dots + a_n x_n) \; \coloneqq\; a_1 f(x_1) + \dots + a_n f(x_n) .
 $$
 (Why is this function linear?)
 Therefore we have a functor $F:\Set\to\Vect$. 
 
 Let's now show that $F\ladj U$. We have to prove that for all sets $X$ and vector spaces $V$, there is a natural bijection
 $$
 \Hom_\cat{Vect}(FX, V) \; \xrightarrow{\cong} \;\Hom_\cat{Set}(X, UV) ,
 $$
 between linear maps $FX\to V$ and functions from $X$ to the set underlying the vector space $V$. Now, if we have a linear map $l:FX\to V$, we can restrict it to the basis $X$, and obtain a function from $X$ to $UV$ (that is, $V$, but considered as a set). Conversely, given a map $f:X\to UV$, there is a \emph{linear extension} $FX\to V$ that extends $f$ linearly on all linear combinations of elements of $X$, and this linear extension is unique, because $X$ is a basis of $FX$ (why exactly?).
 Therefore we have a bijection. Again, naturality is left as exercise below.
\end{eg}

\begin{ex}[linear algebra]
 Show, again, that the bijection above is natural in both arguments.
\end{ex}

\begin{ex}[linear algebra]
 This idea of \emph{linear extension} sounds like a universal property. Can you make this precise? If not, just wait until the next section.
\end{ex}

\begin{ex}[group theory]\label{adjgrpset}
 What is the left-adjoint to the forgetful functor $\Grp\to\Set$?
\end{ex}

\subsection{Galois connections}

If the categories in question are posets, like most category-theoretical concepts, adjunctions looks particularly simple, and they were introduced much earlier than adjunctions (the first example was given by Évariste Galois). 

\begin{deph}
 Let $(X,\le)$ and $(Y,\le)$ be posets, let $f:X\to Y$ and $g:Y\to X$ be monotone maps. An adjunction $f\ladj g$ is called a \emph{Galois connection} or \emph{Galois correspondence}. 
 The left-adjoint $f$ is also called the \emph{lower adjoint}, and the right-adjoint $g$ is also called the \emph{upper adjoint}.
\end{deph}

Concretely, an adjunction $f\ladj g$ amounts to the following condition (can you see why?): for every $x\in X$ and $y\in Y$,
$$
f(x) \le y \quad \mbox{if and only if} \quad x \le g(y) ,
$$
and the naturality condition is trivial.

We have seen that in many cases, adjunctions between categories can be interpreted in terms of ``forgetting structure'' and ``trying to recover the structure in a default way''. Analogously, for posets, Galois connections can be often interpreted in terms of ``forgetting properties'', and ``trying to recover them in a default way''.\footnote{In category theory, often structures and properties behave in similar way. If you are interested in this, see for example \href{https://ncatlab.org/nlab/show/stuff\%2C+structure\%2C+property}{the nLab page on ``stuff, structure, property'' (link)}.}

\begin{eg}[sets and relations, analysis]\label{convsubsets}
 Let $X$ be the poset of subsets of $\R^2$, ordered by inclusion. Let $Y$ be the poset of \emph{convex} subsets of $\R^2$ (circles, squares, triangles\dots). We have a canonical inclusion map $i:Y\to X$, which we can think of as ``forgetful'': it forgets the convexity. 
 Let's now show that $i$ has a lower adjoint. The \emph{convex hull} of a subset $S\subseteq \R^2$ is defined as either 
 \begin{itemize}
  \item The smallest convex subset of $\R^2$ containing $S$;
  \item The intersection of all convex subsets of $\R^2$ containing $S$;
  \item The set obtained by closing $S$ under all possible convex combinations. 
 \end{itemize}
 For why these construction are equivalent, see the next exercise. 
 Define now $c:X\to Y$ to be the map assigning to each $S\in X$, i.e.~to each subset of $\R^2$, its convex hull. This map is monotone (why?). 
 
 Let us now show that $c$ is left-adjoint (or lower adjoint) to $i$. We have to show that for each subset $S\subseteq \R^2$ and each \emph{convex} subset $C\subseteq \R^2$, 
 $$
 c(S) \subseteq C \quad\mbox{if and only if}\quad S\subseteq C 
 $$
 (note that we have omitted the inclusion map $i$).
 Now, suppose that the convex hull $c(S)$ is contained in the convex set $C$. Then, since $c(S)$ contains $S$, clearly $S$ is contained in $C$ as well. Conversely, suppose that $S$ is contained in $C$. Then since $C$ is convex and $c(S)$ is the \emph{smallest} convex set containing $S$, necessarily $c(S)$ is contained in $C$. Therefore $c$ is left-adjoint to $i$.
\end{eg}

\begin{ex}[sets and relations, analysis]
 Show that the constructions of the convex hull given above are all equivalent. (Hint: the first and the second are basically the same.)
\end{ex}

Here are other examples of Galois connection with the same interpretation of ``forgetting properties''.

\begin{ex}[topology]
 Show that there is a Galois connection between subsets of $\R^2$ and \emph{closed} subsets of $\R^2$.
\end{ex}

\begin{ex}[group theory]
 Show that there is a Galois connection between \emph{subsets} of a given group $G$ and \emph{subgroups} of $G$. 
\end{ex}

Here are more general exercises.

\begin{ex}[sets and relations]\label{round}
 Show that the inclusion $i:\Z\to\R$ (with the usual order of numbers) has both an upper and a lower adjoint. 
\end{ex}

\begin{ex}[sets and relations, analysis]
 Show that, for all examples above, the result of applying the lower adjoint is always the solution to some constrained optimization problem, namely, the \emph{smallest element satisfying a particular constraint}. Express this in terms of a universal property, and show that this universal property is true in general for any lower adjoint of a Galois connection.
 
 (Of course, the upper adjoint satisfies a dual property.)
\end{ex}

\begin{ex}[sets and relations; difficult!]
 If you have solved the exercise above, show that a similar universal property holds for all left-adjoint functors, not necessarily between posets. 
 (Hint: The ``smallest element'' in a poset corresponds to the ``initial object'' of some category\dots \emph{which category}?)
\end{ex}

\section{Unit and counit as universal arrows}\label{unitcounit}

Let's now see in detail the universal arrows defined by an adjunction. Let $\cat{C}$ and $\cat{D}$ be categories, let $F:\cat{C}\to\cat{D}$ and $G:\cat{D}\to\cat{C}$ be functors, and consider an adjunction $F\ladj G$, i.e.~a natural bijection
$$
\Hom_\cat{D}(FC, D) \; \xrightarrow{\cong} \;\Hom_\cat{C}(C, GD) 
$$
for all objects $C$ of $\cat{C}$ and $D$ of $\cat{D}$. Now fix the object $C$. The bijection above, which is natural in $D$, can be seen as a natural isomorphism between the functors
$$
\Hom_\cat{D}(FC, -) \quad \mbox{and} \quad \Hom_\cat{C}(C, G-)  : \cat{D} \to \Set .
$$
In other words, for all objects $C$ of $\cat{C}$, the functor $\Hom_\cat{C}(C, G-)  : \cat{D} \to \Set$ is representable, and it is represented by the object $FC$. 

How can we see this intuitively?
In terms of representable functor as ``probes'', we can interpret the situation as follows. The functor $\Hom_\cat{C}(C, G-)$ takes an object $D$ of $\cat{D}$ and extracts some of its features by first mapping it to $\cat{C}$ via the functor $G$, and then probes the resulting object $GD$ by looking at all possible ways of mapping $C$ into it. The natural isomorphism above says that the features extracted this way from $D$ are the same as those that can be extracted by directly probing $D$ with $FC$. 
So, for example, in the adjunction between sets and vector spaces of \Cref{adjvectset}, given a set $X$ and a vector space $V$, all the possible ways to map $X$ to $V$ (taken as a set) correspond exactly to all the possible ways of mapping $FX$ linearly to $V$ (taken as a vector space). 

Let's now see what is happening in terms of universal properties. By the Yoneda lemma, we know that the natural isomorphism above is specified uniquely by an element of the set 
$$
\Hom_\cat{C}(C, GFC) ,
$$
i.e.~a particular (``universal'') morphism $C\to GFC$ of $\cat{C}$. As usual (\Cref{uniprop}), this morphism has to be the image of the identity of $FC$ under the bijection. That is, set $D$ to be exactly $FC$, so that we have a bijection 
$$
\Hom_\cat{D}(FC, FC) \; \xrightarrow{\cong} \;\Hom_\cat{C}(C, GFC) 
$$
In the set on the left there is a distinguished element, namely $\id_{FC}$. The universal morphism $C\to GFC$ is the image of $\id_{FC}$ under the isomorphism above, that is, it is exactly $(\id_{FC})^\flat:C\to GFC$. 

\begin{deph}
 The universal map $C\to GFC$ induced by the adjunction, at the object $C$, is called the \emph{unit} of the adjunction (at $C$), and denoted by $\eta_C:C\to GFC$. 
\end{deph}

So $\eta_C=(\id_{FC})^\flat$ and $(\eta_C)^\sharp = \id_{FC}$.
Here is how the universal property of $\eta_C$ works in practice, following the ideas of \Cref{uniprop}. Let $D$ be an object of $\cat{D}$, and apply the functor $G$ to it, to get the object $GD$ of $\cat{C}$. Then the bijection $\Hom_\cat{D}(FC, D) \to \Hom_\cat{C}(C, GD)$ says that given any morphism $f^\flat:C\to GD$ of $\cat{C}$, there is a unique morphism $ f^\sharp:FC\to D$ of $\cat{D}$ such that the following triangle in $\cat{C}$ commutes:
\begin{equation}\label{unipropeta}
\begin{tikzcd}
 C \ar{r}{\eta_C} \ar{dr}[swap]{f^\flat} & GFC \ar[""{name=GFBAR,right}]{d}{G f^\sharp} && FC \ar[dashed,""{name=FBAR,left}]{d}{ f^\sharp} \\
  & GD && D
  \ar[mapsto,from=FBAR,to=GFBAR,"G",start anchor={[xshift=-1ex]},end anchor={[xshift=4ex]}]
\end{tikzcd}
\end{equation}

\begin{eg}[linear algebra]\label{unitvect}
 Consider again the adjunction between sets and vector spaces of \Cref{adjvectset}. Given a set $X$, the unit $\eta_X:X\to UFX$ is a function between the set $X$ and the underlying set of $FX$, i.e.~the \emph{set} of formal linear combinations of elements of $X$ (so, $FX$ seen as a set, not as a vector space). Which map is now $\eta_X$? We know that $(\eta_X)^\sharp = \id_{FX}$, in other words, $\eta_X$ is the map whose linear extension is the identity on $FX$. Now, the (unique) map whose linear extension is the identity is the map $X\to UFX$ which assigns to each $x\in X$ the ``trivial'' linear combination $1\cdot x\in UFX$. That is, $\eta_X:X\to UFX$ is the canonical \emph{embedding of $X$} into the set underlying $FX$. 
 
 The universal property now reads: for every vector space $V$, and for every function $f:X\to UV$, that is, function from $X$ to $U$ seen as a set, there is a unique linear extension $ f^\sharp:FX\to V$ of $f$ such that the following diagram commutes:
 $$
\begin{tikzcd}
 X \ar{r}{\eta_X} \ar{dr}[swap]{f} & UFX \ar[""{name=GFBAR,right}]{d}{U f^\sharp} && FX \ar[dashed,""{name=FBAR,left}]{d}{ f^\sharp} \\
  & UD && D
  \ar[mapsto,from=FBAR,to=GFBAR,"U",start anchor={[xshift=-1ex]},end anchor={[xshift=4ex]}]
 \end{tikzcd}
 $$
 The triangle commuting on the left says that if we compose $U f^\sharp$ with $\eta_X$, that is, if we restrict $ f^\sharp$ to the ``trivial'' linear combinations (the elements of $X$), then it agrees with the original $f$. In other words, $ f^\sharp$ is \emph{really extending} $f$.
\end{eg}

\begin{ex}[linear algebra]
 Why is the embedding $\eta_X:X\to UFX$ \emph{canonical}? Is anything natural here? (Spoiler: to know the answer, you can also read further down).
\end{ex}

\begin{ex}[group theory, topology]
 What are the units of the adjunctions between sets and groups (\Cref{adjgrpset}) and between sets and topological spaces (\Cref{adjtopset})?
\end{ex}

\begin{eg}[sets and relations, analysis]
 Consider now the Galois connection of \Cref{convsubsets} between subsets and convex subsets of $\R^2$. The unit of the adjunction $c\ladj i$ amounts to a relation of containment $S\subseteq c(S)$ for all subsets $S\subseteq \R^2$ (why?). This inequality indeed holds, since every set is contained in its convex hull.
\end{eg}

So far, we have kept the object $C$ fixed. If we now let $C$ vary, naturality in $C$ implies that the map $\eta_C$ is actually natural in $C$.

We first start with the following very useful technical lemma.

\begin{lemma}[{\cite[Lemma~4.1.3]{ctcontext}}]\label{usefuldiag}
 The naturality condition of the adjunction implies that for every for every $f^\sharp:FC\to D$, $g^\sharp:FC'\to D'$, $h:C\to C'$ of $\cat{C}$ and $k:D\to D'$ of $\cat{D}$, the diagram on the left commutes if and only if the diagram on the right commutes.
 $$
 \begin{tikzcd}
  FC \ar{d}{Fh} \ar{r}{f^\sharp} & D \ar{d}{k} \\
  FC' \ar{r}{g^\sharp} & D'
 \end{tikzcd}
 \qquad\qquad
 \begin{tikzcd}
  C \ar{d}{h} \ar{r}{f^\flat} & GD \ar{d}{Gk} \\
  C' \ar{r}{g^\flat} & GD'
 \end{tikzcd}
 $$
\end{lemma}

Note that the diagram on the left is in $\cat{D}$, while the diagram on the right is in $\cat{C}$. One way to interpret the lemma is that the ``sharp'' and ``flat'' isomorphisms permit to pass ``rigidly'' from $\Hom_\cat{D}(FC, D)$ to $\Hom_\cat{C}(C, GD)$, i.e.~in a way which plays well with the morphisms of $\cat{C}$ and $\cat{D}$. 

\begin{proof}[Proof of \Cref{usefuldiag}]
 Naturality in $D$ means that for every $k:D\to D'$ of $\cat{D}$, the following diagram must commute,
 $$
 \begin{tikzcd}
  \Hom_\cat{D}(FC, D) \ar{d}{k\circ -} \ar{r}{\flat}[swap]{\cong} & \Hom_\cat{C}(C, GD) \ar{d}{Gk\circ -} \\
  \Hom_\cat{D}(FC, D') \ar{r}{\flat}[swap]{\cong} & \Hom_\cat{C}(C, GD')
 \end{tikzcd}
 $$
 which means that for every $f^\sharp:FC\to D$, we must have 
 \begin{equation}\label{natD}
  (k\circ f^\sharp)^\flat \;=\; Gk\circ f^\flat   .
 \end{equation}
 Just as well, naturality in $C$ means that for every $h:C\to C'$ of $\cat{C}$, the following diagram must commute,
 $$
 \begin{tikzcd}
  \Hom_\cat{D}(FC', D') \ar{d}{-\circ Fh} \ar{r}{\flat}[swap]{\cong} & \Hom_\cat{C}(C', GD') \ar{d}{-\circ h} \\
  \Hom_\cat{D}(FC, D') \ar{r}{\flat}[swap]{\cong} & \Hom_\cat{C}(C, GD')
 \end{tikzcd}
 $$
 which means that for every $g^\sharp:FC'\to D'$ we must have
 \begin{equation}\label{natC}
    (g^\sharp\circ Fh)^\flat \;=\; g^\flat \circ h .
 \end{equation}
 Now by \Cref{natD,natC}, and since the maps $\flat$ and $\sharp$ are bijections, we have that 
 $$
 k\circ f^\sharp \;=\; g^\sharp\circ Fh
 $$
 if and only if
 $$
 Gk\circ f^\flat \;=\; g^\flat \circ h ,
 $$
 that is, the diagram on the left commutes if and only if the diagram on the right does.
\end{proof}

\begin{ex}
 Show that \Cref{usefuldiag} has a converse, namely, that the equivalence between the commutativity of the two diagrams (for all the arrows involved) is actually \emph{equivalent} to naturality of the bijections.
\end{ex}

Now let's see in detail why and how $\eta$ is natural.

\begin{lemma}\label{etanatural}
 The maps $\eta_C:C\to GFC$ assemble to a natural transformation $\eta:\id_\cat{C}\Rightarrow G\circ F$ of endofunctors $\cat{C}\to\cat{C}$.
\end{lemma}

\begin{deph}
 The natural transformation $\eta:\id_\cat{C}\Rightarrow G\circ F$ of components $\eta_C:C\to GFC$ is called the \emph{unit} of the adjunction. 
\end{deph}

\begin{proof}[Proof of \Cref{etanatural}]
 We have to prove that for all $h:C\to C'$ in $\cat{C}$, this diagram in $\cat{C}$ commutes.
 $$
 \begin{tikzcd}
  C \ar{r}{\eta_C} \ar{d}{h} & GFC \ar{d}{GFh} \\
  C' \ar{r}{\eta_{C'}} & GFC'
 \end{tikzcd}
 $$
 Now we can use \Cref{usefuldiag} to ``rigidly replace $G$ on the right of the diagram by $F$ on the left of the diagram'', replacing also the horizontal morphisms with their ``sharp'' counterpart. We obtain therefore the equivalent diagram in $\cat{D}$
 $$
 \begin{tikzcd}
  FC \ar{r}{\id_{FC}} \ar{d}{Fh} & FC \ar{d}{Fh} \\
  FC' \ar{r}{\id_{FC'}} & FC'
 \end{tikzcd}
 $$
 which clearly commutes. 
\end{proof}

So, the adjunction $F\ladj G$ defines a natural transformation $\eta:\id_\cat{C}\Rightarrow G\circ F$ whose components satisfy a universal property. 
Dually, we can also obtain a natural transformation on the other side of the adjunction. Namely, let's now keep $D$ fixed in the natural bijection 
$$
\Hom_\cat{D}(FC, D) \; \xrightarrow{\cong} \;\Hom_\cat{C}(C, GD) .
$$
We have then a natural isomorphism between the presheaves
$$
\Hom_\cat{D}(F-, D) \quad\mbox{and}\quad\Hom_\cat{C}(-, GD) : \cat{C}^\op\to\Set ,
$$
in other words, we are saying that the presheaf $\Hom_\cat{D}(F-, D): \cat{C}^\op\to\Set$ is representable, and represented by the object $GD$ of $\cat{C}$. 

\begin{ex}
 What is the interpretation of this condition in terms of ``representable presheaves as observations''?
\end{ex}

Dually to the case of $\eta$, we get a universal arrow $FGD\to D$ of $\cat{D}$ for each object $D$, which is given by $(\id_{GD})^\sharp$. Again, this resulting universal map is natural in $D$, giving a natural transformation $F\circ G\Rightarrow\id_\cat{D}$ of endofunctors $\cat{D}\to\cat{D}$.

\begin{deph}
 The natural transformation $F\circ G\Rightarrow\id_\cat{D}$ induced by the adjunction $F\ladj G$ is called the \emph{counit} of the adjunction, and it denoted by $\e:F\circ G\Rightarrow\id_\cat{D}$.
\end{deph}

So, for each object $D$ of $\cat{D}$, the component of the counit at $D$ satisfies $\e_D=(\id_{GD})^\sharp$ and $(\e_D)^\flat=\id_{GD}$.

\begin{ex}[important!]
 Derive in detail the natural transformation $\e$, as we did above for $\eta$, and show that it is natural. (Hint: dualize what was done above.)
\end{ex}

Let's see how the universal property of $\e_D$ looks in practice. For every object $C$ of $\cat{C}$, apply $F$ to it, to get the object $FC$ of $\cat{D}$. Now the bijection $\Hom_\cat{D}(FC, D) \to \Hom_\cat{C}(C, GD)$ says that given any morphism $g^\sharp:FC\to D$ of $\cat{D}$, there is a unique morphism $g^\flat:C\to GD$ of $\cat{C}$ such that the following triangle of $\cat{D}$ commutes.
\begin{equation}\label{unipropepsilon}
\begin{tikzcd}
 C \ar[dashed,""{name=GFLAT,right}]{d}{ g^\flat} && FC \ar{d}[""{name=FGFLAT,left},swap]{Fg^\flat} \ar{dr}{g^\sharp} \\
 GD && FGD \ar{r}[swap]{\e_D} & D
 \ar[mapsto,from=GFLAT,to=FGFLAT,"F",start anchor={[xshift=3ex]}]
\end{tikzcd}
\end{equation}

\begin{eg}[linear algebra]\label{counitvect}
 Let's look at the counit of the adjunction between sets and vector spaces (\Cref{adjvectset}). Let $V$ be a vector space. 
 The counit $\e_V:FUV\to V$ is a map from the vector space of \emph{formal linear combinations of the vectors of $V$} to $V$. Which map in particular? We know that this map must be equal to $(\id_{UV})^\sharp$, that is, must be the linear extension of the identity of $V$ (seen as a map between sets). In other words, $\e_V:FUV\to V$ must be the unique linear map such that for each $v\in V$, $\e_V(1v)=v$, and for all nontrivial linear combinations, the value of $\e_V$ is uniquely specified by linearity. Therefore, applying $\e_D$ to generic expression 
 $$
 a_1 v_1 + \dots + a_n v_n
 $$
 of $FUV$, we get the \emph{result in $V$} of the expression above. Now mind the difference between a \emph{formal linear combination} and the \emph{actual result of the linear combination}: for example, if $V=\R$, we have
 $$
 \e_\R(2\cdot 3 + 2\cdot 1)  \;=\; 8 .
 $$
 Note that, \emph{as formal expressions}, $``2\cdot 3 + 2\cdot 1''$ and $``8''$ are not the same (but they have the same result). The map $\e_\R$ \emph{actually does} the operations prescribed by the formal expression.
 
 The universal property of $\e_V$ says now that for every set $X$, we can form the free vector space $FX$, and given any linear map $g:FX\to V$, there is a unique function $g^\flat$ from $X$ to $V$ (seen as a set) such that the following diagram commutes:
 $$
\begin{tikzcd}
 X \ar[dashed,""{name=GFLAT,right}]{d}{ g^\flat} && FX \ar{d}[""{name=FGFLAT,left},swap]{Fg^\flat} \ar{dr}{g} \\
 UV && FUV \ar{r}[swap]{\e_V} & V
 \ar[mapsto,from=GFLAT,to=FGFLAT,"F",start anchor={[xshift=3ex]}]
\end{tikzcd}
$$
This is precisely the universal property of a basis, it is saying that $X$ is a basis of $FX$.
\end{eg}

As we have seen above, the counit map $\e_V:FUV\to V$ turns a \emph{formal linear combination of element of $V$} into their \emph{actual} linear combination. In other words, the map $\e_V$ \emph{is} that operation (such as the addition is a map $+:\R\times\R\to\R$). In other words again, this map is exactly what makes $V$ a vector space, it is its vector space structure (a vector space is precisely a set equipped with an operation of linear combination). 
This fact, that the ``extra structure'' is captured by the counit, is a very common phenomenon, it will be made precise by the concept of \emph{Eilenberg-Moore} or \emph{monadic adjunction} (see Sections~\ref{emadj} and~\ref{adjmonads}). Not every adjunction has this property, for example, the one between sets and topological spaces (\Cref{adjtopset}) does not (why?).

\begin{ex}[group theory]
 Show that the counit of the adjunction between sets and groups (\Cref{adjgrpset}), analogously to the case of vector spaces, can be seen as the map giving the group structure.
\end{ex}

 Can you think of other forgetful functors giving rise to adjunctions that have a similar property?

\subsection{Alternative definition of adjunctions}

We have seen how to obtain the unit and counit natural transformations from an adjunction, and we know that by the Yoneda lemma these natural transformations (or better, their components) satisfy a universal property, and so determine the natural bijections $\Hom_\cat{D}(FC,D)\to \Hom_\cat{C}(C,GD)$ and $\Hom_\cat{C}(C,GD)\to \Hom_\cat{D}(FC,D)$ uniquely. We can now ask the converse question: given functors $F:\cat{C}\to\cat{D}$ and $G:\cat{D}\to\cat{C}$, does a pair of natural transformations $\eta:\id_\cat{C}\Rightarrow G\circ F$ and $\e:F\circ G\Rightarrow \id_\cat{D}$ induce an adjunction $F\ladj G$? The answer is, almost, but not quite. The reason is that, thanks to the Yoneda lemma, the universal maps $\eta_C$ and $\e_D$ specify natural \emph{transformations}, but not necessarily natural \emph{isomorphisms} $\Hom_\cat{D}(FC,D)\to \Hom_\cat{C}(C,GD)$. In order for $\eta_C$ and $\e_D$ to specify isomorphisms we need a couple of additional conditions to be satisfied, known as the \emph{triangle identities}. 

\begin{lemma}[Triangle identities]\label{triangleidlemma}
 Let $\eta:\id_\cat{C}\Rightarrow G\circ F$ and $\e:F\circ G\Rightarrow \id_\cat{D}$ be the unit and counit of an adjunction $F\ladj G$. Then the following diagrams of natural transformations commute.
 \begin{equation}\label{trianglegeneral}
  \begin{tikzcd}
   F \nat{r}{F \eta} \nat{dr}[swap]{\id_F} & FGF \nat{d}{\e F} \\
   & F
  \end{tikzcd}
  \qquad\qquad
  \begin{tikzcd}
   G \nat{r}{\eta G} \nat{dr}[swap]{\id_G} & GFG \nat{d}{G \e} \\
   & G
  \end{tikzcd}
 \end{equation}
\end{lemma}

In components, the triangle identities read as follows (why?), for every object $C$ of $\cat{C}$ and $D$ of $\cat{D}$.
 \begin{equation}\label{trianglecomp}
  \begin{tikzcd}
   FC \ar{r}{F \eta_C} \ar{dr}[swap]{\id_{FC}} & FGFC \ar{d}{\e_{FC}} \\
   & FC
  \end{tikzcd}
  \qquad\qquad
  \begin{tikzcd}
   GD \ar{r}{\eta_{GD}} \ar{dr}[swap]{\id_{GD}} & GFGD \ar{d}{G\e_D} \\
   & GD
  \end{tikzcd}
 \end{equation}
Note that the first diagram is in $\cat{D}$, and the second diagram is in $\cat{C}$.

\begin{proof}[Proof of \Cref{triangleidlemma}]
 We can rewrite the diagrams of \eqref{trianglecomp} as the following squares,
 $$
  \begin{tikzcd}
   FC \ar{r}{F \eta_C} \ar{d}[swap]{\id_{FC}} & FGFC \ar{d}{\e_{FC}} \\
   FC \ar{r}[swap]{\id_{FC}} & FC
  \end{tikzcd}
  \qquad\qquad
  \begin{tikzcd}
   GD \ar{r}{\eta_{GD}} \ar{d}[swap]{\id_{GD}} & GFGD \ar{d}{G\e_D} \\
   GD \ar{r}[swap]{\id_{GD}}& GD
  \end{tikzcd}
 $$
 and flip the first one (for convenience),
 $$
  \begin{tikzcd}
   FC \ar{d}[swap]{F \eta_C} \ar{r}{\id_{FC}} & FC \ar{d}{\id_{FC}} \\
   FGFC \ar{r}[swap]{\e_{FC}} & FC
  \end{tikzcd}
  \qquad\qquad
  \begin{tikzcd}
   GD \ar{r}{\eta_{GD}} \ar{d}[swap]{\id_{GD}} & GFGD \ar{d}{G\e_D} \\
   GD \ar{r}[swap]{\id_{GD}}& GD
  \end{tikzcd}
 $$
 Now using \Cref{usefuldiag}, we can apply ``flat'' to the first diagram, replacing $F$ on the left by $G$ on the right, and replacing the horizontal arrows with their ``flat'' versions. Similarly, we can apply ``sharp'' to the second diagram, replacing $G$ on the right of the diagram with $F$ on the left, and replacing the horizontal morphisms by their ``sharp'' version. We get the following two diagrams, equivalent to the previous ones,
 $$
  \begin{tikzcd}
   C \ar{d}[swap]{ \eta_C} \ar{r}{\eta_C} & GFC \ar{d}{\id_{GFC}} \\
   GFC \ar{r}[swap]{\id_{GFC}} & GFC
  \end{tikzcd}
  \qquad\qquad
  \begin{tikzcd}
   FGD \ar{r}{\id_{FGD}} \ar{d}[swap]{\id_{FGD}} & FGD \ar{d}{\e_D} \\
   FGD \ar{r}[swap]{\e_{D}}& D
  \end{tikzcd}
 $$
 and these clearly commute. 
\end{proof}

So the unit and the counit of an adjunction must always satisfy the triangle identities \eqref{trianglegeneral}. This condition is actually sufficient, if a pair of natural transformations $\eta:\id_\cat{C}\Rightarrow G\circ F$ and $\e:F\circ G\Rightarrow \id_\cat{D}$ satisfies the triangle identities, then the natural transformations $\Hom_\cat{D}(FC,D)\to \Hom_\cat{C}(C,GD)$ and $\Hom_\cat{C}(C,GD)\to \Hom_\cat{D}(FC,D)$ that they induce are bijections. Recall that by the Yoneda lemma, as we have seen in \Cref{uniprop}, the map $\eta_C:C\to GFC$ corresponds to the natural transformation of components (for all objects $D$ of $\cat{D}$)
\begin{equation}\label{isoeta}
\begin{tikzcd}[row sep=0]
 \Hom_\cat{D}(FC,D) \ar{r}{\flat} & \Hom_\cat{C}(C,GD) \\
 f^\sharp \ar[mapsto]{r} & f^\flat = Gf^\sharp \circ \eta_C  .
\end{tikzcd}
\end{equation}
(To see this more explicitly, look at the diagram \eqref{unipropeta}.)
Just as well, the map $\e_D:D\to GFD$ corresponds to the natural transformation of components (for all objects $C$ of $\cat{C}$)
\begin{equation}\label{isoepsilon}
\begin{tikzcd}[row sep=0]
 \Hom_\cat{C}(C,GD) \ar{r}{\sharp} & \Hom_\cat{D}(FC,D) \\
 g^\flat \ar[mapsto]{r} & g^\sharp = \e_D\circ Fg^\flat  .
\end{tikzcd}
\end{equation}
(To see this more explicitly, look at the diagram \eqref{unipropepsilon}.)

\begin{lemma}\label{triangleidsufficient}
 Let $F:\cat{C}\to\cat{D}$ and $G:\cat{D}\to\cat{C}$ be functors, and let $\eta:\id_\cat{C}\Rightarrow G\circ F$ and $\e:F\circ G\Rightarrow \id_\cat{D}$ satisfying the triangle identities \eqref{trianglegeneral}. Then the assignments \eqref{isoeta} and \eqref{isoepsilon} are mutually inverse, giving a bijection $\Hom_\cat{D}(FC,D)\to \Hom_\cat{C}(C,GD)$ (and so, an adjunction $F\ladj G$). 
\end{lemma}

\begin{proof}
 Let's start with $f^\sharp:FC\to D$. We apply \eqref{isoeta}, to get $Gf^\sharp \circ \eta_C:C\to GD$, and then we apply \eqref{isoepsilon} to it, to get $\e_D\circ F(Gf^\sharp \circ \eta_C):FC\to D$. We have to prove that 
 \begin{equation}\label{fsharpisfsharp}
 \e_D\circ F(Gf^\sharp \circ \eta_C) \;=\; f^\sharp.
 \end{equation}
 Now consider the following diagram,
 $$
 \begin{tikzcd}
  FC \ar{dr}[swap]{\id_{FC}} \ar{r}{F\eta_C} & FGFC \ar{d}{\e_{FC}} \ar{r}{FGf^\sharp} & FGD \ar{d}{\e_D} \\
  & FC \ar{r}{f^\sharp} & D
 \end{tikzcd}
 $$
 The triangle on the left commutes, as it's one of the triangle identities \eqref{trianglecomp}, and the square on the right commutes by naturality of $\e$. So \Cref{fsharpisfsharp} follows. 
 
 Dually, let's now take $g^\flat:C\to GD$. We apply \eqref{isoepsilon}, to get $\e_D\circ Fg^\flat:FC\to D$, and then \eqref{isoeta}, to get $G(\e_D\circ Fg^\flat) \circ \eta_C$. We have to prove that 
 \begin{equation}\label{gflatisgflat}
 G(\e_D\circ Fg^\flat) \circ \eta_C \;=\; g^\flat.
 \end{equation}
 So consider the diagram
 $$
 \begin{tikzcd}
  C \ar{d}[swap]{\eta_C} \ar{r}{g^\flat} & GD \ar{d}[swap]{\eta_{GD}} \ar{dr}{\id_{GD}} \\
  GFC \ar{r}[swap]{GFg^\flat} & GFGD \ar{r}[swap]{G\e} & GD
 \end{tikzcd}
 $$
 The square on the left commutes by naturality of $\eta$, and the triangle on the right is the other triangle identity of \eqref{trianglecomp}. Therefore \Cref{gflatisgflat} follows, which means that \eqref{isoeta} and \eqref{isoepsilon} are mutually inverse.
\end{proof}

We have therefore proven:
\begin{thm}[Alternative definition of adjunction]
 An adjunction $F\ladj G$ is equivalently given by a pair of natural transformations $\eta:\id_\cat{C}\Rightarrow G\circ F$ and $\e:F\circ G\Rightarrow \id_\cat{D}$ satisfying the triangle identities \eqref{trianglegeneral}.
\end{thm}

This result is very convenient, since the unit and counit are more explicit, and usually easier to work with in mathematical practice. 

\begin{ex}[linear algebra]\label{trianglevect}
 Show that the triangle identities hold for the adjunction between sets and vector spaces of \Cref{adjvectset}. 
\end{ex}

\begin{ex}[group theory]\label{trianglegrp}
 Show that the triangle identities hold for the adjunction between  sets and groups of \Cref{adjgrpset}.
\end{ex}

\subsection{Example: the adjunction between categories and multigraphs}\label{adjcatgraph}

Here is an example of how to describe an adjunction in terms of its units and counits, in this case between categories and multigraphs. 

Let $\cat{C}$ be a small category. We can associate to $\cat{C}$ its \emph{underlying multigraph} $U(\cat{C})$, where the vertices of $U(\cat{C})$ are the objects of $\cat{C}$, and the edges of $U(\cat{C})$ are the morphisms of $\cat{C}$ (all of them, including the identities). 
Take now another small category $\cat{D}$ and a functor $F:\cat{C}\to\cat{D}$. There is an induced morphism of multigraphs $U(\cat{C})\to U(\cat{D})$, which we denote by $U(F)$, which maps the vertex of $U(\cat{C})$ corresponding to an object $X$ of $\cat{C}$ to the vertex of $U(\cat{D})$ corresponding to the object $FX$ of $\cat{D}$. On edges, $U(F)$ maps the edge of $U(\cat{C})$ corresponding to the morphism $f:X\to Y$ of $\cat{C}$ to the edge of $U(\cat{D})$ corresponding to the morphism $Ff:FX\to FY$ of $\cat{D}$.
The assignment $\cat{C}\mapsto U(\cat{C}), F\mapsto U(F)$ preserves identities and composition, and so we have a functor $U:\cat{Cat}\to\cat{MGraph}$. We can view $U$ as a forgetful functor, which takes a (small) category and forgets its unit and composition structure, giving only its underlying graph. 

We want now to construct a left-adjoint to the functor $U$. So consider a directed multigraph $G$, and let's try to obtain a (small) category from it in a canonical way, which we will call $\cat{P}(G)$. The objects of $\cat{P}(G)$ are the vertices of $G$. As morphisms, we would like to have all possible walks in $G$, including the trivial walk at each vertex (which will be the identity). Now, recall from \Cref{graphtoset} that an $n$-chain of $G$ is an $n$-tuple of edges 
$$
(e_1,\dots,e_n)
$$ 
of $G$ which are head-to-tail, that is, such that the target of $e_i$ is the source of $e_{i+1}$ for all $i$. Let's also, for convenience, call $1$-chains simply the edges of $G$, and $0$-chains the vertices, seen as ``trivial chains''. Denoting by $\mathit{Chain}_n(G)$ the set of $n$-chains of $G$, we now define the set of morphism of $\cat{P}(G)$ to be the disjoint union
$$
\coprod_{n\ge 0} \mathit{Chain}_n(G) ,
$$
i.e.~the morphisms of $\cat{P}(G)$ are all the chains of arbitrary (but finite) length in $G$. We define the source of $(e_1,\dots,e_n)$ to be the source of $e_1$, and the target to be the target of $e_n$. The identities will be the chains of length zero, and the composition is given by the composition of chains. 

\begin{ex}
 Show that the composition of chains, defined as above, is associative and unital, so that $\cat{P}(G)$ is indeed a category. (Hint: there is not that much to prove.)
\end{ex}

\begin{deph}
 The category $\cat{P}(G)$ associated to the multigraph $G$ goes under the names of \emph{free category over $G$}, or \emph{fundamental category of $G$}, or also \emph{path category of $G$}. 
\end{deph}

Now let's define the action of $\cat{P}$ on the morphisms. Given multigraphs $G,H$ and a morphism of multigraphs $f:G\to H$, we define the functor $\cat{P}f:\cat{P}(G)\to\cat{P}(G)$ to be:
\begin{itemize}
 \item On objects of $\cat{P}(G)$, it maps the object corresponding to the vertex $x$ of $G$ to the object corresponding to the vertex $f(x)$ of $H$;
 \item On morphisms of $\cat{P}(G)$, it maps the morphism corresponding to the chain $(e_1,\dots,e_n)$ of $G$ to the morphism corresponding to the chain $(f(e_1),\dots,f(e_n))$ of $H$. This is again a chain, since $f$ preserves adjacency (it is a morphism of multigraphs).
\end{itemize}

This makes the assignment $G\mapsto \cat{P}G, f\mapsto \cat{P}f$ a functor $\cat{P}:\cat{MGraph}\to\cat{Cat}$.

\begin{ex}
 Prove that this assignment preserves identities and composition, so that it is indeed a functor. (Hint: there is not that much to prove.) 
\end{ex}

Mind the different levels, which are possibly confusing: $\cat{P}$ is a functor $\cat{MGraph}\to\cat{Cat}$. But also, given a morphism $f:G\to H$ of $\cat{MGraph}$, its image $\cat{P}f:\cat{P}G\to\cat{P}H$ is a functor too, in the sense of a morphism of $\cat{Cat}$.

We will now show that $\cat{P}$ is left-adjoint to $U$, and describe the adjunction in terms of the unit and counit. Let's start with the unit. Consider a multigraph $G$. We can form the category $\cat{P}G$, and then look at the underlying multigraph $U\cat{P}G$. This is in general not the same as the multigraph we started with: there will be additional edges, corresponding to the identities and compositions that we have added. For example, if we start with the graph $G$ given by
$$
\begin{tikzcd}
 x \ar{r}{e_1} & y \ar{r}{e_2} & z ,
\end{tikzcd}
$$
the multigraph $U\cat{P}G$ will be given by
$$
\begin{tikzcd}
x \ar[out=60,in=120,loop,swap,distance=0.8cm] \ar[bend right,swap]{rr}{(e_1,e_2)} \ar{r}{e_1} & y \ar[out=60,in=120,loop,swap,distance=0.8cm] \ar{r}{e_2} & z \ar[out=60,in=120,loop,swap,distance=0.8cm]
\end{tikzcd}
$$
where the loops are coming from the identities in $\cat{P}G$.

\begin{remark}
This is related to the fact that the category \emph{generated} by this diagram:
 $$
\begin{tikzcd}
 X \ar{r}{f} & Y \ar{r}{g} & Z ,
\end{tikzcd}
$$
is implicitly the one that \emph{actually} looks like this:
$$
\begin{tikzcd}
X \ar[out=60,in=120,loop,"\id_X",swap,distance=0.8cm] \ar[bend right,swap]{rr}{g\circ f} \ar{r}{f} & Y \ar[out=60,in=120,loop,"\id_Y",swap,distance=0.8cm] \ar{r}{g} & Z \ar[out=60,in=120,loop,"\id_Z",swap,distance=0.8cm]
\end{tikzcd}
$$
The functor $\cat{P}$ is actually the mathematically rigorous way of generating a category by ``adding in identities and composites''. We have implicitly used it all the times we have used diagrams to represent pieces of a category. 
\end{remark}

So, in general, $G$ and $U\cat{P}G$ are different. However, there is a canonical inclusion of $G$ into $U\cat{P}G$ (see also the example above):
\begin{itemize}
 \item $G$ and $U\cat{P}G$ have the same vertices;
 \item All the edges of $G$ are also edges of $U\cat{P}G$ (but the latter may have more edges). This is because all the edges of $G$ appear in the morphisms of $\cat{P}G$ (they are the 1-chains), which in turns are the edges of $U\cat{P}G$.
\end{itemize}
We denote the inclusion of $G$ into $U\cat{P}G$ by $\eta_G:G\to U\cat{P}G$. It respects incidence, and so it is a morphism of $\cat{MGraph}$.

\begin{ex}
 Show that the map $\eta_G:G\to U\cat{P}G$ is natural in $G$. That is, show that for each morphism of multigraphs $f:G\to H$, the following diagram commutes.
 $$
 \begin{tikzcd}
  G \ar{d}{f} \ar{r}{\eta_G } & U\cat{P}G \ar{d}{U\cat{P}f} \\
  H \ar{r}{\eta_H } & U\cat{P}H
 \end{tikzcd}
 $$
 (Hint: this becomes almost trivial once you see how the morphism $U\cat{P}f$ acts on those edges of $U\cat{P}G$ which lie in the image of $\eta_G$.)
\end{ex}

By the exercise above, we have a natural transformation $\eta:\id_\cat{MGraph}\Rightarrow \cat{P}\circ U$. 
This will be the unit of our adjunction. Note the analogy with the unit of the adjunction of between sets and vector spaces (\Cref{unitvect}): here, just like there, the unit is the embedding of the ``trivial cases''. Here it selects the ``primitive morphisms'', the ones not obtained by composition of other morphisms, in \Cref{unitvect} it selected the ``primitive vectors'', the ones obtained by the trivial linear combinations in the form $1x$ for $x\in X$.

Let's now turn to the counit. Let $\cat{D}$ be a small category. We can look at its underlying multigraph $U\cat{D}$, and then at the category $\cat{P}U\cat{D}$ generated by it. Again, this will in general be different from the category we started with. In particular,
\begin{itemize}
 \item The objects of $\cat{D}$ and $\cat{P}U\cat{D}$ are the same;
 \item The morphisms of $\cat{P}U\cat{D}$ are \emph{tuples} of composable morphisms of $\cat{D}$, including empty tuples one for each object of $\cat{D}$. This is because each morphism of $\cat{D}$ gives rise to an edge of the multigraph $U\cat{D}$, and by applying $\cat{P}$, the morphisms of $\cat{P}U\cat{D}$ will be chains of edges of $U\cat{D}$ (including the 0-chains), i.e.~chains of composable morphisms.
\end{itemize}
Now, mind the difference between a \emph{chain of composable morphisms} $(f_1,\dots,f_n)$ and their \emph{actual composition} $f_n\circ\dots\circ f_1$. The morphisms of $\cat{P}U\cat{D}$ are in the form $(f_1,\dots,f_n)$, where all the $f_i$ are morphisms of $\cat{D}$. 
This should already give a feeling of what the counit map will be: it will be the \emph{map actually performing the composition}. In rigor, $\e_\cat{D}:\cat{P}U\cat{D}\to\cat{D}$ is the functor which
\begin{itemize}
 \item On objects, it maps an object of $\cat{P}U\cat{D}$ to the corresponding object in $\cat{D}$ (remember, the two categories have the same objects);
 \item On morphisms, it maps a tuple of composable morphisms of $\cat{D}$ $(f_1,\dots,f_n)$ (which make up a morphism of $\cat{P}U\cat{D}$) to their actual composition in $\cat{D}$, $f_n\circ\dots\circ f_1$:
 $$
\begin{tikzcd}[column sep=small]
 A \ar{r}{f_1} & B \ar{r}{f_2} & C \ar{r}{f_3} & D  & \, \ar[mapsto]{rrrr}{\e_\cat{D}} &&&& \, & A \ar{rrr}{f_3\circ f_2\circ f_1} &&& D
\end{tikzcd}
$$
 It also maps empty tuples (there's one for each object) to the identity of the corresponding object of $\cat{D}$. 
\end{itemize}

This gives a functor $\e_\cat{D}:\cat{P}U\cat{D}\to\cat{D}$ for each small category $\cat{D}$.

\begin{ex}
 Show that $\e_\cat{D}$ is indeed a functor, i.e.~it respects identity and composition. (Hint: we have just said what it does on the identities.)
\end{ex}

So $\e_\cat{D}$ is a morphism of the category $\cat{Cat}$.

\begin{ex}
 Show that $\e_\cat{D}$ is natural in $\cat{D}$. That is, for every small categories $\cat{C},\cat{D}$ and every functor $F:\cat{C}\to\cat{D}$, the following diagram commutes:
 $$
 \begin{tikzcd}
  \cat{P}U\cat{C} \ar{d}[swap]{\cat{P}U(F)} \ar{r}{\e_\cat{C}} & \cat{C} \ar{d}{F} \\
  \cat{P}U\cat{D} \ar{r}{\e_\cat{D}} & \cat{D}
 \end{tikzcd}
 $$
 (Hint: the assert becomes almost trivial once you see what the functor $\cat{P}U(F)$ does on the morphisms of $\cat{P}U\cat{D}$.)
\end{ex}

Therefore we get a natural transformation $\e:\cat{P}\circ U\Rightarrow\id_\cat{Cat}$. This will be the counit of our adjunction. Note the similarity with the counit of the adjunction between sets and vector spaces (\Cref{counitvect}): there, the counit was mapping formal linear expressions to their actual results. Here, the counit is mapping composable morphisms to their actual composition (and trivial chains to identities). Here too, the counit is exactly encoding the extra structure that categories have over graphs: identities and composition. (So, also this adjunction is monadic, as we will see in \Cref{catgraphmonadic}.)

In order to establish the adjunction, we now need to prove the triangle identities \eqref{trianglecomp}. 
The first identity says that for each multigraph $G$, the following diagram of $\cat{Cat}$ must commute,
$$
\begin{tikzcd}
 \cat{P}G \ar{dr}[swap]{\id_{\cat{P}G}} \ar{r}{\cat{P}\eta_G} & \cat{P}U\cat{P}G \ar{d}{\e_{\cat{P}G}} \\
 & \cat{P}G
 \end{tikzcd}
$$
Now, all the categories in the diagram above have the same objects, and all the functors are just the identities on objects. So let's see what happens on morphisms. A morphism of $\cat{P}G$ is a chain
$(e_1,\dots,e_n)$ of edges of $G$. The map $\cat{P}\eta_G$, by definition of $\cat{P}$ on morphisms, acts on $(e_1,\dots,e_n)$ by applying $\eta_G$ to each entry of the tuple, that is, it gives the tuple $(\eta_G(e_1),\dots,\eta_G(e_n))$. Each entry of the tuple is in the form $\eta_G(e_i)$. What $\eta_G$ does on an edge $e_i$ is that it includes it in the morphisms of $\cat{P}G$ as the 1-chain $(e_i)$. Therefore we get the chain 
$$
((e_1),\dots,(e_n))
$$
Mind that this has \emph{two} levels of brackets: this is a \emph{chain of 1-chains}.
Applying the counit $\e_{\cat{P}G}$ of $\cat{P}G$ gives the composition 
$$
(e_n)\circ \dots \circ (e_1) .
$$
But the composition in $\cat{P}G$ is just the concatenation of chains, in this case, of 1-chains, and so we get
$$
(e_1,\dots,e_n) ,
$$
which gives the tuple we started with, and so it gives the same result as applying the identity of $\cat{P}G$. Therefore the diagram commutes. 

The second triangle identity says that for every category $\cat{D}$, the following diagram of multigraphs must commute,
$$
\begin{tikzcd}
 U\cat{D} \ar{dr}[swap]{\id_{U\cat{D}}} \ar{r}{\eta_{U\cat{D}}} & U\cat{P}U\cat{D} \ar{d}{U\e_\cat{D}} \\
 & U\cat{D}
\end{tikzcd}
$$
Now all the multigraphs in the diagram have the same vertices, and all the morphisms in the diagram are the identity on vertices, so let's focus on edges. The edges of $U\cat{D}$ are the morphisms of $\cat{D}$, so let $f:X\to Y$ be a morphism of $\cat{D}$, seen as an edge of $U\cat{D}$. The unit map $\eta_{U\cat{D}}$ embeds this edge into the morphisms of $\cat{P}U\cat{D}$ (seen as edges of $U\cat{P}U\cat{D})$ as the 1-chain $(f)$. Now we apply the map $U\e_\cat{D}$, which returns the edge associated to the morphism $\e_\cat{D}(f)$. The latter is a ``trivial composition'' of only one morphism, so it just returns itself, i.e.~$f$ (it's a trivial chain of only one arrow). Therefore we just get $f$, which again corresponds to the edge we started with, so it is the same result as applying the identity of $U\cat{D}$. Therefore this diagram also commutes.

In conclusion, we have proven that the functor $\cat{P}:\cat{MGraph}\to\cat{Cat}$ is left-adjoint to the functor $U:\cat{Cat}\to\cat{MGraph}$. 
This is one of the most important connections between graph theory and category theory. 

\begin{ex}
 Write down explicitly the universal properties of the unit and counit associated to this adjunction.  
\end{ex}

\begin{ex}[linear algebra, group theory]
 If you have solved \Cref{trianglevect} or \Cref{trianglegrp}: can you see the resemblance between the meaning of the triangle identities here and in those contexts?
\end{ex}

\begin{ex}
 Construct an adjunction between sets and monoids similar to the one of \Cref{adjgrpset}. What is the relation between this construction and the adjunction $\cat{P}\ladj U$ that we have just seen?
\end{ex}

\begin{ex}[algebraic topology, graph theory]
 Why is the category $\cat{P}G$ sometimes called the ``fundamental category'' of $G$? What is its relation with the fundamental group (\Cref{fundamentalgroup})?
\end{ex}

\section{Adjunctions, limits and colimits}\label{sec_adj_lim}

Adjoint functors interact very well with limits and colimits.

\begin{thm}\label{radjcont}
 Right-adjoint functors are continuous. 
\end{thm}

Let's see what this statement means. Let $\cat{C}$ and $\cat{D}$ be categories, and let $L:\cat{C}\to\cat{D}$ and $R:\cat{D}\to\cat{C}$ be functors with an adjunction $L\ladj R$. Let also $E:\cat{J}\to\cat{D}$ be a diagram in $\cat{D}$ admitting a limit. Then the limit of the diagram 
$$
\begin{tikzcd}
\cat{J} \ar{r}{E} & \cat{D} \ar{r}{R} & \cat{C}
\end{tikzcd}
$$
exists in $C$, and
$$
\lim R\circ E \;\cong\; R\big( \lim E \big) .
$$

Reversing all the arrows of $\cat{C}$ and $\cat{D}$, we obtain immediately

\begin{cor}\label{ladjcocont}
 Left-adjoint functors are cocontinuous. 
\end{cor}

As we did with \Cref{repfuncontinuous}, let's first try to understand the statement for the simpler case of binary products, and then prove it in general. Finally we will look at examples.

\begin{caveat}
Note that, differently from what we did when we dualized \Cref{repfuncontinuous} (obtaining \Cref{reprpresheaveslim}), we now are reversing arrows the arrows of \emph{both} $\cat{C}$ and $\cat{D}$. Therefore left-adjoint functors do \emph{not} map colimits into limits (as representable preshaves do), they map colimits into \emph{colimits}, i.e.~they preserve colimits.
\end{caveat}

\subsection{Right-adjoints and binary products}

Let $A$ and $B$ be objects of $\cat{D}$ admitting a product $A\times B$. We want to show that 
$$
R(A\times B) \;\cong\; RA \times RB .
$$
To do so, it suffices to check that $R(A\times B)$ satisfies the universal property of the product of $RA$ and $RB$ in $\cat{C}$ (why?). This means that we have to check that for every object $C$ of $\cat{C}$, to each pair of arrows $C\to RA$, $C\to RB$ there corresponds a unique arrow $C\to R(A\times B)$, filling in the following diagram.
$$
\begin{tikzcd}
 & C \ar{dr} \ar{dl} \uni{d} \\
 RA & R(A\times B) \ar{l}{Rp_1} \ar{r}[swap]{Rp_2} & RB
\end{tikzcd}
$$
In other words, we have to give a natural bijection 
$$
\Hom_\cat{C}(C,R(A\times B)) \;\cong\; \Hom_\cat{C}(C,RA) \times \Hom_\cat{C}(C,RB).
$$
Now we have the following chain of natural bijections,
\begin{align*}
\Hom_\cat{C}(C,R(A\times B)) \;&\cong\;\Hom_\cat{D}(LC,A\times B)  \\
\;&\cong\; \Hom_\cat{D}(LC,A)\times \Hom_\cat{D}(LC,B) \\
\;&\cong\; \Hom_\cat{C}(C,RA) \times \Hom_\cat{C}(C,RB) ,
\end{align*}
given respectively by the adjunction, by the universal property of $A\times B$ (or by \Cref{repfuncontinuous}, with naturality coming from \Cref{natisolimit}), and by the adjunction again (twice). 

This suffices to prove the statement for the case of binary products. But before moving on, let's try to interpret this fact intuitively. $R$ being right-adjoint to $L$ means the following. For each object $D$ of $\cat{D}$, the information that we can extract from $D$ by first applying $R$ and then probing the resulting object $RD$ in $\cat{C}$ with an object $C$ (by mapping $C$ to $RD$ in all possible ways), gives the same result as directly probing $D$ with the object $LC$ of $\cat{D}$. Now given the product $A\times B$, the features that we can extract by applying $R$ and then probing with $C$ are the same that we would obtain if we were to experiment with $A$ and $B$ separately. This is because applying $R$ and then probing with $C$ is equivalent to probing $A\times B$ directly with $LC$, and we know that, by the universal property of the product, probing $A\times B$ directly corresponds to probing $A$ and $B$ separately. (See also the remarks after \Cref{repfuncontinuous}.)
So in other words, right-adjoint functors \emph{do not exhibit complex behavior on products} in the sense of \Cref{complexity}.

\begin{ex}
 Can you interpret the behavior of left-adjoint functors dually?
\end{ex}

\subsection{General proof of \Cref{radjcont}}

In order to prove the theorem, we make use of the following lemma, which takes care of the diagram-chasing.

\begin{lemma}\label{adjcone}
 Suppose that $L\ladj R$, and let $E:\cat{J}\to\cat{D}$ be a diagram in $\cat{D}$. Then the adjunction induces a bijection
 $$
 \Cone(LC,E) \;\cong\; \Cone (C, R\circ E)
 $$
 for all objects $C$ of $\cat{C}$, natural in $C$.
\end{lemma}

\begin{proof}[Proof of \Cref{adjcone}]
 A cone $\alpha$ over $E$ with tip $LC$ consists in particular of a tuple of arrows $\alpha_I:LC\to EI$ for each object $I$ of $\cat{J}$, in this form
 $$
 \begin{tikzcd}
 &  LC \ar{dr} \ar{d}{\alpha_{I'}} \ar{dl}[swap]{\alpha_I} \\
 EI & EI' & \cdots
 \end{tikzcd}
 $$
 Moreover, in order for this tuple to be a cone, we need that for every morphism $m:I\to I'$ of $\cat{J}$, $\alpha_{I'}=Em\circ\alpha_{I}$, i.e.~all these triangles must commute,
 \begin{equation}\label{concondL}
 \begin{tikzcd}[column sep=small]
  &  LC \ar{dr}{\alpha_{I'}} \ar{dl}[swap]{\alpha_I} \\
   EI \ar{rr}[swap]{Em} && EI' 
 \end{tikzcd}
 \end{equation}
 The adjunction $\Hom_\cat{D}(LC, D) \to \Hom_\cat{C}(C, RD)$ now maps all these arrows $\alpha_I:LC\to EI$ bijectively to arrows $\alpha_I^\flat:C\to REI$ in the form
 $$
 \begin{tikzcd}
 &  C \ar{dr} \ar{d}{\alpha_{I'}^\flat} \ar{dl}[swap]{\alpha_I^\flat} \\
 REI & REI' & \cdots
 \end{tikzcd}
 $$
 and arrows in this latter form belong to a cone over $R\circ E$ if and only if for every morphism $m:I\to I'$ of $\cat{J}$, we need $\alpha_{I'}^\flat=REm\circ\alpha_{I}^\flat$, i.e.~all these triangles must commute,
 \begin{equation}\label{concondR}
 \begin{tikzcd}[column sep=small]
  &  C \ar{dr}{\alpha_{I'}^\flat} \ar{dl}[swap]{\alpha_I^\flat} \\
   REI \ar{rr}[swap]{REm} && REI' 
 \end{tikzcd}
 \end{equation}
 and moreover every cone over $R\circ E$ of tip $C$ is in this form. But now, by naturality of the bijection $\flat:\Hom_\cat{D}(LC, D) \to \Hom_\cat{C}(C, RD)$ in the argument $D$ (or by \Cref{usefuldiag}), the diagram \eqref{concondL} commutes if and only if the diagram \eqref{concondR} does. Therefore the cones in $\Cone(LC,E)$ are mapped bijectively to the cones in $\Cone (C, R\circ E)$.
 
 The naturality of the bijection is guaranteed again by the naturality of the bijection $\flat:\Hom_\cat{D}(LC, D) \to \Hom_\cat{C}(C, RD)$ in the argument $C$.
\end{proof}

We are now ready to prove the theorem.

\begin{proof}[Proof of \Cref{radjcont}]
 We have to show that the object $R\big( \lim E \big)$ of $\cat{C}$ satisfies the universal property of the limit of $R\circ E$. That is, we have to give a bijection
 $$
 \Cone(C, R\circ E) \;\cong\; \Hom_\cat{C}(C, R(\lim E))
 $$
 for all objects $C$ of $\cat{C}$, natural in $C$. 
 
 Now, we have the following chain of natural bijections,
 \begin{align*}
  \Cone(C, R\circ E) \;\cong\; \Cone(LC,E) \;\cong\; \Hom_\cat{D}(LC,\lim E) \;\cong\; \Hom_\cat{C}(C, R(\lim E)) ,
 \end{align*}
 obtained, in order, by \Cref{adjcone}, by the universal property of $\lim E$, and by the adjunction.
\end{proof}

\subsection{Examples}

\begin{eg}[topology]
 The forgetful functor $\cat{Top}\to\cat{Set}$ has a left-adjoint (\Cref{adjtopset}) as well as a right-adjoint (\Cref{radjtopset}). Therefore it is continuous and cocontinuous. In particular, the underlying set of the coproduct of topological spaces is their disjoint union, and the underlying set of the product of topological spaces is their cartesian product. 
 Compare this also with the other results of \Cref{topsetlim}.
\end{eg}

\begin{eg}[linear algebra]
 The forgetful functor $\cat{Vect}\to\cat{Set}$ has a left-adjoint (\Cref{adjvectset}), therefore it preserves limits. In particular:
 \begin{itemize}
  \item The product of two vector spaces, isomorphic to their disjoint union, has as underlying set the cartesian product of the respective sets;
  \item The equalizer $E$ of a parallel pair of linear maps $f,g:V\to W$ is mapped to the equalizer of the underlying functions, which gives a subset $U(E)\subseteq U(V)$. In other words, $E$ is a vector space whose underlying set is a subset of the set underlying $V$, i.e.~a vector subspace of $V$.
 \end{itemize}
 Compare this also with the other results of \Cref{vecsetlim}. Note in particular that colimits are not preserved in general.

 Dually, the free functor $F:\cat{Set}\to\cat{Vect}$ preserves colimits. Therefore:
 \begin{itemize}
  \item The vector space freely generated by a disjoint union $A \coprod B$ is the direct sum of the spaces generated by $A$ and $B$ respectively;
  \item The quotient set of an equivalence relation gives rise to a quotient vector space (why is this a colimit?);
  \item The empty set generates the zero vector space.
 \end{itemize}
\end{eg}

\begin{ex}[group theory]
 Make statements similar to the example above for the adjunction between sets and groups (\Cref{adjgrpset}).
\end{ex}

\begin{ex}
 Define the product of two categories in $\cat{Cat}$. (Hint: by the adjunction of \Cref{adjcatgraph} between categories and multigraphs, the underlying multigraph to the product of two categories is necessarily the product of the two respective multigraphs.)
\end{ex}

\begin{ex}
 Define the coproduct of two categories in $\cat{Cat}$. Show that the fundamental category of a disjoint union of graphs is the coproduct of the respective fundamental categories.
\end{ex}

\section{The adjoint functor theorem for preorders}

It turns out that, under some conditions, \Cref{radjcont} admits a converse. Converse-like statements to \Cref{radjcont} are known under the name of \emph{adjoint functor theorems}, and say that under some conditions, if a functor $G:\cat{D}\to\cat{C}$ preserves all limits, then it has a left-adjoint. There are several adjoint functor theorems in the literature, see for example Section~4.6 of \cite{ctcontext}, or \href{https://ncatlab.org/nlab/show/adjoint+functor+theorem}{the nLab page on adjoint functor theorems (link)}.

An intuitive interpretation of the phenomenon underlying the adjoint functor theorems is the following. In some cases, we can \emph{reconstruct} what the left-adjoint to $G$ would do, thanks to the fact that $G$ preserves limits, i.e.~does not lose information on how more complex objects are created from simpler ones (for example by forming products). Of course, the condition that $G$ has to preserve all limits is necessary, by \Cref{radjcont}. The adjoint functor theorems say that, in some cases, this is also sufficient.

Here we will look at the simplest of the adjoint functor theorems, the \emph{adjoint functor theorem for preorders}, following the approach of \cite{sevensketches}. 

\begin{thm}[Adjoint functor theorem for preorders]\label{adjfuncthm}
 Let $(X,\le)$ and $(Y,\le)$ be partial orders (or preorders). Suppose that $Y$ has all infima (or meets), and let $g:Y\to X$ be a monotone function preserving all infima. Then $g$ has a lower adjoint $f:X\to Y$, given by 
 $$
 f(x) \;=\; \inf\{ y\in Y \,|\, x\le g(y) \} .
 $$
 
 In particular, a monotone map $g:Y\to X$ is the upper (or right-) adjoint of a Galois connection if and only if it preserves all infima.
\end{thm}

\begin{ex}[sets and relations]
 What is, explicitly, the corresponding statement for preorders?
\end{ex}

Immediately we also get the dual statement,
\begin{cor}
 Suppose that $X$ has all suprema (or joins).
 A monotone map $f:X\to Y$ is lower (or left-) adjoint if and only if it preserves all suprema.
\end{cor}

\begin{ex}[sets and relations]
 In this case, what is the explicit formula for the lower adjoint to $f$?
\end{ex}

Before proving the theorem, let's look at a practical example of how the proof of the theorem works.

\subsection{The case of convex subsets}

Consider again the case of \Cref{convsubsets}. The set $X$ is the set of all subsets of $\R^2$, the set $Y$ is the set of convex subsets of $\R^2$, both sets are ordered by inclusion, and the ``forgetful'' inclusion map $g:Y\to X$ is monotone. 

The infima in both sets are given by the intersection of sets (why?), and the map $g:Y\to X$ preserve those infima, since the intersection of convex sets is again convex (why?).
By \Cref{adjfuncthm}, this is enough to establish that $g$ has a left-adjoint $f$. Plugging in the thesis of the theorem, this map $f$ takes a subset $S\subseteq \R^2$ and gives the convex subset
$$
f(S) \;=\; \bigintersection \{ C \subseteq \R^2 \mbox{ convex} \,|\, S\subseteq C \} ,
$$
that is $f(S)$ must be the intersection of all the convex subsets of $\R^2$ which contain $S$. We know from \Cref{convsubsets} that is exactly the convex hull. 

The actual proof of the theorem follows this same line of reasoning, which, it turns out, is purely about order theory, and not really about convexity.

\subsection{Proof of \Cref{adjfuncthm}}

We will prove the statement for partial orders, the statement for more general preorders is analogous (why?).

Define for each $x$, 
$$
 f(x) \;\coloneqq\; \inf\{ y\in Y \,|\, x\le g(y) \} .
$$
First of all, we have to show that this assignment $x\mapsto f(x)$ is a monotone map (or a functor) $X\to Y$. So suppose that $x\le x'$. Then for all $y$ such that $x'\le g(y)$, we also have $x\le x'\le g(y)$, which means that 
$$
\{ y\in Y \,|\, x\le g(y) \}  \;\supseteq \;\{ y\in Y \,|\, x'\le g(y) \}.
$$
Taking an infimum over a \emph{larger or equal} subset gives us a \emph{smaller or equal} infimum, so 
$$
f(x) \;=\; \inf\{ y\in Y \,|\, x\le g(y) \} \;\le\; \inf\{ y\in Y \,|\, x'\le g(y) \} \;=\; f(x').
$$
So $f$ is monotone. 

Now we need to prove that $f$ is indeed lower adjoint to $g$. We will use the unit and counit characterization of \Cref{unitcounit}. Now since $g$ preserves infima, we have that for each $x\in X$,
$$
g(f(x)) \;=\; g\left( \inf\{ y\in Y \,|\, x\le g(y) \} \right) \;=\; \inf\{ g(y) \,|\, y\in Y , \,x\le g(y) \}
$$
and in particular the infimum on the right-hand side exists.
Since clearly $x$ is less or equal than every element of the set $\{ g(y) \,|\, y\in Y , \,x\le g(y) \}$, by the universal property of the infimum $x$ is also less or equal than the infimum of that set, which means 
$$
x \;\le\; \inf\{ g(y) \,|\, y\in Y , \,x\le g(y) \}\;=\; g(f(x)) .
$$
This inequality is the unit of the adjunction.
For the counit, let $z\in Y$. Then 
$$
f(g(z)) \;=\; \inf\{ y\in Y \,|\, g(z)\le g(y) \} .
$$
Now since $g(z)\le g(z)$, the element $z$ belongs to the set $\{ y\in Y \,|\, g(z)\le g(y) \}$, and so $z$ must be larger or equal than the infimum over that set, that is, 
$$
z \;\ge\; \inf\{ y\in Y \,|\, g(z)\le g(y) \} \;=\; f(g(z)) .
$$
This inequality is the counit of the adjunction. The naturality and triangle diagrams, since we are in a preorder, commute trivially. Therefore $f$ is left-adjoint to $g$.

\subsection{Further considerations and examples}\label{substructures}

\Cref{adjfuncthm} says that a monotone map is the upper adjoint of a Galois connection if and only if it preserves infima, but this says nothing about suprema. That is, an upper adjoint in general does \emph{not} preserve suprema. (\emph{Lower} adjoints preserve suprema, of course.)

This is why the intersection of convex subsets is convex, but the union of convex subsets, in general, is not. This is a very general phenomenon: many special properties of subsets are preserved by intersections but not by unions -- this is a sign that there is a Galois connection involved. 

\begin{eg}[linear algebra, group theory]
 Consider the following statements:
 \begin{itemize}
  \item The intersection of vector subspaces of $\R^3$ is always a vector subspace. The union is in general not (take the union of two lines for example).
  \item The intersection of subgroups of a group $G$ is again a subgroup. Their union is in general not (can you give a counterexample?).
 \end{itemize}
\end{eg}

\begin{ex}[sets and relations, linear algebra, group theory]
 Show that, for the two examples above, there is a Galois connection involved. (Hint: use \Cref{adjfuncthm}.)
\end{ex}

\begin{eg}[topology]\label{topclosure}
 Let $T$ be a topological space. Let $X$ be the poset of subsets of $T$, ordered by inclusion, and let $Y$ be the subposet of \emph{closed} subsets of $X$. There is a ``forgetful inclusion'' $g:Y\to X$, which maps closed subsets to just subsets. This map preserves infima, because the intersection of an arbitrary number of closed sets is closed. Therefore it has a lower adjoint $f:X\to Y$. Following the construction of \Cref{adjfuncthm}, this lower adjoint takes a set $S\subseteq T$, and gives the intersection of all the closed subsets of $T$ containing $S$. This is known as the \emph{closure of $S$} in topology.
\end{eg}

\begin{ex}[topology, tricky!]
 The \emph{union} of two closed sets is again closed. So does the map $g:Y\to X$ above also admit an \emph{upper} adjoint?
\end{ex}

\begin{eg}[sets and relations]
The inclusion $\Z\to\R$ clearly preserves both infima and suprema. Define now $\bar{\Z}$ and $\bar{\R}$ to be $\Z$ and $\R$, extended to have $+\infty$ and $-\infty$ too. The inclusion $\bar{\Z}\to\bar{\R}$ still preserves infima and suprema. Therefore it has an upper and a lower adjoint. This gives a slick solution to \Cref{round}.
\end{eg}

\newpage
\chapter{Monads and comonads}\label{secmonads}

\begin{deph}\label{defmonad}
	Let $\cat{C}$ be a category. A \emph{monad on $\cat{C}$} consists of:
	\begin{itemize}
		\item A functor $T:\cat{C}\to\cat{C}$;
		\item A natural transformation $\eta:\id_{\cat{C}}\Rightarrow T$ called \emph{unit};
		\item A natural transformation $\mu:TT \Rightarrow T$ called \emph{composition} or \emph{multiplication};
	\end{itemize}
	such that the following diagrams commute, called ``left and right unitality'' and ``associativity'', respectively.
	\begin{equation*}
		\begin{tikzcd}
			T \nat{r}{\eta T} \nat{dr}[swap]{\id} & TT \nat{d}{\mu}  \\
			& T
		\end{tikzcd}
		\qquad
		\begin{tikzcd}
			T \nat{r}{T \eta} \nat{dr}[swap]{\id} & TT \nat{d}{\mu} \\
			& T
		\end{tikzcd}
		\qquad
		\begin{tikzcd}
			TTT \nat{r}{T \mu} \nat{d}{\mu T} & TT \nat{d}{\mu} \\
			TT \nat{r}{\mu} & T
		\end{tikzcd}
	\end{equation*}
\end{deph}

We can also write down the natural transformations in terms of their components. For each object $X$ of $\cat{C}$, the unit is a morphism $\eta_X:X\to TX$, and the multiplication is a morphism $\mu_X:TTX\to TX$, such that the following diagrams commute.
\begin{equation}\label{monaddiagrams}
\begin{tikzcd}
TX \ar{r}{\eta_{TX}} \ar{dr}[swap]{\id} & TTX \ar{d}{\mu_X}  \\
& TX
\end{tikzcd}
\qquad
\begin{tikzcd}
TX \ar{r}{T \eta_X} \ar{dr}[swap]{\id} & TTX \ar{d}{\mu_X} \\
& TX
\end{tikzcd}
\qquad
\begin{tikzcd}
TTTX \ar{r}{T \mu_X} \ar{d}{\mu_{TX}} & TTX \ar{d}{\mu_X} \\
TTX \ar{r}{\mu_X} & TX
\end{tikzcd}
\end{equation}

\begin{deph}
	A \emph{comonad} on $\cat{C}$ is a monad on $\cat{C}^\op$.
\end{deph}

Explicitly, a comonad on $\cat{C}$ consists of:
	\begin{itemize}
		\item A functor $C:\cat{C}\to\cat{C}$;
		\item A natural transformation $\e:C\Rightarrow \id_{\cat{C}}$ called \emph{counit};
		\item A natural transformation $\nu:C \Rightarrow CC$ called \emph{comultiplication};
	\end{itemize}
	such that the following diagrams, in components, commute for each object $X$, called ``left and right counitality'' and ``coassociativity'', respectively.
\begin{equation}\label{comonaddiagrams}
\begin{tikzcd}
CX \ar{r}{\nu_X} \ar{dr}[swap]{\id} & CCX \ar{d}{\e_{CX}}  \\
& CX
\end{tikzcd}
\qquad
\begin{tikzcd}
CX \ar{r}{\nu_X} \ar{dr}[swap]{\id} & CCX \ar{d}{C \e_X} \\
& CX
\end{tikzcd}
\qquad
\begin{tikzcd}
CX \ar{r}{\nu_X} \ar{d}{\nu_X} & CCX \ar{d}{\nu_{CX}} \\
CCX \ar{r}{C\nu_X} & CCCX
\end{tikzcd}
\end{equation}

\section{Monads as extensions of spaces}

Here is the first intuitive idea of what a monad can look like. 
\begin{idea} 
	A monad is a consistent way of extending spaces to include generalized elements and generalized functions of a specific kind.
\end{idea}

Let's illustrate this aspect with some examples.

\begin{eg}[sets and relations]\label{powersetmonad}
	Consider the power set functor (\Cref{egpowerset}) on the category $\cat{Set}$ of sets and functions. Given a set $X$, its power set $PX$ can be considered an extension of $X$, where ``subsets generalize elements''. Given sets $X$ and $Y$ and a function $f:X\to Y$, we get the function $Pf:PX\to PY$ given by the image, it maps each $A\in PX$, which is a subset of $X$, to the subset of $Y$ given by the image of $A$ under $f$. Equivalently, it gives the subset of $Y$ obtained by applying $f$ to all elements of $A$.
	One can consider the subsets of $X$ as ``generalized elements'', and $Pf$ as the ``obvious'' extension of $f$ to the generalized elements of $X$ (mapping them to generalized elements of $Y$). 
	
	In which way can we see subsets as generalized elements?
	Strictly speaking, elements of $X$ are not particular subsets, an element is not technically a subset. However, each element $x\in X$ defines a subset canonically: the singleton $\{x\}$. As we say in \Cref{egsingletonmap}, the embedding $X\to PX$ given by singletons is part of a natural transformation $\sigma:\id_{\cat{C}}\Rightarrow P$. Explicitly:
	\begin{enumerate}
		\item For each $X$ we consider the map given by singletons $\sigma_X:X\to PX$, The interpretation is that $PX$, the extension, ``includes'' the old space $X$, or that $PX$ ``extends'' $X$.
		\item For each $f:X\to Y$, the extended function $Pf$ must agree with $f$ on the ``old elements'', i.e.~the elements coming from $X$ via $\sigma$. In other words, this diagram has to commute,
		\begin{equation*}
			\begin{tikzcd}
				X \ar{d}{\sigma_X} \ar{r}{f} & Y \ar{d}{\sigma_Y} \\
				PX \ar{r}{Pf} & PY
			\end{tikzcd}
		\end{equation*}
	\end{enumerate}
	
	Let's now turn to the multiplication map. The set $PPX$ contains \emph{subsets of subsets} of $X$. Given a subset of subsets of $X$, there is a canonical way of obtaining a subset of $X$: via the union. For example, if $x,y,z\in X$, a subset of subsets has the form
	\begin{equation*}
		\mathcal{A} \;=\; \{ \{x,y\}, \{y,z\}, \{\;\} \}  \;\in\; PPX.
	\end{equation*}
	From the element above, we can take the union of the subsets contained in it, which is
	\begin{equation*}
		\bigcup_{A\in \mathcal{A}} A \;=\; \{ x,y,z \} \;\in\; PX.
	\end{equation*}
	We can view $\union$ as the map which ``removes the inner brackets''.
	This gives an assignment $\cup:PPX\to PX$.
	Given a function $f:X\to Y$ the following diagram commutes,
	\begin{equation*}
		\begin{tikzcd}
			PPX \ar{d}{\cup_X} \ar{r}{PPf} & PPY \ar{d}{\cup_Y} \\
			PX \ar{r}{Pf} & PY
		\end{tikzcd}
	\end{equation*} 
	which means that that the union of the images is the image of the union. In symbols: 
	\begin{equation*}
		f \left( \bigcup_{A\in \mathcal{A}} A \right) \;=\;  \bigcup_{A\in \mathcal{A}} f(A) .
	\end{equation*}
	Therefore the map $\cup$ gives a natural transformation $PP\Rightarrow P$. 
	Additional motivation for these natural transformations will be given in \Cref{kleisli}.
	
	Let's now see how $(P,\sigma,\cup)$ is a monad. We have to show that the diagrams \eqref{monaddiagrams} commute. Now, the left unitality diagram
	$$
	\begin{tikzcd}
	PX \ar{r}{\sigma_{PX}} \ar{dr}[swap]{\id} & PPX \ar{d}{\cup_X}  \\
	& PX
	\end{tikzcd}
	$$
	says the following. Take a subset of $X$, say $\{x,y,z\}$ with $x,y,z\in X$, and we form the singleton subset of $PX$ containing just this set, i.e.~$\{\{x,y,z\}\}$ (mind the double brackets: now we have a set of subsets of $X$, containing just one subset). 
	Now take the union of all the sets in this set. There is just one set in there, so we get back $\{x,y,z\}$. In symbols,
	$$
	\begin{tikzcd}
	\{x,y,z\} \ar[mapsto]{r}{\sigma} \ar[mapsto]{dr}[swap]{\id} & \{\{x,y,z\}\} \ar[mapsto]{d}{\cup}  \\
	& \{x,y,z\}.
	\end{tikzcd}
	$$
	The right unitality diagram 
	$$
	\begin{tikzcd}
	PX \ar{r}{P \sigma_X} \ar{dr}[swap]{\id} & PPX \ar{d}{\cup_X} \\
	& PX
	\end{tikzcd}
	$$
	says something similar. Namely, take again a subset of $X$ such as $\{x,y,z\}$ as above. Now take the \emph{image} of the map $\sigma_X$, that is, apply $\sigma$ to each element of $\{x,y,z\}$. This replaces every element with its corresponding singleton set, giving the set of subsets $\{\{x\},\{y\},\{z\}\}$. Taking the union of the sets in this set gives again the original set $\{x,y,z\}$. in symbols,
	$$
	\begin{tikzcd}
	\{x,y,z\} \ar[mapsto]{r}{P\sigma} \ar[mapsto]{dr}[swap]{\id} & \{\{x\},\{y\},\{z\}\} \ar[mapsto]{d}{\cup}  \\
	& \{x,y,z\}.
	\end{tikzcd}
	$$
	The associativity diagram
	$$
	\begin{tikzcd}
	PPPX \ar{r}{P \cup_X} \ar{d}{\cup_{PX}} & PPX \ar{d}{\cup_X} \\
	PPX \ar{r}{\cup_X} & PX
	\end{tikzcd}
	$$
	says as follows. Take a subset of subsets of subsets of $X$ (three times), for example
	$$
	\{\{\{x,y\},\{x,z\}\},\{\{a,b\}\}\} .
	$$
	Then removing the innermost brackets and then the (remaining) innermost brackets has the same result as removing the mid-level brackets and then the (remaining) innermost brackets. In symbols,
	$$
	\begin{tikzcd}
	\{\{\{x,y\},\{x,z\}\},\{\{a,b\}\}\} \ar[mapsto]{r}{P \cup_X} \ar[mapsto]{d}{\cup_{PX}} & \{\{x,y,z\},\{a,b\}\} \ar[mapsto]{d}{\cup_X} \\
	\{\{x,y\},\{x,z\},\{a,b\}\} \ar[mapsto]{r}{\cup_X} & \{x,y,z,a,b\}
	\end{tikzcd}
	$$
	Since the diagrams \eqref{monaddiagrams} commute, the triplet $(P,\sigma,\cup)$ is a monad on $\cat{Set}$.
\end{eg}

\begin{eg}[basic probability]\label{probmonad}
	The probability functors given in \Cref{egprob}, \Cref{giryfunctor} and \Cref{radonfunctor} can be also given a monad structure, in a way that's similar to the power set case. Let's do it for the set case, the other cases are left as exercises (see below).
	
	Recall that the probability functor $\mathcal{P}$ of  \Cref{egprob} takes a set $X$ and returns the set $X$ of finitely supported probability measures on $X$. 
	Every element $x\in X$ gives a ``deterministic'' probability measure, namely $\delta_x$ (see \Cref{setdelta}). There is a natural embedding $\delta:X\to \mathcal{P}X$ which we can interpret as follows. Considering a probability measure $p\in \mathcal{P}X$ as a ``random point'' of $X$, then $\mathcal{P}X$ is the extension of $X$ to account for ``random points'' too, with the old deterministic points included via the unit map $\delta:X\to \mathcal{P}X$ (compare with the power set case).
	
	The multiplication map $E:\mathcal{PPX}\to\mathcal{PX}$ is given as follows. For $\pi\in PPX$, 
	$$
	E(\pi) (x) \;\coloneqq\; \sum_{p\in \mathcal{P}X} \pi(p)\, p(x) .  
	$$
	Here is a way to interpret this map. Suppose that you have two coins in your pocket. Suppose that one coin is fair, with ``heads'' on one face and ``tails'' on the other face; suppose the second coin has ``heads'' on both sides. Suppose now that you draw a coin randomly, and flip it.  We can sketch the probabilities in the following way:
 \begin{equation*}
  \begin{tikzcd}[column sep=tiny]
   &&& ? \ar{dll}[swap]{1/2} \ar{drr}{1/2} \\
   & \mbox{coin 1} \ar{dl}[swap]{1/2} \ar{dr}{1/2} &&&& \mbox{coin 2} \ar{dl}[swap]{1} \ar{dr}{0} \\
   \mbox{heads} && \mbox{tails} && \mbox{heads} && \mbox{tails}
  \end{tikzcd}
 \end{equation*}
 Let $X$ be the set $\{\mbox{``heads''},\mbox{``tails''}\}$. A coin gives a \emph{law} according to which we will obtain ``heads'' or ``tails'', so it determines an element of $\mathcal{P}X$. Since the choice of coin is also random (we also have a \emph{law on the coins}), the law on the coins determines an element of $\mathcal{P}\mathcal{P}X$.
 By averaging, the resulting overall probabilities are 
 \begin{equation*}
  \begin{tikzcd}[column sep=tiny]
   & ? \ar{dl}[swap]{3/4} \ar{dr}{1/4} \\
   \mbox{heads} && \mbox{tails} 
  \end{tikzcd}
 \end{equation*}
 In other words, the ``average'' or ``composition'' can be thought of as an assignment $E:\mathcal{P}\mathcal{P}X\to \mathcal{P}X$, from laws of ``random random variables'' to laws of ordinary random variables.
 
 To show that the map $E$ is natural, let $f:X\to Y$ be a function. Since the map $\mathcal{P}f:\mathcal{P}X\to\mathcal{P}Y$ is given by the pushforward of measures, which is usually written as $f_*$, denote for brevity $\mathcal{P}f$ by $f_*$. Then the following diagram commutes.
 $$
 \begin{tikzcd}
  \mathcal{PP}X \ar{d}{E} \ar{r}{(f_*)_*} & \mathcal{PP}Y \ar{d}{E} \\
  \mathcal{P}X \ar{r}{f_*} & \mathcal{P}Y .
 \end{tikzcd}
 $$
 Indeed, for each $\pi\in \mathcal{PP}X$ and $y\in Y$, 
 \begin{align*}
 E((f_*)_*(\pi))(y) \;&=\; \sum_{q\in \mathcal{P}Y}  (f_*)_*(\pi) (q)  \,q(y) \\
 &=\; \sum_{q\in \mathcal{P}Y}  \sum_{p\in (f_*)^{-1}(q)} \pi (p)  \,q(y) \\
 &=\; \sum_{p\in \mathcal{P}X} \pi (p)\, f_*(p)(y) \\
 &=\; \sum_{x\in f^{-1}(y)} \sum_{p\in \mathcal{P}X} \pi(p)\,p(x) \\
 &=\; \sum_{x\in f^{-1}(y)} E\pi (x) \\
 &=\; f_* (E\pi) (y) .
 \end{align*}
 The next exercise shows that $\delta$ and $E$ satisfy the monad axioms. Therefore $(\mathcal{P},\delta,E)$ is a monad on $\cat{Set}$. It is known in the literature under many names, such as \emph{distribution monad} or \emph{convex combination monad}. 
\end{eg}

\begin{ex}[basic probability]
 Show that $(\mathcal{P},\delta,E)$ is a monad, that is, that the following diagrams commute.
$$
\begin{tikzcd}
\mathcal{P}X \ar{r}{\delta_{\mathcal{P}X}} \ar{dr}[swap]{\id} & \mathcal{P}^2X \ar{d}{E_X}  \\
& \mathcal{P}X
\end{tikzcd}
\qquad
\begin{tikzcd}
\mathcal{P}X \ar{r}{\mathcal{P} \delta_X} \ar{dr}[swap]{\id} & \mathcal{P}^2X \ar{d}{E_X} \\
& \mathcal{P}X
\end{tikzcd}
\qquad
\begin{tikzcd}
\mathcal{P}^3X \ar{r}{\mathcal{P} E_X} \ar{d}{E_{\mathcal{P}X}} & \mathcal{P}^2X \ar{d}{E_X} \\
\mathcal{P}^2X \ar{r}{E_X} & \mathcal{P}X
\end{tikzcd}
$$
\end{ex}

The following exercises give the analogous construction for more general probability measures (using measure theory).

\begin{ex}[measure theory, probability]\label{girymonad}
	Consider the Giry functor defined on $\cat{Meas}$ in \Cref{giryfunctor}. 
	Define a monad structure analogous to the one in the exercise above, with unit $\delta:X\to \mathcal{P}X$ given by Dirac measures (see \Cref{measdelta}), and multiplication $E:\mathcal{PP}X\to \mathcal{P}X$ given by integration as follows,
	$$
	E(\pi) (A) \;\coloneqq\; \int_{\mathcal{P}X} p(A) \,d\pi(p) 
	$$
	for each measurable $A\subseteq X$.
	Note that you have to prove that $\delta$ and $E$ are measurable and natural, and that they satisfy the diagrams \eqref{monaddiagrams}. 
    (Hint: to prove that $E$ is measurable, first prove that for every measurable function $f:X\to[0,1]$ the ``integration'' map $i_f:\mathcal{P}X\to [0,1]$ given by
    $$
    p \;\longmapsto\; \int f\, dp
    $$
    is measurable as a function of $p$, for the $\sigma$-algebra defined on $\mathcal{P}X$ in \Cref{giryfunctor}.)
	
	The resulting monad is known as the the \emph{Giry monad}. 
\end{ex}

\begin{ex}[measure theory, probability]\label{radonmonad}
	Consider the Radon functor on $\cat{CHaus}$ defined in \Cref{radonfunctor}. Define a monad structure in a way analogous to the exercises above.
	
	The resulting monad is known as the \emph{Radon monad}.
\end{ex}

Monads are widely used in computer science. Here is a basic example, given by the list construction.

\begin{ex}[basic computer science, combinatorics]\label{listmonad}
Consider the list functor of \Cref{listfunctor}. Equip it with a monad structure, with the unit given by the map of \Cref{singlelist}, and the multiplication given by flattening a double list (or equivalently, concatenating lists). 

The ``extension'' interpretation is that the set of lists $LX$ extends the old set $X$, which we can view as having as elements only lists of length one.
\end{ex}

The following monad, in different forms, appears almost everywhere in mathematics. It is used in applied mathematics and computer science to model processes that have ``extra costs'', or produce ``additional output'', and in pure mathematics to talk about fiber bundles and group actions. 
 
\begin{eg}[several fields]\label{writermonad}
 Let $M$ be a monoid or a group, let's write it additively (denote its neutral element by $0$, and its binary operation by $+$).  
 Given any set $X$, forming the set $X\times M$ is a functorial operation, where a function $f:X\to Y$ is mapped to $f\times\id_M:X\times M\to Y\times M$. That is, the pair $(x,m)\in X\times M$ is mapped to $(f(x),m)\in Y\times M$ (see \Cref{productfunctorial}).
 
 This functor admits a canonical monad structure, inherited by the monoid structure of $M$. 
 The unit is given by the map 
 $$
 \begin{tikzcd}[row sep=0]
  X \ar{r}{\eta} & X\times M \\
  x \ar[mapsto]{r} & (x,0) ,
 \end{tikzcd}
 $$
 so it is defined using the neutral element of $M$.
 The multiplication is given by the map 
 $$
 \begin{tikzcd}[row sep=0]
  X\times M\times M \ar{r}{\mu} & X\times M \\
  (x,m,n) \ar[mapsto]{r} & (x,m + n) ,
 \end{tikzcd}
 $$
 so it is defined using the multiplication of $M$.
 (Why are these maps natural?)
 
 We can interpret the elements of $M$ as being ``costs'' or ``side effects'' of some kind. If we have two of them, we can multiply them (or sum them) using the multiplication map of the monoid. The neutral element $0$ correspond to zero cost. The ``extension interpretation'' is as follows: the elements $(x,m)$ can be seen as more general than the $x$, if we see the latter as $(x,0)$ (via the unit map). We are extending $X$ to account for elements which have a ``cost'' or ``impact'' or ``side effect''.
 
 Note that, since $M$ is a monoid (is associative and unital), the monad axioms \eqref{monaddiagrams} hold automatically, as they correspond to associativity and unitality of $M$ (this example is one of the reasons why they are called that way). 
 The left unitality diagram says that for each $(x,m)\in X\times M$, $\mu\circ\eta(x,m)=(x,m)$, which explicitly says $(x,m+0)=(x,m)$. This is true, since $m+0=m$ (since $M$ is a monoid, and so it's unital). Analogously, right unitality boils down to saying $0+m=m$, and associativity corresponds exactly to associativity of $M$. Therefore we have a monad on $\cat{Set}$. Let's denote this monad by $(T_M,\eta,\mu)$. 
 
 This monad is known in computer science as the \emph{writer monad}, for reasons that will be explained in \Cref{writerkleisli}. In pure mathematics it is either known as the \emph{trivial bundle monad} (since $X\times M$ is the trivial $M$-bundle over $X$) or as the $M$-action monad, for reasons that will be explained in \Cref{actionmonad}.
\end{eg}

\subsection{Kleisli morphisms}\label{kleisli}

Given a monad $T$, we can not only talk about generalized elements (of an extended space), but also, and mostly, of generalized \emph{functions}. Given spaces $X$ and $Y$, we can form functions which intuitively have as output generalized elements of $Y$. That is, functions $k:X\to TY$.

\begin{deph}
	Let $(T,\eta,\mu)$ be a monad on a category $\cat{C}$. A \emph{Kleisli morphism} of $T$ from $X$ to $Y$ is a morphism $k:X\to TY$ of $\cat{C}$.
\end{deph}

In mathematics it often happens that one would like to obtain a function from $X$ to $Y$ from some construction (for example, a limit), but the result is not always well-defined or unique, or not always existing. Maybe the output of a process is subject to fluctuations, so that a probabilistic output models the situation better. Or, there could be extra side effects...and so on. Allowing more general results or outputs is sometimes a better description of the system, that is, replacing $Y$ with the extension $TY$. Generalized functions include ordinary functions in the same way as generalized elements include ordinary elements: via the map $\eta$. A function $f:X\to Y$ uniquely defines a map $X\to TY$ given by $\eta\circ f$. Note that this is \emph{different} from extending an existing $f:X\to Y$ to $TX$: we are not extending an existing function to generalized elements, we are allowing more general functions on $X$ which \emph{take values} in elements of $TY$ which may not come from $Y$.
In particular, a generalized element can be seen as a constant generalized function.

\begin{eg}[sets and relations]
In the case of the power set, Kleisli morphisms, or generalized maps, are precisely \emph{relations}: given sets $X$ and $Y$, a map $k:X\to PY$ assigns to each element of $X$ a \emph{subset} of $Y$ (possibly empty), i.e.~it is a multi-valued (possibly no-valued) function.
\end{eg}

\begin{eg}[probability]
 A Kleisli morphism for the distribution monad $\mathcal{P}$ of \Cref{probmonad} is a function $X\to \mathcal{P}Y$, which we can view as a \emph{stochastic map} $X\to Y$.
 
 Analogously, a Kleisli morphism for the Giry monad of \Cref{girymonad} is a \emph{Markov kernel}: it consists precisely of a measurable assignment $X\to \mathcal{P}Y$. 
\end{eg}

Kleisli morphisms, according to the definition above, are just ordinary morphisms of a particular form. The reason why they are important enough to deserve their own name is because of how they \emph{compose}. As we will see in the examples, this gives a very fruitful construction.

\begin{deph}
	Let $(T,\eta,\mu)$ be a monad on a category $\cat{C}$. Let $k:X\to TY$ and $h: Y\to TZ$. We define the \emph{Kleisli composition} of $k$ and $h$ to be the morphism $(h\circ_{kl}k):X\to TZ$ given by:
	\begin{equation}\label{kleislicomp}
	\begin{tikzcd}
	X \ar{r}{k} & TY \ar{r}{Th} & TTZ \ar{r}{\mu} & TZ .
	\end{tikzcd}
	\end{equation}
\end{deph}
In other words, the Kleisli composition permits to compose generalized functions from $X$ to $Y$ with generalized functions from $Y$ to $Z$ to give generalized functions from $X$ to $Z$.

\begin{eg}[sets and relations]
	Let's see what happens for the power set monad. Relations can be composed as follows. Given relations $k:X\to PY$ and $h:Y\to PZ$, as in the following picture
	
	\begin{center}
		
		\begin{tikzpicture}[baseline=(current  bounding  box.center),>=stealth]
		\foreach \y in {1,2,3,4}
		\node[bullet,label=left:$x_{\y}$] (x\y) at (0,-\y) {};
		
		\foreach \y in {1,2,3,4}
		\node[bullet,label=below:$y_{\y}$] (y\y) at (4,-\y) {};
		
		\foreach \y in {1,2,3,4}
		\node[bullet,label=right:$z_{\y}$] (z\y) at (8,-\y) {};
		
		\node[draw,fit=(x1) (x4),minimum width=2.5cm,label=below:$X$] {} ;
		\node[draw,fit=(y1) (y4),minimum width=2.5cm,label=below:$Y$] {} ;
		\node[draw,fit=(z1) (z4),minimum width=2.5cm,label=below:$Z$] {} ;
		
		\draw[relation] (x2) -- (y1);
		\draw[relation] (x2) -- (y2);
		\draw[relation] (x3) -- (y3);
		\draw[relation] (x3) -- (y4);
		
		\draw[relation] (y1) -- (z1);
		\draw[relation] (y1) -- (z2);
		\draw[relation] (y2) -- (z2);
		\draw[relation] (y2) -- (z3);
		\draw[relation] (y4) -- (z4);
		\end{tikzpicture}
		
	\end{center}
	
	we can compose the two and forget about $Y$, obtaining a relation $X\to PZ$
	
	\begin{center}
		
		\begin{tikzpicture}[baseline=(current  bounding  box.center),>=stealth]
		\foreach \y in {1,2,3,4}
		\node[bullet,label=left:$x_{\y}$] (x\y) at (0,-\y) {};
		
		\foreach \y in {1,2,3,4}
		\node[bullet,label=right:$z_{\y}$] (z\y) at (8,-\y) {};
		
		\node[draw,fit=(x1) (x4),minimum width=2.5cm,label=below:$X$] {} ;
		\node[draw,fit=(z1) (z4),minimum width=2.5cm,label=below:$Z$] {} ;
		
		\draw[relation] (x2) -- (z1);
		\draw[relation] (x2) -- (z2);
		\draw[relation] (x2) -- (z3);
		\draw[relation] (x3) -- (z4);
		\end{tikzpicture}
		
	\end{center}
	The interpretation is that we can go from $x\in X$ to $z\in Z$ if and only there is a $y\in Y$ such that we can go from $x$ to $y$ and from $y$ to $z$.
	
	What happened formally, though, is that we have first applied $k:X\to PY$, which assigns to each $x\in X$ a subset of $Y$.
	
	\begin{center}
		
		\begin{tikzpicture}[baseline=(current  bounding  box.center),>=stealth]
		\foreach \y in {1,2,3,4}
		\node[bullet,label=left:$x_{\y}$] (x\y) at (0,-\y) {};
		
		\foreach \y in {1,2,3,4}
		\node[bullet,label=below:$y_{\y}$] (y\y) at (4,-\y) {};
		
		\foreach \y in {1,2,3,4}
		\node[bullet,label=right:$z_{\y}$] (z\y) at (8,-\y) {};
		
		\node[draw,fit=(x1) (x4),minimum width=2.5cm,label=below:$X$] {} ;
		\node[draw,fit=(y1) (y4),minimum width=2.5cm,label=below:$Y$] {} ;
		\node[draw,fit=(z1) (z4),minimum width=2.5cm,label=below:$Z$] {} ;
		
		\node[draw,fit=(y1) (y2),minimum width=1cm,dashed] (y12) {} ;
		\node[draw,fit=(y3) (y4),minimum width=1cm,dashed] (y34) {} ;

		\draw[function] (x2) -- (y12);
		\draw[function] (x3) -- (y34);
		
		\node[label=$k$] at (2,-3) {};
		\end{tikzpicture}
		
	\end{center}
	
	Then we have applied $h$ to \emph{elementwise to each subset in the image of $k$}.
	
	\begin{center}
		
		\begin{tikzpicture}[baseline=(current  bounding  box.center),>=stealth]
		\foreach \y in {1,2,3,4}
		\node[bullet,label=left:$x_{\y}$] (x\y) at (0,-\y) {};
		
		\foreach \y in {1,2,3,4}
		\node[bullet,label=below:$y_{\y}$] (y\y) at (4,-\y) {};
		
		\foreach \y in {1,2,3,4}
		\node[bullet,label=right:$z_{\y}$] (z\y) at (8,-\y) {};
		
		\node[draw,fit=(x1) (x4),minimum width=2.5cm,label=below:$X$] {} ;
		\node[draw,fit=(y1) (y4),minimum width=2.5cm,label=below:$Y$] {} ;
		\node[draw,fit=(z1) (z4),minimum width=2.5cm,label=below:$Z$] {} ;
		
		\node[draw,fit=(y1) (y2),minimum width=1cm,dashed] (y12) {} ;
		\node[draw,fit=(y3) (y4),minimum width=1cm,dashed] (y34) {} ;
		
		\node[draw,fit=(z1) (z2),minimum width=1cm,dashed] (z12) {} ;
		\node[draw,fit=(z2) (z3),minimum width=1cm,dashed] (z23) {} ;
		\node[draw,fit=(z4),minimum width=0.8cm,minimum height=0.8cm,dashed] (z44) {} ;
		
		\draw[function] (x2) -- (y12);
		\draw[function] (x3) -- (y34);
		
		\draw[function] (y1) -- (z12);
		\draw[function] (y2) -- (z23);
		\draw[function] (y4) -- (z44);
		
		\node[label=$k$] at (2,-3) {};
		\node[label=$h$] at (6,-3.5) {};
		\end{tikzpicture}
		
	\end{center}
	
	In other words, we have taken the \emph{image} of $h:Y\to PZ$, which we know is the map $Ph:PY\to PPZ$. 
	Technically, to each subset of $Y$ we have a \emph{subset of subsets} of $Z$, which contains the images of $h$.
	
	\begin{center}
		
		\begin{tikzpicture}[baseline=(current  bounding  box.center),>=stealth]
		\foreach \y in {1,2,3,4}
		\node[bullet,label=left:$x_{\y}$] (x\y) at (0,-\y) {};
		
		\foreach \y in {1,2,3,4}
		\node[bullet,label=below:$y_{\y}$] (y\y) at (4,-\y) {};
		
		\foreach \y in {1,2,3,4}
		\node[bullet,label=right:$z_{\y}$] (z\y) at (8,-\y) {};
		
		\node[draw,fit=(x1) (x4),minimum width=2.5cm,label=below:$X$] {} ;
		\node[draw,fit=(y1) (y4),minimum width=2.5cm,label=below:$Y$] {} ;
		\node[draw,fit=(z1) (z4),minimum width=2.5cm,label=below:$Z$] {} ;
		
		\node[draw,fit=(y1) (y2),minimum width=1cm,dashed] (y12) {} ;
		\node[draw,fit=(y3) (y4),minimum width=1cm,dashed] (y34) {} ;
		
		\node[draw,fit=(z1) (z2),minimum width=1cm,dashed] (z12) {} ;
		\node[draw,fit=(z2) (z3),minimum width=1cm,dashed] (z23) {} ;
		\node[draw,fit=(z4),minimum width=0.8cm,minimum height=0.8cm,dashed] (z44) {} ;
		
		\node[draw,fit=(z1) (z3),minimum width=1.5cm,dotted] (z123) {} ;
		\node[draw,fit=(z4),minimum width=1.1cm,minimum height=1.1cm,dotted] (z444) {} ;
		
		\draw[function] (x2) -- (y12);
		\draw[function] (x3) -- (y34);
		
		\draw[function] (y12) -- (z123);
		\draw[function] (y34) -- (z444);
		
		\node[label=$k$] at (2,-3) {};
		\node[label=$Ph$] at (6,-3.2) {};
		\end{tikzpicture}
		
	\end{center}
	
	Now for each subset of $Y$, we take the \emph{union} of the subsets in its image.
	
	\begin{center}
		
		\begin{tikzpicture}[baseline=(current  bounding  box.center),>=stealth]
		\foreach \y in {1,2,3,4}
		\node[bullet,label=left:$x_{\y}$] (x\y) at (0,-\y) {};
		
		\foreach \y in {1,2,3,4}
		\node[bullet,label=below:$y_{\y}$] (y\y) at (4,-\y) {};
		
		\foreach \y in {1,2,3,4}
		\node[bullet,label=right:$z_{\y}$] (z\y) at (8,-\y) {};
		
		\node[draw,fit=(x1) (x4),minimum width=2.5cm,label=below:$X$] {} ;
		\node[draw,fit=(y1) (y4),minimum width=2.5cm,label=below:$Y$] {} ;
		\node[draw,fit=(z1) (z4),minimum width=2.5cm,label=below:$Z$] {} ;
		
		\node[draw,fit=(y1) (y2),minimum width=1cm,dashed] (y12) {} ;
		\node[draw,fit=(y3) (y4),minimum width=1cm,dashed] (y34) {} ;
		
		\node[draw,fit=(z4),minimum width=0.8cm,minimum height=0.8cm,dashed] (z44) {} ;
		
		\node[draw,fit=(z1) (z3),minimum width=1cm,dashed] (z123) {} ;
		
		\draw[function] (x2) -- (y12);
		\draw[function] (x3) -- (y34);
		
		\draw[function] (y12) -- (z123);
		\draw[function] (y34) -- (z44);
		
		\node[label=$k$] at (2,-3) {};
		\node[label=$\cup\circ Ph$] at (6,-3.2) {};
		\end{tikzpicture}
		
	\end{center}
	
	thereby obtaining the composite relation $X\to PZ$:
	
	\begin{center}
		
		\begin{tikzpicture}[baseline=(current  bounding  box.center),>=stealth]
		\foreach \y in {1,2,3,4}
		\node[bullet,label=left:$x_{\y}$] (x\y) at (0,-\y) {};
		
		\foreach \y in {1,2,3,4}
		\node[bullet,label=right:$z_{\y}$] (z\y) at (8,-\y) {};
		
		\node[draw,fit=(x1) (x4),minimum width=2.5cm,label=below:$X$] {} ;
		\node[draw,fit=(z1) (z4),minimum width=2.5cm,label=below:$Z$] {} ;
		
		\node[draw,fit=(z4),minimum width=0.8cm,minimum height=0.8cm,dashed] (z44) {} ;
		
		\node[draw,fit=(z1) (z3),minimum width=1cm,dashed] (z123) {} ;
		
		\draw[function] (x2) -- (z123);
		\draw[function] (x3) -- (z44);
		
		\node[label=$\cup\circ (Ph) \circ k$] at (4,-3.2) {};
		\end{tikzpicture}
		
	\end{center}
\end{eg}

\begin{eg}[probability]
 Kleisli composition for the distribution monad works as follows. Given sets $X$, $Y$ and $Z$, and maps $k:X\to \mathcal{P}Y$ and $h:Y\to \mathcal{P}Z$,
 the composition~\eqref{kleislicomp} gives us
 \begin{equation*}
 h\circ_{kl}k\;=\; E \circ \mathcal{P}h \circ k
\end{equation*}
which for each $x\in X$ maps $z\in Z$ to
$$
 (h\circ_{kl}k)_x(z) = \sum_{q\in\mathcal{P}Z} q(z) \, h_* k(x)(q) = \sum_{y\in Y} h(y)(z) \, k(x)(y) .
$$
 In other words, interpreting the stochastic maps as conditionals (i.e.~writing $k(x)(y)$ as $p(y|x)$), we are composing conditionals by summing over all the intermediate states:
 \begin{equation}\label{chapman}
 p(z|x) \;=\; \sum_{y\in Y} p(z|y)\,p(y|x) .
 \end{equation}
 This is the famous \emph{Chapman-Kolmogorov formula}.

 Let's see now what happens more generally with the Giry monad. Given Markov kernels $k:X\to \mathcal{P}Y$ and $h:Y\to \mathcal{P}Z$, the composition formula gives us for each $x\in X$ and measurable set $A\subseteq Z$,
$$
 (h\circ_{kl}k)(x)(A) = \int_{\mathcal{P}Z} q(A) \, d(h_* k(x))(q) = \int_{Y} h(y)(A) \, dk(x)(y) .
$$
Again, in conditional notation,
$$
p(A|x) \;=\; \int_{Y} p(A|y)\;dp(y|x) ,
$$
which is analogous to \Cref{chapman} if one replaces sums by integrals.

The Kleisli composition for probability monads is exactly the usual composition of stochastic maps and of Markov kernels.
\end{eg}

As the previous two examples show, Kleisli morphisms capture known structures in mathematics (in this case, relations and stochastic maps) not just in how they look, but also in how they behave.

Let's now look at the Kleisli morphisms of the \emph{writer monad} of \Cref{writermonad}, and see why it is called that way.

\begin{eg}[several fields]\label{writerkleisli}
 Let $M$ be a monoid, which we again write additively. A Kleisli morphism of the ``writer monad'' $T_M$ is a map $k:X\to Y\times M$. We can interpret it as a process which, when given an input $x\in X$, does not just produce an output $y\in Y$, but also an element of $M$. 
 For example, it could be energy released by a chemical reaction, or waste, or a cost of the transaction. In computer science, this is the behaviour of a function that computes a certain value, but that also writes into a log file (or to the standard output) that something has happened (the monoid operation being the concatenation of strings). For example, when you compile a \LaTeX~document, a log file is produced alongside your output file. Hence the name ``writer monad''.
 
 Let's now look at the Kleisli composition. If we have processes $k:X\to Y\times M$ and $h:Y\to Z\times M$, then $h\circ_{kl} k:X\to Z\times M$ is given by
 $$
 \begin{tikzcd}[column sep=large, row sep=0]
  X \ar{r}{k} & Y\times M \ar{r}{h\times \id_M} & Z\times M \times M \ar{r}{\id_Z\times +} & Z\times M \\
  x \ar[mapsto]{r} & (y,m) \ar[mapsto]{r} & (z,n,m) \ar[mapsto]{r} & (z, n+m) .
 \end{tikzcd}
 $$
 What it does is as follows:
 \begin{enumerate}
  \item It executes the process $k$ with an input $x\in X$, giving as output an element of $y\in Y$ as well as a cost (or extra output) $m\in M$.
  \item It executes the process $h$ taking as input the $y\in Y$ produced by $k$, giving an element $z\in Z$ as well as an extra cost $n\in M$. (All of this while keeping track of the first cost $m$.)
  \item The two costs $m$ and $n$ are summed (or the extra outputs are concatenated). 
 \end{enumerate}

 So, for example, the cost of executing two processes one after another is the sum of the costs. The same is true about the release of energy in a chemical reaction, and about waste. 
 Just as well, executing two programs one after another will produce a concatenation of text in a log file (or two log files).
 
 When you compile a \LaTeX~file, after compilation your compiler tells you how many errors (and warnings, and badboxes) happened during the process. This can be modelled in terms of Kleisli morphisms. Each step in the compilation consists of a Kleisli morphism taking an input (for example a keyword) and giving an output (for example a mathematical symbol) together with a number, the number of errors encountered. The compilation consists of several steps that pass information to one another, and the final number of errors will be the sum of the errors encountered during all the steps (hopefully zero). 
\end{eg}

Again, Kleisli composition is the ``right'' notion of composition for these functions with generalized output. 

The example of the writer monad (\Cref{writermonad}) motivated the usage of the terms ``unit'', ``multiplication'', ``associativity'' and ``unitality'' for monads. Here is another, almost independent reason for those names. 

\begin{prop}
	Let $(T,\eta,\mu)$ be a monad on a category $\cat{C}$. The Kleisli morphisms of $T$ form themselves a category, where
	\begin{itemize}
		\item The objects are the objects of $\cat{C}$;
		\item The morphisms are the Kleisli morphisms of $T$;
		\item The identities are given by the units $\eta:X\to TX$ for each object $X$;
		\item The composition is given by the Kleisli composition. 
	\end{itemize}
\end{prop}

In other words, the Kleisli morphisms form themselves a category, which we can think of as ``having as morphisms the generalized maps''. 

\begin{deph} The category defined above is called the \emph{Kleisli category} of $T$, and it is denoted by $\cat{C}_T$.
\end{deph}

\begin{proof}
	In order for $\cat{C}_T$ to be a category we need the identities (i.e.~the unit maps) to behave indeed like identities, and the composition (i.e.~the Kleisli composition) to be associative. 
	\begin{itemize}
		\item The left unitality condition for $T$, together with the naturality of $\eta$, for each $k:X\to TY$, gives a commutative diagram
		\begin{equation*}
			\begin{tikzcd}
				& TX \ar{r}{Tk} & TTY \ar{d}{\mu} \\
				X \ar{r}{k} \ar{ur}{\eta} & TY \ar{r}{\id} \ar{ur}{\eta} & TY
			\end{tikzcd}
		\end{equation*}
		which means that $\eta \circ_{kl} k = k$, i.e.~$\eta$ behaves like an identity (on the left side) for the Kleisli composition. Hence the name ``left unitality''.
		\item The right unitality condition, for each $k:X\to TY$, gives a commutative diagram
		\begin{equation*}
			\begin{tikzcd}
				X \ar{r}{k} & TY \ar{r}{T\eta} \ar[swap]{dr}{\id} & TTY \ar{d}{\mu} \\
				& & TY
			\end{tikzcd}
		\end{equation*}
		which means that $k \circ_{kl} \eta = k$, i.e.~$\eta$ behaves like an identity (on the right) for the Kleisli composition (hence the name ``right unitality''). So the maps $\eta$ are indeed the identities of $\cat{C}_T$.
		\item The associativity square, together with naturality of $\mu$, gives for each $\ell:W\to TX$, $k:X\to TY$, and $h:Y\to TZ$ a commutative diagram
		\begin{equation*}
			\begin{tikzcd}
				W \ar{r}{\ell} & TX \ar{r}{Tk} & TTY \ar{r}{TTh} \ar{d}{\mu} & TTTZ \ar{r}{T\mu} \ar{d}{\mu} & TTZ \ar{d}{\mu} \\
				& & TY \ar{r}{Th} & TTZ \ar{r}{\mu} & TZ
			\end{tikzcd}
		\end{equation*}
		which means that $h \circ_{kl} (k \circ_{kl} l) = (h \circ_{kl} k) \circ_{kl} l$ (why?), i.e.~the Kleisli composition is associative. Hence the name ``associativity square''. \qedhere
	\end{itemize}
\end{proof}

\begin{eg}[sets and relations]
 The Kleisli category of the power set monad is the category of sets and \emph{relations}. The identity is given by the singleton map $\sigma:X\to PX$. As said before, the interpretation is that the power set ``forms spaces of generalized elements in a consistent way'', and that the associated ``consistent generalization'' of functions is relations. 
\end{eg}

\begin{eg}[probability]
 The Kleisli category of the distribution monad is the category of sets and \emph{stochastic maps}. The identity is given by the delta map $\delta:X\to \mathcal{P}X$. Once again, the interpretation is that the probability construction ``forms spaces of generalized elements in a consistent way'' (which we can see as ``random elements''), and that the associated ``consistent generalization'' of functions is stochastic maps.
 
 For the Giry monad, the Kleisli category is the category of measurable spaces and Markov kernels.
\end{eg}

\begin{ex}[basic computer science]
 What is the Kleisli category of the list monad?
\end{ex}

\begin{ex}[several fields]
 How does the Kleisli category of the writer monad look, explicitly?
\end{ex}

\subsection{The Kleisli adjunction}

Let $(T,\eta,\mu)$ be a monad on a category $\cat{C}$.
As we said already, every ordinary morphism $f:X\to Y$ defines ``trivially'' a Kleisli morphism via
$$
\begin{tikzcd}
X \ar{r}{f} & Y \ar{r}{\eta} & TY .
\end{tikzcd}
$$
For example, every function defines in particular a (very special) relation, given by ``$x$ is related to $y$ if and only if $y=f(x)$''. Analogously, every function defines a \emph{deterministic} stochastic map, given by ``$y=f(x)$ with probability one''.

This assignment is actually a functor $\cat{C}\to\cat{C}_T$, which we denote by $L_T$, as the next exercise shows.

\begin{ex}[important!]
	Prove that $L_T$ is indeed a functor, where on objects, $L_T(X)=X$, and on morphisms, $L_T(f)=\eta\circ f$. 
	
	It has to be proven that it preserves identities (i.e.~$L_T(\id_X)=\eta_X$ for each object $X$) and composition (i.e.~$L_T(f\circ g) = L_T(f)\circ_{kl}L_T(g)$). 
\end{ex}

In our power set example, this says in particular that if we consider functions as special relations, then composing them as functions or as relations gives the same result. The same is true for stochastic maps.

Conversely, suppose we have a Kleisli morphism $k:X\to TY$. In general, this does not come from a map $X\to Y$, and cannot be used to define such a map. However, we can obtain from $k$ a map between the \emph{extended} spaces, i.e.~$TX\to TY$, as follows.
$$
\begin{tikzcd}
 TX \ar{r}{Tk} & TTY \ar{r}{\mu} & TY .
\end{tikzcd}
$$
Denote this map by $R_T(k)$. This is again a functorial assignment. That is, the assignment $X\mapsto TX$, $k\mapsto \mu\circ Tk$ is a functor $\cat{C}_T\to\cat{C}$.

\begin{ex}[important!]
	Prove that $R_T$ is indeed a functor. This means that it preserves identities (i.e.~$R_T(\eta_X)=\id_{TX}$ for each object $X$) and composition (i.e.~$R_T(k\circ_{kl} h) = R_T(k)\circ R_T(h)$). 
\end{ex}

\begin{eg}[sets and relations]\label{fca}
	Let $k:X\to PY$ be a Kleisli morphism of the power set monad, i.e.~a relation between $X$ and $Y$. We can assign to a subset $S\subseteq X$ the subset of all the points of $Y$ which are related to at least one element of $S$. This gives the set 
	$$
	\bigcup_{x\in S} k(x) \;=\; \cup\circ Pk (S).
	$$
	If we do this for each $S\in PX$, we have then a map $PX\to PY$.
\end{eg}

\begin{eg}[probability]\label{stomat}
	Let $k:X\to \mathcal{P}Y$ be a Markov kernel (or a stochastic map, for the readers that are unfamiliar with measure theory). Given $p\in\mathcal{P}X$ we can form the measure on $Y$
	$$
	A \;\longmapsto\; \int_X k(x)(A) \, dp(x) .
	$$
	Doing this for each $p\in \mathcal{P}X$ gives the desired assignment $\mathcal{P}X\to \mathcal{P}Y$. This is sometimes called the \emph{pushforward of measures over a Markov kernel}, but for our formalism this name is slightly inaccurate: it corresponds precisely to the pushforward over the map $k:X\to \mathcal{P}Y$, \emph{followed by integration}. In symbols,
	$$
	\begin{tikzcd}
     \mathcal{P}X \ar{r}{\mathcal{P}k} & \mathcal{P}\mathcal{P}Y \ar{r}{\mu} & \mathcal{P}Y .
    \end{tikzcd}
	$$
\end{eg}

\begin{prop}\label{kleisliadj}
 Let $(T,\eta,\mu)$ be a monad on $\cat{C}$. 
 \begin{enumerate}
  \item The composite functor $R_T\circ L_T:\cat{C}\to\cat{C}$ is naturally isomorphic to $T$;
  \item The functor $L_T$ is left-adjoint to $R_T$;
  \item The unit of the adjunction is given by the unit of the monad $\eta$.
 \end{enumerate}
\end{prop}

We call the adjunction above the \emph{Kleisli adjunction}. 
This is one of the many ways in which monads and adjunctions are related, more will come in the next sections.

\begin{proof}
First of all, let $X$ and $Y$ be of $\cat{C}$, but consider $Y$ as an object of $\cat{C}_T$. Note that $L_T(X)=X$, and $R_T(Y)=TY$, from the way we constructed the functors. 

\begin{enumerate}
 \item Set now $Y=L_T(X)$. We get $R_T(L_T(X))=R_T(X)=TX$. 
Let now $f:X\to X'$ be a morphism of $\cat{C}$. We have that 
$$
R_T(L_T(f)) \;=\; R_T(\eta\circ f) \;=\; \mu\circ T(\eta\circ f) \;=\; \mu\circ T\eta\circ Tf.
$$
By the right unitality triangle of \eqref{monaddiagrams}, $\mu\circ T\eta$ is the identity, and so we are left with $Tf$. Therefore $R_T\circ L_T=T$.

\item Let's now turn to the adjunction. 
There is a natural bijection between functions $X\to TY$ and Kleisli morphisms from $X$ to $Y$, given by
$$
\begin{tikzcd}[row sep=0]
 \Hom_\cat{C} (X,TY) \ar{r}{\cong} & \Hom_{\cat{C}_T} (X,Y) \\
 k \ar[mapsto]{r} & k .
\end{tikzcd}
$$
The fact that this is a bijection follows directly from the definition of Kleisli morphism: Kleisli morphisms from $X$ to $Y$ are \emph{by definition} morphisms $X\to TY$ of $\cat{C}$. 

Naturality in $X$ says that for $g:W\to X$, the following diagram commutes. 
$$
\begin{tikzcd}
 \Hom_\cat{C} (X,TY) \ar{d}{-\circ g} \ar{r}{\cong} & \Hom_{\cat{C}_T} (X,Y) \ar{d}{-\circ_{kl}L_T(g)} \\
 \Hom_\cat{C} (W,TY) \ar{r}{\cong} & \Hom_{\cat{C}_T} (W,Y)
\end{tikzcd}
$$
In other words, we need to show that for all $k:X\to TY$, $k\circ g=k\circ_{kl}L_T(g)$, or more explicitly, that $k\circ g=\mu\circ Tk\circ \eta \circ g$. This can be shown by forming the following diagram,
$$
\begin{tikzcd}
 & TX \ar{r}{Tk} & TTY \ar{dr}{\mu} \\
 W \ar{r}{g} & X \ar{u}{\eta} \ar{r}{k} & TY \ar{u}{\eta} \ar{r}[swap]{\id} & TY 
\end{tikzcd}
$$
which commutes by naturality of $\eta$ and by the left unitality diagram of \eqref{monaddiagrams}.

Naturality in $Y$ says that for every $h:Y\to TZ$, the following diagram commutes. 
$$
\begin{tikzcd}
 \Hom_\cat{C} (X,TY) \ar{r}{\cong} \ar{d}{R_T(h)\circ -} & \Hom_{\cat{C}_T} (X,Y) \ar{d}{h\circ_{kl}-} \\
\Hom_\cat{C} (X,TZ) \ar{r}{\cong} & \Hom_{\cat{C}_T} (X,Z) 
\end{tikzcd}
$$
In other words, we need to show that for all $k:X\to TY$, $R_T(h)\circ k = h \circ_{kl} k$. But both sides of the equation translate explicitly to $\mu\circ Th \circ h$, and so they coincide.

\item To see what the unit of the adjunction is, set $X=Y$, and consider the Kleisli identity $Y\to Y$. This corresponds, under the bijection above, with the map $\eta:Y\to TY$. Therefore $\eta$ is the unit of the adjunction. \qedhere
\end{enumerate}

\end{proof}

\begin{ex}
 Write down explicitly the universal properties associated to this adjunction.
\end{ex}

\subsection{Closure operators and idempotent monads}

The interpretation of monads as extensions is helpful also for the case of monads on posets. These are simpler construction which go under the name of \emph{closure operators}. Let's see this in detail. Let $(X,\le)$ be a poset (or a preorder). Plugging in the definition, a monad on $X$ amounts to a monotone map $t:X\to X$ (the ``endofunctor''), and with pointwise inequalities (``natural transformations'') $\id_X\le t$ and $t^2\le t$, which mean that
for every $x\in X$, $x\le t(x)$ and $t(t(x))\le t(x)$. Since $t$ is monotone, the first inequality implies that $t(x)\le t(t(x))$ too, therefore equivalently we have the following.

\begin{deph}
 A \emph{closure operator} on a poset $(X,\le)$ is a map $t:X\to X$ satisfying the following properties.
 \begin{enumerate}
  \item Monotonicity: for each $x\le y\in X$, $t(x)\le t(y)$;
  \item Extensivity: for each $x\in X$, $x\le t(x)$ (i.e. ``applying the map lets us step up'');
  \item Idempotency: for each $x\in X$, $t(t(x))=t(x)$ (i.e.~once we apply the map, applying it again has no effect).
 \end{enumerate}
\end{deph}

\begin{caveat}
 In topology one calls ``closure operator'' a function satisfying, in addition to the properties above, the additional property of preserving finite joins. This is a distinct notion. In order to avoid confusion, the operators in the form we have given above are also called ``Moore closure operators'', while those in topology are called ``Kuratowski closure operators''.
\end{caveat}

The following exercises show why closure operators can be interpreted as ``extensions'' or ``completions'' of some kind.
Before reading the following exercises, make sure you understand the examples given in \Cref{substructures}.

\begin{ex}[sets and relations, analysis]
 Let $S$ be a subset of $\R^2$, and denote by $t(S)$ its convex hull, see \Cref{convsubsets}. (In the notation of \Cref{convsubsets}, $t=i\circ c$.) Prove that $t$ is a closure operator on $P(\R^2)$.
\end{ex}

\begin{ex}[topology]
 Prove that the topological closure (see \Cref{topclosure}) is indeed a closure operator on the subsets of a topological space. 
\end{ex}

\begin{ex}[linear algebra]
 Let $V$ be a vector space. Recall that the \emph{span} of a subset $S\subseteq V$ is the smallest vector subspace of $V$ containing $S$, or equivalently, the set of all vectors of $V$ which can be expressed as linear combinations of elements of $S$. 
 
 Show that the span gives a closure operator on the subsets of $V$.
\end{ex}

You may have noticed a pattern in these exercises. Here it is:
\begin{ex}[important!]\label{galoisclosure}
 Let $X$ and $Y$ be posets, and let $f:X\to Y$ and $g:Y\to X$ form a Galois connection $f\ladj g$.
 Prove that $g\circ f$ is a closure operator.
\end{ex}

This statement will be generalized in \Cref{adjmonads}.

Monads can be ``idempotent'' even when they are defined on a category which is not a poset -- in that case, one has to allow for ``idempotency up to isomorphism''.

\begin{deph}
 A monad $(T,\eta,\mu)$ on a category $\cat{C}$ is called \emph{idempotent} if the multiplication $\mu:TT\Rightarrow T$ is a natural isomorphism.
\end{deph}

Idempotent monads model for example situations in which, after ``extending'' our spaces, we have ``completed'' our space in a way that cannot be further extended, ``everything is there already''. 
Here is a standard example.

\begin{ex}[analysis]\label{cauchycompletion}
 Let $X$ be a metric space. The \emph{Cauchy completion} of $X$ is a space $CX$ which we can think of as ``containing all the limit points of $X$ too'', and is constructed as follows. Let $SX$ be the set of Cauchy sequences in $X$. Given sequences $\{x_n\}$ and $\{y_n\}$, define 
 $$
 d \big( \{x_n\}, \{y_n\} \big) \;\coloneqq\; \lim_{n\to\infty} d(x_n,y_n) ,
 $$
 where the distance on the right is the one of $X$. 
 Let now $CX$ be the quotient of $SX$ given by identifying all the elements of $SX$ which have distance zero (i.e.~sequences which ``would have the same limit'', even if the limit does not exist in $X$). 
 The space $CX$ is complete (why?). 
 Moreover, there is a canonical dense isometric embedding $\eta:X\to CX$ mapping the point $x$ to the constant sequence at $x$, which obviously has $x$ as limit. 
 Therefore, $CX$ extends $X$ by allowing for ``limits of points of $X$'' which may not have been in $X$ before.
 If $X$ is complete to begin with, then $CX$ is isometric to $X$ via the map $\eta$: indeed, an inverse to $\eta$ is given by the map 
 \begin{equation}\label{cauchylimit}
 \{x_n\}\;\longmapsto\; \lim_{n\to\infty} x_n
 \end{equation}
 which always exists in $X$ if (and only if) $X$ is complete. (Why is this an inverse?)
 
 The Cauchy completion is functorial in the following sense. Let $X$ and $Y$ be metric spaces, and $f:X\to Y$ be Lipschitz (or even just uniformly continuous). We can extend $f$ ``by continuity'' to a map $Cf:CX\to CY$ by setting 
 $$
 Cf(\{x_n\}) \;\coloneqq\;  \{f(x_n)\}.
 $$
 This is well-defined, as it does not depend on the choice of the representative sequence (why?), and so $C$ is a functor on the category $\cat{Lip}$ of metric spaces and Lipschitz maps. 
 The embedding $\eta:X\to CX$ is natural (why?). 
 
 We can make $C$ an idempotent monad as follows. First of all, since $CX$ is complete, we have that $\eta:CX\to CCX$ is an isometry. Set then $\mu:CCX\to CX$ to be the inverse of $\eta$, i.e.~the ``lim'' map of~\Cref{cauchylimit}. 
 
 Show that $(C,\eta,\mu)$ is a monad on $\cat{Lip}$. Note that by construction, the monad is idempotent.
\end{ex}

We conclude this section by showing that \emph{not all monads can be thought of as extensions}. The interpretation as extensions is often helpful, but not always accurate, and the reason is that the unit $\eta:X\to TX$ of the monad is not always a monomorphism. This is particularly true for some idempotent monads. Sometimes, for example, the map $\eta$ is even a quotient, so that instead of ``extending'' $X$ it is ``compressing'' it, by identifying different elements. The basic example is the following.

\begin{ex}[sets and relations]
 Let $\cat{Eq}$ be the category whose objects are sets equipped with an equivalence relation $(X,\sim)$, and morphisms are functions respecting the equivalence, i.e.~$x\sim x'$ implies $f(x)\sim f(x')$ (recall \Cref{equequ}).   
 Assign now to each $(X,\sim)$ the quotient set $X/\sim$ obtained by identified all elements of $X$ which are related by $\sim$ (recall \Cref{quotient}). Assume $X/\sim$ equipped with the identity relation. Denote by $q$ the quotient map $(X,\sim)\to X/\sim$.
 Show that given a map $f:(X,\sim)\to(Y,\sim)$ respecting the equivalence there exists a unique map $\tilde{f}:X/\sim\to Y/\sim$ making the following diagram commute.
 $$
 \begin{tikzcd}
  (X,\sim) \ar{d}{q} \ar{r}{f} & (Y,\sim) \ar{d}{q} \\
  X/\sim \uni{r}{\tilde{f}} & Y/\sim
 \end{tikzcd}
 $$
 Why is this map well-defined? (Reading \Cref{eq} again may help you.)
 The assignment $(X,\sim)\mapsto X/\sim$, $f\mapsto\tilde{f}$ is therefore an endofunctor on $\cat{Eq}$.
 
 Show that this functor induces an idempotent monad on $\cat{Eq}$, with unit given by the quotient map $q$.
\end{ex}

Here are similar examples from different fields.

\begin{ex}[topology]\label{kolmogorovquotient}
 Show that the Kolmogorov quotient of topological spaces gives an idempotent monad on $\cat{Top}$.
\end{ex}

\begin{ex}[group theory]\label{abelianization}
 Show that the abelianization of a group gives an idempotent monad on $\cat{Grp}$. 
\end{ex}

In the most general case, the unit of the monad is not monic nor epi. For example, the unit of the Giry monad is in general neither (to see why it is not always monic, take as $X$ a codiscrete topological space with its Borel $\sigma$-algebra).

\section{Monads as theories of operations}

Here is another way to look at monads, very useful in algebra.
\begin{idea} 
	A monad is a consistent choice of formal expressions of a specific kind, together with ways to evaluate them. 
\end{idea}

Before looking at the examples, recall the discussion about formal expressions as opposed to their result given in \Cref{counitvect}. A formal expression is something that ``looks like an operation'', but may have no result defined in any sense. For example, if $x,y\in X$, we can write the expression $x+y$, even if there is no sum defined on $X$. The expression $x+y$ does not \emph{actually} mean that we are summing them, it is only written there, formally. 
Whenever we work with variables instead of numbers we do this all the time: we know that $2+1=3$, but what's $a+b$? The expression $a+b$ cannot be further evaluated unless we know the values of $a$ and $b$ (or, in computer science: unless the variables have been assigned a value, and we have access to it). Even if $a+b$ has no result, we can still use it in mathematical manipulations, because there are some facts which will be true regardless of the values of $a$ and $b$. For example, that $a+b=b+a$.\footnote{By, ``addition'', here we mean an operation satisfying the axioms of a commutative monoid, like the addition of natural numbers. We will always denote such an operation, formal or not, by ``$+$''.}

\begin{eg}[algebra]\label{fcm}
 Let $X$ be a set. We can form the set $FX$ of \emph{formal sums of elements of $X$}. That is, formal expressions 
 $$
 x_1 + \dots + x_n ,
 $$
 of finite length (but arbitrarily large), including the empty expression, satisfying the usual commutativity law for addition. 
 
 We can make $F$ functorial as follows. Let $f:X\to Y$ be a function. We define the function $Ff:FX\to FY$ to ``apply $f$ elementwise in the expression''. That is, 
 $$
 Ff(x_1 + \dots + x_n) \;\coloneqq\; f(x_1) + \dots + f(x_n) ,
 $$
 and $Ff$ applied to the empty expression on $X$ gives the empty expression on $Y$. This way, $F$ is a functor $\cat{Set}\to\cat{Set}$. 
 
 Let's now equip $F$ with a monad structure.
 Given $x\in X$, we can form the trivial formal expression 
 $$
 x 
 $$
 in which $x$ is the only addendum, there is ``nothing else to add''. Doing this for every $x\in X$ we get then an assignment $\eta:X\to FX$. This is natural in $X$: the naturality diagram
 $$
 \begin{tikzcd}
  X \ar{d}{\eta} \ar{r}{f} & Y \ar{d}{\eta} \\
  FX \ar{r}{Ff} & FY
 \end{tikzcd}
 $$
 commutes, as for each $x\in X$ both paths in the diagram give the trivial formal expression $f(x)\in FY$.
 
 Consider now the space $FFX$, whose elements are \emph{formal expressions of formal expressions}. A way to represent its elements is to add brackets, as in 
 $$
 (x_1+x_2)+(x_1+x_3) ,
 $$
 and so on. The addenda in the sum above are themselves formal sums.
 Given such a nested formal expression we can always remove the brackets, obtaining a plain formal expression
 $$
 x_1 + x_2 + x_1 + x_3 .
 $$
 This gives an assignment $\mu:FFX\to FX$, which will be the multiplication of the monad. (Why is this natural?)
 
 The maps $\eta$ and $\mu$ satisfy the monad axioms \eqref{monaddiagrams}. Let's see what they mean.
 \begin{itemize}
  \item The left unitality diagram says that if we take a formal expression, such as 
  $$
  x_1+x_2+x_3 ,
  $$
  and we embed it into $FFX$ via $\eta$ forming the formal expression 
  $$
  (x_1 + x_2 + x_3)
  $$
  (which is a nested formal expression with only one addendum), then removing the brackets via the map $\mu$ gives the original expression
  $$
  x_1+x_2+x_3 .
  $$
  
  \item The right unitality diagram says that if we again take a formal expression, as above, and now embed it into $FFX$ via the map $F\eta$, that is, embedding each term as its own trivial formal expression,
  $$
  (x_1) + (x_2) + (x_3)
  $$
  (this is a nested formal expression whose addenda are themselves formal expressions of one addendum), then removing the the brackets via the map $\mu$ gives again the original expression.
  
  \item The associativity square says the following. Take a a formal expression of formal expressions of formal expressions (three times, so two levels of nesting), such as
  $$
  ((x_1+x_2)+(x_3))+((x_4)).
  $$
  We can then either remove the (outer) brackets, via the map $\mu$,
  $$
  (x_1+x_2)+(x_3)+(x_4)
  $$
  or instead apply the map $F\mu$, which removes the brackets in each term, i.e.~the inner brackets,
  $$
  (x_1+x_2+x_3)+(x_4).
  $$
  These two expressions differ, but applying $\mu$ one more time to either of them removes the remaining brackets, and makes the two expressions coincide,
  $$
  x_1 + x_2 + x_3 + x_4 .
  $$
  
  Therefore $(F,\eta,\mu)$ is a monad. It is called the \emph{(free) commutative monoid monad}. ``Commutative monoid'' because the formal operations that it encodes, additions, are exactly those giving the algebraic structure of a commutative monoid. Additional motivation will be given in \Cref{algmonoids}. The word ``free'' will be explained in \Cref{freealgebras}.
 \end{itemize}
\end{eg}

\begin{ex}[several fields]\label{freemonoidmonad}
 Consider the list monad of \Cref{listmonad}. Given a set $X$, we can view the elements of $LX$, instead of as ``lists on $X$'', as ``formal products of elements of $X$'', writing 
 $$
 [x_1,\dots,x_n]
 $$
 as
 $$
 x_1 \,\dots\, x_n .
 $$
 The unit and multiplication have a similar interpretation, as in \Cref{fcm}, in terms of ``one-term expression'' and ``removing brackets''.
 Note that this time, differently from \Cref{fcm}, the order of the terms matters. This encodes the structure of a monoid, but not necessarily commutative. Because of this, the list monad is also known under the name of \emph{free monoid monad}.
 
 (Another example of a monoid where the product is not commutative is given by the product of square matrices of a fixed dimension.)
\end{ex}

\begin{ex}[analysis, probability]\label{convexcombmonad}
 Just as the list monad can be reinterpreted in terms of formal expressions, the same can be done for the probability monad of \Cref{probmonad}. Given a set $X$, we can view an element $p\in \mathcal{P}X$ as a \emph{formal average} or \emph{formal convex combination} of elements of $x$. Here's an example. 
 
 Consider a coin flip, where ``heads'' and ``tails'' both have probability $1/2$. Then \emph{in some sense}, this is a formal convex combination of ``heads'' and ``tails''. The word ``\emph{formal}'' here is the key: the set $\{\mbox{``heads''},\mbox{``tails''}\}$ is not a convex space, so one can't really take actual averages of its elements. There is no ``intermediate state between heads and tails'', how would the coin land? It is a \emph{formal} convex combination, without a result. Every probability measure on $X$ can be interpreted in a similar way. The unit and multiplication, once again, can be interpreted in terms of ``one-term convex combination'' and ``removing the brackets'' (how exactly?). 
\end{ex}

\begin{ex}[algebra, group theory]\label{freegroupmonad}
 Given a set $X$, consider the set $GX$ of words whose letter are either elements of $X$ or formal inverses thereof, such as
 $$
 x_1 \, {x_2}^{-1}\, x_3 .
 $$
 Can you equip $GX$ with a monad structure?
\end{ex}

\subsection{Algebras of a monad}

Given a monad, we don't only have a way to form formal expressions, but also a setting in which those expressions have a result. 

\begin{deph}\label{defalgebra}
 Let $(T,\eta,\mu)$ be a monad on a category $\cat{C}$. An \emph{algebra of $T$}, or \emph{algebra over $T$}, or \emph{$T$-algebra}, consists of
 \begin{itemize}
  \item An object $A$ of $\cat{C}$;
  \item A morphism $e:TA\to A$ of $\cat{C}$,
 \end{itemize}
 such that the following diagrams commute, called ``unit'' and ``composition'', respectively.
 \begin{equation}\label{algebradiagrams}
  \begin{tikzcd}
   A \ar[swap]{dr}{\id} \ar{r}{\eta} & TA \ar{d}{e} \\
   & A
  \end{tikzcd}
  \qquad
  \begin{tikzcd}
   TTA \ar{d}{\mu} \ar{r}{Te} & TA \ar{d}{e} \\
   TA \ar{r}{e} & A
  \end{tikzcd}
 \end{equation}
\end{deph}

\begin{caveat}
 The word ``algebra'' has other meanings in mathematics. In order to avoid confusion, the algebras of a monad are sometimes also called ``Eilenberg-Moore algebras''. 
\end{caveat}

The interpretation of a $T$-algebra in terms of formal expressions is that an algebra is a place which is \emph{equipped} with the operations specified by the monad, i.e.~those expressions have an actual result. The map $e:TA\to A$ maps a formal expression to its actual result. 

\begin{eg}[algebra]\label{algmonoids}
 Let's show that the algebras of the ``free commutative monoid monad'' given in \Cref{fcm} are exactly commutative monoids.
 
 First of all, let $A$ be a commutative monoid (for example, natural numbers with addition). 
 Take as map $e:FA\to A$ the function that maps each formal sum in $A$ to its actual result (and the empty expression to zero). 
 For example, $2+2+1\mapsto 5$. 
 This satisfies the algebra axioms~\eqref{algebradiagrams}.
\begin{itemize}
 \item The unit diagram of the algebra says that if we evaluate a one-element sum, the result is just that element. For example, given $a\in A$ the evaluation of the trivial formal sum containing only $a$ gives as result again $a$;
 \item If we have a formal expression of formal expressions, we can either first remove the brackets and then evaluate the result, or first evaluate the content of the brackets, and then evaluate the resulting expression. The composition diagram says that the result will be the same. For example, the expression in the top left corner of the following diagram can be evaluated in these two equivalent ways:
 \begin{equation*}
  \begin{tikzcd}
   (2+3) + (1+2) \ar[mapsto]{d}{\mu} \ar[mapsto]{r}{Fe} & 5+3 \ar[mapsto]{d}{e}\\
   2+3+1+2 \ar[mapsto]{r}{e} & 8 
  \end{tikzcd}
 \end{equation*}
\end{itemize}
 Therefore, every monoid is an $F$-algebra. 
 
 Conversely, let $A$ be an $F$-algebra with structure map $e:FA\to A$. Then $A$ has a canonical monoid structure given as follows.
 \begin{itemize}
  \item The neutral elements is the element of $A$ obtained by applying $e$ to the empty expression;
  \item Given $a,b\in A$, their (actual) sum is defined as the result (via $e$) of the formal expression $a+b$, i.e.~$e(a+b)$. 
 \end{itemize}
 This operation is associative, unital, and commutative (why?), and so $A$ is canonically a commutative monoid.
\end{eg}

\begin{ex}[several fields]
 Prove that the algebras of the free monoid monad (or list monad) of \Cref{listmonad} and \Cref{freemonoidmonad} are exactly monoids. 
\end{ex}

\begin{ex}[algebra, group theory]
 If you have solved \Cref{freegroupmonad}, what are the algebras of the resulting monad?
\end{ex}

\begin{eg}[analysis, probability]\label{convexspaces}
 The algebras of the probability monad (see \Cref{probmonad} and \Cref{convexcombmonad}) are called \emph{convex spaces}. These are spaces $A$ where one can (actually) take weighted averages, as specified by the elements of $\mathcal{P}A$ (the nonzero values of $p\in \mathcal{P}A$ are the weights).
 
 For example, the unit interval $[0,1]$ is a $\mathcal{P}$-algebra with the usual convex structure: the average
 $$
 \dfrac{1}{2} \, 0 + \dfrac{1}{2} \, 1 
 $$
 gives as result $1/2$. Compare with the coin example of \Cref{convexcombmonad}, where the average of ``heads'' and ``tails'' had no result. 
 More generally, every convex subset of a vector space is a $\mathcal{P}$-algebra (the converse is not true: there $\mathcal{P}$-algebras which cannot be embedded into vector spaces). 
 
 In probability theory, the operation of \emph{expectation value} is one of the most important -- this can be encoded by the notion of an algebra of the probability monad. (For non-finitely supported measures, take for example the Giry monad). 
\end{eg}

In the next example we see that the writer monad is related to group and monoid actions on spaces.

\begin{eg}[several fields]\label{actionmonad}
 Let's study the algebras of the writer monad $T_M$, where $M$ is a monoid (\Cref{writermonad}).
 In order for this interpretation to be as suggestive as possible it is helpful to change the notation slightly. We write $M\times A$ instead of $A\times M$, we write the monoid $M$ multiplicatively instead of additively (with neutral element 1), and we denote the map $e:M\times A \to A$ simply by a dot, i.e.~we write $e(m,a)$ by $m\cdot a$. 
 
 Plugging in the definition, a $T_M$-algebra is then a set $A$ together with a map $M\times A\to A$, such that the following diagrams commute. 
 $$
 \begin{tikzcd}
 A \ar{dr}[swap]{\id} \ar{r}{\eta} & M\times A \ar{d}{e} \\
 & A
 \end{tikzcd}
 \qquad
 \begin{tikzcd}[column sep=large]
  M \times M \times A \ar{d}{\mu} \ar{r}{\id_M\times e} & M\times A \ar{d}{e} \\
  M \times A\ar{r}{e} & A .
 \end{tikzcd}
 $$
 Explicitly, the diagram say that for all $a\in A$ and $m,n\in M$,
 $$
 1\cdot a= a \qquad\mbox{and}\qquad (mn)\cdot a = m\cdot (n\cdot a) .
 $$
 In other words, a $T_M$-algebra is exactly a set equipped with an $M$-action (also called an $M$-set). 
 This is why the writer monad is also called the \emph{action monad} or \emph{$M$-action monad}.
 
 We have already seen a way to talk about monoids and groups acting on sets (and vector spaces, and so on): namely, a monoid $M$ acting on a set is the same as a functor $\cat{B}M\to\cat{Set}$ (see \Cref{egrepresentation}, \Cref{perm}, \Cref{egmorphsym}, and \Cref{egmorphdynsys}). This is a different (but related) way to talk categorically about the same structure.
 
 This also says that we can interpret the monad $T_M$ in terms of formal expressions: an element $(m,a)\in M\times A$ is a ``formal action'' or ``formal move'': $m$ is ``ready to act on $a$'', but the operation is not evaluated (yet). (As usual, this can be done also on sets which are not necessarily algebras, in that case formal actions are never evaluated.)
 
 From the point of view of computer science, remember that we were interpreting $M$, for example as a ``cost'' or ``side effect'' of some kind. Algebras are settings in which the side effect can be reincorporated, by possibly ``acting'' on $A$, i.e.~it has an effect on the output data. Using again the \LaTeX~ example, if a reference in the document is wrong or missing, the compiler not only reports an error, but also adds question marks in the compiled file, in place of the reference. The error has an effect on the main output data. 
\end{eg}

The algebras of a given monad form a category, whose morphisms we can think of as ``being compatible with the specified operations''.
Let's see the general definition. 

\begin{deph}\label{defemcat}
 Let $(A,e)$ and $(B,e)$ be algebras of a monad $T$ on $\cat{C}$. A \emph{morphism of $T$-algebras}, or \emph{$T$-morphism}, is a morphism $f:A\to B$ of $\cat{C}$ such that the following diagram commutes. 
 \begin{equation}\label{Tmorphism}
 \begin{tikzcd}
  TA \ar{d}{e} \ar{r}{Tf} & TB \ar{d}{e} \\
  A \ar{r}{f} & B
 \end{tikzcd}
\end{equation}

 The category of $T$-algebras and $T$-morphisms is called the \emph{Eilenberg-Moore category} or \emph{category of algebras of $T$} and it is denoted by $\cat{C}^T$. 
\end{deph}

\begin{eg}[algebra]
 Consider two algebras of the free commutative monad (\Cref{fcm}), i.e.~two commutative monoids, $A$ and $B$. Not every function between them respects addition. The function $f$ preserves the addition and neutral elements (a property called \emph{additivity}) if and only if evaluating expressions before or after applying $f$ does not change the result. For example, if $f(a+b)=f(a)+f(b)$.
In other words, $f$ is additive if an only if it makes the diagram~\eqref{Tmorphism} commute.

Therefore the Eilenberg-Moore category of the free commutative monoid monad is the category of commutative monoids and their morphisms.
\end{eg}

\begin{ex}[algebra]
 Prove that the Eilenberg-Moore category of the list monad (or free monoid monad, \Cref{listmonad}) is the category of monoids and their morphisms.
\end{ex}

\begin{ex}[algebra, group theory]
 If you have solved \Cref{freegroupmonad}, what is the Eilenberg-Moore category of the resulting monad?
\end{ex}

\begin{eg}[several fields]
 Consider the writer monad (or action monad) $T_M$, with $M$ a monoid (\Cref{actionmonad}). For $T_M$-algebras (i.e.~$M$-sets) $A$ and $B$, a function $f:A\to B$ is a morphism of algebras if and only if the following commutes.
 $$
 \begin{tikzcd}
  M\times A \ar{d}{e} \ar{r}{\id_M\times f} & M\times B \ar{d}{e} \\
  A \ar{r}{f} & B
 \end{tikzcd}
 $$
 Equivalently, if and only if for every $m\in M$ and $a\in A$, 
 $$
 f(m\cdot a) \;=\; m \cdot f(a). 
 $$
 In other words, the morphisms of algebras are precisely the equivariant maps, i.e.~those preserving ``symmetries'' or ``specified operations'' (see \Cref{equivariant}).
 
 Therefore the Eilenberg-Moore of $T_M$ is equivalent to the functor category $[\cat{B}M,\cat{Set}]$ (see \Cref{fcats}). 
\end{eg}

\begin{ex}[linear algebra]\label{vectorspacemonad}
 Construct a monad on $\cat{Set}$ whose Eilenberg-Moore category is the category $\cat{Vect}$ of vector spaces and linear maps. (Hint: adapt to the vector space case the construction of the free commutative monad. Reading \Cref{counitvect} again could help.) 
\end{ex}

\begin{ex}[probability, analysis]
 What are the morphisms of algebras of the probability monad of \Cref{probmonad}? 
\end{ex}

\begin{ex}[sets and relations, basic computer science]
 The \emph{maybe monad} is a monad on $\cat{Set}$ whose functor part adds an extra point to each set, i.e.~it maps $X$ to $X\sqcup 1$ (where $1$ is the singleton). What can its unit and multiplication be?
 
 Show that the Eilenberg-Moore category of the maybe monad is (equivalent to) the category of pointed sets and base point-preserving functions. 
\end{ex}

\begin{ex}[sets and relations; difficult!]
 Show that the algebras of the power set monad of \Cref{powersetmonad} are complete semilattices, with the algebra structure map given by the join (or the meet). 
 What are the morphisms of algebras?
\end{ex}

\begin{ex}[difficult!]
 Show that for every $T$-algebra $(A,e)$, the diagram
 $$
 \begin{tikzcd}
 TTA \ar[shift left]{r}{\mu} \ar[shift right]{r}[swap]{Te} & TA \ar{r}{e} & A
 \end{tikzcd}
 $$
 is a coequalizer diagram. (Hint: the unique maps giving the universal property of the coequalizer can be obtained using the unit $\eta$.)
\end{ex}

\begin{ex}\label{stmapiso}
 Show that a monad is idempotent if and only if all its algebra structure maps are isomorphisms. (Hint: use the exercise above, after proving that for every algebra $(A,e)$ of an idempotent monad $T$, $\mu$ and $Te$ must coincide.)
\end{ex}

Note that this implies that, for example,
\begin{itemize}
 \item The algebras of the Cauchy completion monad of \Cref{cauchycompletion} are the complete metric spaces;
 \item The algebras of the abelianization monad of \Cref{abelianization} are the abelian groups;
 \item The algebras of the Kolmogorov quotient monad of \Cref{kolmogorovquotient} are the $T_0$ (or \emph{Kolmogorov}) topological spaces.
\end{itemize}
What are the morphism of algebras in each case?
If you can't answer this question, see the next exercise.

\begin{ex}
 Using \Cref{stmapiso}, show that if $T$ is an idempotent monad on $\cat{C}$, given any two $T$-algebras $A$ and $B$, any morphism $f:A\to B$ of $\cat{C}$ is automatically a morphism of algebras. 
\end{ex}

\subsection{Free algebras}\label{freealgebras}

Let $(T,\eta,\mu)$ be a monad on $\cat{C}$. Given an object $X$ of $\cat{C}$, the object $TX$ is canonically a $T$-algebra, with structure morphism $\mu:TTX\to TX$. The diagrams \eqref{algebradiagrams} in this setting become 
$$
\begin{tikzcd}
 TX \ar{dr}[swap]{\id} \ar{r}{\eta} & TTX \ar{d}{\mu} \\
 & TX
\end{tikzcd}
\qquad
\begin{tikzcd}
 TTTX \ar{r}{T\mu} \ar{d}{\mu} & TTX \ar{d}{\mu} \\
 TTX \ar{r}{\mu} & TX
\end{tikzcd}
$$
which are just the left unitality and associativity diagram of the monad, as in \eqref{monaddiagrams}, and so they commute.

\begin{deph}
 We call a $T$-algebra in the form $(TX,\mu)$ for some $X$ of $\cat{C}$ a \emph{free $T$-algebra}.
\end{deph}

Moreover, let $f:X\to Y$ be a morphism of $\cat{C}$. The induced map $Tf:TX\to TY$ is a morphism of algebras, since the diagram
$$
\begin{tikzcd}
 TTX \ar{d}{\mu} \ar{r}{TTf} & TTY \ar{d}{\mu} \\
 TX \ar{r}{Tf} & TY
\end{tikzcd}
$$
commutes by naturality of $\mu$.

An interpretation of free algebras, in terms of formal expressions, is that they contain formal expressions, and that the structure map $\mu$ is just the simplification of formal expressions (or the juxtaposition of them). Let's see some examples. The word ``free'', as we said before, will be motivated in \Cref{emadj}.

\begin{eg}[several fields]
 Consider the list monad $L$ of \Cref{listmonad}. We know its algebras are monoids. Now given a set $X$, the set $LX$ of \emph{lists} or \emph{words} in $X$ is canonically a monoid, with as neutral element the empty list, and as multiplication the concatenation of strings,
 $$
 [x_1,\dots,x_n]\cdot[y_1,\dots,y_m] \;\coloneqq\;[x_1,\dots,x_n,y_1,\dots,y_m] .
 $$
\end{eg}

\begin{eg}[algebra]
 Consider the free commutative monoid monad $F$ of \Cref{fcm}. Given a set $X$, the free algebra $FX$ contains formal sums such as 
 $$
 x_1 + \dots + x_n ,
 $$
 and the structure map, which is the multiplication of the monad, sums formal expressions by just removing the brackets,
 $$
 (x_1 + \dots + x_n) + (y_1 + \dots + y_m) 
 \;\coloneqq\; (x_1 + \dots + x_n + y_1 + \dots + y_m).
 $$
\end{eg}

\begin{eg}[probability]
 Consider the probability monad $\mathcal{P}$ of \Cref{probmonad}, or any of its variants in the categories of measurable or topological spaces. 
 We have seen that its algebras can be interpreted as ``convex spaces'' of some kind (\Cref{convexspaces}). The free algebras are now the convex spaces in the form $\mathcal{P}X$ for some set (or space) $X$, with the convex combination operation given by the mixture of probability measures, such as 
 $$
 \dfrac{1}{2} \left( \dfrac{1}{2}\,\mbox{heads} + \dfrac{1}{2}\,\mbox{tails} \right) + \dfrac{1}{2} \left( 1\,\mbox{heads} + 0\,\mbox{tails} \right) \;\coloneqq\; \dfrac{3}{4}\,\mbox{heads} + \dfrac{1}{4}\,\mbox{tails}.
 $$
 Note that we are not mixing ``heads'' and ``tails'' (we know that's impossible), we are mixing \emph{probability measures over them}. 
 
 Spaces in the form $\mathcal{P}X$ are also known as \emph{simplices}, and the points in the image of $\delta:X\to \mathcal{P}X$ are called the \emph{extreme points} of the simplex.
 Can you see why?
\end{eg}

\begin{ex}[linear algebra]
 Consider the vector space monad of \Cref{vectorspacemonad}. Using the fact that \emph{every vector space has a basis}, prove that each algebra is free.
\end{ex}

\begin{ex}[group theory]
 What are the free algebras of the monad of \Cref{freegroupmonad}?
\end{ex}

\subsection{The Eilenberg-Moore adjunction}\label{emadj}

Let $(T,\eta,\mu)$ be a monad on a category $\cat{C}$, and let $A$ and $B$ be $T$-algebras. By construction, any morphism of algebras $f:A\to B$ is first of all a morphism of $\cat{C}$. For example, an additive map between monoids is first of all a function. Therefore there is a fully faithful ``forgetful'' functor $\cat{C}^T\to\cat{C}$. Let's denote this functor by $R^T$.
We have seen some of those forgetful functors already:
\begin{itemize}
 \item The forgetful functor $\cat{Vect}\to\cat{Set}$;
 \item The forgetful functor $\cat{Mon}\to\cat{Set}$;
 \item The forgetful functor $\cat{Grp}\to\cat{Set}$.
\end{itemize}
(Which monads have the categories above as Eilenberg-Moore categories?)

We have also seen that, in many cases, those functors have left-adjoints. For example, the forgetful functor $\cat{Vect}\to\cat{Set}$ has a left-adjoint (see \Cref{adjvectset}), with the intuition of ``forming formal linear combinations of elements out of a set''.
This idea of ``forming formal expressions'', which we now know can be formalized by the ideas of a monad and of free algebras, works in general.
Let's see how. 

\begin{deph}
 Let $(T,\eta,\mu)$ be a monad on a category $\cat{C}$. We define the functor $L^T:\cat{C}\to\cat{C}^T$ as follows.
 \begin{itemize}
  \item It maps an object $X$ to the free $T$-algebra $(TX,\mu)$;
  \item It maps a morphism $f:X\to Y$ to the morphism of algebras $Tf:TX\to TY$.
 \end{itemize}
\end{deph}

The interpretation of this construction is: given $X$, form spaces of formal expressions out of $X$, and given $f:X\to Y$, apply $f$ to each term in the formal expressions to get a formal expression in $Y$. 

We now have a result which is similar, in some sense even dual, to \Cref{kleisliadj}.

\begin{prop}\label{emadjprop}
 Let $(T,\eta,\mu)$ be a monad on $\cat{C}$. 
 \begin{enumerate}
  \item The composite functor $R^T\circ L^T:\cat{C}\to\cat{C}$ is naturally isomorphic to $T$;
  \item The functor $L^T$ is left-adjoint to $R^T$;
  \item\label{unitem} The unit of the adjunction is given by the unit of the monad $\eta$;
  \item\label{counitem} The counit of the adjunction is given by the structure maps $e$ of the algebras.
 \end{enumerate}
\end{prop}

\begin{proof}

 \begin{enumerate}
  \item The composite functor $R^T\circ L^T$ acts as follows,
  $$
  \begin{tikzcd}[sep=small]
   X \ar{dd}{f} &&& (TX,\mu) \ar{dd}{Tf} &&& TX \ar{dd}{Tf} \\
  & \; \ar[mapsto]{r}{L^T} &\; &&\; \ar[mapsto]{r}{R^T} &\; \\
   Y &&& (TY,\mu) &&& TY
  \end{tikzcd}
  $$
  therefore, almost by definition, it has the same action on objects and morphisms as $T$.
  
  \item We will define the adjunction in terms of the unit and counit, which will show \ref{unitem} and \ref{counitem} as well. So first of all, the unit of the monad $\eta:\id_\cat{C}\Rightarrow T = R^T\circ L^T$ is already in the desired form for a unit of the adjunction. The counit $\e:L^T\circ R^T\Rightarrow \id_{\cat{C}^T}$ is defined as follows. Given a $T$-algebra $(A,e)$, 
  $$
  L^T\circ R^T (A,e) \;=\; L^T(A) \;=\; (TA,\mu).
  $$
  The very structure map $e:TA\to A$ defines a morphism of algebras $(TA,\mu)\to (A,e)$, since the following diagram commutes,
  $$
  \begin{tikzcd}
  TTA \ar{d}{\mu} \ar{r}{Te} & TA \ar{d}{e} \\
  TA \ar{r}{e} & A
  \end{tikzcd}
  $$
  being simply the multiplication diagram of \eqref{algebradiagrams}. 
  Moreover, every morphism of $\cat{C}^T$ commutes with the structure maps $e$ of the algebras, so that $e:(TA,\mu)\to (A,e)$ is actually natural in the algebra $A$, and so it induces a natural transformation $\e:L^T\circ R^T\Rightarrow \id_{\cat{C}^T}$.  
  
  In order to have an adjunction we have to show that $\eta$ and $\e$ satisfy the triangle identities \eqref{trianglegeneral}. The first one says (why?) that the following diagram of $\cat{C}^T$ has to commute for all objects $X$ of $\cat{C}$.
  $$
  \begin{tikzcd}
   (TX,\mu) \ar{dr}[swap]{\id} \ar{r}{T\eta} & (TTX,\mu) \ar{d}{\mu} \\
   & (TX,\mu)
  \end{tikzcd}
  $$
  (Note that the map $\mu$ appearing on the right is the structure map of the free algebra $(TX,\mu)$.) This diagram commutes since it corresponds to the right unitality diagram of the monad \eqref{monaddiagrams}.
  
  The second triangle identity says (why?) that the following diagram of $\cat{C}$ has to commute for each $T$-algebra $(A,e)$.
  $$
  \begin{tikzcd}
   A \ar{dr}[swap]{\id} \ar{r}{\eta} & TA \ar{d}{e} \\
   & A
  \end{tikzcd}
  $$
  This commutes, since it is the unit condition for algebras \eqref{algebradiagrams}. \qedhere
 \end{enumerate}
\end{proof}

\begin{cor}\label{uniem}
The adjunction amounts to a natural bijection
  $$
  \begin{tikzcd}
   \Hom_{\cat{C}} (X, A) \ar{r}{\cong} & \Hom_{\cat{C}^T}\big( (TX,\mu), (A,e) \big)
  \end{tikzcd}
  $$
  for each object $X$ of $\cat{C}$ and each $T$-algebra $(A,e)$.
In other words, it says that given a morphism $f:X\to A$, where $A$ is (the underlying object of) a $T$-algebra, there is a unique morphism of $T$-algebras $TX\to A$ such that the following diagram of $\cat{C}$ commutes.
$$
\begin{tikzcd}
X \ar{dr}{f} \ar{d}[swap]{\eta} \\
TX \uni{r} & A
\end{tikzcd}
$$
\end{cor}

\begin{eg}[linear algebra]
 We have already seen the adjunction between sets and vector spaces of \Cref{adjvectset}. The morphism $TX\to A$, in that case, is the linear extension of a map from the basis $X$ of $TX$ to the vector space $A$.
\end{eg}

\begin{eg}[algebra]
 If $F$ is the free commutative monoid monad of \Cref{fcm}, this says that given a set $X$ and a monoid $M$, there is a unique morphism of monoids $\tilde{f}:FX\to M$ such that $f=\tilde{f}\circ \eta$. This is the universal property of free monoids, it says that $FX$ is ``freely generated'' by $X$ (in our interpretation, by forming formal sums of elements of $X$). 
\end{eg}

\begin{eg}[probability, analysis]
For the probability monad of \Cref{probmonad}, the statement above says that given a convex space $A$ (for example, $\R$ or $[0,1]$) and a set (or space) $X$, there is a bijection between functions $f:X\to A$ and \emph{affine maps} from the simplex, $\mathcal{P}X\to A$, which agree with $f$ on the extrema of the simplex.

Similar statements holds for the different variants of probability monads, in their respective categories.
\end{eg}

\begin{ex}[group theory]
 Write down the adjunction explicitly for the case of groups, \Cref{freegroupmonad}. Compare with the adjunction of \Cref{adjgrpset}.
\end{ex}

\section{Comonads as extra information}

We now turn to comonads. 
Let's look at a way in which comonads can be motivated, and which helps understanding many comonads arising in practice.

\begin{idea} 
	A comonad is a consistent way to equip spaces with extra information of a specific kind, and let some morphisms access that information.
\end{idea}

As usual, let's show this by giving some examples.

\begin{eg}[several fields]\label{readercomonad}
 Let $E$ be a set. Given any (other) set $X$, the assignment $X\mapsto X\times E$ is functorial (see \Cref{productfunctorial} and \Cref{writermonad}).
 Denote this functor by $C_E$. So $C_E(X)=X\times E$, and given $f:X\to Y$, we have $C_E(f)= f\times\id_E:X\times E\to Y\times E$, which maps $(x,e)$ to $(f(x),e)$.
 
 If we view $E$ as a set of ``extra data'', then an element $(x,e)$ of $C_E(X)=X\times E$ has more information than just $x$. We can equip the functor $C_E$ with a comonad structure following this intuition. 
 \begin{itemize}
  \item The counit $\e: C_E(X)\to X$ is given by the projection $(x,e)\mapsto x$. In other words, this is \emph{forgetting or discarding the extra information}.
  \item Information cannot only be discarded, it can also be copied.\footnote{This is not true for quantum information, see \Cref{nocloning}.} The comultiplication $\nu:C_E(X)\to C_E(C_E(X))$ then \emph{copies the extra information}, it is given by $(x,e)\mapsto(x,e,e)$. 
 \end{itemize}

 It is easy to see that these maps are natural. The comonad axioms \eqref{comonaddiagrams} say now the following. 
 \begin{itemize}
  \item The left counitality diagram says that $\e\circ\nu=\id_{C_E(X)}$. Explicitly, if we start with $(x,e)\in X\times E$, then copying the extra information (getting $(x,e,e)$) and then discarding the last part (getting $(x,e)$) is the same as doing nothing. 
  \item The right unitality diagram says that $C\e\circ\nu=\id_{C_E(X)}$. Again, if we start with $(x,e)$, copying the extra information (getting $(x,e,e)$) and then discarding the first extra datum (the first of the ``$e$'', giving $(x,e)$) is again the same as doing nothing.
  \item The coassociativity diagram says that $\nu\circ\nu=C\nu\circ\nu$. Explicitly, starting again with $(x,e)$, if we again copy the extra information (obtaining $(x,e,e)$), then copying either the first bit or the second bit of the extra information gives the same result, $(x,e,e,e)$. 
 \end{itemize}

 Therefore $(C_E,\e,\nu)$ is a comonad on $\cat{Set}$. We can call it the \emph{reader comonad}. The name will be motivated in \Cref{readerkleisli}. For now, notice that it is somewhat dual to the writer monad of \Cref{writermonad}: there we had ``extra stuff'', that can be put together (sum), here we have ``extra information'', that can be copied. 
\end{eg}

\begin{eg}[basic computer science, dynamical systems]\label{streamcomonad}
 Let $X$ be a set. Denote by $SX$ the space of \emph{infinite sequences} or \emph{streams} in $X$, that is, functions $\N\to X$. (Equivalently, we can see $SX$ as the cartesian product of $\N$-many copies of $X$ with itself.)
 We can denote the elements by $\{x_n\}_{n\in\N}$, or more briefly $\{x_n\}$, sometimes more explicitly by 
 $$
 \{x_0,x_1,x_2,\dots\} .
 $$
 
 This construction is canonically functorial: given $f:X\to Y$, we can define $Sf:SX\to SY$ as the function applying $f$ elementwise, that is,
 $$
 \{x_0,x_1,x_2,\dots\} \;\longmapsto\; \{f(x_0),f(x_1),f(x_2),\dots\} .
 $$
 
 Clearly, a stream on $X$ has more information than just an element of $X$. We can now equip $S$ with a comonad structure as follows. 
 \begin{itemize}
  \item First of all, define the counit $\e:SX\to X$ as the map 
 $$
 \{x_0,x_1,x_2,\dots\}  \;\longmapsto\; x_0 ,
 $$
 which keeps only the first (well, zero-th) value and discards the rest of the stream.
 \item The comultiplication map $\nu:SX\to SSX$ has to map a stream to a \emph{stream of streams} (imagine an infinite two-dimensional matrix). Given a stream 
 $$
 \{x_0,x_1,x_2,\dots\} ,
 $$
 we form the stream of streams
 \begin{align}\label{historyhistory}
  \{ &\{ x_0,x_1,x_2,\dots \} \notag \\
  &\{ x_1,x_2,x_3,\dots \} \\
  &\{ x_2,x_3,x_4,\dots \} \notag \\
  & \dots \qquad \qquad \quad \} \notag
 \end{align}
 In other words, we form a stream of stream which has as first element the original stream, as second element the original stream but starting at $x_1$ (in computer science one says ``popping'', in dynamical systems one says ``shifting''). As third element we have a stream that starts at $x_2$, and so on.
 In symbols,
 $$
 \nu(\{x_n\}_{n\in \N}) \;=\; \{\{ x_{n+m} \}_{n\in\N}\}_{m\in \N} .
 $$
 \end{itemize} 
 One way to interpret $SX$ is that elements of $SX$ are elements of $X$ together with their \emph{history}: $x_0$ is where it was a second ago, $x_2$ where it was two seconds ago, and so on.
 The unit forgets the past and keeps only the present state. The multiplication looks at the \emph{history of the history}: one second ago, the history was only until $x_1$, two seconds ago, the history was only until $x_2$, and so on.
 
 Let's check that the comonad axioms \Cref{comonaddiagrams} hold. 
 \begin{itemize}
  \item The left counitality says that if we take the stream $\{x_0,x_1,x_2,\dots\}$ and look at its history \eqref{historyhistory}, the first element in there is the original stream.
  \item The right counitality diagram says that if we again take the stream $\{x_0,x_1,x_2,\dots\}$ and look at its history \eqref{historyhistory}, then the stream forming by \emph{taking the first element of each stream in the history} gives the original stream.
  \item The coassociativity diagram says that if we once again take the stream $\{x_0,x_1,x_2,\dots\}$ and look at its history \eqref{historyhistory}, then we can look once again at its history (a 3-dimensional stream). We can also form the 3-dimensional stream formed by looking at the history of each stream in the 2-dimensional stream \eqref{historyhistory}. Both constructions give the same 3-dimensional stream, namely
  $$
  \{\{\{ x_{n+m+k} \}_{n\in\N}\}_{m\in \N}\}_{k\in\N}
  $$
 \end{itemize}
 
 Therefore $(S,\e,\mu)$ is a comonad. It is known as the \emph{stream comonad}.
\end{eg}

\begin{ex}[several fields]\label{monoidstream}
 Let $M$ be a monoid. Construct a comonad on $\cat{Set}$ analogous to the stream comonad, but with $M$ in place of $\N$. 
 
 If $M=\Z$ this means that we are looking not just at the past, but also at the future. If $M=\R$ we are looking at continuous time instead of discrete time...and so on. 
\end{ex}

A fairly general phenomenon is that, as for many monads the unit is an embedding of $X$ into its ``extension'' $TX$, for many comonads the counit is a quotient of $CX$ onto $X$, which ``forgets the extra data''. 
In the next exercise, for the readers which have a background in geometry and topology, you can show that the universal covering is an idempotent comonad. The ``extra information'' that it encodes can be interpreted as the ``winding number'', or ``number of times we have gone around a hole''.
(If you find the next example too technical, but you still find these concepts interesting, you can try to read \Cref{ucgraphs}, which has less prerequisite knowledge.)

\begin{ex}[algebraic topology; difficult!]\label{universalcovering}
 Let $\cat{PCLC}_*$ be the category whose
 \begin{itemize}
  \item Objects are path-connected, locally contractible,\footnote{The condition of local contractibility can be weakened, see \cite[Section~1.3]{hatcher}.} pointed topological spaces;
  \item Morphisms are continuous maps preserving the base point. 
 \end{itemize}
 Given a space $(X,x)$ in the category above, denote by $p:UX\to X$ its universal covering (where $UX$ is the space of homotopy classes of paths in $X$ starting at $x$, with its usual topology, and the map $p$ takes the endpoint of each path).
 
 Show that $U$ is a functor as follows. Given $f:(X,x)\to (Y,y)$, and given a path $\gamma$ starting at $x$, $f(\gamma)$ is a path starting at $y$. Since $f$ is continuous, this assignment respects homotopy, and so we have a well-defined map $UX\to UY$. Show that this map is continuous. (Hint: prove that the preimage of a basic open of $UY$ is open in $UX$, using the fact that $X$ is locally contractible.)
 Assume $UX$ based at the constant path at $x$.
 
 Show moreover that the covering map $p:UX\to X$ is natural in $X$. This will form the counit of the comonad. (This does discard the extra information: it literally forgets the path and keeps track only of the endpoint.)
 
 Since $UX$ is simply connected, the map $p:UUX\to UX$ obtained by taking again the universal covering is a homeomorphism. Let $\nu$ be the inverse of this map. Show that $(U,p,\nu)$ is a (necessarily idempotent) comonad  on $\cat{PCLC}_*$. 
\end{ex}

\subsection{Co-Kleisli morphisms}

\begin{deph}
 Let $(C,\e,\nu)$ be a comonad on $\cat{C}$. A \emph{co-Kleisli morphism} of $C$ from $X$ to $Y$ is a morphism $CX\to Y$ of $\cat{C}$. 
\end{deph}

Here is a typical situation in science. Suppose that, during an experiment, we have a process $f$ taking an input $x\in X$ and giving an output $y\in Y$. This could be for example a survey in which we ask people of different age ($X$) what their political views are ($Y$). 
Suppose that we repeat the experiment, feeding again the \emph{same} input $x$ to $f$. It could happen that this time we get a \emph{different} outcome than before, $y'\ne y$. 
For example, suppose a person of age $x$ expresses political view $y$. Another person of the same age $x$ may express a different political view $y'$ (this may even happen by asking the same person twice, in different moments). 
The usual conclusion that we draw, in science, is: $X$ alone (for example, age alone) is \emph{not enough} to determine $Y$. There was some ``hidden'', ``extra'' data that the process $f$ has access to in order to determine $Y$. 
It could depend on hidden information, it could depend on past outcomes, and so on. 
Therefore a better mathematical model of the situation is not quite $f:X\to Y$, but rather $f:CX\to Y$, where $CX$ contains more information than just $X$, as specified by the comonad $C$. 

Note that this is somewhat dual to Kleisli morphisms: there, functions were allowed to have more possible \emph{outcomes}. Here, they are allowed to \emph{depend} on possibly more information.

\begin{eg}[several fields]\label{readerkleisli}
 A co-Kleisli morphism for the reader comonad of \Cref{readercomonad} is a map $k:X\times A\to Y$. We can view it as a map that, when fed $x\in X$, also needs to read a $a\in A$ to give its output.
 This can model the experimental situation of ``hidden variables'' described above.
 
 An example in computer science is a function that needs an extra input in order to carry out the computation, either via user interface (say, keyboard), or by having access to some external environment. This motivates the name ``reader comonad''.
\end{eg}

\begin{eg}[basic computer science, dynamical systems]\label{streamkleisli}
 A co-Kleisli morphism for the stream comonad is a map $SX\to Y$, in other words, it is a function that potentially depends on the whole stream, not just on the current value. 
 
 In dynamical systems (and probability, and so on) this corresponds to a dynamic that \emph{depends on the history}, or \emph{has memory}. Examples are given by non-Markovian stochastic processes (which, in addition, are also random), and by delayed differential equations.
\end{eg}

\begin{eg}[complex analysis]
 Let $z$ be a nonzero complex number.
 The \emph{complex logarithm of $z$} is given by the integral 
 $$
 \int_1^z \dfrac{1}{z'} \, dz'.
 $$
 This does not quite depend only on $z$, but also on the path of integration. More specifically, on the homotopy class of the path of integration. Therefore, the integral above, rather than a function $\C\backslash\{0\}\to\C$, is a function $U(\C\backslash\{0\})\to\C$, where $U$ is the universal covering comonad of \Cref{universalcovering}, and the spaces $\C\backslash\{0\}$ and $\C$ are assumed based at $1$ and $0$, respectively. In other words, the complex logarithm is, rather than an ordinary function, a co-Kleisli morphism. It takes the extra information of ``how many times we have gone around $0$''.
\end{eg}

\begin{eg}[quantum physics]
 The wave function of an electron in a one-dimensional periodic material (such as a crystal) can be obtained by solving Schrödinger's equation with a periodic potential and periodic boundary conditions, or equivalently, by solving Schrodinger's equation on the circle. Denoting a point of the circle by $x_0$, we can see the circle as a pointed space $(S^1,x_0)$.
 By Bloch's theorem, the wave function of our electron will have the form 
 $$
 \psi(x) \;=\; e^{i k (x-x_0)} \, u(x),
 $$
 where $k$ is a real number, and $u(x)$ is a function with the same periodicity as the potential. 
 Note that the probability density of the electron is equal to 
 $$
 \|\psi(x)\|^2 \;=\; \| u(x) \|^2,
 $$
 which is a periodic function, and so it can be written as a function on the graph $(S^1,x_0)$. However, the wave function itself is not periodic, because of its phase ($k$ is not necessarily related to the period of the potential), and so $\psi$ it is not a well-defined function on the circle. Instead, we can model $\psi$ as a (well-defined) \emph{co-Kleisli morphism}, for the universal covering comonad of \Cref{universalcovering}. That is, $\psi:US^1\to \C$. The ``extra information'' to which $\psi$ has access is ``how far we are from $x_0$ not only on the circle, but on its universal covering,  or on the actual periodic material'' (we may be a few entire cells away).
 Note that this extra information is reflected by the phase, which alone does not encode physical information, but plays a role whenever two or more electrons are present (there can be interference).
\end{eg}

As it is the case for Kleisli morphisms, co-Kleisli morphisms also have a meaningful notion of composition.

\begin{deph} 
 Let $(C,\e,\nu)$ be a comonad on $\cat{C}$. Let $k:CX\to Y$ and $h:CY\to Z$ be co-Kleisli morphisms. Their \emph{co-Kleisli composition} is the co-Kleisli morphism $h\circ_{ck} k:CX\to Z$ given by
 $$
 \begin{tikzcd}
  CX \ar{r}{\nu} & CCX \ar{r}{Ck} & CY \ar{r}{h} & Z .
 \end{tikzcd}
 $$
\end{deph}

\begin{eg}[several fields]
 Let $k:X\times E\to Y$ and $h:Y\times E\to Z$ be co-Kleisli morphisms of the reader comonad of \Cref{readercomonad}. Their co-Kleisli composition proceeds as follows.
 $$
 \begin{tikzcd}[row sep=0, column sep=large]
  X\times E \ar{r}{\nu} & X\times E\times E \ar{r}{k\times\id_E} & Y\times E \ar{r}{h} & Z \\
  (x,e) \ar[mapsto]{r} & (x,e,e) \ar[mapsto]{r} & (y,e) \ar[mapsto]{r} & z 
 \end{tikzcd}
 $$
 We first start with $(x,e)\in X\times E$. We copy the extra information, to get $(x,e,e)$. We then feed the first two arguments to $k$, which gives us an element $y\in Y$, and keep track of the additional $e$. We then feed $(y,e)$ to $h$, to get $z\in Z$. 
\end{eg}

\begin{ex}[basic computer science, dynamical systems]
 Write down explicitly the co-Kleisli composition of two composable co-Kleisli morphisms of the stream comonad of \Cref{streamcomonad}.
\end{ex}

Note that, just as comonads are monads in the opposite category, co-Kleisli morphisms are Kleisli morphisms in the opposite category. Therefore they form a category too. In the next exercise you can prove this explicitly.

\begin{ex}
 Prove explicitly that co-Kleisli morphisms also form a category, with identities given by the counits, and composition given by the co-Kleisli composition. 
\end{ex}

\begin{deph}
  This category is called the \emph{co-Kleisli category}, and is denoted, similarly to the Kleisli category, by $\cat{C}_C$. 
\end{deph}

\subsection{The co-Kleisli adjunction}

Just as Kleisli morphisms ``include'' ordinary morphisms via the unit, co-Kleisli morphisms also include ordinary morphisms, via the counit. Explicitly, let $(C,\e,\nu)$ be a comonad on $\cat{C}$. Then an ordinary morphism $f:X\to Y$ of $C$ defines canonically a co-Kleisli morphism
$$
\begin{tikzcd}
 CX \ar{r}{\e} & X \ar{r}{f} & Y.
\end{tikzcd}
$$
The interpretation is that if we start with extra information, we first discard it, and then apply $f$. In other words, $f$ does not really need, nor use, the extra information. 

As for the Kleisli case, this assignment is functorial, and defines an adjunction, as we will see. Denote for convenience $f\circ\e$ by $R_C(f)$.

\begin{ex}
 Show that $R_C$ is part of a functor $\cat{C}\to\cat{C}_C$ which is the identity on objects, with $R_C(\id_X)=\e_X$ and $R_C(f\circ g) = R_C(f)\circ_{ck} R_C(g)$. 
\end{ex}

Conversely, suppose we have a co-Kleisli morphism $k:CX\to Y$. This gives canonically an ordinary morphism $CX\to CY$ as
$$
\begin{tikzcd}
 CX \ar{r}{\nu} & CCX \ar{r}{Ck} & CY .
\end{tikzcd}
$$
Denote this composite $Ck\circ\nu$ by $L_C$. 

\begin{eg}[several fields]
 For the reader comonad, this assignment takes a function $k:X\times E \to Y$ and gives a function $L_C(k):X\times E \to Y\times E$ which maps $(x,e)$ to $(k(x,e),e)$, copying the extra information $e$ to the output. 
\end{eg}

\begin{ex}[basic computer science, dynamical systems]
 What does this give for the stream comonad? (Hint: one looks at all the values that $k$ ``would have assumed in the past''.)
\end{ex}

\begin{ex}
 Prove that $L_C$ is part of a functor $\cat{C}_C\to \cat{C}$, which maps an object $X$ to $CX$. In particular, show that $L_C(\e_X)=\id_X$, and that $L_C(k\circ_{ck} h) = L_C(k)\circ L_C(k)$. 
\end{ex}

\begin{ex}
 Prove the dual to \Cref{kleisliadj}, namely that
 \begin{itemize}
  \item $L_C\circ R_C= C$;
  \item $L_C$ is left-adjoint to $R_C$;
  \item The counit of the adjunction is given by the counit $\e$ of the comonad.
 \end{itemize}
\end{ex}

This adjunction is called the \emph{co-Kleisli adjunction}. Note that the role of the left- and of the right-adjoints are reversed compared to \Cref{kleisliadj}.

\section{Comonads as processes on spaces}

\begin{idea} 
	A comonad is a consistent way to construct, from spaces, processes of a specified structure, and give selected strategies or trajectories.
\end{idea}

\begin{eg}[basic computer science, dynamical systems]
 Consider the stream comonad $S$ on $\cat{Set}$ (\Cref{streamcomonad}). We have seen that a possible interpretation is that a stream in $X$ is the ``history'' of a point of $X$, where the first element of the stream is the current state. 
 We can ``reverse the arrow of time'' and interpret instead a stream as the \emph{future} positions of a point of $x$, which the first element given by the current state. This can be thought of as a process of some kind. In fact, $SX$ is even canonically a \emph{dynamical system}, with the map $SX\to SX$ given by shifts, i.e.~mapping the stream
 $$
 \{x_0,x_1,x_2,\dots\}  \quad\mbox{to}\quad \{x_1,x_2,x_3,\dots\} .
 $$
\end{eg}

\begin{eg}[several fields]
 The stream comonad for a different choice of monoid (\Cref{monoidstream}) has the same interpretation. Notice that even if we choose a group, such as $\Z$ or $\R$, which intuitively indexes both the past and the future, we still have to reverse the direction of time in the interpretation. Moreover, by choosing different indexing monoids, such as $\Z^2$, the process can grow with more general shapes.
\end{eg}

The following example can be considered a discrete analogue of the universal covering comonad of \Cref{universalcovering}.
Before starting with the example, let's give a preliminary definition: a \emph{rooted tree} is a directed graph with a distinguished vertex $x$, called the \emph{root}, and such that for every (other) vertex $y$ there is a unique walk from $x$ to $y$. 
Rooted trees are graphs that look for example like this:
$$
\begin{tikzcd}[sep=tiny]
 &&& x \ar{dl} \ar{dr} \\
 && y \ar{dl} \ar{dr} && z \ar{dr} \\
 &a \ar{dl} \ar{dr} && b  && c \\
 p && q
\end{tikzcd}
$$
and \emph{not} like these:
$$
\begin{tikzcd}[sep=tiny]
 && x \ar{dl} \ar{dr} \\
 & y \ar{dr} && z \ar{dl} \\
  && a
\end{tikzcd}
\qquad
\begin{tikzcd}[sep=tiny]
 && x \ar{ddl}  \\
 \\
  & a \ar{rr} && b \ar{uul} 
\end{tikzcd}
\qquad
\begin{tikzcd}
 x \ar[out=60,in=120,loop,distance=1cm] 
\end{tikzcd}
$$
Let's now give our example.
\begin{eg}[graph theory]\label{ucgraphs}
 Let $\cat{MGraph}_*$ be the category of directed multigraphs with a distinguished vertex (``base point''), and multigraph morphisms preserving the base point (and incidence). 
 
 Consider an object $(G,x)$, which is a multigraph $G$ with a distinguished vertex $x$. Denote now by $UG$ the graph constructed as follows.
 \begin{itemize}
  \item Vertices are \emph{walks} in $G$ (chains of head-to-tail edges) starting at $x$, of finite length, including the trivial walk at $x$. That is, a vertex is a tuple of the form
  $$
  (e_1,\dots,e_n)
  $$
  where, denoting source and target by $s$ and $t$, $s(e_1)=x$, and $t(e_i)=s(e_{i+1})$ for all $i=1,\dots,n-1$.
  \item There is a unique edge of $UG$ from $q$ to $q'$ if and only if they are in the form 
  $$
  q=(e_1,\dots,e_n) \quad\mbox{and}\quad q'=(e_1,\dots,e_n,e_{n+1}) ,
  $$
  with $s(e_{n+1})=t(e_n)$,
  i.e.~if $q'$ can be obtained from $q$ by adding a consecutive edge. 
 \end{itemize}
 Denote the trivial walk at $x$ again by $x$, so that we have a pointed multigraph $(UX,x)$. 
 
 This construction is functorial. Indeed, let $f:(G,x)\to (H,y)$ be a morphism, i.e.~an incidence-preserving map with $f(x)=y$. We can construct a map $Uf:UG\to UH$ which maps 
 \begin{itemize}
  \item The trivial walk at $x$ to the trivial walk at $y$;
  \item The walk $(e_1,\dots,e_n)$ to the tuple $(f(e_1),\dots,f(e_n))$, which is a walk in $H$, since $f$ preserves incidence (and so the edges are consecutive in $H$ too), starting at $y$ (since $f(x)=y$). 
 \end{itemize}
 This map preserves the base point by construction, and it also preserves incidence: the unique edge between
 $$
 (e_1,\dots,e_n) \quad\mbox{and}\quad (e_1,\dots,e_n,e_{n+1})
 $$
 is mapped to the unique edge between
 $$
 (f(e_1),\dots,f(e_n)) \quad\mbox{and}\quad (f(e_1),\dots,f(e_n),f(e_{n+1})) .
 $$
 Therefore $Uf$ is a morphism of pointed multigraphs.
 It is easy to check the functoriality axioms, so we have an endofunctor $U:\cat{MGraph}_*\to\cat{MGraph}_*$. 

 Let's now give to $U$ a comonad structure. First of all, the unit $p:(UG,x)\to(G,x)$ is given by the endpoints of walks. That is, we define, on vertices, $p(x)\coloneqq x$, and
 $$
 p(e_1,\dots,e_n) \;\coloneqq\; t(e_n),
 $$
 the target of the last edge $e_n$. On edges, we map the unique edge between
 $$
 (e_1,\dots,e_n) \quad\mbox{and}\quad (e_1,\dots,e_n,e_{n+1})
 $$
 to $e_{n+1}$. This respects incidence (why?), and by construction it preserves the base points, so we have a morphism of pointed multigraphs. Moreover, this is natural (why?), we have a natural transformation $p:U\Rightarrow \id_{\cat{MGraph}_*}$. 
 
 In order to construct the multiplication, first note that $UG$ is always a rooted tree, with root $x$. Indeed, to a vertex $(e_1,\dots,e_n)$ of $UG$, there is always a unique walk from $x$, namely
 \begin{equation}\label{walkofwalks}
 \begin{tikzcd}[sep=small]
  x \ar{r} & (e_1) \ar{r} & (e_1,e_2) \ar{r} &\dots \ar{r} & (e_1,\dots,e_n) ,
 \end{tikzcd}
 \end{equation}
 where all the arrows above denote the unique walks in $UG$.
 This gives a bijection $\nu:UG\to UUG$,
 which is even an isomorphism of multigraphs (why does it preserve incidence?). It is part of a natural isomorphism $\nu:U\Rightarrow UU$ (check naturality!).
 
 To show that $(U,p,\nu)$ is a comonad, let's check that the diagrams \Cref{comonaddiagrams} commute. 
 \begin{itemize}
  \item The left counitality diagram says that given a walk $(e_1,\dots,e_n)$, and applying $\nu$ to form the walk of walks \eqref{walkofwalks},
  then the endpoint of \eqref{walkofwalks} is the original $(e_1,\dots,e_n)$.
  \item The right counitality diagram says that again, given a walk $(e_1,\dots,e_n)$, and applying $\nu$ to form the walk of walks \eqref{walkofwalks}, if we take the tuple formed by \emph{the endpoints of each arrow in \eqref{walkofwalks}}, then we get again $(e_1,\dots,e_n)$.
  \item The coassociativity diagram says that if we repeat the construction \eqref{walkofwalks} twice, we get the same result as if we took the construction \eqref{walkofwalks} for each element of \eqref{walkofwalks} itself. (Can you write down which \emph{walk of walks of walks} you get that way?)
  
  Therefore $(U,p,\nu)$ is an idempotent comonad. We call it the \emph{comonad of rooted trees} (see \Cref{rootedtrees} for why), or also \emph{discrete universal covering comonad} (see \Cref{ucboth} for why). 
 \end{itemize}
\end{eg}

Also in this example the comonad constructs some ``processes'' on our spaces, in this case ``walks''. The counit forgets the process, and just tells us ``where we are''.

\begin{ex}[graph theory]\label{ucugraph}
 Give an analogous construction for \emph{un}directed (multi)graphs. (Hint: you can consider an undirected edge a pair of directed edges in opposite directions.)
\end{ex}

\begin{ex}[graph theory, algebraic topology]\label{ucboth}
 If you have solved \Cref{universalcovering} and \Cref{ucugraph}, establish the connection between the universal covering for graphs, and the one for topological spaces given in \Cref{universalcovering}.
 
 Also the universal covering comonad of topological spaces can be interpreted in terms of ``processes'': the points of the universal cover are again paths from the base point.
\end{ex}

\begin{ex}[graph theory, dynamical systems]
 Can you think of a function or transition process on a graph which is best modelled by a co-Kleisli morphism of the rooted tree comonad?
 (Hint: think of something that ``depends really on the path, not just on the endpoints''.)
\end{ex}

\subsection{Coalgebras of a comonad}

A comonad not only encodes a ``process'' of some kind, but also a setting in which a the process can be started in a ``default'' way, or a ``strategy''.

\begin{deph}\label{defcoalgebra}
	Let $(C,\e,\mu)$ be a comonad on a category $\cat{C}$. A \emph{coalgebra} of $C$, or \emph{over $C$}, or $C$-coalgebra, consists of
	\begin{itemize}
		\item An object $A$ of $\cat{C}$;
		\item A morphism $i:A\to CA$ of $\cat{C}$,
	\end{itemize}
	such that the following diagrams commute, called ``counit'', and ``comultiplication'' or ``coalgebra square'', respectively.
	\begin{equation}\label{coalgebradiagrams}
		\begin{tikzcd}
			A \ar[swap]{dr}{\id} \ar{r}{i} & CA \ar{d}{\e} \\
			& A
		\end{tikzcd}
		\qquad
		\begin{tikzcd}
			A \ar{d}{i} \ar{r}{i} & CA \ar{d}{\nu} \\
			CA \ar{r}{Ci} & CCA
		\end{tikzcd}
	\end{equation}
\end{deph}

The intuition for the map $i$ is that it assigns to each $a\in A$ a distinguished process, which we can think of as being ``started'' or ``triggered'' by $a$. Let's see some examples.

\begin{eg}[several fields]\label{dynsysascoalg}
 The coalgebras of the stream comonad $S$ (\Cref{streamcomonad}) are exactly dynamical systems. Let's see why.
 Plugging in the definition, a coalgebra consists of a set $A$ together with a map $i:A\to SA$, which we can write as
 $$
 a \;\longmapsto\; \{i_0(a), i_1(a), i_2(a), \dots\} ,
 $$ 
 such that the diagrams \eqref{coalgebradiagrams} commute. The counit square says that for each $a\in A$, $\e(i(a))=a$, which means precisely that $i_0(a)=a$. 
 Therefore $i$ is in the form
 \begin{equation}\label{iform}
  a \;\longmapsto\; \{a, i_1(a), i_2(a), \dots\} .
 \end{equation}
 Let's now turn to the comultiplication square. For each $a\in A$, after forming the string \eqref{iform}, we can either applying $\nu$ to form the ``history'' string of (shifted) strings
 \begin{align*}
  \{ &\{ a,i_1(a),i_2(a),\dots \} \\
  &\{ i_1(a),i_2(a),i_3(a),\dots \} \\
  &\{ i_2(a),i_3(a),i_4(a),\dots \}  \\
  & \dots \qquad \qquad \quad \} ,
 \end{align*}
 or we can apply $i$ again to each element of \eqref{iform}, to form the string of strings
 \begin{align*}
  \{ &\{ a,i_1(a),i_2(a),\dots \} \\
  &\{ i_1(a),i_1(i_1(a)),i_1(i_2(a)),\dots \} \\
  &\{ i_2(a),i_2(i_1(a)),i_2(i_2(a)),\dots \}  \\
  & \dots \qquad \qquad \quad \} .
 \end{align*}
 The comultiplication square says that these two strings of strings must agree. This implies that all the corresponding elements must agree, which means that $i_2(a)=i_1(i_1(a))$, and more generally, $i_n(a)={i_1}^n(a)$.
 Therefore $i$ is in the form 
 $$
 a \;\longmapsto\; \{a, f(a), f(f(a)), \dots\} 
 $$
 for some $f:A\to A$. That is, a coalgebra structure gives a dynamical system on $A$.
 
 Conversely, given a function $f:A\to A$, the map assigning the orbits
 $$
 a \;\longmapsto\; \{a, f(a), f(f(a)), \dots\} 
 $$
 gives a coalgebra structure. Therefore the coalgebras of the stream comonad are exactly dynamical systems (indexed by $\N$, that is, in discrete time).
\end{eg}

 We can view this as a ``canonical way to obtain a stream from each element of $A$''. Every element defines a ``path'' or ``orbit''. This phenomenon happens with many other comonads too.

\begin{ex}[dynamical systems]
 Show that if we form the stream comonad with the monoid $\R_{\ge 0}$ instead of $\N$ (as in \Cref{monoidstream}), its coalgebras are the \emph{continuous-time dynamical systems} on a set. 
\end{ex}

In the exercise above, despite having a dynamical system in continuous time, the trajectories are not continuous in any sense, since we are simply working in the category of sets. 

\begin{ex}[dynamical systems, topology; difficult!]
 Construct a comonad on the category $\cat{CHaus}$ of compact Hausdorff spaces and continuous maps, similar to the stream comonad, where as monoid you take the extended half-line $[0,\infty]$ (with $x+\infty=\infty$ for every $x$). 
 As streams, take \emph{continuous} maps $[0,\infty]\to X$. 
 
 Show that the coalgebras of this comonad are the continuous-time dynamical systems in the traditional sense, i.e.~with continuous trajectories.
\end{ex}

\begin{ex}[dynamical systems]
 What are the coalgebras of the stream comonad for different indexing monoids?
\end{ex}

Here is another very important example.

\begin{eg}[graph theory]\label{rootedtrees}
 The coalgebras of the rooted tree comonad $U$ (\Cref{ucgraphs}) are rooted trees (hence the name). 
 Let's see why. Plugging in the definition, a coalgebra consists of a pointed multigraph $(G,x)$ together with a morphism $i:G\to UG$ preserving the base point (and incidence), and making the diagrams \eqref{coalgebradiagrams} commute. The map $i$ has to map a vertex $y$ of $G$ to a vertex $i(y)$ of $UG$, i.e. a walk in $G$ starting at $x$. The counit condition says that $p(i(y))=y$, that is, the walk $i(y)$ has to end at $y$. Therefore, the map $i$, on vertices is in the form 
 $$
 y \;\longmapsto\; 
 \begin{tikzcd}[sep=small]
  (x \ar{r} &\dots \ar{r} & y ) .
 \end{tikzcd}
 $$
 In particular $i(x)=x$, the trivial walk.
 Let now $e$ be an edge of $G$ from $y$ to $z$. In order to be incidence-preserving, the map $i$ has to map $e$ to an edge $i(e)$ of $UG$ from $i(y)$ to $i(z)$. This means necessarily that $i(z)$ is exactly given by the walk which is the composite of the walk $i(y)$ and $e$,
 \begin{align*}
 y \;&\longmapsto\; 
 \begin{tikzcd}[ampersand replacement=\&, sep=small]
  (x \ar{r} \&\dots \ar{r} \& y )   
 \end{tikzcd} \\
   z \;&\longmapsto\; 
 \begin{tikzcd}[ampersand replacement=\&, sep=small]
  (x \ar{r} \&\dots \ar{r} \& y \ar{r}{e} \& z ) .   
 \end{tikzcd}
 \end{align*}
 The counit diagram now says that, on edges, $p(i(e))=e$.
 Let's show that $G$ is a tree, rooted at $x$. Let $y$ be a vertex of $G$. Consider $i(x)=x$ and $i(y)$. Since $UG$ is a tree rooted at $x$, there is a unique walk $q$ from $x$ to $i(y)$. Applying $p$ to this walk $q$ we get a walk in $G$ from $x$ to $y$. Therefore there is a walk from $x$ to $y$. To show that this walk is unique, suppose that we had walks $r$ and $r'$ from $x$ to $y$. 
 Then $i(r)$ and $i(r')$ both would be walks in $UG$ from $x$ to $i(y)$. But since $UG$ is a tree, necessarily $i(r)=i(r')$. Applying $p$ we get 
 $$
 r \;=\; p(i(r)) \;=\; p(i(r')) \;=\; r' .
 $$
 Therefore $G$ is itself a tree rooted at $x$. 
 (We don't even have to look at the comultiplication square.)
 
 Conversely, suppose that $(G,x)$ is a tree rooted at $x$. Then for each vertex $y$ there exists a unique walk from $x$ to $y$, and this gives a morphism $i:G\to UG$ preserving base point and incidence, which by construction satisfies the counit diagram. Let's see that the comultiplication diagram commutes too. Let $y$ be a vertex of $G$, consider the unique walk $i(y)$ from $x$ to $y$, and denote its edges by $e_1,\dots,e_n$. Then applying the comultiplication map $\nu$ to $i(y)$ we get 
 $$
 \begin{tikzcd}[sep=small]
  x \ar{r} & (e_1) \ar{r} & (e_1,e_2) \ar{r} &\dots \ar{r} & (e_1,\dots,e_n) ,
 \end{tikzcd}
 $$
 and if we apply $Ci$ to $i(y)$, i.e.~apply $i$ to each of intermediate points between $x$ and $y$, we get the same result. Therefore the comultiplication diagram commutes on vertices.  
 We don't have to check that the diagram commutes on edges since we are in a rooted tree, and so between any two vertices there is at most one edge. Therefore $G$ with this map $i$ is a coalgebra.
\end{eg}

Again, here, we see that each vertex of the graph defines canonically a walk leading to it (or starting from it, depending on the point of view).

Another way to prove the characterization above is as follows (why?).

\begin{ex}
 Prove the dual statement to \Cref{stmapiso}, namely, that a comonad is idempotent if and only if all its coalgebras structure maps are isomorphisms. 
\end{ex}

By the same line of reasoning, we also have a similar statement for the universal covering of topological spaces.

\begin{eg}[algebraic topology]
 The exercise above implies that the coalgebra of the universal covering comonad (see \Cref{universalcovering}) are the simply connected (and locally contractible) topological spaces. 
\end{eg}

\begin{ex}[several fields]
 Prove that the coalgebras of the reader comonad $C_E$ (\Cref{readercomonad}) are sets $A$ equipped with a function $e:A\to E$. 
 
 We can interpret this function as a ``default value for the extra information'' that every element $a$ on $A$ carries.
\end{ex}

Dually to the case of algebras over monads, coalgebras over comonads have a notion of morphisms between them.

\begin{deph}\label{defcoemcat}
 Let $(A,i)$ and $(B,i)$ be coalgebras of a comonad $C$ on $\cat{C}$. A \emph{morphism of $C$-coalgebras}, or \emph{$C$-morphism}, is a morphism $f:A\to B$ of $\cat{C}$ such that the following diagram commutes. 
 \begin{equation}\label{Cmorphism}
 \begin{tikzcd}
  A \ar{d}{i} \ar{r}{f} & B \ar{d}{i} \\
  CA \ar{r}{Cf} & CB
 \end{tikzcd}
\end{equation}

 The category of $C$-algebras and $C$-morphisms is called the \emph{category of coalgebras of $C$} or, sometimes, \emph{co-Eilenberg-Moore category}, and it is denoted, similarly to the Eilenberg-Moore category, by $\cat{C}^C$. 
\end{deph}

These morphisms can be interpreted as morphisms which ``respect the dynamics of the processes'' or ``respect the default choices or strategies''.

\begin{eg}[dynamical systems]
 Let $(A,f)$ and $(B,g)$ be dynamical systems in the category of sets, or equivalently, coalgebras over the stream comonad. A morphism of coalgebras, plugging in the definition, is a map $m:A\to B$ such that for each $a\in A$, $Cm(i(a))=i(m(a))$. Recalling that the map $i$ is given by the orbits of $a$, this means that
 $$
 \{m(a), m(f(a)), m(f(f(a)),\dots\}
 $$
 has to be equal to 
 $$
 \{m(a), f(m(a)), f(f(m(a))),\dots\}.
 $$
 This is equivalent to say that $f\circ m=m\circ f$, i.e.~that the following diagram commutes.
 $$
 \begin{tikzcd}
  A \ar{r}{f} \ar{d}{m} & A \ar{d}{m} \\
  B \ar{r}{g} & B
 \end{tikzcd}
 $$
 In other words, the morphisms of coalgebras are precisely the morphisms of dynamical systems. 
\end{eg}

\begin{eg}[several fields]
 Let $(A,e)$ and $(B,e)$ be coalgebras of the reader comonad of \Cref{readercomonad}, i.e.~sets equipped with functions $e:A\to E$ and $e:B\to E$. A function $f:A\to B$ is a morphism of coalgebras if and only if the following diagram commutes (why?),
 $$
 \begin{tikzcd}
 A \ar{r}{f} \ar{d}{e} & B \ar{d}{e} \\
 E \ar{r}{\id} & E
 \end{tikzcd}
 $$
 which can be interpreted as the fact that $f$ has to preserve the ``default choice of extra data''.
\end{eg}

\begin{ex}[graph theory]\label{5.4.19}
 What are the coalgebras of the rooted tree comonads? (Hint: the comonad is idempotent.)
\end{ex}

\subsection{The adjunction of coalgebras}

As in the case of the Eilenberg-Moore category, every morphism of coalgebras over a given comonad $C$ is in particular a morphism of the underlying category $\cat{C}$. Therefore there is a fully faithful ``forgetful'' functor $\cat{C}^C\to \cat{C}$. Let's denote this functor by $L^C$. 

Dually to the case of monads, this forgetful functor has a right-adjoint, which is constructed as follows. Let $X$ be an object of $C$. Then $CX$ is canonically a $C$-coalgebra with structure map $\nu$. (Coalgebras of this form are sometimes called ``cofree'', since they are the dual to free algebras.)
Just as well, given $f:X\to Y$, the morphism $Cf:CX\to CY$ is canonically a morphism of coalgebras (why?). 
This construction gives the desired functor $R^C:\cat{C}\to\cat{C}^C$. 

\begin{eg}[dynamical systems]
 We know that coalgebras of the stream comonads are dynamical systems. Given a set $X$, we can canonically form a dynamical system from it, the one given by \emph{streams and their shifts}. That is, we can form the set of streams $SX$, with the structure map given by $\nu$. Recall from \Cref{dynsysascoalg} that the structure map of a coalgebra $A$ of the stream comonad corresponds to a function $f:A\to A$. The structure map $\nu:SX\to SSX$ corresponds to the function $SX\to SX$ given by shifts, 
 $$
 \{x_0,x_1,x_2,\dots\}\;\longmapsto\; \{x_1,x_2,x_2,\dots\} .
 $$
\end{eg}

\begin{eg}
 Consider the reader comonad of \Cref{readercomonad}. Given a set $X$ we can always form a coalgebra by taking the set $X\times E$, with the map $X\times E\to X\times E\times E$ given by copying the ``extra information'' $E$. Recall that, intuitively, a coalgebra of the reader comonad is a set equipped with a default choice for the extra information. Here, the ``default value'' of the extra information is trivially just the information that we already have.  
\end{eg}

\begin{ex}
 Prove that $R^C$ is indeed right-adjoint to $L^C$ (Hint: this is dual to \Cref{emadjprop}.)
\end{ex}

\begin{cor}
 The adjunction above gives the following universal property. Let $(C,\e,\nu)$ be a comonad on $\cat{C}$. Let $X$ be an object of $\cat{C}$, and let $(A,i)$ be a $C$-coalgebra. For each morphism $f:A\to X$ of $\cat{C}$ there exists a unique morphism of coalgebras $(A,i)\to (CX,\nu)$ such that the following diagram commutes.
 $$
 \begin{tikzcd}
 & CX \ar{d}{\e} \\
 A \uni{ur} \ar{r}[swap]{f} & X
 \end{tikzcd}
 $$
\end{cor}

Let's see examples of this universal property. Just as the map for the monad case (\Cref{uniem}) can be interpreted as an \emph{extension} to arbitrary formal expressions (see the examples there), here we can interpret this unique map as a \emph{lifting to the dynamics}.

\begin{eg}[dynamical systems]
 Let $(A,d:A\to A)$ be a dynamical system on the category of sets, and let $X$ be a set. Given a map $g:A\to X$ we can form the morphism of dynamical systems $A\to CX$ given by 
 $$
 a\;\longmapsto\; \{g(a), g(d(a)), g(d(d(a))),\dots\} .
 $$
 This respects the dynamics, since replacing $a$ by $d(a)$ has the same effect as shifting. This map is the unique morphism of dynamical systems $A\to CX$ such that its first component agrees with $g$. This map can be seen as \emph{lifting} $g$ to the dynamics over $X$, or as defining a dynamics on $X$ derived from the dynamics on $A$, with $g$ fixing the initial condition.
\end{eg}

\begin{eg}[graph theory]
 Let $(T,x)$ be a rooted tree, and let $(G,y)$ be any multigraph with a distinguished vertex (``base point''). Given a map $f:T\to G$ preserving incidence and base points, there is a unique map $T\to UG$, preserving incidence and base points, and lifting $f$ to $UG$. This map is constructed as follows. A vertex $v$ of $T$ identifies a unique walk in $T$ from $x$ to $v$, which we had denoted by $i(v)\in UT$. Taking the image of $i(v)$ under $f$ we get a walk in $G$ from $y$ to $f(v)$. This walk can be seen as a vertex of $UG$. Since $T$ and $UG$ are trees, there is only one way of extending this assignment to edges while preserving incidence (how?), this gives the desired map $T\to UG$. 
\end{eg}

\begin{ex}[algebraic topology]
 What's the analogous of the example above for topological spaces and the universal covering comonad?
 
 (Constructions of this kind are known in topology as \emph{homotopy lifting properties}.)
\end{ex}

\section{Adjunctions, monads and comonads}\label{adjmonads}

We have seen that, whenever we have a monad or a comonad, we can obtain an adjunction in two canonical ways. Conversely, whenever we have an adjunction, we can obtain a monad and a comonad canonically.

\begin{thm}
 Let $\cat{C}$ and $\cat{D}$ be categories, and let $F:\cat{C}\to\cat{D}$ and $G:\cat{D}\to\cat{C}$ be adjoint functors, with $F\ladj G$. Denote the unit and the counit by $\eta:\id_\cat{C}\Rightarrow G\circ F$ and $\e:F\circ G\Rightarrow \id_\cat{D}$, respectively. Then
 \begin{enumerate}
  \item\label{monadpart} $G\circ F$ is a monad on $\cat{C}$, with unit $\eta$ and multiplication $G\e F$;
  \item\label{comonadpart} $F\circ G$ is a comonad on $\cat{D}$, with counit $\e$ and comultiplication $F\eta G$.
 \end{enumerate}
\end{thm}

Note that this theorem is a wide generalization of \Cref{galoisclosure}. 

Since the two statements in the theorem are dual to each other, we will only prove one.

\begin{proof}[Proof of \ref{monadpart}]
The specified natural transformations are by construction natural, and already in the desired form. We only have to prove that they satisfy the monad axioms \eqref{monaddiagrams}. Explicitly, we have to show that the following diagrams commute for each object $C$ of $\cat{C}$.
$$
\begin{tikzcd}
 GFC \ar{dr}[swap]{\id} \ar{r}{\eta} & GFGFC \ar{d}{G\e } \\
 & GFC
\end{tikzcd}
\qquad
\begin{tikzcd}
 GFC \ar{dr}[swap]{\id} \ar{r}{GF\eta} & GFGFC \ar{d}{G\e } \\
 & GFC
\end{tikzcd}
\qquad
\begin{tikzcd}
 GFGFGFC \ar{r}{GFG\e } \ar{d}{G\e } & GFGFC \ar{d}{G\e } \\
 GFGFC \ar{r}{G\e } & GFC
\end{tikzcd}
$$
Now
\begin{itemize}
 \item The first diagram commutes since it corresponds exactly to the second triangle identity \eqref{trianglecomp}, for $D=FC$.
 \item The second diagram commutes since it corresponds to the image under $G$ of the first triangle identity \eqref{trianglecomp}.
 \item The last diagram commutes since it is the image under $G$ of a naturality diagram for $\e$. \qedhere
\end{itemize}
\end{proof}

\begin{ex}
Prove \ref{comonadpart} explicitly. 
\end{ex}

We can say more. Not only does every adjunction give rise to a monad and a comonad, but moreover, the adjunction will always lie ``between'' the Kleisli and Eilenberg-Moore adjunctions, in the way made precise and explained below.

\begin{thm}[{\cite[Proposition~5.2.12]{ctcontext}}]\label{comparisonfunctors}
 Let $\cat{C}$ and $\cat{D}$ be categories, and let $F:\cat{C}\to\cat{D}$ and $G:\cat{D}\to\cat{C}$ be adjoint functors, with $F\ladj G$, and unit and counit $\eta$ and $\e$. Denote the induced monad $G\circ F$ by $T$. 
 \begin{enumerate}
  \item\label{klpart} There is a canonical ``comparison'' functor $J$ from the Kleisli category of $T$ to $\cat{D}$, unique up to isomorphism, which makes the following diagrams commute (up to isomorphism).
  \begin{equation}\label{compfunkl}
  \begin{tikzcd}[column sep=small]
   \cat{C}_T \ar{dr}[swap]{R_T} \uni{rr}{J} && \cat{D} \ar{dl}{G} \\
   & \cat{C}
  \end{tikzcd}
  \qquad
  \begin{tikzcd}[column sep=small]
   \cat{C}_T  \uni{rr}{J} && \cat{D}  \\
   & \cat{C} \ar{ul}{L_T} \ar{ur}[swap]{F}
  \end{tikzcd}
  \end{equation}
  \item\label{empart} There is a canonical ``comparison'' functor $K$ from $\cat{D}$ to the Eilenberg-Moore category of $T$, unique up to isomorphism, which makes the following diagrams commute (up to isomorphism).
  \begin{equation}\label{compfunem}
  \begin{tikzcd}[column sep=small]
   \cat{D} \ar{dr}[swap]{G} \uni{rr}{K} && \cat{C}^T \ar{dl}{R^T} \\
   & \cat{C}
  \end{tikzcd}
  \qquad
  \begin{tikzcd}[column sep=small]
   \cat{D}  \uni{rr}{K} && \cat{C}^T  \\
   & \cat{C} \ar{ul}{F} \ar{ur}[swap]{L^T}
  \end{tikzcd}
  \end{equation}
 \end{enumerate}
\end{thm}

The proof, and the exercises that follow, use the basic concepts of adjunctions given in \Cref{adjunctions}.

\begin{proof}[Proof of \ref{klpart}]
 Let's construct the functor $J$ explicitly.
 Recall that the objects of the Kleisli category are just the objects of $\cat{C}$, and $L_T$ is the identity on objects. In order for the diagram on the right of \eqref{compfunkl} to commute, we are then forced to define, on objects, $J(X)\coloneqq F(X)$. 
 Note that this, on objects, makes also the diagram on the left of \eqref{compfunkl} commute: for each object $X$, $G JX=GFX=TX$, and recall that, on objects, $R_T(X)=TX$. 
 
 On morphisms, let $k:X\to TY=GFY$ be a Kleisli morphism of $T$, considered as a morphism of $\cat{C}$. Under the adjunction $F\ladj G$, this corresponds to a morphism $k^\sharp:FX\to FY$ of $\cat{D}$, obtained as usual by $k^\sharp=\e\circ Fk$. Let then $J(k)\coloneqq k^\sharp$.
 With this definition, the diagrams \eqref{compfunkl} commute on morphisms. Indeed, given $k:X\to GFY$, we have that 
 $$
 GJ(k) \;=\; G\e\circ GFk \;=\; G\e\circ Tk \;=\; \mu\circ Tk \;=\; R_T(k).
 $$
 Just as well, given $f:X\to Y$, by the triangle identities,
 $$
 JL_T(k) \;=\; \e\circ F (\eta\circ f) \;=\; \e\circ F \eta\circ Ff \;=\; Ff .
 $$
 
 It remains to be shown that $J$ is indeed a functor. To see that it preserves identities, we have that at each object $X$ of $\cat{C}$, $J(\eta)=\eta^\sharp=\id_{FX}$. For composition, let $k:X\to GFY$ and $h:Y\to GFZ$. We have that, by naturality of $\e$, the following diagram commutes.
 $$
 \begin{tikzcd}
  FGFY \ar{d}{\e} \ar{r}{FGFh} & FGFGZ \ar{d}{\e} \ar{r}{FG\e} & FGFZ \ar{d}{\e} \\
  FY \ar{r}{Fh} & FGFZ \ar{r}{\e} & FZ 
 \end{tikzcd}
 $$
 Therefore,
 \begin{align*}
 J(h\circ_{kl}k) \;&=\; \e\circ F(\mu\circ Th \circ k) \\
 &=\; \e\circ F(G\e  \circ GFh \circ k) \\
 &=\; \e\circ FG\e  \circ FGFh \circ Fk \\
 &=\; \e\circ Fh \circ \e\circ Fk \\
 &=\; J(h) \circ J(k) . \qedhere
 \end{align*}
\end{proof}

\begin{ex}[important!]\label{5.5.4}
 Prove \ref{empart}. (Hint: you can write an analogous proof to the one of \ref{klpart}.)
\end{ex}

\begin{ex}[important!]
 Write down the corresponding statements for comonads, dual to \Cref{comparisonfunctors}. (Hint: be extra careful with the direction of the arrows.)
\end{ex}

\begin{deph}
 Let $\cat{C}$ and $\cat{D}$ be categories, and let $F:\cat{C}\to\cat{D}$ and $G:\cat{D}\to\cat{C}$ be adjoint functors, with $F\ladj G$.
 The adjunction $F\ladj G$ is called \emph{monadic} if and only if the comparison functor $\cat{D}\to\cat{C}^T$ is an equivalence of categories. 
 In that case we also call the right-adjoint $G$ a \emph{monadic functor}.
\end{deph}

In other words, an adjunction is monadic if and only if the category $\cat{D}$ is, up to equivalence, the category of algebras of the induced monad.

\begin{eg}[several fields]
 Most adjunctions which we interpret as ``free-forgetful'' are monadic:
 \begin{itemize}
  \item The adjunction between sets and vector spaces of \Cref{adjgrpset};
  \item The adjunction between sets and groups of \Cref{adjgrpset};
  \item By construction, all the other examples of \Cref{emadj}.
 \end{itemize}
\end{eg}

To see easy examples of non-monadic adjunctions, take the Kleisli adjunctions of most monads -- usually, the Kleisli category and the Eilenberg-Moore category are not equivalent. 
We can say more. By \Cref{comparisonfunctors}, given a monad $T$ on $\cat{C}$, there is a canonical comparison functor $J:\cat{C}_T\to\cat{C}^T$, making the following diagram commute (up to isomorphism).
  \begin{equation}
  \begin{tikzcd}[column sep=small]
   \cat{C}_T \ar{dr}[swap]{R_T} \uni{rr}{J} && \cat{C}^T \ar{dl}{R^T} \\
   & \cat{C}
  \end{tikzcd}
  \qquad
  \begin{tikzcd}[column sep=small]
   \cat{C}_T  \uni{rr}{J} && \cat{C}^T  \\
   & \cat{C} \ar{ul}{L_T} \ar{ur}[swap]{L^T}
  \end{tikzcd}
  \end{equation}
  
\begin{prop}\label{eqkleisli}
 The comparison functor $J:\cat{C}_T\to\cat{C}^T$ given above establishes an equivalence between the Kleisli category $\cat{C}_T$ and the full subcategory of $\cat{C}^T$ whose objects are the \emph{free} algebras.
\end{prop}

\begin{proof}
 First of all, let $X$ be an object of $\cat{C}$ (or equivalently, of $\cat{C}_T$). We have that, instantiating \Cref{comparisonfunctors} (or, its proof), $J(X)=L^T(X)=(TX,\mu)$. Therefore all the objects in the image of $J$ are free algebras. Conversely, every free algebra is in the image of $J$: given a free algebra $(TY,\mu)$, the object $Y$ is such that $J(Y)=(TY,\mu)$.
 
 It remains to be shown that $J$ is fully faithful. In other words, we have to prove that given objects $X$ and $Y$ of $\cat{C}$, $J$ induces a bijection from Kleisli morphisms between $X$ and $Y$ and morphisms of algebras between $TX$ and $TY$. 
 But this is given exactly by the universal property associated to the Eilenberg-Moore adjunction, as given in \Cref{uniem}, by setting $(A,e)=(TY,\mu)$. 
\end{proof}

\begin{eg}[probability]
 We have see in \Cref{stomat} that, given a stochastic map (or a Markov kernel) $k:X\to\mathcal{P}Y$, we can canonically obtain a map $\mathcal{P}X\to\mathcal{P}Y$. We now know that this map is the unique morphism of (free) algebras making the following diagram commute.
 $$
 \begin{tikzcd}
  X \ar{dr}{k} \ar{d}[swap]{\delta} \\
  \mathcal{P}X \uni{r} & \mathcal{P}Y.
 \end{tikzcd}
 $$
 In the finite case, such a map is an affine map, and so it can be represented by a particular matrix, called a \emph{stochastic matrix}. \Cref{eqkleisli} implies that stochastic maps and stochastic matrices encode the same information: there is an equivalence of categories between the category whose morphisms are stochastic maps between sets, and the category whose morphisms are stochastic matrices between simplices.
\end{eg}

Here is an example where the comparison functor given above is an equivalence.
\begin{ex}[linear algebra]
 Prove that the Kleisli adjunction of the vector space monad (\Cref{vectorspacemonad}) \emph{is} monadic. (Hint: in this case, all algebras are free.)
 
 Why is this related to \Cref{matcat}?
\end{ex}

\begin{ex}[topology]
 Is the forgetful functor $\cat{Top}\to\cat{Set}$ of \Cref{adjtopset} monadic? (Hint: what is the induced monad on $\cat{Set}$?)
\end{ex}

A classic result of category theory, Beck's monadicity theorem, gives a necessary and sufficient condition for an adjunction to be monadic. We will not cover Beck's theorem here, we refer the interested reader to \cite[Chapter~5]{ctcontext}. 

We conclude these notes with the following example, once again on the connection between graphs and categories.

\subsection{The adjunction between categories and multigraphs is monadic}\label{catgraphmonadic}

Consider the adjunction between categories and multigraphs of \Cref{adjcatgraph}.
 Let's write explicitly the monad $T=U\circ\cat{P}$ associated to the adjunction. 
 Given a multigraph $G$, the graph $TG$ has the same vertices as $G$, but, as edges, it has the \emph{walks} of $G$, including the trivial walks at each vertex. Given a morphism $f:G\to H$ of multigraphs, $T$ gives a morphism $Tf:TG\to TH$ given by taking images of walks under $f$ (since $f$ preserves incidence, we get again a walk). 
 
 An algebra over $T$ is a graph $A$ together with a map $c:TA\to A$ preserving incidence, which we can think of as mapping a walk to a single edge by \emph{composing the edges of the walk}. The unit condition of \eqref{algebradiagrams} implies first of all that $c$ has to be the identity on vertices (why?). Moreover, it has to map the edges in $TA$ which come from (single) edges to $A$ to the corresponding edge of $A$.
 
 Let's show that every algebra $A$ is canonically a category (necessarily small, why?), with objects and morphisms given by vertices and edges, and composition given indeed by the map $c$. 
 \begin{itemize}
  \item For each vertex $x$ of $A$ we have the trivial walk of $TA$ at $x$, denote it by $1_x$. Define then the identity morphism $\id_x$ as the edge $c(1_x)$ of $A$. 
  \item For each pairs of composable edges $e_1$ and $e_2$, define their composite $e_2\circ e_1$ as the edge $c(e_1,e_2)$. 
  \item Given an edge $e$ from $x$ to $y$, we have
  $$
  e \circ \id_x \;=\; c(1_x,e) \;=\; c(e) \;=\;e
  $$
  (since composing with a trivial walk gives the same edge $e$), and 
  $$
  \id_y\circ e \;=\; c(e,1_y) \;=\; c(e) \;=\; e ,
  $$
  so the composition given by $c$ is unital.
  \item Given three consecutive edges $e,f,g$ we have by the multiplication square that
  $$
  g\circ (f\circ e) \;=\; c(c(e,f),c(g)) \;=\; c(e,f,g) \;=\; c(c(e),c(f,g)) \;=\; (g\circ f)\circ e 
  $$
 \end{itemize}
 (recall that $\mu$ gives the concatenation of walks),
 so the composition given by $c$ is associative.
 Therefore $(A,c)$ is a category.

 Conversely, every (small) category $\cat{C}$ is canonically a $T$-algebra, with the map $c$ given by identities and composition. In detail,
 \begin{itemize}
  \item For a trivial walk $1_X$ at $X$, we define $c(1_X)=\id_X$;
  \item For a one-morphism walk $f$, we define $c(f)=f$;
  \item For two (or more) composable morphisms $f$ and $g$, we define $c(f,g)=g\circ f$. 
 \end{itemize}
 The unit condition for algebras says that the composition of a single morphism gives again that morphism, this is satisfied by construction. The multiplication condition says that if we have a composable walk of composable walks of morphisms 
 $$
 ((f_{1,1},\dots,f_{1,n_1}),\dots,(f_{m,1},\dots,f_{m,n_m}))
 $$
 then we can either first ``flatten the array'' 
 $$
 (f_{1,1},\dots,f_{1,n_1},\dots,f_{m,1},\dots,f_{m,n_m})
 $$
 and take the composite,
 $$
 f_{1,1}\circ\dots\circ f_{1,n_1}\circ\dots\circ f_{m,1}\circ\dots\circ f_{m,n_m}
 $$
 or we can first take the compositions inside the brackets,
 $$
 (f_{1,1}\circ\dots\circ f_{1,n_1},\dots,f_{m,1}\circ\dots\circ f_{m,n_m})
 $$
 and then compose the resulting morphism,
 $$
 (f_{1,1}\circ\dots\circ f_{1,n_1})\circ\dots\circ (f_{m,1}\circ\dots\circ f_{m,n_m})
 $$
 By associativity of composition, the two procedures have the same result. Therefore $\cat{C}$ is a $T$-algebra.
 
 Consider now two (small) categories $\cat{C}$ and $\cat{D}$ -- we know these are equivalently $T$-algebras. Every functor $\cat{C}\to\cat{D}$ preserves incidence on the underlying graphs, and different functors give different morphisms of graphs. A morphism of graphs $f:\cat{C}\to\cat{D}$ is a morphism of $T$-algebras if and only if the following diagram commutes.
 $$
 \begin{tikzcd}
  T\cat{C} \ar{d}{c} \ar{r}{Tf} & T\cat{D} \ar{d}{c} \\
  \cat{C} \ar{r}{f} & \cat{D}
 \end{tikzcd}
 $$
 Now, functors preserve compositions and identities, and so they make the diagram above commute. In detail, for $1_X$ a trivial walk in $TC$ at $X$, we have that, by functoriality,
 $$
 f(c(1_X)) \;=\; f(\id_X) \;=\; \id_{f(C)} \;=\; c(Tf(1_X)).
 $$
 For a walk given by just one morphism $g$, the diagram commutes since both ways just return $g$.
 For a walk given by two (or more) composable morphisms $e_1,\dots,e_n$, we have that, again by functoriality,
 $$
 f(c(e_1,\dots,e_n)) \;=\; f(e_n\circ\dots\circ e_1) \;=\;f(e_n)\circ\dots\circ f(e_1) \;=\; c(f(e_1),\dots,f(e_n)) .
 $$
 Therefore every functor is a morphism of $T$-algebras.
 
 Conversely, if $f$ makes the diagram above commute, then for each object $X$ of $\cat{C}$,
 $$
 f(\id_X) \;=\; f(c(1_X)) \;=\; c(Tf(1_X)) \;=\; c(1_{f(X)}) \;=\; \id_{f(X)} ,
 $$
 and for each composable morphisms $g:X\to Y$ and $h:Y\to Z$, 
 $$
 f(h\circ g) \;=\; f(c(g,h)) \;=\; c(f(g),f(h)) \;=\; f(h) \circ f(g) .
 $$
 Therefore $f$ is a functor. 
 
 In summary, the $T$-algebras are precisely the (small) categories, and their morphisms are precisely the functors between them. In other words, the extra structure that a graph needs to have in order to be a category is exactly encoded by a monad, whose algebras are categories.

\newpage
\chapter*{Conclusion}
\addcontentsline{toc}{chapter}{\currentname}

This is the end!  Hopefully, these notes have lit your interest in categories and their applications, and are enough to get you started on more advanced material. 

Where does it go from here?
\begin{itemize}
 \item If you are interested in learning pure category theory, or in  applying category theory to areas of pure math such as algebraic geometry or topology, the next step is a real category theory textbook. 
 You can look at the book which these notes have become, \cite{startingcats}, which has extra material, for example an entire chapter on monoidal categories.
 You can also use Emily Riehl's recent book \cite{ctcontext}, which largely inspired this course, as well as the classic texts \cite{borceux} and \cite{joycats}. 
 
 \item If you are interested in algebraic topology or homotopy theory, you can read Emily Riehl's notes \cite{riehl-ssets}, and then her book \cite{riehl-cathom}.
 For algebraic topology from the categorical viewpoint, you can look at J.~P.~May's classic~\cite{may-at}. 
 I recommend however that before approaching those you learn a bit more category theory.
 An excellent reference for classical algebraic topology is \cite{hatcher}.
 
 \item If you are interested in learning applied category theory, the next step could be Brendan Fong and David Spivak's recent book \cite{sevensketches}. It has an overview of possible applications of category theory, from which you can choose where to focus in particular. 
 
 \item If you want to learn about string diagrams, the classic paper is Ross Street and André Joyal's \cite{geotens1}. More introductory references are again \cite{sevensketches}, and Dan Marsden's \cite{marsden-string}.
 
 \item If you are interested in applications to physics, especially quantum information theory using string diagrams, check out Peter Selinger's \cite{selinger-graph} and the book of Bob Coecke and Aleks Kissinger \cite{coecke-quantum}.
\end{itemize}
These are only a few references, the category theory literature is quite vast.

There are a lot of category theory resources online as well. Here are some.
\begin{itemize}
 \item The \href{http://ncatlab.org}{nLab (http://ncatlab.org)} is a wiki about category theory, higher category theory, and their applications. You can learn from it, but also contribute yourself. 
 \item The \href{https://nforum.ncatlab.org/}{nForum (https://nforum.ncatlab.org/)} is the forum of the nLab, where category-related questions and discussions are more than welcome.
 \item The \href{https://golem.ph.utexas.edu/category/}{n-category café (https://golem.ph.utexas.edu/category/)} is a blog on mathematics with a focus on category theory. Many recent ideas in category theory and their applications were born in the comment threads of this blog. 
\end{itemize}

\newpage
\bibliographystyle{alpha}
\bibliography{notes}

\end{document}